\documentclass[a4paper,oneside,english,reqno]{amsart}

\usepackage[utf8]{inputenx}
\usepackage{amsthm, amsmath, amssymb, amsfonts, mathrsfs, mathtools, stmaryrd}
\usepackage{babel, textcomp, url, enumerate, latexsym, graphicx, varioref, hyperref, multirow, layout, siunitx, booktabs}
\usepackage{pdflscape}
\usepackage{verbatim}
\usepackage{geometry}
\usepackage{dsfont}
\usepackage{csquotes}

\usepackage[T1]{fontenc}
\usepackage{imakeidx}
\makeindex

\RequirePackage{enumitem}
\setlist[itemize]{font = \upshape, before = \leavevmode}
\setlist[enumerate]{font = \upshape, before = \leavevmode}
\setlist[description]{before = \leavevmode}

\usepackage[usenames]{color}
\usepackage{colortbl}

\usepackage{tikz, babel, rotating, tikz-cd, pigpen}
\usetikzlibrary{matrix, arrows, decorations.pathmorphing}
\RequirePackage{zref-clever}
\zcsetup{
    abbrev=true, 
    }

\newcommand{\cref}[1]{\zcref{#1}}
\newcommand{\Cref}[1]{\zcref[S]{#1}}

\ExplSyntaxOn
\newcommand{\zcrefglobalstringname}[2]{
  \__zrefclever_opt_varname_lang_type:nnnn{\languagename}{#1}{#2}{tl}
}
\newcommand{\zcreflocalstringname}[2]{
  \__zrefclever_opt_varname_type:een{#1}{#2}{tl}
}
\newcommand{\zcrefgetstring}[2]{
  \__zrefclever_provide_langfile:e { \languagename }
  \__zrefclever_process_language_settings:
  \ifcsvoid{\zcreflocalstringname{#1}{#2}}
    {\csname \zcrefglobalstringname{#1}{#2} \endcsname}
    {\csname \zcreflocalstringname{#1}{#2} \endcsname}
}
\ExplSyntaxOff

\NewDocumentCommand{\newzctheorem}{mO{#1}m}{
  \newtheorem{#1}[sharedtheoremcounter]{#3}
    \AddToHook{env/#1/begin}{%
      \zcsetup{countertype={sharedtheoremcounter=#1}}}
}
\RequirePackage[color       = blue!20,
                bordercolor = black,
                textsize    = tiny,
                figwidth    = 0.99\linewidth]{todonotes}
\usepackage[all]{xy}
\usepackage{microtype}

\RequirePackage[backend    = biber,
                sortcites  = true,
                giveninits = true,
                doi        = false,
                isbn       = false,
                url        = false,
                maxnames   = 6,
                style      = numeric]{biblatex}
\DeclareFieldFormat*{title}{#1}
\renewbibmacro{in:}{}

\usepackage[pagewise,mathlines]{lineno}

\numberwithin{equation}{section}

\newcommand{\xrightleftarrows}[1]{\mathrel{\substack{\xrightarrow{\rule{#1}{0pt}} \\[-0.7ex] \xleftarrow{\rule{#1}{0pt}}}}}

\DeclarePairedDelimiter{\p}{\lparen}{\rparen}

\DeclarePairedDelimiter{\ip}{\langle}{\rangle}

\theoremstyle{plain}
\newtheorem{theorem}{Theorem}[section]
\newtheorem{lemma}[theorem]{Lemma}

\newtheorem{prop}[theorem]{Proposition}
\newtheorem{conjecture}[theorem]{Conjecture}
\newtheorem{corollary}[theorem]{Corollary}

\theoremstyle{definition}
\newtheorem{definition}[theorem]{Definition}

\newtheorem{example}[theorem]{Example}

\theoremstyle{remark}
\newtheorem{remark}[theorem]{Remark}
\newtheorem{notation}[theorem]{Notation}

\newcommand{\wt}{\widetilde}
\newcommand{\wh}{\widehat}

\newcommand{\ul}{\underline}
\newcommand{\defeq}{\vcentcolon=}

\newcommand{\gp}{\mathrm{gp}}

\newcommand{\Nis}{\mathrm{Nis}}

\newcommand{\Zar}{\mathrm{Zar}}

\newcommand{\op}{\mathrm{op}}

\newcommand{\fr}{\mathrm{fr}}
\newcommand{\Et}{\mathrm{\acute{E}t}}

\newcommand{\A}{\mathbb{A}}

\newcommand{\G}{\mathbb{G}}
\newcommand{\HH}{\mathbf{H}}
\newcommand{\HHtriv}{\mathbf{H}_{s}}

\newcommand{\HHtrivptd}{\HH_{s,\bullet}}
\newcommand{\HHtrivrigid}{\HH_{\rigid}}

\newcommand{\PP}{\mathbb{P}}

\newcommand{\Z}{\mathbb{Z}}

\newcommand{\calC}{\mathcal{C}}

\newcommand{\calE}{\mathcal{E}}
\newcommand{\calF}{\mathcal{F}}
\newcommand{\calG}{\mathcal{G}}

\newcommand{\calL}{\mathcal{L}}
\newcommand{\calM}{\mathcal{M}}

\newcommand{\calO}{\mathcal{O}}

\newcommand{\calS}{\mathcal{S}}

\newcommand{\calX}{\mathcal{X}}

\newcommand{\DM}{\mathbf{DM}}

\newcommand{\Ab}{\mathrm{Ab}}

\newcommand{\Sch}{\mathrm{Sch}}

\newcommand{\Sm}{\mathrm{Sm}}
\newcommand{\EssSm}{\mathrm{EssSm}}

\newcommand{\Lrep}{\mathcal L}
\newcommand{\Rep}{\mathscr L}

\newcommand{\Loc}{L}

\newcommand{\SH}{\mathbf{SH}}

\newcommand{\Ho}{\mathrm{Ho}}
\newcommand{\Spc}{\mathbf{Spc}}
\newcommand{\Spt}{\mathbf{Spt}}
\newcommand{\Pre}{\mathbf{Pre}}

\newcommand{\Set}{\mathrm{Set}}
\newcommand{\sSet}{s\mathrm{Set}}
\newcommand{\Shv}{\mathrm{Shv}}
\newcommand{\Shvbf}{\mathbf{Shv}}

\newcommand{\tr}{\mathrm{tr}}

\newcommand{\Gm}{\G_m}

\newcommand{\red}{\mathrm{red}}
\newcommand{\can}{\mathrm{can}}

\newcommand{\id}{\mathrm{id}}

\newcommand{\Corr}{\mathrm{Corr}}

\newcommand{\nis}{\mathrm{Nis}}
\newcommand{\triv}{\mathrm{triv}}
\newcommand{\mot}{\mathrm{mot}}
\newcommand{\Irig}{\mathrm{R}_\rig}

\DeclareMathOperator{\hofib}{hofib}
\DeclareMathOperator{\hocofib}{hocofib}
\DeclareMathOperator{\hocoeq}{hocoeq}

\DeclareMathOperator{\GL}{GL}

\DeclareMathOperator{\SL}{SL}

\DeclareMathOperator{\Fun}{Fun}

\DeclareMathOperator{\coker}{coker}
\DeclareMathOperator{\Spec}{Spec}
\DeclareMathOperator{\pr}{pr}

\DeclareMathOperator{\Supp}{Supp}

\DeclareMathOperator{\codim}{codim}

\DeclareMathOperator{\Map}{Map}

\DeclareMathOperator{\Cor}{Cor}

\DeclareMathOperator{\varlim}{\displaystyle{\lim_{\longleftarrow}}}

\DeclareMathOperator{\colim}{colim}
\DeclareMathOperator{\holim}{holim}
\DeclareMathOperator{\hocolim}{hocolim}

\DeclareMathOperator{\Cyl}{Tel}

\DeclareMathOperator{\cofib}{cofib}

\newcommand{\Fr}{\mathrm{Fr}}
\newcommand{\ovFr}{\overline{\mathrm{Fr}}}
\newcommand{\ZFr}{\mathbb Z\Fr}

\newcommand{\SheafHom}{\mathscr{H}\kern-3pt om}
\SelectTips{cm}{10}

\address{Andrei Druzhinin, Chebyshev Laboratory, 
St. Petersburg State University \&
St. Petersburg Department of Steklov Mathematical Institute 
of Russian Academy of Sciences, Russia}
\email{\href{mailto:andrei.druzh@gmail.com}{andrei.druzh@gmail.com}}

\address{H{\aa}kon Kolderup, Western Norway University of Applied Sciences, Norway}
\email{\href{mailto:hakon.kolderup@hotmail.com}{hakon.kolderup@hotmail.com}}

\address{Paul Arne {\O}stv{\ae}r, Department of Mathematics Federigo Enriques, University of Milan, Italy \& 
Department of Mathematics, University of Oslo, Norway}
\email{\href{mailto:paul.oestvaer@unimi.it}{paul.oestvaer@unimi.it} \& 
\href{mailto:paularne@math.uio.no}{paularne@math.uio.no}}

\newcommand{\AffSch}{\mathrm{AffSch}}
\newcommand{\Aff}{\mathrm{Aff}}

\newcommand{\Smat}{\Sm^{\mathrm{cci}}}

\newcommand{\Schcci}{\Sch^\cci}
\newcommand{\SmAff}{\mathrm{SmAff}}

\newcommand{\aff}{\mathrm{aff}}
\newcommand{\cci}{\mathrm{cci}}
\newcommand{\EssSmB}{\mathrm{EssSm}_B}

\newcommand{\SmB}{\Sm_{B}}
\newcommand{\SmBcZ}{\Sm_{B,Z}}
\newcommand{\EssSmBcZ}{\EssSm_{B,Z}}
\newcommand{\SmBlZ}{\Sm_{B\ast Z}}
\newcommand{\SmBmZ}{\Sm_{B-Z}}
\newcommand{\SmZ}{\Sm_Z}
\newcommand{\SmatB}{\Smat_{B}}
\newcommand{\SmatBlZ}{\Smat_{B\ast Z}}
\newcommand{\SmatBcZ}{\Smat_{B,Z}}
\newcommand{\SmatZ}{\Smat_{Z}}
\newcommand{\SmAffSlZ}{\SmAff_{B\ast Z}}
\newcommand{\SmAffScZ}{\SmAff_{B,Z}}
\newcommand{\SmAffSmZ}{\SmAff_{B-Z}}
\newcommand{\SmAffBlZ}{\SmAff_{B\ast Z}}
\newcommand{\SmAffBcZ}{\SmAff_{B,Z}}

\newcommand{\tf}{\mathrm{tf}}

\newcommand{\tif}{\widetilde i^!}
\newcommand{\tids}{\widetilde i_*}

\newcommand{\ids}{i_*}
\newcommand{\iuf}{i^!}
\newcommand{\jus}{j^*}
\newcommand{\jds}{j_*}
\newcommand{\uus}{u^*}
\newcommand{\uds}{u_*}
\newcommand{\nothook}{\not\hookrightarrow}
\newcommand{\nothookleft}{\reflectbox{$\nothook$}}

\newcommand{\Sh}{\mathrm{Shv}}

\newcommand{\SHstf}{\SH_{s,\tf}}

\newcommand{\SHI}{strict $\A^{1}$-invariance }

\newcommand{\XhZ}{X^h_Z}

\newcommand{\XtZ}{X\times_B Z}

\newcommand{\wX}{\widetilde{X}}

\newcommand{\PreA}{\mathbf{H}_{\A^{1}}}
\newcommand{\rigid}{\mathrm{rig}}

\newcommand{\PreRig}{\mathbf{H}_{\rigid}}
\newcommand{\rig}{\mathrm{rig}}

\newcommand{\SHRig}{\mathbf{SH}_{s,\rigid}}

\newcommand{\EM}{\mathrm{EM}} 

\newcommand{\Coker}{\operatorname{coker}}

\newcommand{\ZF}{\mathbb Z \mathrm F}
\newcommand{\ovZF}{\overline{\ZF}}

\newcommand{\ovX}{\overline{X}}

\newcommand{\tildeRep}{\tilde\Rep}

\newcommand{\ovC}{\overline{C}}
\newcommand{\ovW}{\overline{W}}

\newcommand{\hocolimABZ}{\hocolim^{\Delta_{B,Z}}}
\newcommand{\hocolimAZ}{\hocolim^{\Delta_Z}}

\newcommand{\catS}{\calS_{B,Z}}

\newcommand{\Corrfr}{\mathrm{Corr}^\mathrm{fr}}

\newcommand{\smot}{\mathrm{smot}}
\newcommand{\tfsmot}{{\tf\text{-}\smot}}
\newcommand{\nissmot}{{\smot}}

\newcommand{\Id}{\mathrm{Id}}

\newcommand{\fib}{\operatorname{fib}}

\newcommand{\overarrow}[1]{\overline{#1}}

\newcommand{\PSpt}{\mathrm{PSpt}}

\newcommand{\SpcAtf}{\Spc^{\A^1}_\tf}

\newcommand{\catSH}{\mathrm{SH}}
\newcommand{\catLrep}{\mathcal L}
\newcommand{\catPre}{\mathrm{Pre}}
\newcommand{\catSpt}{\mathrm{Spt}}

\newcommand{\catSpc}{\mathrm{Spc}}
\newcommand{\cSpc}{\catSpc}

\newcommand{\cPre}{\catPre}
\newcommand{\cPretf}{\cPre_\tf}
\newcommand{\cPrenis}{\cPre_\nis}
\newcommand{\cPreA}{\cPre_{\A^{1}}} 
\newcommand{\cPreAtf}{\cPre_{\A^1,\tf}}
\newcommand{\cPreAnis}{\cPre_{\A^1,\nis}}
\newcommand{\cPrefr}{\cPre^\fr}
\newcommand{\cPrefrtf}{\cPre^\fr_\tf}
\newcommand{\cPrefrnis}{\cPre^\fr_\nis}
\newcommand{\cPrefrA}{\cPre^\fr_{\A^{1}}}
\newcommand{\cPrefrAtf}{\cPre^\fr_{\A^1,\tf}}
\newcommand{\cPrefrAnis}{\cPre^\fr_{\A^1,\nis}}

\newcommand{\cHAnis}{\catPre_{\A^1,\nis}}

\newcommand{\cSHAnis}{\catSH_{\A^1,\nis}}
\newcommand{\cSHsAtf}{\catSH^{s}_{\A^1,\tf}}
\newcommand{\cSHsAnis}{\catSH^{s}_{\A^1,\nis}}
\newcommand{\cSHstAtf}{\catSH^{\sct}_{\A^1,\tf}}
\newcommand{\cSHstAnis}{\catSH^{\sct}_{\A^1,\nis}}

\newcommand{\swt}{{s\wedge t}}
\newcommand{\sct}{{s, t}}

\newcommand{\sS}{s}
\newcommand{\cSpts}{\catSpt^{\sS}}

\newcommand{\cSptsfr}{\catSpt^{\sS,\fr}}
\newcommand{\cSptsfrtf}{\catSpt^{\sS,\fr}_\tf}
\newcommand{\cSptsfrnis}{\catSpt^{\sS,\fr}_\nis}
\newcommand{\cSptsfrA}{\catSpt^{\sS,\fr}_{\A^{1}}}
\newcommand{\cSptsfrAtf}{\catSpt^{\sS,\fr}_{\A^1,\tf}}

\newcommand{\cSHsfrAtf}{\catSH^{\sS,\fr}_{\A^1,\tf}}
\newcommand{\cSHsfrAnis}{\catSH^{\sS,\fr}_{\A^1,\nis}}
\newcommand{\cSptst}{\catSpt^{\sct}}

\newcommand{\cSptstAnis}{\catSpt^{\sct}_{\A^1,\nis}}

\newcommand{\cPSpts}{\PSpt^{s}}

\newcommand{\cPSptst}{\PSpt^{\sct}}

\newcommand{\cPSptswt}{\PSpt^{\swt}}

\newcommand{\cSptstfrtf}{\catSpt^{\sct,\fr}_\tf}

\newcommand{\cSptstfrAnis}{\catSpt^{\sct,\fr}_{\A^1,\nis}}

\newcommand{\cSHstfrAtf}{\catSH^{\sct,\fr}_{\A^1,\tf}}
\newcommand{\cSHstfrAnis}{\catSH^{\sct,\fr}_{\A^1,\nis}}
\newcommand{\cPSptsfr}{\PSpt^{s,\fr}}
\newcommand{\cPSptsfrtf}{\PSpt^{s,\fr}_\tf}

\newcommand{\cPSptstfr}{\PSpt^{\sct,\fr}}
\newcommand{\cPSptstfrtf}{\PSpt^{\sct,\fr}_\tf}

\newcommand{\cSptswt}{\catSpt^{\swt}}

\newcommand{\cSptswtAnis}{\catSpt^{\swt}_{\A^1,\nis}}

\newcommand{\cSptP}{\catSpt^{\PP^1}}
\newcommand{\cSptPAnis}{\catSpt^{\PP^1}_{\A^1,\nis}}

\makeatletter
\@namedef{subjclassname@2020}{%
  \textup{2020} Mathematics Subject Classification}
\makeatother

\newcommand{\OmegaSigma}{\Theta}

\newcommand{\bbZ}{\mathbb Z}

\newcommand{\PSptt}{\PSpt^{t}}
\newcommand{\Xlocx}{X^\mathrm{loc}_{x}}

\newcommand{\SptstfrAtf}{\mathrm{Spt}^{s,t,\fr}_{\A^{1},\tf}}
\newcommand{\SptstfrANis}{\mathrm{Spt}^{s,t,\fr}_{\A^{1},\nis}}
\newcommand{\Sptst}{\mathrm{Spt}^{s,t}}
\newcommand{\SptstNis}{\mathrm{Spt}^{s,t}_{\nis}}
\newcommand{\OmegaSigmatffrAnis}{\OmegaSigma_{\A^{1},\nis}}
\newcommand{\OmegaSigmatffrAtf}{\OmegaSigma_{\A^{1},\tf}}

\begin{document}

\title[The trivial fiber topology and framed motives over the integers]
{The trivial fiber topology\\ and framed motives over the integers}
\author{Andrei Druzhinin, Håkon Kolderup and Paul Arne Østvær}

\subjclass[2020]{14F35, 14F42, 19E15, 55P99}
\keywords{
Motivic homotopy theory, 
framed correspondences, 
trivial fiber topology, 
strict $\A^{1}$-invariance, 
$\A^{1}$-connectivity for Nisnevich sheaves of stable motivic homotopy groups}

\begin{abstract}
This paper introduces the trivial fiber topology on schemes. 
For one-dimensional base schemes, 
we use it to describe fibrant replacements in the stable motivic homotopy category and motivic infinite loop spaces. 
We also extend the Garkusha-Panin and Voevodsky strict $\A^{1}$-invariance theorems to one-dimensional base schemes. 
The trivial fiber topology plays a central role in the proof of refined localization results for motivic 
homotopy categories.
Moreover, 
we extend Morel's $\A^{1}$-connectivity theorem on Nisnevich sheaves of stable motivic homotopy groups.
These results open new vistas for computations of motivic invariants over deeper base schemes of arithmetic interest.
\end{abstract}

\maketitle

\tableofcontents
\addtocontents{toc}{\protect\setcounter{tocdepth}{1}}

\newpage

\section{Introduction}
\label{section:overview}

\subsection{Background}
\label{subsection:overview}
Voevodsky initiated the theory of framed motives in \cite{VVfc}, 
to answer a profound computational problem regarding the stable motivic homotopy category $\SH(k)$ over a field $k$.
Freely translated from Russian, his handwritten text reads:
\begin{displayquote}
The construction of the motivic category $\mathbf{DM}(k)$ in terms of finite correspondences gives the formula
\[
\mathrm{Hom}_{\mathbf{DM}(k)}(M(X),M(Y)[i])\cong {\mathbb H}_{\mathrm{Zar}}^i(X,C^*\mathbb{Z}_{\mathrm{tr}}(Y)).
\]
This works because the theories we are interested in have transfers. 
However, 
it does not work for more general theories like $K$-theory and Witt-theory.
Fabien Morel and I constructed $\SH(k)$ for this purpose. 
However, 
this category has an unsatisfactory description as there is no formula for the corresponding Hom-groups
\begin{equation}
\label{eq:HomsetSH(k)SigmaPXSimgaPY}
\mathrm{Hom}_{\SH(k)}(\Sigma^{\infty}_{\PP^1}X_{+},\Sigma^{\infty}_{\PP^1}Y_{+}[i]).
\end{equation}
The theory of framed correspondences aims to describe $\SH(k)$ in a similar way as $\mathbf{DM}(k)$ 
can be described by finite correspondences.
\end{displayquote}
\vspace{0.1in}

Fast-forward 15 years, 
the theory of framed motives has become one of the pillars for motivic homotopy theory \cite{Levine-app}.
Many foundational questions remain to be answered,  
computations remain difficult but more accessible with new methods, 
and numerous outstanding conjectures sharpen the focus of this expansion. 
To a smooth and separated $k$-scheme of finite type $X\in\Sm_k$ one associates the motivic suspension spectrum 
$\Sigma^{\infty}_{\PP^1}X_{+}\in\SH(k)$ 
--- 
for early accounts of motivic homotopy theory, 
see \cite{Voe-motives,Cisinski-Deglise,Morel-Voevodsky,morel-trieste,Jardine-spt,mot-functors,Nordfjordeid,AyoubI}. 
The motivic suspension spectrum contains in-depth cohomological information about $X$ pertaining to algebraic cycles.
The question of finding a model for $\Sigma^{\infty}_{\PP^1}X_{+}$ amenable to computations lies 
at the core of the subject.
\vspace{0.1in}

To set the notation,
let $B$ be a qcqs base scheme.
We write $\Delta^{\bullet}_{B}$ for the standard co-simplicial $B$-scheme and $\Fr$ for 
Voevodsky's geometrically defined framed correspondences.
For any smooth simplicial scheme $\calX$ over $B$, 
we discuss the equivalence explaining the fundamental importance of Voevodsky's framed correspondences
\begin{equation}
\label{eq:Xsuspension}
\Sigma^{\infty}_{\PP^1}{\mathcal X}_{+}
\simeq
\{\Rep_{\nis}\Fr(-\times\Delta^{\bullet}_{B},{\mathcal X}_{+}\wedge S^n\wedge \Gm^{\wedge n})^{\gp}\}_{n\geq 0}
\in 
\SH(B).
\end{equation}
Here, 
$\gp$ is shorthand for group completion and $\Rep_{\nis}$ denotes the Nisnevich fibrant replacement.
We note that $\Fr(-\times\Delta^{\bullet}_{B},{\mathcal X}_{+}\wedge S^n\wedge \Gm^{\wedge n})$ is a 
group-like motivic space for every $n\geq 1$ by \cite[Theorem 6.5]{Framed}, 
\cite[Appendix B]{SmAffOpPairs},
and for any $B$ there are 
structure maps 
\begin{equation}
\label{eq:structureP1spectrummorphisms}
\Rep_{\nis}\Fr(-\times\Delta^{\bullet}_{B},{\mathcal X}_{+}\wedge S^n\wedge \Gm^{\wedge n})^{\gp}
\rightarrow
\Omega_{\PP^{1}}
\Rep_{\nis}\Fr(-\times\Delta^{\bullet}_{B},{\mathcal X}_{+}\wedge S^{n+1}\wedge \Gm^{\wedge n+1})^{\gp}.
\end{equation}
The latter morphisms assemble into a $\PP^1$-spectrum, 
which is not necessarily an $\Omega_{\PP^1}$-spectrum   
(in the $\infty$-categorical setting we can view it as a prespectrum of Nisnevich sheaves 
--- see \Cref{section:StablelocalizationDefinitionsNotations}). 
In particular, the right-hand object and the morphism in \eqref{eq:Xsuspension} 
are well defined in $\SH(B)$ for any qcqs scheme $B$.
The proof of the equivalence \eqref{eq:Xsuspension} in \cite[Theorem 11.4]{Framed} for infinite perfect fields 
applies more generally to all fields by appeal to the cone theorems in \cite{ConeTheGNP,DrKyllfinFrpi00}.
Due to the motivic equivalence between $\Fr$ and tangentially framed correspondences, 
see \cite[Corollary 2.3.25]{five-authors}, 
\eqref{eq:Xsuspension} is equivalent to the $\infty$-categorical reconstruction theorem for the stable 
motivic homotopy category
\begin{equation}
\label{eq:SH(B)eqSHtfr(B)}
\SH(B)\simeq \SH^\fr(B)
\end{equation}
proven for general base schemes in \cite{Hoyois-framed-loc} by localization techniques and reconstruction 
for fields (see \cite[Theorem 3.5.12]{five-authors} for perfect fields).
Thus, \eqref{eq:Xsuspension} holds over any base scheme $B$.

To gain computational leverage from the equivalence \eqref{eq:Xsuspension}, 
one invokes \cite[Theorem 11.7]{Framed}:
if the base scheme is an infinite perfect field, 
then the structure maps \eqref{eq:structureP1spectrummorphisms} are schemewise simplicial weak equivalences of 
$\A^1$-invariant Nisnevich sheaves.
That is, 
the right side of \eqref{eq:Xsuspension} is an $\Omega_{\PP^1}$-spectrum in $\SH(k)$, 
which directly implies Morel's $\PP^{1}$-stable connectivity theorem \cite{morel-trieste,Morel-connectivity}
(see also \Cref{subsection:A1connectivity} for a review).
\vspace{0.1in}

Let $k$ be a perfect field.
By \cite[Theorem 1.3, Theorem 11.7(2)]{Framed} or \cite[Corollary 3.5.15]{five-authors}, 
the composite functor
\begin{equation}
\label{eq:XtoLnisFrDeltaX}
\Sm_k
\xrightarrow{\Sigma^{\infty}_{\PP^1}}
\SH(k)
\xrightarrow{\Omega^\infty_{\PP^1}} 
\Spc_*(k) 
\end{equation}
taking values in the homotopy category of pointed simplicial presheaves is given by 
\begin{equation}
\label{eq:XmapstoLnisFrDeltakX+}
\begin{array}{lcl}
Y
\mapsto
\Rep_{\nis}\Fr(-\times\Delta^{\bullet}_{k}, Y_{+})^{\gp},
\end{array}
\end{equation}
see \eqref{eq:notation:LLrepRep} for the definition of $\Rep_{\nis}$.
As a consequence, 
this answers Voevodsky's question about \eqref{eq:HomsetSH(k)SigmaPXSimgaPY} on account of the isomorphism
\begin{equation} 
\label{eq:pinot}
\mathrm{Hom}_{\SH(k)}(\Sigma^{\infty}_{\PP^1}X_{+},\Sigma^{\infty}_{\PP^1}Y_{+}[i]) 
\cong \pi_{n}\Rep_{\nis}\Fr(X\times\Delta^{\bullet}_{k},Y_{+}\wedge S^{n+i})^{\gp}
\end{equation}
for any integer $n\in\mathbb{Z}$ such that $n,n+i\geq 0$.
Furthermore, 
the explicit description of \eqref{eq:XtoLnisFrDeltaX} given by \eqref{eq:XmapstoLnisFrDeltakX+} 
is pivotal in the recognition theorems for 
very effective and effective motivic spectra ($\SH^\mathrm{veff}(k)$ resp.~$\SH^\mathrm{eff}(k)$),
see \cite{motivictwisted}, \cite{voevodsky.open}.
We refer to \cite{Framed,five-authors,BigFrmotives,FramedGamma} for precise statements involving 
framed motives, motivic infinite loop spaces and motivic $\Gamma$-spaces.
Moreover, 
the main result of Garkusha-Panin in \cite{BigFrmotives} reconstructs the stable motivic homotopy 
category of $k$ in terms of Nisnevich local framed $(s,t)$-bispectra
\begin{equation} 
\label{eq:gpreconstruction}
\SH(k)
\simeq 
SH^{\fr}_{\Nis}(k).
\end{equation}
We note that
$SH^{\fr}_{\Nis}(k)$
is defined by framed correspondences and Nisnevich local equivalences. 
Its construction does not involve any $\A^1$-localization functors.
\vspace{0.1in}

A verbatim generalization of the schemewise equivalence in 
\eqref{eq:structureP1spectrummorphisms} and the comparison between 
\eqref{eq:XtoLnisFrDeltaX} and 
\eqref{eq:XmapstoLnisFrDeltakX+} 
does not hold over general base schemes.
Indeed, Ayoub's counterexample to the $\PP^{1}$-stable connectivity theorem for a normal surface 
\cite{AyoubcontrexempleA1connexite} shows that dimension is an important parameter in these considerations.
In \Cref{sect:counterexample},
we give counterexamples over every positive dimensional base scheme $B$.
To prove the correct generalization, our guiding idea is to decompose $\SH(B)$ into two parts;
on the easy part, the theory of framed motives works as over a field, 
while understanding the more complicated part requires refined techniques. 
We achieve the decomposition using the trivial fiber topology introduced in \Cref{section:ttft}.
\vspace{0.1in}

Framed motives yield an essential computational result for the motivic sphere $\mathbf 1$ over $k$
initially proven by Morel \cite{Morel-sphere-spt} using different techniques.
That is, 
\cite{Nesh-FrKMW} (for fields of characteristic zero) and 
\cite{DrKyllfinFrpi00} (for perfect fields)
use \eqref{eq:pinot} to identify the graded endomorphism ring of 
$\mathbf 1$ with Milnor-Witt $K$-theory 
\begin{equation}
\label{equation:KMW1}
\bigoplus_{n\in\Z}
\pi_{n,n}(\mathbf 1)
\cong
\mathrm{K}^\mathrm{MW}_{\ast}(k).
\end{equation}
It is worth pointing out that both the comparison between  
\eqref{eq:XtoLnisFrDeltaX} and \eqref{eq:XmapstoLnisFrDeltakX+}, 
and also the isomorphism \eqref{eq:pinot}
hold over arbitrary fields due to the work of the first author in \cite{nonperfect-SHI}.
As a referee pointed out,
the general form of \eqref{equation:KMW1} follows from the case of perfect fields 
using connectivity of the unit map to very effective hermitian $K$-theory \cite{zbMATH07179287}.
\vspace{0.1in}

\Cref{subsection:smfr,subsection:A1connectivity,subsection:sA1is} review some of the fundamental theorems for 
framed motives over fields. 
This sets the stage for our current program on framed motives over base schemes such as the integers.
With a view to future applications, 
we note some guiding problems that have become approachable but are beyond the scope of this work 
and will be pursued elsewhere:
\vspace{0.05in}
\begin{itemize}
\item 
Show that the right-hand side of the equivalence \eqref{eq:Xsuspension} is an $\Omega_{\PP^1}$-spectrum in $\SH(B)$.
Deduce a stable motivic connectivity theorem for general base schemes.
\vspace{0.05in}
\item Give a local reconstruction of $\SH(B)$ as in \eqref{eq:gpreconstruction}.
\vspace{0.05in}
\item Give a Milnor-Witt $K$-theory type of presentation for the lowest degree nontrivial 
stable motivic homotopy groups of $B$ analogous to \eqref{equation:KMW1}.
\vspace{0.05in}
\item Set up a motivic infinite loop space machine to reconstruct $\SH^{\mathrm{veff}}(B)$ and $\SH^{\mathrm{eff}}(B)$
--- see \cite{zbMATH07303324} for a review of the computational significance of these subcategories of $\SH(B)$.
\end{itemize}

\subsection{Stable motivic fibrant replacements}
\label{subsection:smfr}
A fundamental result in Garkusha-Panin's theory of framed motives \cite{Framed} describes an explicit 
fibrant replacement functor on motivic spectra
(the extension of their result to all fields follows from \cite{nonperfect-SHI}).

\begin{theorem}[\protect{\cite[Theorem 11.7]{Framed}, \cite{DrKyllfinFrpi00}}]
\label{citeth:GP14MotivcReplacement}
Let $k$ be a field and $X\in\Sm_k$. 
There is a canonical stable motivic equivalence
\begin{equation}
\label{eq:SigmaCorSGXLnisLAFrSigmaTX2}
\Sigma_{s,t}^\infty X_+
\xrightarrow{\simeq}
\Rep_\nis\Rep_{\A^{1}}\Fr(-,\Sigma^\infty_{s,t} X_{+})
\end{equation}
in the category of $(s,t)$-bispectra $\mathrm{Spt}_{s,t}(k)$.
Here $\Rep_\nis\Rep_{\A^{1}}\Fr(-,\Sigma^\infty_{s,t} X_{+})$ is a motivic local 
$\Omega_{s,t}$-bispectrum in positive degrees with respect to the simplicial circle $S^{1}$
(see \Cref{subsubsect:model-cat,subsect:LrepA1notdef} for the definitions of the localization 
functors $\Rep_\nis$ and $\Rep_{\A^{1}}$).
\end{theorem}

\begin{remark}
Note that the diagonal of the $(s,t)$-bispectrum \eqref{eq:SigmaCorSGXLnisLAFrSigmaTX2} defines 
the $\PP^1$-spectrum \eqref{eq:Xsuspension}.
\end{remark}

\subsection{$\A^{1}$-connectivity}
\label{subsection:A1connectivity}
Since each simplicial presheaf 
$\Fr(-\times\Delta^{\bullet}_{B},{\mathcal X}_{+}\wedge S^{n}\wedge \Gm^{\wedge n})$
is $n$-connective,
the $\A^{1}$-local $(s,t)$-bispectrum 
$\Rep_{\A^{1}}\Fr(-,\Sigma^\infty_{s,t} X_{+})$ in 
\eqref{eq:SigmaCorSGXLnisLAFrSigmaTX2} is connective. 
Consequently, motivic localization preserves connective $(s,t)$-bispectra, which is precisely the assertion of Morel’s stable connectivity theorem \cite{morel-trieste,Morel-connectivity}. In particular, it implies the vanishing of negative stable homotopy sheaves of smooth schemes:

\begin{theorem}[\protect{\cite{morel-trieste}}]
\label{theorem:fieldconnectivity}
Let $k$ be a field.
For any $X\in \Sm_k$,
the sheaves of motivic stable homotopy groups $\ul\pi_{p,q}(X)$ are Nisnevich locally trivial for all $p<q$.
\end{theorem}

\subsubsection{Morel's homotopy $t$-structure on the stable motivic homotopy category of fields}
\label{subs:MorelSHI}
Let $\SH_{s}(k)$ denote the $S^{1}$-stable motivic homotopy category of $k$, 
see \cite{morel-trieste}. 
The heart $\SH_{s}(k)^\heartsuit$ of its homotopy $t$-structure consists of strictly $\A^{1}$-invariant 
Nisnevich sheaves of abelian groups on $\Sm_{k}$, 
and the heart $\SH(k)^\heartsuit$ is equivalent to the category of homotopy modules,
see \cite[Theorem 5.2.6]{morel-trieste}.
Morel's result can be expressed in terms of framed homotopy modules,
see Ananyevskiy-Neshitov \cite{framed-MW}, 
and in terms of Milnor-Witt cycle modules, 
see Feld \cite{Niels-Feld}.

\subsection{Strictly $\A^{1}$-invariant sheaves} 
\label{subsection:sA1is}
Voevodsky's notion of strictly $\A^{1}$-invariant Nisnevich sheaves has been central in the development of motives 
and motivic homotopy theory \cite{Levine-app}.
Recall that every field $k$ has an associated linear category of finite correspondences $\Cor_k$ in the sense of 
Friedlander-Suslin-Voevodsky \cite{cycles-book}.
A presheaf with transfers, 
i.e., an additive functor 
$$
\mathcal{F}\colon\Cor_k^\op\to \Ab
$$ 
is called $\A^{1}$-invariant if, 
for any smooth separated $k$-scheme $X$, 
the projection $X\times\A^{1}\to X$ from the affine line induces an isomorphism of abelian groups
$$
\mathcal{F}(X)\xrightarrow{\cong} \mathcal{F}(X\times\A^{1}).
$$
Moreover, 
$\mathcal{F}$ is called Nisnevich strictly $\A^{1}$-invariant if, 
for all $i\geq 0$,
the cohomology presheaf
$$
X
\mapsto
H^{i}_{\nis}(X,\mathcal{F}_{\nis})
$$
is an $\A^{1}$-invariant presheaf with transfers.
The same definitions apply verbatim to presheaves of abelian groups on the category $\Sm_{k}$ 
of smooth separated $k$-schemes of finite type.

\subsubsection{Voevodsky's strict $\A^{1}$-invariance theorem for presheaves with transfers over fields}
\label{subs:VoevSHI}
A key technical input in Voevodsky's construction of motives is his result on strict $\A^{1}$-invariance 
of radditive presheaves with transfers in \cite[\S3.2]{Voe-hty-inv}.

\begin{theorem}
\label{th:shiCor}
If $k$ is a perfect field, 
then every $\A^{1}$-invariant presheaf with transfers
\[
\mathcal{F}
\colon
\Cor_k^\op
\to 
\Ab
\]
is Nisnevich strictly $\A^{1}$-invariant.
\end{theorem}

Let $\Delta^{\bullet}_{k}$ denote the standard cosimplicial $k$-scheme.
A notable consequence of \Cref{th:shiCor} is that the associated Suslin complex
\[
C_*^{\A^{1}}(\mathcal{F}_{\nis})
\defeq
\mathcal{F}_{\nis}(\Delta^{\bullet}_{k}\times -)
=
(
\cdots \to 
\mathcal{F}_{\nis}(\Delta^2_k\times -)\to
\mathcal{F}_{\nis}(\Delta^1_k\times -)\to
\mathcal{F}_{\nis}
\to
0)
\] 
of $\calF_{\nis}$ is $\A^{1}$-local, 
and hence an object in Voevodsky's category of effective motives $\mathbf{DM}^\mathrm{eff}(k)$.
More generally,
see, e.g., \cite[\S4.2]{Ayoub-guide}, 
the endofunctor on the derived category of bounded above chain complexes of Nisnevich sheaves with transfers 
given by the totalization 
\begin{equation}
\label{CSus}    
\calF^\bullet\mapsto \mathrm{Tot}(C_*^{\A^{1}}(\calF^\bullet))
\end{equation}
of the bicomplex $C_*^{\A^{1}}(\calF^\bullet)$ computes the $\A^{1}$-localization functor with values in 
$\mathbf{DM}^\mathrm{eff}(k)$.

\begin{remark}
Suslin \cite{Sus17} showed that 
\Cref{th:shiCor} holds for every field after inverting the exponential characteristic.
\end{remark}

\subsubsection{Garkusha-Panin's strict $\A^{1}$-invariance theorem for framed presheaves over fields}
\label{subs:GarkushaPaninSHI}

An appealing feature of Voevodsky's theory \cite{Voe-notes} is that framed correspondences encode transfer maps 
for motivic cohomology theories, 
see \cite[Lemma 7.9]{rigid-AnanievDruzh} and \cite{five-transfers}. 
Framed presheaves restrict to presheaves with transfers because framed correspondences form in a precise way 
the universal correspondence category.
In \cite{hty-inv},  
Garkusha-Panin prove a strict $\A^{1}$-invariance theorem for $\A^{1}$-invariant quasi-stable radditive 
framed presheaves over perfect fields, 
see also \Cref{definition:qsradditive}.
The said result is fundamental for their construction of framed motives, 
which in turn provides explicit geometric models for fibrant replacements of motivic spectra \cite{Framed}.
Recall from \cite[\S2]{Framed} the category of framed correspondences $\Fr_{+}(k)$ with objects 
smooth finite type $k$-schemes, 
and morphisms pointed sets $\bigvee_n \mathrm{Fr}_n(X,Y)$ of all level $n$ framed correspondences as in 
\Cref{sectionApp:FrCor}.

\begin{theorem}[\protect{\cite[Theorem 1.1,\S 17]{hty-inv}}]
\label{th:str_field}
If $k$ is a perfect field, 
then every $\A^{1}$-invariant quasi-stable radditive framed presheaf 
\[
\mathcal{F}
\colon
\Fr_{+}(k)^\op
\to 
\Ab
\]
is Nisnevich strictly $\A^{1}$-invariant.
\end{theorem}

\begin{remark}
For the proof of \Cref{th:str_field}, 
see \cite[\S 17]{hty-inv} for infinite fields of odd characteristic and fields of even characteristic 
(assuming $\mathbb Z[1/2]$-coefficients), 
and \cite{surj-etale-exc} for fields of characteristic two.
The case of finite fields is established independently in \cite{DrKyllfinFrpi00} and \cite[Appendix B]{five-authors}.
We also provide a short proof in \Cref{theorem:shifinitefields}.
\end{remark}

\begin{remark}
By \cite{nonperfect-SHI}, \Cref{th:str_field} holds also for imperfect fields.
\end{remark}

The proof of \Cref{citeth:GP14MotivcReplacement} makes use of the $S^{1}$-spectrum 
of motivic spaces with framed transfers 
\[
\calM_\fr(X)
\defeq
\left(
\Fr(\Delta^{\bullet}_{k}\times -, X_{+}), 
\Fr(\Delta^{\bullet}_{k}\times -,  X_{+}\wedge S^{1}), 
\dots, 
\Fr(\Delta^{\bullet}_{k}\times -,  X_{+}\wedge S^i),
\dots
\right)
\] 
from \cite[Definition 5.2 (3)]{Framed}.
Both strict $\A^{1}$-invariance and additivity for framed presheaves \cite[Theorem 6.4]{Framed} 
are used in the proof of \cite[Corollary 7.6]{Framed} to show the levelwise Nisnevich fibrant 
replacement $\calM_\fr(X)_{\nis}$ is a motivic fibrant $\Omega_{s}$-spectrum in positive degrees 
(the notation indicates $S^{1}$-loops).
In \cite{five-authors}, 
it is shown that the motivic localization of $\calM_\fr(X)$ coincides with the group completion 
of the Nisnevich local replacement of $\calM_\fr(X)$.

\subsection{What is done in this paper?}
\label{subsection:widitp}
In short, 
we generalize the results in \Cref{subsection:smfr,subsection:A1connectivity,subsection:sA1is}
to one-dimensional base schemes.
This project has proved to be anything but routine, 
and this article attempts to report on our findings so far.
First, we introduce the trivial fiber topology in \Cref{subsection:tftopologyintro}, 
one of the new ingredients in our approach.
Our main results are stated in \Cref{subsection:mr}. 

We invoke strict $\A^{1}$-invariance for abelian groups over imperfect fields, 
as recently shown by the first author in \cite{nonperfect-SHI}.\footnote{
The cautious reader may want to assume that our one-dimensional separated noetherian 
base scheme $B$ has perfect residue fields.}

\subsection{The $\tf$-topology}
\label{subsection:tftopologyintro}
It is known that strict $\A^{1}$-invariance in the sense of \Cref{th:shiCor} fails for every positive 
dimensional scheme, 
see \cite[Remark 4.9]{Ayoub-guide} and \Cref{example:H(A1U)} for details.  
To remedy the situation, we introduce the $\tf$-topology, 
where $\tf$ is shorthand for ``trivial fiber".

Let $B$ be a finite dimensional noetherian base scheme and let $\Sm_{B}$ denote smooth separated $B$-schemes 
of finite type.
A pullback diagram in $\Sch_{B}$ of the form
\begin{align}
\label{eq:TfgSq}
\xymatrix{
X^\prime\times_B {(B-Z)} \ar[r]\ar[d] &X^\prime\ar[d]^{\varphi} \\ 
X\times_B {(B-Z)} \ar[r] & X
}
\end{align}
is called a $\tf$-square if it is a Nisnevich square as in \cite[Definition 1.3, p. 96]{Morel-Voevodsky}, 
$\varphi$ is affine, 
and the closed immersion $Z\not\hookrightarrow B$ induces an isomorphism of $B$-schemes
\[
X^{\prime} \times_{B} Z
\cong 
X\times_{B} Z.
\]

The $\tf$-squares form a cd-structure on $\Sm_{B}$ in the sense of \cite[Definition 2.1]{VV:cd}.
The associated $\tf$-topology on $\Sm_B$ is the Grothendieck topology generated by $\tf$-squares 
as in \eqref{eq:TfgSq} and the empty covering sieve $\emptyset$. 
The $\tf$-topology on $\Sm_{B}$ is intermediate between the trivial topology and the Nisnevich topology.
Over a field, 
the $\tf$-topology coincides with the trivial topology, 
see \Cref{prop:Proptftop} (iii).
On the other hand, 
every Nisnevich covering of $B$ is also a $\tf$-covering. 
Thus, the $\tf$-topology is finer than the topology generated by pullbacks 
of Nisnevich coverings of $B$ along some structure morphism in $\Sm_{B}$.
One feature of the $\tf$-topology is that the $\tf$-cohomological dimension of any $X\in\Sch_B$ 
is bounded above by $\dim B$ (see \Cref{prop:cdstructurecompleteregularbounded}). 
The assumption that $\varphi$ is affine simplifies the description of the points of the $\tf$-topology,
see \Cref{rem:tfsquare:etalemorphismafineness}, 
but it is not essential for our main results.

\subsection{Main results}
\label{subsection:mr}

The $\tf$-topology allows us 
to decompose the motivic infinite loop space construction via certain adjunctions
\begin{equation}
\label{eq:Spc+frSHfrA1tfSHfrA1NSHAN}
\Spc_*(B)\stackrel{\gamma^*\dashv\gamma_*}{\rightleftarrows}
\Spc^\fr_*(B)\stackrel{\Sigma^\infty\dashv\Omega^\infty}{\rightleftarrows} \SH^{\fr}_{\A^1,\tf}(B)
\stackrel{L_{\nis}\dashv R_{\nis}}{\rightleftarrows} 
\SH^{\fr}_{\A^1,\nis}(B)\stackrel{\gamma_*\dashv\gamma^*}{\simeq}\SH(B).
\end{equation}
Here, starting with the homotopy category of pointed presheaves $\Spc_*(B)$ on the left, 
and proceeding to the homotopy category of $\A^1$-invariant $\tf$-local framed $(s,t)$-bispectra 
$\SH^{\fr}_{\A^1,\tf}(B)$ via the functor $\Sigma^\infty$, 
which simultaneously carries out $\tf$-motivic localization and stabilization, our analysis of 
these functors in \Cref{sect:stablemotiviclocalization,sect:ReusltFormulations} yields the following result.


\begin{theorem}
\label{theorem:homcatfunctrsshiftNisstalks}
\begin{itemize}
\item[(1)] 
If $B$ is a one-dimensional base scheme, 
then for any $\mathcal F\in \SH^\fr_{\A^{1},\tf}(B)$ and essentially smooth local henselian scheme $U$,
there is an isomorphism of bigraded abelian groups 
\begin{equation}
\pi_{*,*}^{\SH^\fr_{\A^{1},\tf}(B)}(\mathcal F)(U)
\xrightarrow{\cong}
\pi_{*,*}^{\SH_{\A^{1},\nis}(B)}(L_{\nis}\mathcal F)(U).
\end{equation}
Equivalently, 
there is a naturally induced Nisnevich local isomorphism of 
presheaves
\begin{equation}\label{eq:piSHfrAtfB(F)congNispiSHAnisB(LnisF)}
\pi_{*,*}^{\SH^\fr_{\A^{1},\tf}(B)}(\mathcal F)
\xrightarrow{\cong_{\Nis}}
\pi_{*,*}^{\SH_{\A^{1},\nis}(B)}(L_{\nis}\mathcal F).
\end{equation}
Therefore, 
$L_\nis$ is exact with respect to the homotopy $t$-structures on 
$\SH^\fr_{\A^{1},\tf}(B)$ and $\SH_{\A^{1},\nis}(B)$.
\item[(2)]
For any base scheme $B$ of Krull dimension $d$, 
the functor $\Sigma^\infty$ 
in \eqref{eq:Spc+frSHfrA1tfSHfrA1NSHAN} takes values in schemewise $(-d)$-connective objects,
see \Cref{sect:stabhomgroupsheavesconnectivity}.
\end{itemize}
\end{theorem}

\begin{remark}
\label{rem:framedspectrarepalemcetcooectivity}
The homotopy $t$-structures on $\SH^\fr_{\A^{1},\tf}(B)$ is defined 
via the $\tf$-sheaf associated to the left side of \eqref{eq:piSHfrAtfB(F)congNispiSHAnisB(LnisF)}, 
and likewise for the Nisnevich sheaf associated to the right side of \eqref{eq:piSHfrAtfB(F)congNispiSHAnisB(LnisF)}.
We note that $\gamma^*$ and $\gamma_*$ on the right side of \eqref{eq:Spc+frSHfrA1tfSHfrA1NSHAN}
are exact with respect to the homotopy $t$-structures.
\end{remark}

\begin{corollary}
\label{th:connectivitysmotsheaves}
Let $B$ be a one-dimensional base scheme.
Then the suspension functor 
\begin{equation}
\label{eq:SigmainftySpcbulletSHA1Nis(B)}
\Sigma^\infty\colon \Spc_*(B)\to \SH(B)
\end{equation} 
takes values in $(-1)$-Nisnevich locally connective objects.
For any $Y\in \Sm_B$ the Nisnevich sheaf 
$\underline{\pi}_{i,j}(Y)$ of stable motivic homotopy groups is trivial for $i<j-1$
(see \Cref{sect:stabhomgroupsheavesconnectivity}).
\end{corollary}

We emphasize that while \Cref{th:connectivitysmotsheaves} recovers the connectivity results for 
motivic $\PP^1$-spectra in \cite{ConnDodekindDomains,ConnGabPresLemNoethDominffield,ConnBase},
\Cref{theorem:homcatfunctrsshiftNisstalks} has additional computational applications.
First, 
the connectivity property of $\Sigma^\infty$ in \eqref{eq:Spc+frSHfrA1tfSHfrA1NSHAN} 
is stronger than for $\Sigma^\infty$ in \eqref{eq:SigmainftySpcbulletSHA1Nis(B)}.  
Second, \Cref{theorem:homcatfunctrsshiftNisstalks} can be applied to compute 
non-trivial Nisnevich stalks of stable motivic homotopy groups.
We use these facts in \Cref{subsect:stabletfmotivicdecomposition,sect:results:loopspaces} 
to study the stable motivic localization functor and motivic infinite loop spaces.

\begin{remark}
\Cref{theorem:fieldconnectivity} is pivotal for convergence and subsequent calculations of Voevodsky's 
slice spectral sequence over fields,  
see \cite{beo}, \cite{zbMATH06220359}, \cite{pi-one-new}, and \cite{pi-two}.
\Cref{th:connectivitysmotsheaves} is a first step towards such results for one-dimensional base schemes.
\end{remark}

\subsubsection{Stable motivic localization}
\label{subsect:stabletfmotivicdecomposition}

\Cref{theorem:homcatfunctrsshiftNisstalks} follows from our computation of the 
stable motivic localization functor through framed correspondences.
First, 
we note there is a natural stable motivic weak equivalences of $(s,t)$-bispectra 
\[
\Sigma^\infty_{\Gm,S^1}X_{+}
\simeq 
\Fr(-,\Sigma^\infty_{\Gm,S^1} X_{+}).
\]
For fields, 
see \cite{Framed}, 
and the general case follows from the reconstruction theorem  \cite{Hoyois-framed-loc}.
\vspace{0.1in}

Let $\Rep_{\nissmot}$ denote the fibrant replacement endofunctor on the category of 
$(s,t)$-bispectra of presheaves on $\Sm_{B}$ equipped with the stable motivic model structure 
\cite[\S2.3]{Nordfjordeid}, \cite{Jardine-spt}.
In \Cref{th:Lsmotspecrtaofframedpresheaves} we give the following description of $\Rep_{\nissmot}$.

\begin{theorem}
\label{theorem:results:Fr}
\par(1) 
Suppose that $B$ is a one-dimensional base scheme.
For any $(s,t)$-bispectrum of quasi-stable framed radditive presheaves $\calF$ and $X\in\Sm_{B}$,
there are canonical levelwise schemewise equivalences
\begin{equation}
\label{eqintro:LnismotdecompostionLnisLtfsmot}
\Rep_{\nissmot}(\calF)
\simeq 
\Rep_\nis\Rep_{\Gm}\Rep_{\A^{1},\tf}(\calF)^{\gp},
\end{equation}
and 
\begin{equation*}
\Rep_{\nissmot}(\Sigma^\infty_{s,t}X_+)
\simeq 
\Rep_\nis\Rep_{\Gm}\Rep_{ \A^{1},\tf}\Fr(-,\Sigma^\infty_{s,t}X_+)^{\gp}.
\end{equation*}
Here 
$\Rep_{ \A^{1},\tf}$ is the levelwise 
$\tf$-motivic localization endofunctor and 
$\Rep_{\nis}$ is the levelwise Nisnevich localization endofunctor.
The endofunctor $\Rep_{\Gm}$ maps $(s,t)$-bispectra to $S^1$-spectra of $\Omega_{\Gm}$-spectra.
\par(2) For any base scheme $B$ of Krull dimension $d$, 
the endofunctors $\Rep_{\Gm}$, 
$\Rep_{\A^{1},\tf}$ send schemewise connective objects to schemewise $(-d)$-connective objects.
\par(3) For $B$ be a one-dimensional base scheme, 
the endofunctor $\Rep_{\nissmot}$ sends Nisnevich locally connective objects to 
$(-1)$-Nisnevich locally connective ones.
\end{theorem}

\begin{remark}
For any essentially smooth local henselian scheme $U$, 
there is an equivalence $\calF(U)\simeq\Lrep_{\nis}\calF(U)$.
Therefore, 
$\Rep_\nis$ preserves the Nisnevich local connective objects.
\end{remark}

\subsubsection{Motivic infinite loop spaces}\label{sect:results:loopspaces}
The description of the stable motivic localization functor given in \Cref{subsect:stabletfmotivicdecomposition} 
allows us to deduce the following result on motivic infinite loop spaces
(see \Cref{th:Lacalizationfunctordecomposition}).

\begin{theorem}
\label{th:Lacalizationfunctordecomposition2}
Let $B$ be a one-dimensional base scheme.
The motivic infinite loop space of the motivic suspension spectrum of $X\in\Sm_B$ in 
$\SH(B)$ is Nisnevich locally equivalent to the pointed simplicial presheaf
\begin{equation*}
\label{eq:OmegastTGSAOmegastTTinftyLmot(Y)}
\colim
\Omega^l_{\Gm}(\Rep_{\A^{1},\tf}\Fr(X_{+}\wedge \Gm^{\wedge l})^{\gp}).
\end{equation*}
Here $\Fr(\mathcal X)=\Fr(-,\mathcal X)$ is the presheaf of framed correspondences in the sense of \cite{Framed} 
for any pointed simplicial smooth $B$-scheme $\mathcal X$.
\end{theorem}
\begin{remark}
\Cref{th:Lacalizationfunctordecomposition2} implies that the functor 
$\Omega^\infty_{\Gm}\colon \DM(B)\to \mathbf{D}(\Shv^\Ab_\nis(\mathrm{Cor}_B))$ 
takes the motive of $X\in\Sm_B$ to a complex of Nisnevich sheaves with transfers which is quasi-isomorphic to 
\begin{equation*}
\colim\Omega^l_{\Gm^{\wedge 1}}\Lrep_{\A^{1},\tf}(\mathbb Z_\tr(X_{+}\wedge \Gm^{\wedge l}))
\in 
\mathbf{D}(\Shv^\Ab_\nis(\mathrm{Cor}_B)).
\end{equation*}
\end{remark}

\subsubsection{Strict $\A^{1}$-invariance}
If $\tau$ is a topology on $\Sm_{B}$, 
a presheaf of abelian groups $\calF$ on $\Fr_{+}(B)$ is $\tau$-strictly $\A^{1}$-invariant if the cohomology presheaf
$H_\tau^i(-,\calF_\tau)$ is $\A^{1}$-invariant for all $i\geq 0$. 
We refer to \Cref{sectionApp:FrCor} for a discussion of framed correspondences $\Fr_{+}(B)$ over $B$ 
and quasi-stable radditive framed presheaves.
Over fields, 
this specializes to \cite[Definition 2.3]{Framed}.

\begin{definition}
\label{def:str_xibase}
A base scheme $B$ satisfies 
$\tf$-Nisnevich strict $\A^{1}$-invariance for abelian groups if every 
$\tf$-strictly $\A^{1}$-invariant quasi-stable radditive framed presheaf 
\[
\mathcal{F}
\colon
\Fr_{+}(B)^\op
\to 
\Ab
\]
is Nisnevich strictly $\A^{1}$-invariant. 
\end{definition}

As noted in \Cref{subsection:tftopologyintro}, 
the $\tf$-topology of any field is trivial. 
Hence for fields,
the notion of $\tf$-Nisnevich strict $\A^{1}$-invariance for abelian groups agrees with the notion of 
strict $\A^{1}$-invariance for abelian groups.
We show the following permanence result for strict $\A^{1}$-invariance:

\begin{theorem}
\label{thm:main}
Let $B$ be a one-dimensional base scheme.
Then $B$ satisfies $\tf$-Nisnevich strict $\A^{1}$-invariance for abelian groups.
\end{theorem}

\begin{remark}
If $B$ is a one-dimensional scheme, 
then for any presheaf of abelian groups $\calF$ on $\Sm_B$ and $X\in \Sm_B$, 
there is similarly to \Cref{lm:tfcohomologyonedimB} an isomorphism
\begin{equation}
\label{eq:Htf(X,F)Bdim1}
H_\tf^i(X,\calF_\tf)
\cong
\begin{cases}
\ker(\textrm{res}_{\calF}) & i=0\\
\coker(\textrm{res}_{\calF}) & i=1\\
0 & i\ge 2.
\end{cases}
\end{equation}
We write $\textrm{res}_{\calF}$ is the naturally induced restriction map
\begin{equation*}
\textrm{res}_{\calF}
\colon
\bigoplus_{z\in B^{(1)}}\calF(X^h_z)\oplus \bigoplus_{\eta\in B^{(0)}}\calF(X_\eta)
\to
\bigoplus_{z\in B^{(1)}, \eta\in B^{(0)}} \calF((X^h_z)_\eta).
\end{equation*}
Here $B^{(p)}$ denotes the set of codimension $p$ points of $B$,
$X^h_z$ denotes the henselization of $X$ along the closed subscheme $X_z=X\times_B z$;
$X_\eta=X\times_B \eta$, 
and $(X^h_z)_\eta=(X^h_z)\times_B \eta$.
Suppose $\calF$ is a quasi-stable radditive framed abelian presheaf on $\Sm_{B}$.
\Cref{thm:main} is equivalent to saying that if the presheaf
$X\mapsto H_\tf^i(X,\calF_\tf)$ in \eqref{eq:Htf(X,F)Bdim1} is $\A^{1}$-invariant, 
then the Nisnevich sheafification of $\calF$ is Nisnevich strictly $\A^{1}$-invariant.
\end{remark}

\subsection{Scholium for strict $\A^{1}$-invariance}
\label{subsection:squareAtimesV}
To elucidate the role of the $\tf$-topology, 
we note that a base scheme $B$ satisfies $\tf$-Nisnevich strict $\A^{1}$-invariance for abelian groups 
if and only if for every $X\in\Sm_{B}$,  
$x\in X$,
and $\calF\colon\Fr_{+}(B)^\op\to\Ab$ as in \Cref{def:str_xibase}, 
there is similarly to \Cref{corollary:ident-Nis-coh-loc-sch-corol} an isomorphism
\begin{equation}
\label{eq:Hnis(A1oloc)}
H^i_{\nis}(X_x^h\times \A^{1},\calF_{\nis})
\cong
\begin{cases}
\calF(X_x^h) & i=0\\
0 & i>0.
\end{cases}
\end{equation}
Here $X_x^h$ denotes the henselization of the local scheme $X_{(x)}$ at $x\in X$.
By the same reasoning, 
the assumption that $\calF$ is $\tf$-strictly $\A^{1}$-invariant implies the isomorphism
\begin{equation}
\label{eq:Htf(A1oloc)}
H^i_\tf(X_x^h\times \A^{1},\calF_\tf)\cong
\begin{cases}
\calF(X_x^h) & i=0\\
0 & i>0.
\end{cases}
\end{equation}
The formula \eqref{eq:Htf(A1oloc)} holds trivially over fields for any $\A^{1}$-invariant presheaf $\calF$.
To deduce \eqref{eq:Hnis(A1oloc)} we consider the Nisnevich topology on the small {\'e}tale site 
$\Et_{\A^{1}\times{X^h_x}}$ of the affine line $\A^{1}_{X^h_x}$ for all $x\in X$.
The latter topology is generated by the Nisnevich squares 
\begin{equation}
\label{eq:NissqA1FrContr}
\xymatrix{
W\times_V V^\prime \ar[r]\ar[d]& V^\prime\ar[d]\\
W\ar[r] & V
.}
\end{equation}
We consider the corresponding contractible chain complexes of presheaves of abelian groups
\[
\overline{\ZF}(-,W\times_V V^\prime)
\to 
\overline{\ZF}(-,W)\oplus \overline{\ZF}(-,V^\prime)
\to 
\overline{\ZF}(-,V). 
\]
Here $\overline{\ZF}(-,V)$ is the quotient of the presheaf of free abelian groups $\mathbb Z\Fr_+(-,V)$ 
by $\A^{1}$-homotopy, additivity, and quasi-stabilization, see \Cref{def:ZFandoverlineZF}.
We refer to \eqref{eq:NissqA1FrContr} as an $(\A^{1},\ZF)$-contractible Nisnevich square.
For every positive dimensional base scheme $B$ there exists a Nisnevich square
\begin{equation}
\label{eq:fsvSq}
\xymatrix{
\A^{1}_{U-Z(f)}-Z(r)\ar[r]\ar[d] & \A^{1}_U-Z(r)\ar[d] \\
\A^{1}_{U-Z(f)}\ar[r] & \A^{1}_U. 
}
\end{equation}
Here $U\subset B$ is an open affine subscheme, 
$f\in \mathcal O(U)$ is any regular function with a nonempty vanishing locus $Z(f)$, 
and $r=ft-1\in \mathcal O(\A^{1}_U)$.
In \Cref{sect:counterexample} we show that \eqref{eq:fsvSq} is an example of a 
non-$(\A^{1},\ZF)$-contractible Nisnevich square.
Hence the Nisnevich topology on 
$\A^{1}_{X^h_x}$ cannot be generated by $(\A^{1},\ZF)$-contractible Nisnevich squares.
The $\tf$-squares allow us to fix this problem.
This is decisive for showing that \eqref{eq:Htf(A1oloc)} implies \eqref{eq:Hnis(A1oloc)}.
In order to make this precise we use localization techniques, 
see also \Cref{rm:subsection:squareAtimesV} for a more detailed discussion.

\begin{remark}
In the proof of \Cref{thm:main}, 
we show that $\mathcal{F}_{\nis}$ and its Nisnevich cohomology groups are quasi-stable radditive framed presheaves 
for any base scheme $B$.
\end{remark}

\begin{remark}
\label{sect:AssumptionsBase}
We shall use strict homotopy invariance for all the residue fields of the base scheme $B$.
In fact, 
the strict $\A^1$-homotopy invariance theorem is proven over any field in \cite{nonperfect-SHI}
(if the field is perfect, 
this follows alternatively by combining \cite{hty-inv,surj-etale-exc} with \cite{DrKyllfinFrpi00} or \cite{five-authors}
--- see also \Cref{theorem:shifinitefields} for the case of finite fields).
The proof for imperfect fields expands on the material in \Cref{section:tcogf}.
A crucial step is to view the affine line $\A^1_k$ as a base scheme and analyze 
essentially smooth local schemes over the localization of $\PP^1_k$ at $\infty$.
\end{remark}

\begin{remark}
Suppose that for all $\eta\in B$, $X\in \Sm_{B}$, $z\in X$, and quasi-stable $\A^{1}$-invariant $F\in \Spt^\fr_s(B)$, 
the canonical morphism $F( (X^h_z)_\eta )\to L_{\nis} F( (X^h_z)_\eta)$ is a stable equivalence.
Under this assumption, 
the arguments and results in this paper remain valid over the base scheme $B$.
\end{remark}

\subsection{Further directions}
\Cref{theorem:homcatfunctrsshiftNisstalks,theorem:results:Fr} shift the problems of identifying 
the motive and computing the stable motivic homotopy groups of $X\in \Sm_B$ in $\SH(B)$ to the 
more manageable category $\SH^\fr_{\A^1,\tf}(B)$ and the endofunctor 
$\Lrep_{\A^1,\tf}\gamma_*\gamma^*$ on $\Spc_*(B)$.
If $B$ is one-dimensional, 
we claim that there is a canonical equivalence 
\begin{equation}
\label{eq:LA1tf}
\Lrep_{\A^{1}}\Lrep_{\tf}\Lrep_{\A^{1}}
\xrightarrow{\simeq}
\Lrep_{\A^{1},\tf}.
\end{equation}
The proof of \eqref{eq:LA1tf} occupies a large part of the work in progress \cite{notes}.
We remark that $\Lrep_{\A^{1},\tf}$ is not equivalent to the composite $\Lrep_{\tf}\Lrep_{\A^{1}}$ since 
$\tf$-strict $\A^{1}$-invariance fails for positive dimensional schemes, 
see \Cref{sect:counterexample}.
Hence \Cref{citeth:GP14MotivcReplacement} fails for general base schemes, 
because the right side of \eqref{eq:SigmaCorSGXLnisLAFrSigmaTX2} does not have the claimed properties.
If $B$ is a one-dimensional scheme and $X\in \Sm_B$, 
a combination of \eqref{eqintro:LnismotdecompostionLnisLtfsmot} and \eqref{eq:LA1tf} yields that the 
correct replacement for \eqref{eq:SigmaCorSGXLnisLAFrSigmaTX2} is 
\begin{equation}
\label{eq:LnLGLALtfLA}
\Sigma_{s,t}^\infty X_+
\xrightarrow{\simeq}
\Lrep_{\nis}\Lrep_{\Gm}\Lrep_{\A^{1}}\Lrep_{\tf}\Lrep_{\A^{1}}\Fr(-,\Sigma^\infty_{s,t} X_{+}).
\end{equation}
In \eqref{eq:LnLGLALtfLA}, 
the mixing of the left exact $\A^1$-localization functor $\Lrep_{\A^{1}}$ with the exact Nisnevich 
and $\tf$-localizations gives rise to nontrivial stable motivic homotopy classes in negative degrees.
We offer the following conjecture on the motivic sphere $\mathbf{1}_{B}$ and the $S^{1}$-stable motivic 
homotopy category  $\SH_{s}(B)$ of a base scheme $B$. 
We refer to \Cref{def:FrZFrZFovZF} for the definition of $\ZF(-,\Gm^{\wedge j})$ 
--- see also \cite{Nesh-FrKMW} for the second part.

\begin{conjecture}
\label{conj:S_S} 
\item[(1)] For $j\geq 0$ there is an equivalence
\[
\Rep_\smot(\Gm^{\wedge j})
\simeq 
\Rep_\nis\Fr(\Delta^\bullet_B\times -,\Sigma^\infty_{s,t}\Gm^{\wedge j})^{\gp}.
\]
\item[(2)] The stable motivic homotopy sheaves 
$\underline{\pi}_{i,j}(\mathbf{1}_{B})$ in $\SH(B)$ and $\underline{\pi}_{i}(\mathbf{1}_{B})$ 
in $\SH_{s}(B)$ are trivial in the range $i<j$ and $i<0$, 
respectively. 
\item[(3)] For $j\geq 0$ there is a canonical isomorphism
$$
\underline{\pi}_{-j,-j}(\mathbf{1}_{B})
\cong 
\ZF(-,\Gm^{\wedge j})/\ZF(\Delta^1_B\times -,\Gm^{\wedge j}).
$$
\end{conjecture}

\begin{remark}
The results in this paper reduce \Cref{conj:S_S} for any one-dimensional scheme $B$
to showing that there is a short exact sequence
\[
\calF(E\times_B U)\to
\calF((E\times_B U)^h_z)\oplus\calF((E\times_B U)\times_B (B-z))
\to 
\calF((E\times_B U)^h_z\times_B (B-z))
\]
for each closed point $z\in B$, see also \Cref{lm:tfcohomologyonedimB}.
Here, 
$U$ is an essentially smooth local henselian scheme, $E=\A^1,\Gm$,
and $\calF=H^l_\tf\Fr(\Delta^\bullet_B\times -,\Gm^{\times d})$ for $l,d\geq 0$. 
A weaker version of the second part is shown using different methods in \cite[\S5]{zbMATH07472296}.
\end{remark}

\subsection{Acknowledgments}
The authors thank Joseph Ayoub, Tom Bachmann, Elden Elmanto, Grigory Garkusha, Marc Hoyois, Ivan Panin, 
Markus Spitzweck, and our referees
for comments, discussions, and encouragements.
We acknowledge the support of the Centre for Advanced Study at the Norwegian Academy of Science and Letters in Oslo,
Norway, which funded and hosted our research project ``Motivic Geometry" during the 2020/21 academic year,
and the RCN Project no. 250399 ``Motivic Hopf Equations" and no. 312472 ``Equations in Motivic Homotopy Theory."
Druzhinin was supported by a Young Russian Mathematics award.
{\O}stv{\ae}r gratefully acknowledges support from the Humboldt Foundation and the Radboud Excellence Initiative.
Our work is supported by The European Commission -- Horizon-MSCA-PF-2022 ``Motivic integral $p$-adic cohomologies".

\section{Methods of proof, conventions and notation}
\label{section:shiitoms1s}

Next we discuss the strategy of the proof of the results from \Cref{subsection:mr},
especially the $\tf$-motivic localization theorem.
For simplicity, 
we review strict homotopy invariance for framed presheaves of \texorpdfstring{$S^{1}$}{S1}-spectra
(the proofs of \Cref{theorem:results:Fr,th:Lacalizationfunctordecomposition2} are based on similar principles).

We refer to Bousfield-Friedlander \cite{Bousfield-Friedlander} for $S^{1}$-spectra $\mathrm{Spt}$,
and to Hovey \cite{hovey:ss}, Jardine \cite{Jardine-spt}, \cite{Jardine-local} for the category 
$\mathrm{Spt}_{s}(B)$ of presheaves of $S^{1}$-spectra on $\Sm_{B}$.
We write equivalence for a stable weak equivalence in $\Spt$,
and $\Omega_{s}$-spectra for the fibrant objects in the stable model structure on $\mathrm{Spt}$.
For a Grothendieck topology $\tau$ on $\Sm_{B}$ we have the following model structures on $\mathrm{Spt}_{s}(B)$:
\begin{itemize}
\item[(1)] The levelwise (resp.~stable) $\tau$-local model structure $\Spt_{s,\tau}(B)$ 
(resp.~$\Spt_{\text{st},\tau}(B)$).
\item[(2)] The levelwise (resp.~stable) motivic $\tau$-local model structure $\Spt^{\A^{1}}_{s,\tau}(B)$ 
(resp.~$\Spt^{\A^{1}}_{\text{st},\tau}(B)$).
\end{itemize}
For a quick introduction we refer to \cite[\S 2]{Nordfjordeid}.
Let $\SH_{s,\tau}(B)$ and $\SH^{\A^{1}}_{s,\tau}(B)$ denote the corresponding stable homotopy categories.
The fibrant objects in (1) resp.~(2) are the levelwise fibrant presheaves of $S^{1}$-spectra that are $\tau$-local 
resp.~$\tau$-local and $\A^{1}$-local, 
see \cite[Lemma 20]{MZmod} for a characterization of the stable fibrations.
By \cite[Definition 3.1, Theorem 3.4]{hovey:ss}, 
the fibrant objects in the stable projective model structure on $\mathrm{Spt}_{s}(B)$ are the levelwise fibrant presheaves 
of $\Omega_{s}$-spectra.
If $\tau$ is a completely decomposable topology, 
an object in $\SH_{s,\tau}(B)$ is $\tau$-local if it satisfies (stable) excision with respect to $\tau$-squares in 
the sense of \cite[Corollary 1.4]{Jardine-spt}, \cite[Theorem 5.39]{Jardine-local}.
We say $\calF\in \SH_s(B)$ is $\A^{1}$-local with respect to $\tau$ if,  
for all $X\in \Sm_B$, 
the canonical morphism $\Lrep_\tau\calF(X)\to \Lrep_\tau\calF(X\times\A^{1})$ is an equivalence. 
Here $\Lrep_\tau$ is the $\tau$-localization endofunctor on $\SH_s(B)$.
In this terminology,
$\A^{1}$-local means $\A^{1}$-local with respect to the trivial topology.

Since the Nisnevich topology is finer than the $\tf$-topology, 
the discussion of change-of-topologies in \cite[\S 2]{ELSO} shows there exists a $\tf$-Nisnevich localization 
functor
\begin{equation}
\label{eq:tfNisSH}
\Loc^{\tf}_{\nis}
\colon 
\SH_{s,\tf}(B)
\to 
\SH_{s,\nis}(B).
\end{equation}
Thus for every $\calE\in\SH_{s,\tf}(B)$ there is a canonical morphism $\calE\to \Loc^{\tf}_{\nis}\calE$ that witnesses 
$\Loc^{\tf}_{\nis}\calE$ as the initial Nisnevich-local presheaf of $S^{1}$-spectra receiving a morphism from $\calE$.

\begin{definition}
\label{definition:qsradditive}
A framed presheaf of $S^{1}$-spectra
\[
\calF
\colon 
\Fr_{+}(B)^\op
\to 
\mathrm{Spt}
\]
is quasi-stable if, for all $X\in \Sm_{B}$, the level one framed correspondence $\sigma_X$ in 
\Cref{example:quasistable} induces an equivalence
$$
\sigma_X^*\colon 
\calF(X)
\overset{\simeq}{\to} 
\calF(X).
$$
Moreover, 
$\calF$ is radditive if for all $X,Y\in\Sm_{B}$ the 
natural map 
$$
\calF(X\amalg Y)
\to
\calF(X)\times\calF(Y)
$$
is an equivalence.
\end{definition}

A presheaf of framed $S^{1}$-spectra can be viewed as an object of $\mathrm{Spt}_{s}(B)$ 
via the functor from $\Sm_B$ to $\Fr_+(B)$.
Here we use that every morphism of schemes defines a framed correspondence.

\begin{definition}
\label{definition:SHI}
A base scheme $B$ satisfies 
$\tf$-Nisnevich 
strict $\A^{1}$-invariance for framed presheaves of $S^{1}$-spectra if the 
$\tf$-Nisnevich localization functor 
$L^{\tf}_{\nis}$ in \eqref{eq:tfNisSH} preserves $\A^{1}$-local quasi-stable 
radditive framed objects in $\SH_{s,\tf}(B)$.
\end{definition}

We are ready to state our main result on strict $\A^{1}$-invariance enhancing \Cref{thm:main} to framed presheaves 
of $S^{1}$-spectra.
See \Cref{thm:SHIrc} and \Cref{cor:ptobasedefinition:SHI} for more details, where we prove a slightly more robust result.

\begin{theorem}
\label{thm:S1spectraform0}
Suppose the base scheme $B$ is one-dimensional.
Then $B$ satisfies strict $\A^{1}$-invariance for framed presheaves of $S^{1}$-spectra. 
\end{theorem}

\subsection{$\tf$-motivic recollement}
\label{subsection:tfmlt}

In the proof of \Cref{thm:S1spectraform0} we appeal to a $\tf$-motivic recollement or localization 
theorem akin to Morel-Voevodsky's localization theorem; 
see \cite{AyoubI}, \cite{Cisinski-Deglise}, \cite{Hoyois-framed-loc}, and \cite{Morel-Voevodsky} 
for the Nisnevich topology.
Here, 
we state the stable form of the $\tf$-motivic localization theorem;
it follows immediately from the unstable result in \Cref{th:HA1tflocalizationrecollement}.
For a closed immersion $i\colon Z\not\hookrightarrow B$ with open complement $j\colon U\to B$ 
there are naturally induced adjunctions ($i_*$ and $j^*$ are the left adjoints):
\begin{equation}
\label{eq:locpriSHS1tf}
i_*
\colon
\SH^{\A^{1}}_{s,\tf}(\SmAff_Z)
\xrightleftarrows{1em} 
\SH^{\A^{1}}_{s,\tf}(\SmAff_B)
\colon
i^!
\,\,
\text{and}
\,\,
j^*
\colon
\SH^{\A^{1}}_{s,\tf}(\SmAff_B)
\xrightleftarrows{1em} 
\SH^{\A^{1}}_{s,\tf}(\SmAff_U)
\colon
j_*
\end{equation}

\begin{theorem}[see \Cref{th:HA1tflocalizationrecollement}]
\label{th:locthSHmottf}
The functors $i_*$ and $j_*$ in \eqref{eq:locpriSHS1tf} are fully faithful. 
Moreover, 
the counit $i_*i^!\to\id$ and the unit $\id\to j_*j^*$ induce, 
for every $\calF\in \SH^{\A^{1}}_{s,\tf}(\SmAff_B)$, 
a homotopy fiber sequence
\begin{equation}
\label{eq:idsidcalFjdsjus}
i_*i^!(\calF)\to \calF \to j_*j^*(\calF).  
\end{equation}
\end{theorem}

We use localization as a tool to reduce the questions about $\SH^{\A^{1}}_{s,\tf}(B)$ to $\SH^{\A^{1}}_{s,\tf}(\SmAff_Z)$ 
and $\SH^{\A^{1}}_{s,\tf}(\SmAff_U)$.
To prove \eqref{eq:idsidcalFjdsjus} we work with the subcategory $\SmAff_{B,Z}$ of 
$\SmAff_{B}$ spanned by the 
essentially smooth schemes $X^h_{X\times_B Z}$ defined as in \eqref{eq:proafhenselizationXhY}.
If $X$ is affine, 
we note that $X^h_{X\times_B Z}$ is the henselization of $X$ along $X\times_B Z$.
Following the usual script we use $\SmAff_{B,Z}$ to construct the stable $\tf$-local homotopy category 
$\SH_{s,\tf}(\SmAff_{B,Z})$ and the adjunctions
\begin{equation}
\label{eq:locpriShs1nis(Sms)}
\widetilde i_*
\colon
\SH_{s,\tf}(\SmAff_{B,Z})
\xrightleftarrows{1em} 
\SH_{s,\tf}(\SmAff_{B})
\colon
\widetilde i^!
\,\,
\text{and}
\,\,
j^*
\colon
\SH_{s,\tf}(\SmAff_{B})
\xrightleftarrows{1em} 
\SH_{s,\tf}(\SmAff_{U})
\colon
j_*
\end{equation}

\Cref{prop:filtr-iFj:unstable} implies the following properties for the functors in \eqref{eq:locpriShs1nis(Sms)}.

\begin{prop}[see \Cref{prop:filtr-iFj:unstable},  \Cref{cor:filtr-iFj:stable}]
\label{prop:cor:filtr-iFj:stable}
The functors $\widetilde i_*$ and $j_*$ in \eqref{eq:locpriShs1nis(Sms)} are fully faithful. 
Moreover, 
the counit $\widetilde i_*\widetilde i^!\to\id$ and the unit $\id\to j_*j^*$ induce, 
for every $\calF\in \SH_{s,\tau}(\SmAff_{B})$, 
a homotopy fiber sequence
\begin{equation}
\label{eq:tidstidcalFjdsjus}
\widetilde i_*\widetilde i^!(\calF)\to \calF \to j_*j^*(\calF).  
\end{equation}
\end{prop}
We deduce \eqref{eq:idsidcalFjdsjus} from 
\eqref{eq:tidstidcalFjdsjus} and the equivalence 
\eqref{eq:smtflocalhc}
between stable motivic $\tf$-local homotopy categories
\begin{equation}
\label{eq:smtflocalhc}
\SH_{s,\tf}^{\A^{1}}(\SmAff_{B,Z})
\simeq 
\SH_{s,\tf}^{\A^{1}}(\SmAff_Z).
\end{equation}
provided by \Cref{th:LocA1tfnisstructuresDeformation}(1).
We note that, by \Cref{prop:iAGm}(1), the equivalence holds already between 
(unstable) homotopy categories of $\A^1$-invariant presheaves
\begin{equation*}
\HH^{\A^{1}}(\SmAff_{B,Z})
\simeq 
\HH^{\A^{1}}(\SmAff_{Z}),
\end{equation*}
and this follows essentially from the lifting property for presheaves represented by smooth affine schemes with respect to Henselian pairs of affine $B$-schemes, \Cref{lm:AffSmHenselianLift}.

\subsection{Outline of the proof of strict homotopy invariance}
\label{subsection:atshi}
After setting up the adjunctions 
\[
\SHstf(Z)
\xrightleftarrows{1em}
\SHstf(B,Z)
\xrightleftarrows{1em} 
\SHstf(B)
\]
and proving \Cref{prop:cor:filtr-iFj:stable}, 
we divide the proof into two steps:
(1) strict $\A^{1}$-invariance for $\SHstf(B,Z)$ and $\SHstf(U)$ implies strict $\A^{1}$-invariance for $\SHstf(B)$, 
and 
(2) strict $\A^{1}$-invariance for $\SHstf(Z)$ implies strict $\A^{1}$-invariance for $\SHstf(B,Z)$.

To deduce (1), 
we study the adjunction in \eqref{eq:locpriShs1nis(Sms)} through $\tf$-squares as in \eqref{eq:TfgSq} 
and the generic fiber of an essentially smooth local henselian scheme over a one-dimensional base scheme, 
see \Cref{subsection:Hi(Xhx)_SHIeta}.
To deduce (2), 
we consider the naturally induced adjunction 
\begin{equation}
\label{SHstf_Z-BZ}
\SHstf(Z)
\xrightleftarrows{1em}
\SHstf(B,Z).
\end{equation}
The key geometrical input in this step is a lifting property asserting that 
for every closed immersion $Z\not\hookrightarrow X$ in $\SmAff_B$ and any framed correspondence $Z\to W$ of 
smooth affine $B$-schemes, 
there exists a framed correspondence $X_Z^h\to W$ as indicated by the dotted arrow in the diagram
\begin{align}
\xymatrix{ 
& W \\ 
X_Z^h \ar@{.>}[ur] & Z. \ar[l] \ar[u]}
\label{eq:fr-lift-diagr}
\end{align}
We view \eqref{eq:fr-lift-diagr} as a diagram of framed correspondences between essentially smooth 
$B$-schemes, 
see \Cref{sectionApp:FrCor}.
We show in \Cref{subsect:Rigid}, \Cref{sect:DeformationZSZ} and \Cref{sectapp:Homotopycolimits} that the latter 
lifting property for framed correspondences leads to an equivalence between the subcategories of $\A^1$-invariant 
objects in the adjunction \eqref{SHstf_Z-BZ}.   
While the required geometrical lifting property \eqref{eq:fr-lift-diagr} follows from classical results on henselian pairs,
we also employ more involved homotopical techniques to complete the proof.

\subsection{Conventions and notation}
\label{subsection:candn}

Throughout,
we follow the same conventions as in the Stacks Project \cite{StacksProject}.
A base scheme $B$ refers to a finite dimensional 
quasi-compact and quasi-separated scheme.

\subsubsection{Notation} 
\label{subsubsectionList} 
\begin{enumerate}

\item 
We write $\Sm_{B}^\tau$ for the site associated with the $\tau$-topology on \index{Categories of $B$-schemes!$\Sm_B$}; 
the category of smooth separated finite type $B$-schemes.
We let \index{Categories of $B$-schemes!$\SmAff_B$} denote the full subcategory of $\Sm_B$ spanned by schemes that admit a closed immersion into 
some finite dimensional affine space $\A^n_B$,
and write \index{Categories of $B$-schemes!$\EssSm_B$} for essentially smooth $B$-schemes, 
see \Cref{def:EssentiallySmooth}.

\item 
To a closed immersion $Z\not\hookrightarrow B$ and $B$-scheme $X$, 
we form the fiber product $X_Z:=X\times_B Z$ and the scheme \index{Schemes!$X_Z^h$} defined in \eqref{eq:proafhenselizationXhY};
if $X$ is affine, then $X_Z^h$ is the henselization of $X$ along \index{Schemes!$X_Z$}.

\item For $x\in X$, $X\in\Sch_B$, let \index{Schemes!$X_{(x)}$} denote the local scheme of $X$ at $x$.
         For $X=B$, we use the shorter notation \index{Schemes!$B_\sigma=B_{(\sigma)}$}.
   \item For $\sigma\in B$, $X\in\Sch_B$, let \index{Schemes!$X_\sigma$} denote $X\times_B \sigma$
   (we never apply this notation to $B$), and let \index{Schemes!$X_\sigma^h$} denote $X^h_{X\times_B \sigma}$.
 \item 
We write $Z(I)$ for the vanishing locus of an ideal sheaf $I\subset \calO_{X}$ contained 
in the structure sheaf of a scheme $X$, and we write
$I_{X}(Z)\subset \calO_{X}$ for the ideal sheaf of a closed subscheme $Z$ of $X$.

\item 
For an $\infty$-category $\calC$, 
we write $\Map_{\calC}(-,-)$ or simply $\Map(-,-)$ for the mapping space.
\item 
A presheaf $\calF$ on $\Sm_B$ extends by continuation to a presheaf on $\EssSm_B$, 
see \Cref{section:essahp}.

\item 
For a closed immersion $Z\not\hookrightarrow B$ with complementary open immersion $j\colon U\hookrightarrow B$, 
we have the adjunctions 
\begin{equation*}\begin{array}{lclclcl}
\overarrow{i}^*&\colon& 
\cPre^\fr(\mathcal S_{B,Z})
&\rightleftarrows& 
\cPre^\fr(\mathcal S_Z)
&\colon& 
\overarrow{i}_*,
\\
\tids
&\colon&
\cPre^\fr(\mathcal S_{B,Z})
&\rightleftarrows& 
\cPre^\fr(\mathcal S_{B})
&\colon&
\tif,
\\
\jus
&\colon&
\cPre^\fr(\mathcal S_{B})
&\rightleftarrows& 
\cPre^\fr(\mathcal S_{B-Z})
&\colon&
\jds.
\end{array}\end{equation*}
Here, 
the functors are given by 
\begin{gather*}
\begin{array}{lcll}
\tilde i^!F(X^h_Z) &=& \fib( F(X^h_Z)\to F(X^h_Z-X_Z) ); &X^h_Z\in \mathcal S_{B,Z},\\
\tilde i_*F(X)&=&F(X^h_Z); &X\in\mathcal S_B,\\
j^*F(V)&=&F(V); &V\in \mathcal S_{B-Z},\\
j_*F(X)&=&F(X-X_Z); &X\in \mathcal S_B,\\
\overarrow{i}_*F(X)&=&F(X_Z); & X\in\mathcal S_{B,Z},\\
\overarrow{i}^*(h^\fr(Y^h_Z))&=&h^\fr(Y_Z); & Y\in \mathcal S_B.
\end{array}\end{gather*}
Similar adjunctions exist for the pointed homotopy categories
$\HHtrivptd(\Fr_+(\mathcal S_{B,Z}))$ and 
$\HHtrivptd(\mathcal S_{B,Z})$, 
see \eqref{eq:LocNisZSquare}.

\end{enumerate}

\subsubsection{Stable motivic homotopy sheaves}
\label{sect:stabhomgroupsheavesconnectivity}
We use the notation $\pi_i(Y)$ and $\pi_{i,j}(Y)$ 
for the presheaves of stable motivic homotopy groups given by 
\[
\begin{array}{lcl}
\pi_i(Y)(U) 
&=& \operatorname{Hom}_{\SH(B)}(U_+\wedge S^i,Y_+),\\
\pi_{i,j}(Y)(U)
&=& \operatorname{Hom}_{\SH(B)}(U_+\wedge S^{i-j}\wedge \Gm^{\wedge j},Y_+).
\end{array}\]
Moreover, 
we write $\underline{\pi}_i(Y)$ and $\underline{\pi}_{i,j}(Y)$ for the associated Nisnevich sheaves.
\vspace{0.1in}

An object $\calF$ of $\SH_{s,\tau}(B)$ is called schemewise $n$-connective if for all $X\in\Sm_{B}$ the vanishing 
\[
\operatorname{Hom}_{\SH_{s,\tau}(B)}(X_{+}\wedge S^i,\calF)
=
0
\] 
holds in the range $i<n$.
Moreover, 
$\calF$ is called $\tau$-locally $n$-connective if the $\tau$-sheaf associated to the presheaf 
$X\mapsto \operatorname{Hom}_{\SH_{s,\tau}(B)}(X_{+}\wedge S^i,\calF)$ vanishes in the range $i<n$.
The homotopy $t$-structure on $\SH_{s,\tau}(B)$ is defined via the associated $\tau$-sheaves 
for all $i\in \mathbb Z$.

\subsubsection{$\A^1$-localization}\label{subsect:LrepA1notdef}
We let \index{Schemes!$\A^{1}_{B,Z}$} $\A^{1}_{B,Z}$ be shorthand for $(\A^{1}_B)^h_Z$ defined in \eqref{eq:proafhenselizationXhY}, 
and write $\Delta^\bullet_{B,Z}$ for 
the cosimplicial scheme in $\Sm_{B,Z}$, see \Cref{convnotations:schemescategories}, given by 
\begin{equation*}
\label{equation:cosimplicial}
\Delta^n_{B,Z}
=
(\Delta^n_B)^h_Z
\not\hookrightarrow 
\A^{n+1}_{B,Z}
\defeq
(\A^{n+1}_B)^h_Z
\cong
\underbrace{\A^{1}_{B,Z}\times_{B,Z}\cdots\times_{B,Z}\A^{1}_{B,Z}}_{n+1}.
\end{equation*}

Let $\Rep_{\A^{1}}^{[1]}$ denote the endofunctor $F\mapsto F_f(\Delta^\bullet_{{B,Z}}\times_{{B,Z}} -)$ 
on $\Spc_{s}(\Sm_{B,Z})$.
Here $(-)_f$ is the fibrant replacement in the injective model structure,
and the symmetric monoidal product $\times_{B,Z}$ on $\Sm_{B,Z}$ is defined in \eqref{eq:timesBZ}.
Setting $\Rep_{\A^{1}}^{[l]} = (\Rep_{\A^{1}}^{[1]})^l$ and appealing to \cite[p.107]{Morel-Voevodsky}, 
we have 
\begin{equation}
\label{eq:LAl->LA}
\Rep_{\A^{1}}
\cong 
(\hocolim_{l}\Rep_{\A^{1}}^{[l]})_f. 
\end{equation}
The functors $F_f(\Delta^n\times -)$, 
$n\in \bbZ$,
and $\hocolim(-)$ preserve weak equivalences, 
see \cite[\S 41.2]{DHKS_Hlmodcatandhcat}.
Consequently, 
$\Rep_{\A^{1}}^{[l]}$, $(\Rep_{\A^{1}}^{[1]})$ and $\Rep_{\A^{1}}$ preserve weak equivalences too.
By \cite[Lemma 1.2.2]{Hovey} the latter functors descend to the homotopy category $\HH_{s}(\Sm_{B,Z})$.
In related contexts, 
such as in \Cref{sect:SmBlZA1rigid}, 
we will denote functors defined similarly as above by 
$
\Lrep_{\A^{1}}^{[l]}
$,
$(\Lrep_{\A^{1}}^{[1]})$,
$\Lrep_{\A^{1}}$.

\section{The trivial fiber topology}
\label{section:ttft}

In this section, 
we introduce the trivial fiber topology, or $\tf$-topology for short, on $\Sch_{B}$ and discuss its basic properties. 

\subsection{$\tf$-coverings}
\begin{definition}
\label{definition:tfsquares}
A pullback square in $\Sch_{B}$ 
\begin{equation}
\label{equation:Nistfsquare}
\xymatrix{ 
X^\prime-Y^\prime \ar[r]\ar[d]& X^\prime\ar^{\varphi}[d]\\ 
X-Y\ar[r] & X 
}
\end{equation}
for closed subschemes $Y\not\hookrightarrow X$ and $Y^\prime\not\hookrightarrow X^\prime$
is called a trivial fiber square, 
or a $\tf$-square for short,
if the following holds.
\begin{itemize}
\item[(i)] There exists a closed subscheme $Z\not\hookrightarrow B$ such that 
$Y=X\times_B Z$, 
and 
$Y^\prime=X^\prime\times_X Y$.
\item[(ii)] The morphism $\varphi$ is affine {\'e}tale and induces an isomorphism
$Y^\prime \overset{\cong}{\to} Y$. 
\end{itemize}
That is, 
a $\tf$-square is a Nisnevich square of the form \eqref{equation:Nistfsquare}
for which $\varphi$ is affine and there exists a closed subscheme 
$Z\not\hookrightarrow B$ such that $Y\cong X\times_B Z$.

An affine $\tf$-square over $B$ is a $\tf$-square in $\Aff_B$; 
that is, 
an affine Nisnevich square of the form \eqref{equation:Nistfsquare} 
\cite[Example 2.1.2(5)]{zbMATH06773295}
for which there exists a closed subscheme 
$Z\not\hookrightarrow B$ such that $Y\cong X\times_B Z$.
\end{definition}

\begin{remark}\label{rem:definition:tfsquares(ii)}
In \Cref{definition:tfsquares}(ii), it suffices that $\varphi$ induce an isomorphism
$Y^\prime_\red \overset{\cong}{\to} Y_\red$,
because the morphism $\varphi$ is \'etale. Hence we may work with closed subsets rather than closed subschemes.
\end{remark}

\begin{remark}
$\tf$-squares are introduced because \Cref{definition:tfsquares} assembles a collection of 
squares that are not $(\A^{1},\ZF)$-contractible in the sense of \Cref{subsection:squareAtimesV},
and these provide the essential obstruction to the strict homotopy invariance theorem over $B$.
\end{remark}

In the $\tf$-topology over an affine base scheme $B$,
the affineness assumption on $\varphi$ simplifies some basic facts about the $\tf$-topology ---
see \Cref{rem:tfsquare:etalemorphismafineness} for details.
However, 
the said assumption is not strictly needed for our results on $\tf$-localizations.

The class of $\tf$-squares is stable under isomorphism and hence forms a $cd$-structure on $\Sch_{B}$ 
in the sense of Voevodsky \cite[Definition 2.1]{VV:cd}.

\begin{definition}
\label{definition:tftopology}
The $\tf$-topology on $\Sch_{B}$ is the topology generated by the $\tf$-squares in $\Sch_{B}$.
\end{definition}

\begin{prop}
\label{prop:cdstructurecompleteregularbounded}
The $\tf$-squares form a complete and regular $cd$-structure on $\Sch_{B}$. 
If $B$ is of finite Krull dimension $d$, 
then this $cd$-structure is bounded, 
and there is a density structure such that $\dim_B X\leq d$ for all $X\in \Sch_B$.
Moreover, 
the $\tf$-cohomological dimension of any scheme $X\in \Sch_{B}$ is less than or equal to $d$.
\end{prop}
\begin{proof}
If $X=\emptyset$ in the $\tf$-square \eqref{equation:Nistfsquare}, 
then $X-Y=\emptyset$ and $X^\prime=\emptyset$.
For a morphism $X_1\to X$ in $\Sm_B$ we set $X^\prime_1=X^\prime\times_X X_1$, $Y_1=Y\times_X X_1$, 
and $Y^\prime_1=Y^\prime\times_X X_1$. 
Then the schemes $X_1^\prime$, $X^\prime_1-Y^\prime_1$, $X_1-Y_1$, and $X_1$ form a $\tf$-square.
This shows the $\tf$ $cd$-structure is complete.

To show regularity we check the following conditions for any $\tf$-square \eqref{equation:Nistfsquare}:
\begin{itemize}
\item[(i)] The square \eqref{equation:Nistfsquare} is cartesian.
\item[(ii)] The morphism $X-Y\rightarrow X$ is a monomorphism.
\item[(iii)] The square 
\begin{equation}
\label{eq:squareXprime2_X}
\xymatrix{
X-Y \ar[d]\ar[r] & X\ar[d] \\
(X^\prime-Y^\prime)^{\times 2}_{X-Y}\ar[r]& {X^\prime}^{\times 2}_X
}
\end{equation}
is a $\tf$-square.
\end{itemize}
Parts (i) and (ii) follow immediately from \Cref{definition:tfsquares}. 
Since $X^\prime\to X$ is \'etale, 
the diagonal morphism $X^\prime\to X^\prime\times_X X^\prime$ is a clopen immersion; 
in particular, it is affine and \'etale.
Since $X^\prime\times_X Y\cong Y^\prime\cong Y$, 
it follows that 
$(X^\prime-Y^\prime)\times_{X-Y}(X^\prime-Y^\prime)
\cong 
X^\prime\times_X X^\prime- Y^\prime\times_Y Y^\prime 
\cong
X^\prime\times_X X^\prime\times_B (B-Z)$. 
This shows \eqref{eq:squareXprime2_X} is a $\tf$-square.

For boundedness we define a density structure $D_*(-)$ on $\Sch_B$ by associating to 
$X\in \Sch_B$ the family 
$D_d(X)$ comprised of open immersions of the form $X\times_B (B-W)\to X$, 
where $\codim_B W\geq d$. 
With respect to the said density structure,  
every $X\in \Sch_B$ has dimension $\leq \dim B$. 
Next we prove that every $\tf$-square \eqref{equation:Nistfsquare} is reducing with respect to $D_*(-)$.
Suppose $W_\text{o\'e}$, $W_\text{o}$, $W_\text{\'e}$ are closed subschemes of $B$ such that
$\codim_B W_\text{o\'e}\geq d$, $\codim_B W_\text{o},\codim_B W_\text{\'e}\geq d+1$.
There are open immersions in the density structure
\[
\begin{array}{ccccc}
X^\prime_{-W}&\defeq&X^\prime\times_B (B-W_\text{\'e})&\to& X^\prime, \\
(X^\prime-Y^\prime)_{-W}&\defeq&(X^\prime-Y^\prime)\times_B (B-W_\text{o\'e})&\to& X^\prime-Y^\prime, \\
(X-Y)_{-W}&\defeq&(X-Y)\times_B (B-W_\text{o})&\to& X-Y.
\end{array}
\]
We set $T=W_\text{\'e}\cup W_\text{o}\cup \overline{W_\text{o\'e}\cap (B-Z)}\cap Z$, 
where $\overline{W_\text{o\'e}\cap (B-Z)}$ is the closure of ${W_\text{o\'e}\cap (B-Z)}$ in $B$.
We have that $\codim_B T\geq d$.
Consider the base change of \eqref{equation:Nistfsquare} along the open immersion $B-T\to B$, 
and let $X_1=X\times_B (B-T)$, $X^\prime_1=X^\prime\times_B (B-T)$, $Y_1=X_1\times_B Z$, $Y^\prime_1=X^\prime\times_B Z$.
The canonically induced morphisms
$X_1^\prime-Y^\prime_1\to X^\prime-Y^\prime$, $X_1^\prime\to X^\prime$, $X_1-Y_1\to X-Y$ 
factor through $(X^\prime-Y^\prime)_{-W}$, $X^\prime_{-W}$, $(X-Y)_{-W}$, respectively.
This implies \eqref{equation:Nistfsquare} is reducing. 

The claim about the $\tf$-cohomological dimension follows now from \cite[Theorem 2.26]{VV:cd}.
\end{proof}

\begin{remark}\label{rem:tftopologySmBAffBSmAffB}
The $\tf$-cd structure on $\Sch_B$ restricts to its subcategories 
$\Sm_B$, $\Aff_B$, and $\SmAff_B$.
\Cref{prop:cdstructurecompleteregularbounded} holds for the said restrictions. 
This defines the $\tf$-topology on $\Sm_B$ and the affine $\tf$-topologies on 
$\Aff_B$ and $\SmAff_B$.
\end{remark}

\begin{example}\label{ex:tftopcovering}
We give some examples to elucidate the notion of $\tf$-coverings.
\begin{itemize}
\item[(i)] 
For any $X\in \Sch_{B}$ and any affine Nisnevich covering $\widetilde B\to B$, 
the morphism $X\times_B \widetilde B\to X$ is a $\tf$-covering.
Thus every 
Nisnevich covering of the base scheme $B$ has a refinement that is a $\tf$-covering.
\item[(ii)] The Nisnevich covering of the affine line $\A^{1}_B$ by its open subschemes 
$\A^{1}_B-0_B$ and $\A^{1}_B-1_B$ is not a $\tf$-covering in $\Sm_{B}$.
\item[(iii)] Let $f\in\mathcal{O}(B)$ be a regular function and let $t$ denote the coordinate on $\A^{1}_B$. 
The Zariski covering of the affine line $\A^{1}_B$ by $\A^{1}_{B-Z(f)}$ and $\A^{1}_B-Z(ft-1)$ is a tf-covering
defined by a $\tf$-square, 
because the open immersion $\A^{1}_B-Z(ft-1)\to\A^{1}_B$ is an affine morphism of $B$-schemes.
\end{itemize}
\end{example}

\begin{lemma}
\label{lm:tfsquareglueing}
Suppose $Z\not\hookrightarrow B$ is a closed immersion, and $X\in\Sm_B$.
Then a morphism $\calF\to \calG$ of $\tf$-sheaves of sets 
on $\Sm_{B}$ induces an isomorphism $\calF(X)\to\calG(X)$
if it induces isomorphisms
\begin{equation}\label{eq:FX(-Z)XhZ}
\calF(U\times_B (B-Z))\overset{\cong}{\to} 
\calG(U\times_B (B-Z))
\,\,
\text{and} 
\,\,
\calF(X^h_Z)\overset{\cong}{\to} \calG(X^h_Z), 
\end{equation}
where $U=X,X^h_Z$.
\end{lemma}
\begin{proof}
Suppose $\calF\to \calG$ induces isomorphisms as in \eqref{eq:FX(-Z)XhZ}. 
Here, 
$\calF(X^h_Z)$ is short-hand for the colimit $\varinjlim\calF(X^\prime)$ 
indexed over the \'etale neighborhoods $X^\prime$ 
of the closed subscheme $X\times_B Z$ of $X$, 
i.e., there is a commutative diagram where the vertical map is \'etale
\[\xymatrix{
& X^\prime \ar[d]\\  X\times_B Z\ar[r]\ar[ru]& X.
}\]
Similarly,
$\calF(X^h_Z\times_B(B-Z))$ is short-hand for the colimit $\varinjlim\calF(X^\prime\times_B(B-Z))$. 
It follows that 
\[\calF(X^\prime\times_B (B-Z))\cong
\calG(X^\prime\times_B (B-Z))
\,\,
\text{and} 
\,\,
\calF(X^\prime)\cong\calG(X^\prime)\]
for some 
\'etale neighborhood $X^\prime$ as above.
Now $X$, $X^\prime$, $X\times_B (B-Z)$, and $X^\prime\times_B (B-Z)$ form a $\tf$-square as in \eqref{equation:Nistfsquare}.
Since $\calF$ and $\calG$ are $\tf$-sheaves we are done.
\end{proof}

We follow the same conventions on sites and topoi as in \cite[Tag 00UZ]{StacksProject}.
Let $X_{\Et}$ denote the small {\'e}tale site of $X\in \Sm_{B}$, 
see \cite[Tag 021A]{StacksProject}.
Its underlying category of {\'e}tale $X$-schemes with coverage given by $\tf$-coverings 
form a site we denote by $X_{\Et}^\tf$.

If $\mathcal S$ is a site we denote its associated topos by $\wt\calS$.
Note that every object $X\in \mathcal S$ defines a morphism of topoi 
$X\colon\mathrm{Sets}\rightarrow \wt\calS$ by sending a sheaf $\calF$ to $\calF(X)$.

\begin{prop}
\label{prop:Proptftop}
Let $B$ be a separated scheme of finite Krull dimension.
The following properties hold for the $\tf$-topology on $\Sm_{B}$.
\begin{itemize}
\item[(i)]
Any covering on the small Zariski site of $B$ is a $\tf$-covering.
\item[(ii)]
The $\tf$-topology is generated by {\'e}tale coverings $\widetilde X\to X$ such that  for every $\sigma\in B$, the following is satisfied:
(1) there exists a lifting in the diagram
\[
\xymatrix{
& \widetilde X \ar[d] \\
X_{\sigma}:=X\times_{B}\sigma \ar@{.>}[ur] \ar[r] & X,
}
\]
and (2) there exists a Zariski open neighborhood $U$ of $\sigma$ in $B$ such 
that the induced morphism $\widetilde X\times_B U\to X\times_B U$ is affine.
\item[(iii)]  
If $B$ is a field, 
then the $\tf$-topology on $\Sm_{B}$ is trivial.
\item[(iv)]
For each $X\in\Sm_{B}$ and $\sigma\in B$,
the functor $\calF\mapsto \calF(X^h_\sigma)$ 
defines a $\tf$-point denoted $X_\sigma^{\tf}$.
\item[(v)]
For $X\in\SmAff_{B}$ and $\sigma\in B$, 
the filtering system of $\tf$-neighborhoods of $X_\sigma$ is cofinal in 
the filtering system of Nisnevich neighborhoods of $X_\sigma$.
\item[(vi)] 
For $X\in \Sm_B$,
a morphism of $\tf$-sheaves of sets $\calF\to \calG$ on $\Sm_{B}$ is an isomorphism 
on $X$-sections if for every $\sigma\in B$ and an \'etale $X$-scheme $U$,
the $\tf$-point $U_\sigma^{\tf}$ induces an isomorphism of stalks 
$\calF_{U_\sigma^{\tf}}\to \calG_{U_\sigma^{\tf}}$.
\item[(vii)]
There is a naturally induced conservative family of morphisms of topoi
\[
\{
\widetilde{\Sm}_\sigma^\tf\to \widetilde{\Sm}^\tf_B
\}_{\sigma\in B}.
\]
Moreover,
the naturally induced family of morphisms of topoi
\begin{equation}\label{eq:Set-Xsigma>toposSmtfsigma->toposSmtf}
\mathrm{Sets}\xrightarrow{X_\sigma^{\tf}}
\widetilde{\Sm}_\sigma^\tf
\to 
\widetilde{\Sm}^\tf_B
\end{equation}
indexed by all $X\in \Sm_{B}$ and $\sigma\in B$ forms a conservative set of points for $\widetilde{\Sm}^\tf_B$.

The same holds for $\widetilde{\SmAff}^\tf_\sigma$ and $\widetilde{\SmAff}^\tf_B$.
\item[(viii)]
For every $X\in \Sm_{B}$, the naturally induced family of morphisms of topoi
\begin{equation}\label{eq:Set-Usigma>toposEttfsigma->toposEttf}
\mathrm{Sets}\xrightarrow{U_\sigma^{\tf}}
\widetilde{(X_\sigma)}^\tf_{\Et}\to
\widetilde{X}^\tf_{\Et}
\end{equation}
indexed by all $U\in X_{\Et}$ and $\sigma\in B$ forms a conservative set of points for $\widetilde{X}^\tf_{\Et}$.

The same conclusion holds for every $X\in\SmAff_B$ and the $\tf$-site of affine {\'e}tale $X$-schemes.
\end{itemize}
\end{prop}
\begin{proof}
We begin with (i). 
If $X$ is an open subscheme of $B$ it suffices to show that every Zariski covering $\coprod_{i=0}^{l} U_i\to X$ 
by affine open subschemes $U_i$ is a $\tf$-covering.
We set $X_{l-1}\defeq\bigcup_{i=0}^{l-1}U_i$ and proceed by induction on $l$.
Since $B$ (and $X$) are separated, the morphism $U_l\to X$ is affine. 
Thus $X_{l-1}\cap U_l$, $X_{l-1}$, $U_l$ and $X$ form a $\tf$-square, 
i.e., 
$X_{l-1}\amalg U_l\to X$ is a $\tf$-covering. 
Since $\coprod_{i=0}^{l-1}U_i\to X_{l-1}$ is a $\tf$-covering by assumption, this finishes the proof of (i).
\vspace{0.1in}

To prove (ii) we let $\tf^\prime$ be the topology generated by the given {\'e}tale covering $\widetilde X\to X$.
For any $\tf$-square \eqref{equation:Nistfsquare} the morphism $(X-Y)\amalg X^\prime\to X$ is a $\tf^\prime$-covering, 
and hence $\tf^\prime\supset \tf$.
Conversely, 
we claim that every $\tf^\prime$-covering $\widetilde X\to X$ is a $\tf$-covering.
By definition, 
$\tf^\prime$ contains the Zariski topology on the small Zariski site over $B$, 
which by (i) is contained in the $\tf$-topology.
Hence we may assume $B$ is a local scheme of finite Krull dimension.

The claim holds trivially if $B=\emptyset$, since then $\widetilde X\cong X=\emptyset$.
Suppose the claim holds inductively for all base schemes of dimension less than $\dim B$, 
and let $\sigma\in B$ be 
a closed point.
The lifting $X^h_\sigma\to \widetilde X$, 
afforded by the definition of a $\tf^\prime$-covering, 
gives rise to a lifting $X^\prime\to \widetilde X$ for an \'etale neighborhood of $X_\sigma$ in $X$ given by an 
affine morphism $X^\prime\to X$. 
Note that $(X-X_\sigma)\amalg X^\prime\to X$ is a $\tf$-covering.
Since $\widetilde X\times_B X^h_\sigma\to X^h_\sigma$ has a right inverse, 
it is a covering in the trivial topology and thus also in the $\tf$-topology.
The inductive assumption implies that $\widetilde X\times_B(B-\sigma)\to X\times(B-\sigma)\cong (X-X_\sigma)$ 
is a $\tf$-covering.
In summary, 
$\widetilde X\times_X ((X-X_\sigma)\amalg X^h_\sigma)\to (X-X_\sigma)\amalg X^h_\sigma\to X$ is a composition of 
$\tf$-coverings, 
and we are done with (ii).
\vspace{0.1in}

Part (iii) follows since if $B$ is a field, then any closed subscheme $Z$ of $B$ is either empty or equals $B$.
That is, 
for any $\tf$-square \eqref{equation:Nistfsquare}, 
either $X-Y=X$ or $X^\prime=X\amalg (X^\prime-X)$.
Hence any $\tf$-covering over $B$ admits a section.
\vspace{0.1in}

We proceed with the characterization of points in the $\tf$-topology.
For all $X\in\Sm_B$ and $\sigma\in B$ we claim the only $\tf$-covering of $X^h_\sigma$ is the identity, 
and moreover for every $\tf$-covering $\widetilde X\to X$ there exists an \'etale neighborhood of $X_\sigma$ in $X$ 
given by an affine morphism $X^\prime\to X$ that factors through $\widetilde X$.
In effect, 
we consider a $\tf$-square as in \eqref{equation:Nistfsquare}. 
If $\sigma\not\in Z$ then $Y=\emptyset$ and there is a canonically induced morphism $X^h_\sigma\to X-Y$. 
If $\sigma\in Z$, 
then $X^\prime\to X$ is an affine \'etale morphism and an \'etale neighborhood of $X_\sigma$ in $X^h_\sigma$. 
In this case there exists a lifting $X^h_\sigma\to X^\prime$. 
Since the $\tf$-topology is a $cd$-topology generated by $\tf$-squares, 
it follows that any $\tf$-covering $\widetilde X\to X^h_\sigma$ admits a lifting $X^h_\sigma\to \widetilde X$.
Hence there exists a lifting $X^\prime\to \widetilde X$ for an \'etale neighborhood of $X_\sigma$ given by an 
affine morphism $X^\prime\to X$.
It follows that the functor $\calF\mapsto \calF(X^h_Z)$ defines a point in the $\tf$-topology for any presheaf 
$\calF$ on $\EssSm_B$.
Moreover, 
the functor $\calF\mapsto \varinjlim_{X^\prime}\calF(X^\prime)$ defines a point in the $\tf$-topology on $\Sm_B$.
This proves part (iv).
\vspace{0.1in}

To prove (v), note that any $\tf$-neighborhood is an Nisnevich neighborhood.
Conversely, 
suppose $\widetilde X\to X$ is an \'etale morphism defining a Nisnevich neighborhood of $X_\sigma$.
We may assume that $B$ is local so that $X$ is affine.
Then there exist an affine Zariski neighborhood $\widetilde X^\prime$ of the subscheme 
$\widetilde X_\sigma=\widetilde X\times_B \sigma$ of $\widetilde X$.
Since the map $X_\sigma\to \widetilde X$ factors through $\widetilde X_\sigma$, 
it also factors through $\widetilde X^\prime$.
Hence 
$\widetilde X^\prime$ is a $\tf$-neighborhood of the subscheme $X_\sigma$ of $X$.
\vspace{0.1in}

We assume (vi) holds for all base schemes of dimension less than $\dim B$ and proceed by induction.
Since any affine Zariski covering $\widetilde B\to B$ is a $\tf$-covering by \Cref{ex:tftopcovering}(i),
a morphism $\calF(X)\to \calG(X)$ is an isomorphism if and only if for every $\sigma\in B$, 
the morphism $\calF(X\times_B{B_\sigma})\to \calG(X\times_B{B_\sigma})$ is an isomorphism.
So we may assume $B$ is local.
Applying the inductive assumption to an \'etale $X$-scheme $U$ implies 
$\calF(U\times_B(B-z))\cong \calG(U\times_B(B-z))$, 
where $z\in B$ is the closed point of $B$.
Since $\calF(X^h_z)\cong \calG(X^h_z)$ by applying our assumption to $U=X$ and $\sigma=z$, 
we deduce the desired isomorphism $\calF(X)\cong \calG(X)$ from \Cref{lm:tfsquareglueing}.
\vspace{0.1in}

Parts (vii) and (viii) follow from (iv) and (vi) since for any $X\in \Sm_B$ the functor in
\eqref{eq:Set-Xsigma>toposSmtfsigma->toposSmtf} is given by $\calF\mapsto \calF(X^h_\sigma)$, 
and for $U\in \Et_X$ the functor in \eqref{eq:Set-Usigma>toposEttfsigma->toposEttf} is given by 
$\calF\mapsto \calF(U^h_\sigma)$.
\end{proof}

\begin{remark}
We thank one of our referees for suggestion the terminology that a $\tf$-henselian local 
$B$-scheme is a $B$-scheme of the form $X^h_\sigma$ for some $X\in \Sm_B$, 
$\sigma\in B$.
\end{remark}

\begin{remark}\label{rem:tfsquare:etalemorphismafineness}
The affine trivial fiber topology on $\Aff_B$, \Cref{rem:tftopologySmBAffBSmAffB}, 
is the strongest subtopology of the affine Nisnevich topology, 
which is trivial over fields. 
Moreover, 
it is probably the weakest subtopology of the Nisnevich topology for which localization holds 
in the sense of \Cref{th:LocNisZSHI}.
The strongest subtopology of the Nisnevich topology on $\Sch_S$ which is trivial over fields 
is generated by squares as in \Cref{definition:tfsquares} without the affineness assumption.
According to our definition, 
the latter topology is stronger than the trivial fiber topology.  
The trivial fiber topology, 
with the affineness condition, 
is defined so that the $\tf$-points are of the form $X^h_z$ for $z\in B$.
\end{remark}

\subsection{Base change for points}
\label{subsection:bcfp}

Any morphism of base schemes $f\colon B^\prime\to B$ induces a canonical base change functor
\begin{equation}
\label{equation:basechangefupperstar}
f^*\colon\SH_{s}(B)\to \SH_{s}(B^\prime).
\end{equation}

\begin{lemma}
\label{lm:infZarCoveringBaseChange}
The canonical functor
\begin{equation}
\label{eq:basechangeBcoprodBsigma}
\SH_{s}(B)
\to 
\prod\limits_{\sigma\in B}
\SH_{s}(B_\sigma)
\end{equation}
detects $\A^{1}$-local objects on the subcategory of $\tf$-local objects.
\end{lemma}
\begin{proof}
Let $\calF\in \in \SH_{s}(B)$ be a $\tf$-local object so that the image along 
the functor \eqref{eq:basechangeBcoprodBsigma} is $\A^1$-local. 
For any $\tf$-local $\calF\in \SH_{s}(B)$ the presheaf of $S^{1}$-spectra 
$\calF(\A^{1}\times -)$ is $\tf$-local.
Thus to conclude $\calF$ is $\A^{1}$-local,
it suffices to prove there is a canonically induced isomorphism 
\begin{equation}
\label{equation:detecting}
\calF(X^h_\sigma)\cong \calF(\A^{1}\times X^h_\sigma) 
\end{equation}
for all $X\in \Sm_B$, $\sigma\in B$.
We note that $X^h_\sigma$ is a $B_\sigma$-scheme, see \eqref{eq:proafhenselizationXhY}.
Hence,  
the isomorphism \eqref{equation:detecting} follows because the image of $\calF$ 
under $\SH_{s}(B)\to \SH_{s}(B_\sigma)$ is $\A^{1}$-local by assumption.
\end{proof}

\begin{lemma}
\label{lm:BaseChangeCommutes}
Suppose $B$ is a base scheme and $f\colon B^\prime\to B$ is an open immersion.
Then $f^*$ in \eqref{equation:basechangefupperstar} commutes with $L_{\A^{1}}$, $L_{\nis}$, $L_{\tf}$.
Consequently, 
the functor \eqref{eq:basechangeBcoprodBsigma}
commutes with $L_{\nis}$, $L_{\tf}$, $L_{\A^{1}}$.
In addition, 
the functors
\begin{equation*}
\begin{array}{lcll}
\SH_{s}(\Fr_+(B))&\to& &\SH_{s}(\Fr_+(B^\prime)),\\
\SH_{s}(\Fr_+(B))
&\to& 
\prod\limits_{\sigma\in B}&
\SH_{s}(\Fr_+(B_\sigma))
\end{array}
\end{equation*}
preserve quasi-stable radditive 
objects.
\end{lemma}
\begin{proof}
The functor $f^*$ preserves Nisnevich local equivalences and $\tf$-local equivalences because it is a left 
Quillen adjoint with right adjoint $f_*$.
Since $f$ is an open immersion, 
$f^*$ is also a right Quillen adjoint with left adjoint $f_\#$ for the Nisnevich local and $\tf$-local projective model 
structures on $\Spt_s(\Sm_B)$ and $\Spt_s(\Sm_{B^\prime})$.
Hence $f^*$ preserves local projective fibrant objects.
Thus $f^*$ commutes with $L_{\nis}$ and $L_{\tf}$.
In more details, 
we first use that
the functor
\begin{equation}
\label{eq:XBprimeB}
\Sm_{B^\prime}\to \Sm_B;
(X\to B^\prime)
\mapsto 
(X\to B^\prime \overset{f}{\to} B)
\end{equation} 
preserves \'etale morphisms and open immersions, 
and secondly that  
\eqref{eq:XBprimeB} preserves fibre products,
because $f$ is an open immersion.
This implies that
\eqref{eq:XBprimeB}
preserves Nisnevich- and $\tf$-squares.
Thus, 
$f_\#$,
which is determined by \eqref{eq:XBprimeB}, 
preserves Nisnevich local and $\tf$-local equivalences.
For $L_{\A^{1}}$, 
we note that \eqref{eq:XBprimeB} commutes with the endofunctors on $\Sm_{B^\prime}$ and 
$\Sm_B$ given by $X\mapsto X\times_B\Delta^n_B\cong X\times_{B^\prime}\Delta^n_{B^\prime}$.
\Cref{lm:EssSmContNistfcov} implies the claim for \eqref{eq:basechangeBcoprodBsigma}.

The final claim follows because the functor $\Fr_+(B)\to \Fr_+(B^\prime)$,  
defined as in \eqref{eq:XBprimeB}, 
preserves coproducts, 
and sends the framed correspondence $\sigma_X$ in $\Fr_+(B^\prime)$, 
for $X\in \Sm_{B^\prime}$ as in \Cref{example:quasistable}, 
to $\sigma_X$ in $\Fr_+(B)$.
\end{proof}

In particular, 
the functor in \eqref{eq:basechangeBcoprodBsigma} preserves $\A^{1}$-local objects.

\begin{remark}
Using similar arguments one can show that if $f\colon B^\prime\to B$ is a Zariski covering, 
then $f^*\colon\SH_{s}(B)\to \SH_{s}(B^\prime)$ detects $\A^{1}$-local objects on the subcategory of $\tf$-local objects.
\end{remark}

\section{The Nisnevich- and \texorpdfstring{$\tf$}{tf}-topology for closed immersions}
\label{section:SmScZ}

For a closed immersion $Z\not\hookrightarrow B$, 
\index{Categories of $B$-schemes!$\EssSm_{B,Z}$} $\EssSm_{B,Z}$ is the full subcategory of $\EssSmB$ 
spanned by schemes of the form $X^h_Z:=X^h_{X\times Z}$, where $X\in\EssSm_B$, 
as in \eqref{eq:proafhenselizationXhY}. 
The universal property of the henselization \eqref{eq:proafhenselizationXhY} leads to the adjunction 
(the upper functor is the left adjoint)
\begin{equation}\label{eq:adj:SmBandSmBcZ}\begin{array}{lclcl}
    & \EssSm_B &\rightleftarrows & \EssSmBcZ & \\
    & X & \mapsto & X^h_Z  & \\    
    & X^h_Z & \mapsfrom & X^h_Z &
\end{array}\end{equation}
We can transfer the symmetric monoidal structure on $\EssSm_B$ to one on $\EssSmBcZ$ via the adjunction 
\eqref{eq:adj:SmBandSmBcZ} as follows.
Note that $\EssSm_B$ inherits a symmetric monoidal structure from that on $\Sch_B$ 
(see also \Cref{rem:productonEssSm}).
The counit of the adjunction \eqref{eq:adj:SmBandSmBcZ} is a natural isomorphism on account of 
\Cref{lm:productoverBcZ}.  
Hence the adjunction \eqref{eq:adj:SmBandSmBcZ} is a reflection
(and one may view $\EssSmBcZ$ as the localization of $\EssSm_B$ 
with respect to the morphisms $X\to Y$ inducing an isomorphism $X^h_Z\to Y^h_Z$ on henselizations).
By appealing to Kelly's  
doctrinal adjunction
\cite[Theorem 3.1]{zbMATH03522182} and \Cref{lm:productoverBcZ},
the monoidal product $\times_B$ on $\EssSm_B$ induces a monoidal product 
\index{Categories of $B$-schemes!$\times_{B,Z}$} $\times_{B,Z}$ on $\EssSmBcZ$ given by 
 \begin{equation}
\label{eq:timesBZ}
Y^h_Z\times_{B,Z} W^h_Z 
:=
(Y^h_Z\times_B W^h_Z)^h_Z\cong
(Y\times_B W)^h_Z.
\end{equation}
See also \Cref{def:productoverBcZ}.

We define \index{Categories of $B$-schemes!$\Sm_{B,Z}$} $\Sm_{B,Z}$ as the full subcategory of $\EssSmB$
spanned by $X^h_Z:=X^h_{X\times Z}$, $X\in\Sm_B$, 
as in \eqref{eq:proafhenselizationXhY}.
The symmetric monoidal product on $\EssSm_{B,Z}$ defined in \eqref{eq:timesBZ} 
restricts to a symmetric monoidal product on $\SmBcZ$ 
because $X\times_B Y\in \Sm_B$ for all $X,Y\in\Sm_B$.

For later reference, 
we introduce the Nisnevich- and the $\tf$-topology on $\Sm_{B,Z}$.
\begin{definition}
\label{definition:NiswNis}
The Nisnevich topology on $\SmBcZ$ is the strongest topology $\tau$ for which
\begin{equation}\label{eq:SmZotSmbcZ}
\begin{tikzcd}[row sep=-0.25em]
&\SmZ^{\nis} \ar{r} & \SmBcZ^{\tau} \\ 
&X_Z &\ar[l,mapsto] \XhZ 
\end{tikzcd}
\end{equation}
defines a morphism of sites, see \Cref{rem:eq:SmZotSmbcZ} 
(here, $X_Z:=X\times_B Z$ as in \Cref{subsubsectionList}).
Similarly, 
the weak Nisnevich topology (=$w\nis$ topology) on $\SmBcZ$ is the weakest topology $\tau$ for which we have a morphism 
of sites
\begin{equation}\label{eq:SmBcZotSmB}
\begin{tikzcd}[row sep=-0.25em]
\SmBcZ^{\tau} \ar{r} & \SmB^{\nis} \\ 
\XhZ & \ar[l,mapsto] X.
\end{tikzcd}
\end{equation}
\end{definition}

\begin{remark}\label{rem:eq:SmZotSmbcZ}
Any object of $\SmBcZ$ has the form $\XhZ$ for some $X\in \SmB$, 
and there is an isomorphism $\XhZ\times_B Z\simeq X_Z$.
Thus the functor \eqref{eq:SmZotSmbcZ} is well defined.
\end{remark}

We have the following precise descriptions for coverings in $\nis$ and $w\nis$ topologies.
A \emph{Nisnevich covering} of $X^h_Z\in\SmBcZ$ is a morphism $\wX^h_Z\to X^h_Z$ such that the morphism $\wX_Z\to X_Z$ defined by \eqref{eq:SmZotSmbcZ} is a 
Nisnevich covering in $\Sm_Z$.
A \emph{weak Nisnevich covering} of $X^h_Z\in\SmBcZ$ is a morphism $\wX^h_Z\to X^h_Z$ such that there is some
morphisms $\wX\to X$ in $\Sm_B$ that is a Nisnevich covering in $\Sm_B$ and maps to $\wX^h_Z\to X^h_Z$ along the functor \eqref{eq:SmBcZotSmB}.

\begin{prop}
\label{proposition:NissimeqWNisSmBcZ}
The Nisnevich and weak Nisnevich topologies on $\SmBcZ$ coincide.
Hence there is a canonical isomorphism of sites 
\[
\SmBcZ^\nis
\xrightarrow{\cong}
\SmBcZ^{w\nis}.
\]
\end{prop}
\begin{proof}
Let $\wX^h_Z\to X^h_Z$ be a weak Nisnevich covering in $\SmBcZ$.
Since the functor $\Sm_{B}\to \Sm_{Z}$ given by $X\mapsto X_{Z}$ preserves Nisnevich coverings, 
$\wX_Z\to X_Z$ is a Nisnevich covering in $\Sm_{Z}$.
By definition, 
this means $\wX^h_Z\to X^h_Z$ is a Nisnevich covering in $\SmBcZ$.

Conversely, 
it follows
by \Cref{lm:AffSmHenselianLift}
that
every morphism $\wX^h_Z\to X^h_Z$ in $\SmBcZ$ is obtained from a morphism $f\colon \wX\to X$ in $\Sm_B$. 
If $\wX^h_Z\to X^h_Z$ is a Nisnevich covering,
then the morphism $f_Z\colon \wX_Z\to X_Z$ is a Nisnevich covering in $\Sm_{Z}$.
Since $\wX,X\in \Sm_B$ and $f_Z$ is \'etale, the morphism $f$ is \'etale over $\wX_Z$.
Moreover, 
since $f_Z$ is a Nisnevich covering it follows that $\wX\amalg X_{B-Z}\to X$ is a Nisnevich covering.
Thus ${\wX}^h_Z\amalg (X_{B-Z})^h_Z\to X^h_Z$ is a Nisnevich covering in $\Sm_{B,Z}$. 
This implies ${\wX}^h_Z\to X^h_Z$ is a weak Nisnevich covering in $\Sm_{B,Z}$ since $(X_{B-Z})^h_Z=\emptyset$. 
\end{proof}

\begin{corollary}
\label{lm:topologically-XtoXhs}
The functor $\SmB\to \SmBcZ$ given by $X\mapsto \XhZ$ preserves coverings and points for the Nisnevich topology.
\end{corollary}

\begin{proof}
By definition, 
the functor $\SmB\to \SmBcZ$ given by $X\mapsto \XhZ$ takes Nisnevich coverings and points in $\SmB$ 
to weak Nisnevich coverings and points in $\SmBcZ$.
\Cref{proposition:NissimeqWNisSmBcZ} finishes the proof.
\end{proof}

Owing to \Cref{lm:topologically-XtoXhs}, 
there are induced site morphisms given by 
\[
\begin{tikzcd}[row sep=-0.25em]
\SmBcZ^{\nis} \ar{r} & \SmB^{\nis} & \ar{l}\SmBmZ^{\nis} \\ 
\XhZ & \ar[l,mapsto] X  \ar[r,mapsto] & X\times_B (B-Z).
\end{tikzcd}
\]

\begin{remark}
The results in this section hold for the Zariski topology on $\Sm_{B,Z}$ 
(defined analogously to the Nisnevich topology on $\Sm_{B,Z}$).
\end{remark}

\begin{definition}
\label{definition:tf}
A \emph{$\tf$-covering} of $X^h_Z\in\SmBcZ$ is a morphism $\wX^h_Z\to X^h_Z$ induced by a $\tf$-covering $\wX\to X$ in $\Sm_B$.
\end{definition}

Alternatively, 
the $\tf$-topology on $\SmBcZ$ is the weakest topology $\tau$ for which we have a morphism of sites
\begin{equation}\label{eq:iBZ}
\begin{tikzcd}[row sep=-0.25em]
\SmBcZ^{\tau} \ar{r} & \SmB^{\tf} \\ 
\XhZ & \ar[l,mapsto] X.
\end{tikzcd}
\end{equation}

For $\sigma\in B$, 
we write $B_\sigma$ for the local scheme of $B$ at $\sigma$ and $\Sm_{B,\sigma}$ for $\Sm_{B_\sigma,\sigma}$.

\begin{lemma}
\label{lmSmBcsigmatftopologytrivial}
For $\sigma\in B$, and $X\in \Sm_{B}$,
any $\tf$-covering $v\colon \widetilde{X^h_Z}\to X^h_Z$ in $\Sm_{B,Z}$ admits a section over $X^h_\sigma$, i.e., there is a morphism $X^h_\sigma\to \widetilde{X^h_Z}$ that composition with $v$ equals the morphism $X^h_\sigma\to X^h_Z$.
\end{lemma} 
\begin{proof}
It follows from the definition that the $\tf$-topology on $\Sm_{B,Z}$ is defined by 
the cd-structure formed by the images along the functor $\Sm_B\to \Sm_{B,Z}$ of the 
$\tf$-squares \eqref{equation:Nistfsquare} for all closed immersions $Z\not\hookrightarrow B$.
Given $Z$,
applying the functor $\Sm_B\to \Sm_{B,Z}$ yields a cartesian square:
\[
\xymatrix{
{X^\prime}^h_Z \times_{B,Z} (B-Z)\ar[r]\ar[d] & {X^\prime}^h_Z\ar[d]\\ 
X^h_\sigma\times_{B,Z} (B-Z)\ar[r] & X^h_Z 
}
\]
We are going to construct a section of the morphism 
\begin{equation}
\label{eq:tfsquarehZcovering}
{X^\prime}^h_\sigma\amalg X^h_\sigma\times_{B,\sigma} (B-Z)
\to 
X^h_\sigma.
\end{equation}
If $\sigma\not\in Z$, 
then $X_\sigma\cap X_Z=\emptyset$ and consequently $X^h_\sigma\times_{B,\sigma} (B-Z)\cong X^h_\sigma$. 
Thus \eqref{eq:tfsquarehZcovering} admits a section.
If $\sigma\in Z$, 
then $X_\sigma\subset X_Z$,
and consequently $(X^h_Z)^h_\sigma\cong X^h_\sigma$.
Since $X^\prime\to X$ is affine \'etale and induces an isomorphism $X^\prime\times_B Z\cong X\times_B Z$, 
there is a canonical morphism $X^h_Z\to X^\prime$.
Now the composite 
\[
X^h_\sigma\cong (X^h_Z)^h_\sigma
\to 
(X^\prime)^h_\sigma
\]
induces a section to \eqref{eq:tfsquarehZcovering}.
This shows every $\tf$-covering in $\Sm_{B,\sigma}$ admits a section, 
and hence the $\tf$-topology is trivial.
\end{proof}

\begin{prop}
\label{lm:tf-topologically-XtoXhs}
The functor $\SmB\to \SmBcZ$ given by $X\mapsto \XhZ$ preserves coverings and points for the $\tf$-topology.
\end{prop}
\begin{proof}
The first claim holds by definition.
To prove the second claim we use that every $\tf$-point of $\Sm_B$ is of the form $X^h_\sigma$ for some $X\in\Sm_{B}$, 
$\sigma\in B$, 
see \Cref{prop:Proptftop}(vii).
If $\sigma\not\in Z$, 
then $X_\sigma\cap X_Z=\emptyset$ and $(X^h_\sigma)^h_Z=\emptyset$.
If $\sigma\in Z$ we have $(X^h_\sigma)^h_Z\cong X^h_\sigma$. 
It remains to check that $X^h_\sigma$ defines a $\tf$-point of $\Sm_{B,Z}$.
By \Cref{lmSmBcsigmatftopologytrivial}, 
the claim follows. 
\end{proof}

\section{Reduction to smooth affine schemes with trivial vector bundles} 
\label{section:ReductionSmat}

Let \index{Categories of $B$-schemes!$\SmatB$} $\SmatB$ denote the subcategory of $\SmAff_B$
spanned by schemes $X$ such that $T_X\cong\calO_X^n$ for some $n\in \mathbb Z_{\geq 0}$.
Here ``cci'' refers to ``clopen subschemes of complete intersections.''
$\SmatBcZ$ is the subcategory of $\EssSm_B$ spanned by the schemes $X_Z^h$ for all $X\in\Smat_{B}$.

In what follows, 
we consider the injective model structure \index{Model categories!$\Spc_s(\Sm_B)$} $\Spc_s(\Sm_B)$ of simplicial presheaves on $\Sm_B$.
The cofibrations are the monomorphisms, 
and the equivalences are sectionwise equivalences \cite[\S5.1]{Jardine-local}.
There exist fibrant replacement functors in the model structures %
\index{Model categories!$\Spc_{s,\nis}(\Sm_B)$}
\index{Model categories!$\Spc_{s,\nis}(\Sm_B)$}
\index{Model categories!$\Spc^{\A^{1}}_{s,\triv}(\Sm_B)$}
$\Spc_{s,\nis}(\Sm_B)$, 
$\Spc_{s,\tf}(\Sm_B)$, 
and $\Spc^{\A^{1}}_{s,\triv}(\Sm_B)$.
On the associated homotopy category $\HHtriv(\Sm_B)$, 
we have the corresponding Nisnevich, $\tf$, and $\A^{1}$-localization endofunctors
\[
\Lrep_\nis, \Lrep_\tf, \Lrep_{\A^{1}}\colon \HHtriv(\Sm_B)\to \HHtriv(\Sm_B).
\]
The subcategories $\HH_{\nis}(\Sm_B)$, $\HH_{\tf}(\Sm_B)$, $\HH_{\A^{1}}(\Sm_B)$ of $\HHtriv(\Sm_B)$ are spanned by 
Nisnevich local, $\tf$-local, and $\A^{1}$-local objects. 
These categories arise as Bousfield localizations of $\Spc_s(\Sm_B)$,
with respect to Nisnevich local equivalences, 
$\tf$-local equivalences, and $\A^{1}$-equivalences.
The Morel-Voevodsky motivic homotopy category $\HH_{\mot}(\Sm_B)$ is equivalent to 
$\HH_{\nis}(\Sm_B)\cap \HH_{\A^{1}}(\Sm_B)$ \cite[\S7.2]{Jardine-local}, 
\cite{Morel-Voevodsky}.
Similarly, 
we have $\HHtriv(\SmBcZ)$ and $\HHtriv(\SmatBcZ)$ and the said Bousfield localizations.

Since any scheme in $\SmBcZ$ has a Zariski covering in $\SmatBcZ$,
the categories of Zariski sheaves on $\SmBcZ$ and $\SmatBcZ$
are equivalent, 
and consequently,
there is an equivalence between motivic homotopy categories
\begin{equation*}
\HH_\mot(\SmBcZ)\simeq \HH_\mot(\SmatBcZ).
\end{equation*} 
The aim of this section is to show a refined adjunction before applying motivic localization 
see \Cref{prop:AffSmScZSmat,prop:SmAffSmniseqivalence}.
Recall that $\Lrep_{\A^{1}}^{[1]}=\calF(\Delta^\bullet_{B,Z}\times_{B,Z}-)_f$, 
where $(-)_f$ is the fibrant replacement, 
cf.~\eqref{eq:LAl->LA} for the $\A^{1}$-localization endofunctor $\Lrep_{\A^{1}}$ on $\HHtriv(\SmBcZ)$.
Here $\times_{B,Z}$ is defined in \eqref{eq:timesBZ},  
$\A^{1}_{B,Z}$ is shorthand for $(\A^{1}_B)^h_Z$ as in \eqref{eq:proafhenselizationXhY}, 
and $\Delta^\bullet_{B,Z}$ is the cosimplicial scheme defined in Section \ref{equation:cosimplicial}.
We want to study the adjunctions 
\begin{equation}
\label{eq:H(Smat)H(AffSm)H(Sm)}
\begin{tikzcd}
\HH_{\Nis}(\SmatBcZ)
\ar[r,shift left,"l^\cci",shorten <= 0.3em,shorten >= 0.3em] & 
\HH_{\Nis}(\SmAffScZ)
\ar[l,shift left,"r^\cci",shorten <= 0.3em,shorten >= 0.3em]
\ar[r,shift left,"l^{\aff}",shorten <= 0.3em,shorten >= 0.3em] & 
\HH_{\Nis}(\SmBcZ), 
\ar[l,shift left, "r^{\aff}",shorten <= 0.3em,shorten >= 0.3em]
\end{tikzcd}
\end{equation}
and the behavior of functors $r^\cci$ and $r^{\aff}$ with respect to the $\A^{1}$-local objects. 
In \eqref{eq:H(Smat)H(AffSm)H(Sm)} the right adjoints $r^\aff$ and $r^\mathrm{cci}$ are restriction functors. 
Moreover, 
$l^\aff$ and $l^\mathrm{cci}$ are defined as left Kan extensions along the full embeddings
\begin{equation}
\label{eq:SmatAffSmSm}
\SmatBcZ\to \SmAffScZ\to \SmBcZ.
\end{equation}
In what follows we show that the unit maps in \eqref{eq:H(Smat)H(AffSm)H(Sm)} are natural isomorphisms.
Owing to \Cref{lm:adjunctioninunitconservative},
whose proof is left to the reader, 
it remains to show that the right adjoints $r^\aff$ and $r^\mathrm{cci}$ 
are conservative.

\begin{lemma}
\label{lm:adjunctioninunitconservative}
Suppose 
\begin{equation}
\label{equation:FGadjunction}
F\colon \mathcal C\xrightleftarrows{1em} \mathcal D\colon G
\end{equation}
is an adjunction, 
where the unit $\eta\colon \id_{\mathcal C}\to G\circ F$ is a natural isomorphism and the 
right adjoint functor $G$ is conservative. 
Then \eqref{equation:FGadjunction} is an equivalence of categories.
\end{lemma}

\begin{lemma}
\label{lm:SmSmatZarCov}
If $B$ is an affine scheme, 
then every $X\in\SmBcZ$ admits a Zariski covering by schemes in $\SmatBcZ$.
Thus if $X\in \SmAffScZ$ (resp.~$X\in \SmatBcZ$), 
then every Nisnevich covering of $X$ in $\SmBcZ$ admits a refinement by schemes in $\SmAffScZ$ (resp.~$\SmatBcZ$).
\end{lemma}
\begin{proof}
First, 
we consider the case when $\Sm_{B,Z}=\Sm_B$.
By the assumption on $B$ every $X\in\Sm_{B}$ admits a Zariski covering by affine $B$-schemes.
Since the tangent bundle of a smooth scheme is locally trivial, 
we can refine the Zariski covering by schemes in $\Smat_{B}$.
In the general case, let $X^h_Z\in\Sm_{B,Z}$, where $X\in \Sm_B$. 
The previous case implies there is a Zariski covering and hence a Nisnevich covering $Y\to X$, 
where $Y\in \Smat_B$.
The naturally induced morphism $Y^h_Z\to X^h_Z$ is a Zariski covering and a Nisnevich covering in $\Sm_{B,Z}$ by definition.
\end{proof}

By \eqref{eq:productXhZShZYhZ} and \Cref{lm:prodSmAffSmatBcZ} 
the fibre product in $\SmBcZ$ restricts to the ones in $\SmAffBcZ$, and $\SmatBcZ$,
and consequently,
the Nisnevich topology and $\tf$-topology on $\SmBcZ$ restrict to 
the ones on $\SmAffBcZ$, and $\SmatBcZ$.
\begin{lemma}
\label{lm:raffrcciLA1LNisLtf}
If $B$ is an affine scheme then the right adjoint functors $r^\aff$ and $r^\cci$ commute with the 
localization endofunctors $\Lrep_{\nis}$, $\Lrep_{\tf}$, $\Lrep_{\A^{1}}$.
\end{lemma}
\begin{proof}
As explained above the lemma
the functors in \eqref{eq:SmatAffSmSm} preserve fiber products, 
and also 
Nisnevich and $\tf$-coverings by definition.
Hence the right adjoints $r^\aff$ and $r^\cci$ preserve Nisnevich local objects.
Owing to \Cref{lm:SmSmatZarCov}, 
if $X\in\SmAffScZ$, 
then every Nisnevich covering $Y\to X$ in $\SmBcZ$ admits a refinement in $\SmAffScZ$.
This shows that $r^\aff$ commutes with $\Lrep_\nis$.
If $f\in \mathcal O(B)$ is a regular function, 
then $X\times_B (B-Z(f))\in\SmAffSmZ$ and every {\'e}tale neighborhood of $X\times_B Z(f)$ admits 
a refinement in $\SmAffScZ$.
This shows every $\tf$-covering of $X$ admits a refinement in $\SmAffScZ$.
It follows that $r^\aff$ commutes with $\Lrep_{\tf}$.
If $Y\to X$ is an {\'e}tale morphism in $\SmAffBcZ$ and $X\in \SmatBcZ$, then $Y\in \SmatBcZ$. 
In particular, 
this applies to Nisnevich and $\tf$-coverings,  
and thus $r^\cci$ commutes with $\Lrep_\nis$ and $\Lrep_\tf$.

By \Cref{lm:prodSmAffSmatBcZ}
the endofunctor $X^h_Z\mapsto\A^{1}_{B,Z}\times_{B,Z} X^h_Z$ on $\SmBcZ$, see \eqref{eq:timesBZ}, 
preserves the subcategories
$\SmAffBcZ$ and $\SmatBcZ$.
The latter cylinders commute with $\SmAffScZ\to \SmBcZ$ and $\SmatBcZ\to \SmAffScZ$ in \eqref{eq:SmatAffSmSm}.
By reference to Section \ref{equation:cosimplicial} the same holds for the endofunctors $\Delta^n_{B}\times -$ and 
$\Delta^n_{B,Z}\times -$.
It follows that $\Lrep_{\A^{1}}$ commutes with $r^\aff$ and $r^\cci$.
\end{proof}

\begin{lemma}
\label{lm:XSmXprimeSmatreatect}
If $B$ is an affine scheme then every $X\in\SmAffScZ$ is a retract of some $V\in\SmatBcZ$.
\end{lemma}
\begin{proof}
Since $X\in \SmAffScZ$ we may choose a vector bundle $\xi$ on $X$ such that $T_X\oplus \xi$ is trivial.
The total space of $\xi$ yields an $\A^{1}$-equivalence $p\colon V\to X$, 
where $T_{V}$ is trivial. 
We have $p\circ z=\id_X$, 
where $z\colon X\to V$ is the zero section.
By \Cref{lm:SmatBcZdefinitionsequivalence} $V\in\SmatBcZ$.
\end{proof}

\begin{remark}\label{rem:affsttrivandloctriv}
Any affine scheme with a stably trivial tangent bundle is $\A^{1}$-equivalent to a scheme with a trivial tangent bundle. 
Any affine local scheme has trivial tangent bundle.
\end{remark}

If $f$ is an endomorphism of a simplicial presheaf $\mathcal X$ on $\SmBcZ$, 
we can form the homotopy colimit
\begin{equation}\label{eq:XfcolimXfXf}
\mathcal X[f^{-1}]
\defeq
\hocolim(\mathcal X\xrightarrow{f}\mathcal X\xrightarrow{f} \cdots).
\end{equation}
Note that $f$ induces an equivalence on $\mathcal X[f^{-1}]$.
If $f$ is an equivalence on $\mathcal X$, 
then $\mathcal X\simeq \mathcal X[f^{-1}]$.
We use the same notation for an object $X\in \SmBcZ$ and its images in 
$\Spc_s(\SmBcZ)$ and 
$\HHtriv(\SmBcZ)$.

\begin{lemma}
\label{lm:retractSmH(Sm)}
Suppose $i\colon Y\not\hookrightarrow X$ is a closed immersion in $\SmBcZ$ and $p\colon X\to Y$ is a morphism 
such that $p\circ i=\id_Y$.
Then for $f:=i\circ p\colon X\to X$ there is a canonically induced isomorphism 
\begin{equation*}
\label{eq:Y->X[ip^-1]}
i_\infty
\colon 
Y
\xrightarrow{\cong} 
X[f^{-1}]
\end{equation*}
in $\HHtriv(\SmBcZ)$.
\end{lemma}
\begin{proof}
Since $f$ is an idempotent and $f\circ i=i$, 
there are naturally induced morphisms 
\[
i_\infty\colon Y\simeq Y[\mathrm{id}_Y^{-1}]\to X[f^{-1}],
\,\
p_\infty\colon X[f^{-1}]\to Y[\mathrm{id}_Y^{-1}]\simeq Y.
\]
By assumption the composite $p_\infty\circ i_\infty$ is the identity on $Y$.
Moreover,
the composite $i_\infty\circ p_\infty$ is induced by $f$, 
and hence it is an equivalence on $X[f^{-1}]$.
\end{proof}

\begin{lemma}
\label{lm:SmSmatconservative}
If $B$ is an affine scheme, 
then the restriction functor 
\[
r^\cci
\colon 
\HHtriv(\SmAffScZ)
\to 
\HHtriv(\SmatBcZ)
\]
is conservative. 
\end{lemma}
\begin{proof}
Suppose $\calF\to \mathcal G$ is a morphism in $\HHtriv(\SmAffScZ)$ such that $r^\cci$ induces an isomorphism
\begin{equation}
\label{eq:rccicalFrccicalG}
r^\cci(\calF)
\xrightarrow{\cong} 
r^\cci(\mathcal G)
\in
\HHtriv(\SmatBcZ).
\end{equation} 
\Cref{lm:XSmXprimeSmatreatect} shows that $X\in \SmBcZ$ is a retract of some $V\in\SmatBcZ$. 
Let $z\colon X\to V$ and $p\colon V\to X$ be morphisms with compositions $p\circ z=\id_X$ and $f=z\circ p$.
From \eqref{eq:rccicalFrccicalG} we deduce an isomorphism between hom groups
\begin{equation}
\label{equation:radjoint}
[V[f^{-1}],r^\cci(\calF)]_{\HHtriv(\Smat_{B,Z})}
\cong 
[V[f^{-1}],r^\cci(\mathcal G)]_{\HHtriv(\Smat_{B,Z})}.
\end{equation}
Applying the left adjoint $l^\cci\colon \HHtriv(\SmatBcZ)\to \HHtriv(\SmAffBcZ)$, 
which commutes with homotopy colimits,
yields
\[
l^{\cci}(V)
\simeq 
V
\in 
\HHtriv(\SmAffBcZ),
\,\,
l^{\cci}(V[f^{-1}])
\simeq 
V[f^{-1}]
\in 
\HHtriv(\SmAffBcZ).
\]
Combined with \eqref{equation:radjoint} we deduce the isomorphism 
\begin{equation*}
\label{eq:VfcalFcalG}
[V[f^{-1}],\calF]_{\HHtriv(\SmAff_{B,Z})}
\cong 
[V[f^{-1}],\mathcal G]_{\HHtriv(\SmAff_{B,Z})}.
\end{equation*}
Finally,
using the isomorphism $X\simeq V[f^{-1}]\in \HHtriv(\SmAffBcZ)$ from \Cref{lm:retractSmH(Sm)},
we get
\[
[X,\calF]_{\HHtriv(\SmAffScZ)}
\cong
[V[f^{-1}],\calF]_{\HHtriv(\SmAffScZ)} 
\cong
[V[f^{-1}],\mathcal G]_{\HHtriv(\SmAffScZ)}
\cong
[X,\mathcal G]_{\HHtriv(\SmAffScZ)}.
\]
It follows that $\calF\to \mathcal G$ is an isomorphism in $\HHtriv(\SmAffScZ)$.
\end{proof}

\begin{prop}
\label{prop:AffSmScZSmat}
If $B$ is an affine scheme, 
then the restriction functor 
\[
r^\cci
\colon 
\HHtriv(\SmAffScZ)
\to 
\HHtriv(\SmatBcZ)
\]
is an equivalence. 
Moreover, 
$r^\cci$ commutes with $\Lrep_\nis$, $\Lrep_\tf$, and $\Lrep_{\A^{1}}$. 
In particular, $r^\cci$ preserves and detects Nisnevich local and $\A^{1}$-local objects. 
\end{prop}
\begin{proof}
The first claim follows from \Cref{lm:adjunctioninunitconservative,lm:SmSmatconservative}.
To conclude for $\Lrep_\nis$, $\Lrep_\tf$, and $\Lrep_{\A^{1}}$ we use \Cref{lm:raffrcciLA1LNisLtf}.
\end{proof}

\begin{lemma}
\label{lm:SmSmaffA1surjectiveandpreserves}
If $B$ is an affine scheme, 
then the restriction functor 
\[
r^\aff\colon
\HHtriv(\SmBcZ)\to \HHtriv(\SmAff_{B,Z})
\]
preserves $\A^{1}$-local objects and it is essentially surjective on $\A^{1}$-local objects.
\end{lemma}
\begin{proof}
The first claim follows since for every $X\in\SmAff_{B,Z}$ the cylinder $\A^{1}\times X$ is in $\SmAff_{B,Z}$.
Suppose $\calF\in \HHtriv(\SmAff_{B,Z})$ is $\A^{1}$-local.
Via left Kan extension we deduce the adjunction:
\begin{equation}
\label{equation:rlaffadjunction}  
l^\aff
\colon 
\HHtriv(\SmAff_{B,Z})
\xrightleftarrows{1em}
\HHtriv(\SmBcZ)
\colon 
r^\aff
\end{equation}
Then $\Lrep_{\A^{1}}l^\aff(\calF)\in \HHtriv(\SmBcZ)$ is $\A^{1}$-local, 
and there are isomorphisms 
\[
 r^\aff(\Lrep_{\A^{1}}l^\aff(\calF))
\cong 
\Lrep_{\A^{1}}r^\aff(l^\aff(\calF))
\cong \Lrep_{\A^{1}}\calF
\cong 
\calF.
\]
The first isomorphism follows from \Cref{lm:raffrcciLA1LNisLtf}, 
and the second follows because the unit map of \eqref{equation:rlaffadjunction} is a natural isomorphism.
\end{proof}

\begin{lemma}
\label{lm:SmAffSmCovCheck}
If $B$ is an affine scheme, 
then the restriction functors 
\[
\HH_\nis(\SmBcZ)\to \HH_\nis(\SmAff_{B,Z})\to \HH_\nis(\SmatBcZ) 
\] 
are conservative. 
The same holds for $\HH_\Zar(\SmBcZ)$.
\end{lemma}
\begin{proof}
Suppose $\calF\to \mathcal G$ is a morphism in $\HH_\nis(\SmBcZ)$ or $\HH_\nis(\SmAffScZ)$ that maps to an 
isomorphism in $\HH_\nis(\SmatBcZ)$, 
i.e., 
$\calF(X)\cong\mathcal G(X)$ for all $X\in\SmatBcZ$.
\Cref{lm:SmSmatZarCov} implies that the same holds for all $X\in\SmBcZ$ or $X\in\SmAffScZ$. 
\end{proof}

\begin{prop}
\label{prop:SmAffSmniseqivalence}
If $B$ is an affine scheme, 
then the restriction functor 
$$
\HH_\nis(\SmBcZ)\to \HH_\nis(\SmAff_{B,Z})
$$
is an equivalence.
Moreover, 
the same functor preserves and detects $\A^{1}$-local objects.
The same results hold for $\HH_\Zar(\SmBcZ)$, 
$\HH_\nis(\Fr_+(B,Z))$,
and $\HH_\Zar(\Fr_+(B,Z))$.
\end{prop}
\begin{proof}
The unit map of the adjunction \eqref{equation:rlaffadjunction} is a natural isomorphism.
Since $\HHtriv(B,Z)\rightarrow\HHtriv(\SmAff_{B,Z})$ is Nisnevich exact, 
the same holds for the unit map of the adjunction
\begin{equation}
\label{equation:nisadjunction}
\HH_\nis(\SmAff_{B,Z})\xrightleftarrows{1em}\HH_\nis(\SmBcZ).
\end{equation} 
\Cref{lm:SmAffSmCovCheck} implies \eqref{equation:nisadjunction} is an equivalence of categories.
\Cref{lm:raffrcciLA1LNisLtf} shows the right adjoint in \eqref{equation:nisadjunction} commutes with $\Lrep_{\A^{1}}$;
hence it preserves and detects $\A^{1}$-local objects.
The remaining cases are similar and left to the reader.
\end{proof}

\begin{lemma}
\label{lm:HHmotBwtBconsetw}
If $f\colon \widetilde B\to B$ is a Nisnevich covering, 
then the base change functor
\[
\HH_\nis(\SmBcZ)
\to 
\HH_\nis(\Sm_{\widetilde B,\widetilde Z})
\]
is conservative and preserves $\A^{1}$-local objects.
Here we write $\widetilde Z$ for the fiber product $Z\times_B \widetilde B$.

The same holds for $\HH_\nis(\SmAffScZ)$ and $\HH_\nis(\SmatBcZ)$.
\end{lemma}
\begin{proof}
Suppose the morphism $\calF\to \mathcal G$ in $\HH_\nis(\SmBcZ)$ maps to an isomorphism in 
$\HH_\nis(\Sm_{\widetilde B,\widetilde Z})$.
For every $X\in\Sm_{B,Z}$ and $x\in X$ there exists a lifting $X^h_x\to \widetilde B$ of $X^h_x\to B$ along $f$.
Hence $X^h_x$ is an essentially smooth local henselian scheme over $B$.
Thus by assumption there is an isomorphism $\calF(X^h_x)\cong\mathcal G(X^h_x)$.
This shows the base change functor is conservative.
The second assertion follows since $\Sm_{\widetilde B,\widetilde Z}\to \SmBcZ$ maps
$X\times_{\widetilde B,\widetilde Z}\A^{1}_{\widetilde B,\widetilde Z}\to X$ to $X\times_{B,Z}(\A^{1})_{B,Z}\to X$.

The arguments for $\HH_\nis(\SmAffScZ)$ and $\HH_\nis(\SmatBcZ)$ are similar.
\end{proof}

\section{\texorpdfstring{$\A^{1}$}{A1}-locality and rigidity for closed immersions}
\label{sect:SmBlZA1rigid}

To a fixed closed immersion $Z\not\hookrightarrow B$ such that an open affine neighborhood $B^\prime\subset B$ 
of $Z$ exists, 
we associate a full subcategory $\SmBlZ$ of $\Sch_B$.
Our aim in this section is to establish a close connection between $\SmBlZ$ and $\Sm_Z$ through $\A^{1}$-local 
and rigid objects.
This input is pivotal for the proof of our main result.

\subsection{The category \texorpdfstring{$\SmBlZ$}{smblz} and its Nisnevich site.}

\begin{definition}
Let \index{Categories of $B$-schemes!$\SmBlZ$}
$\SmBlZ$ be the full subcategory of $\Sch_B$ spanned by objects of the form $\XtZ$ and $\XhZ$ for all $X\in \Sm_B$. 
The category \index{Categories of $B$-schemes!$\SmAffSlZ$} $\SmAffSlZ$ (resp.~$\SmatBlZ$) is defined similarly subject to the condition $X\in\SmAff_B$ 
(resp.~$X\in\SmatB$).
\end{definition}

We note that the assumption in the following result holds for closed points.
\begin{lemma}\label{lm:LiftSmatuniqueness}
Assume an open affine neighborhood $B^\prime\subset B$ of $Z$ exists.
For $X^h_Z, (X^\prime)^h_Z\in \SmatBcZ$ and isomorphism $r\colon X_Z\cong (X^\prime)_Z$ 
there exists an isomorphism $e\colon X^h_Z\cong (X^\prime)^h_Z$
in $\SmatBcZ$ such that $i_{B,Z}(e)=r$.
Here $i_{B,Z}$ is the base change functor $\SmatBcZ\to\SmatZ; (-)^h_Z\mapsto (-)_Z$, 
inducing the site morphism \eqref{eq:iBZ}.
\end{lemma}

\begin{proof}
By \Cref{lm:AffSmHenselianLift} there exists a morphism $e\colon (X^\prime)^h_Z\to X^h_Z$ in $\SmatBcZ$
which is \'etale because 
$(X^\prime)^h_Z$ and $X^h_Z$ are essentially smooth schemes of the same relative dimension over $B$, 
and $e$ is \'etale over $X_Z$. 
Next,
by \Cref{lm:AffSmHenselianLift}
applied to 
the closed immersion $Z^\prime\not\hookrightarrow (X^\prime)^h_Z$ and the \'etale morphism $e$, 
there exists a morphism $e^{-1}\colon X^h_Z\to (X^\prime)^h_Z$ such that $e\circ e^{-1}$ is the identity.
By the same reasoning, 
$e^{-1}$ is \'etale and there exists a morphism $(e^{-1})^{-1}\colon (X^\prime)^h_Z\to X^h_Z$ such that 
$e^{-1}\circ (e^{-1})^{-1}$ is the identity.
It follows that
$e= e\circ e^{-1}\circ (e^{-1})^{-1}= (e^{-1})^{-1}$.
\end{proof}

\begin{lemma}
\label{lm:LiftSmat}
Assume an open affine neighborhood $B^\prime\subset B$ of $Z$ exists.
Then for every $Y\in \SmatZ$,
there exists an object $\tilde Y \in\SmatBcZ$, unique up to isomorphism, such that $Y\cong\tilde Y\times_B Z$.
Furthermore, 
there exists $\widetilde Y \in\Smat_{B}$ such that $Y\cong\widetilde Y \times_B Z$.
\end{lemma}

\begin{proof}
Choose a closed immersion $Y\not\hookrightarrow \A^l_Z$ with trivial normal bundle, 
and choose polynomials $f_1,\dots, f_r$ with vanishing locus  
\[
Z(f_1,\dots, f_r)=Y\amalg Y^\prime\subset \A^l_Z
\]
for some $Y'$. Since $B^\prime$ is affine, 
the restriction morphism $\mathcal O(\A^l_{B^\prime})\to \mathcal O(\A^l_Z)$ is surjective on global sections.
Say $\widetilde f_i\in \mathcal O(\A^l_{B^\prime})$ maps to $f_i\in \mathcal O(\A^l_Z)$. 
Then 
\[
Z(f_1,\dots, f_r)\cong Z\times_{B^\prime} Z(\widetilde f_1, \dots, \widetilde f_r).
\]  
where $Z(\widetilde f_1, \dots, \widetilde f_r)\subset \A^l_{B^\prime}$ is the vanishing locus.
By \Cref{cor:henselianpairsplitting} we have a decomposition
\[
Z(\widetilde f_1, \dots, \widetilde f_r)^h_Z\cong\tilde Y\amalg \tilde Y^\prime,
\] 
where $Y\cong \tilde Y\times_B{Z}$.
Since the normal bundle $N_{Y/\A^l_{Z}}$ is trivial and the tangent bundle $T_Y$ is stably trivial, 
$N_{\tilde Y/(\A^l_B)^h_Z}$ is trivial and $T_{\tilde Y}$ is stably trivial.
This shows that $\tilde Y\in \SmatBcZ$ by \Cref{lm:SmatBcZdefinitionsequivalence}. 
The uniqueness follows by \Cref{lm:LiftSmatuniqueness}.
By the definition of $\SmatBcZ$, see \Cref{convnotations:schemescategories}, 
there is $\widetilde Y\in\Smat_B$ such that 
$\tilde Y\cong\widetilde Y^h_Z$.
\end{proof}

\begin{corollary}
\label{cor:LiftSmat}
Assume an open affine neighborhood $B^\prime\subset B$ of $Z$ exists.
There is a fully faithful functor $\SmatZ\to \SmatBlZ$.
\end{corollary}
\begin{proof}
The functor is the identity morphism on objects.
\Cref{lm:LiftSmat} shows that every $Y\in\SmatZ$ is isomorphic to $\widetilde Y\times_B Z$ for some 
$\widetilde Y\in \SmatB$. 
\end{proof}

\begin{definition}\label{rem:rmorpihsmofsites}
The Nisnevich topology on $\SmBlZ$ is the weakest topology $\tau$ for which the following functors 
define site morphisms
\[
\begin{tikzcd}[row sep=-0.25em]
\Sm^\Nis_Z & \Sm^\tau_{B*Z}\ar{r}{u}\ar{l}[swap]{r} & \Sm_{B,Z}^\Nis\\
Y\ar[r,mapsto,"r^{-1}"] & Y,X_Z^h & X_Z^h. \ar[l,mapsto,swap,"u^{-1}"]
\end{tikzcd}
\]
Similarly, for $\SmatBlZ$ and $\SmAffBlZ$.
\end{definition}

\begin{remark}
A morphism $\widetilde X\to X$ in $\SmBlZ$ is a Nisnevich covering if and only if
it is an {\'e}tale morphism of 
$B$-schemes and $\widetilde{X}_Z\to X_Z$ is a Nisnevich covering in $\SmZ$.
\end{remark}

\begin{remark}\label{rem:restrictionscommuteSmBlZSmAffSmat}
The morphism of sites $r$ in \Cref{rem:rmorpihsmofsites} induces morphisms of sites 
$\SmAffBlZ\to\SmAff_Z$, $\SmatBlZ\to\Smat_Z$.
Hence restriction along $r$ commutes with $r^\cci$ and $r^\aff$ in \eqref{eq:H(Smat)H(AffSm)H(Sm)}.
\end{remark}

\begin{lemma}
\label{lemma:Sm-diagrams-recipe}
The site morphism $\SmZ^\nis\rightarrow\SmB^\nis$ factors as
\begin{equation}
\begin{tikzcd}[row sep=-0.25em]\label{eq:Sm-diagrams-recipe}
\Sm_Z^\Nis\ar{r}{i_{B*Z}} & \Sm_{B*Z}^\Nis\ar{r}{u} & \Sm_{B,Z}^\Nis\ar{r} & \Sm_B^\Nis\\
Y,X_Z & Y,X_Z^h\ar[l,mapsto,swap,"i_{B*Z}^{-1}"] & X_Z^h\ar[l,mapsto,swap,"u^{-1}"] & X,\ar[l,mapsto]
\end{tikzcd}\end{equation}
where $X_Z:=X^h_Z\times_B Z$.
\end{lemma}

\subsection{$\A^{1}$-local and rigid presheaves on $\SmBlZ$.}
\label{subsect:Rigid}

\Cref{prop:LrepA1uuscommute} shows that the functor 
$i^*_{B,Z}$ commutes with the $\A^{1}$-localization functor $\Lrep_{\A^{1}}$. 
The key inputs are \Cref{lm:liftingsA1homotopy} and \Cref{lm:hocolimLrepA1SetPresheaf}. 
A central technical tool is what we refer to as a telescope construction
\[
\Cyl^{B,Z}_K \mathbf{X}\in \Pre(\SmatBcZ),
\]
which associates a presheaf on $\SmatBcZ$ to a diagram of schemes, 
or more generally to a morphism of simplicial sets
\[
\mathbf X\colon K\rightarrow N(\SmatBcZ).
\]

In \Cref{def:CylBcZcalCX,def:CylBcZX}, we define the telescope construction as the composite of two procedures. First, we associate to a given diagram of $B$-schemes a simplicial $B$-scheme (see \Cref{def:affinespacessimplcialsiagram}). Second, we apply a relative geometric realization to obtain a presheaf from this simplicial $B$-scheme (see \Cref{def:geomretricrealisation}).
In \Cref{sectapp:Homotopycolimits}, we prove that this construction provides a model for the corresponding homotopy colimit of $\A^1$-localizations of simplicial presheaves on $\SmatBcZ$.
For brevity, we write $\Cyl^{Z}_{-}$ for $\Cyl^{Z,Z}_{-}$.

For any $U\in \Smat_B$,
so that $U$ is affine, 
we can form the diagram of closed immersions
\begin{equation}
\label{eq:diagItoDeltaAUhZ}
(\partial\Delta^n_{B}\times U)^h_Z \nothookleft
\partial\Delta^n_{B}\times U_Z\not\hookrightarrow 
(\Delta^n_B\times U)_Z.
\end{equation}
By \cite[Corollary 3.9]{Schwede-gluing}, 
the corresponding pushout 
\begin{equation}
\label{eq:colimItoDeltaAUhZ2}
(\partial\Delta^n_{B}\times U)^h_Z\amalg_{(\partial\Delta^n_{B}\times U_Z)} (\Delta^n_{B}\times U)_Z
\end{equation}
exists in the category of $B$-schemes.
Moreover, 
since colimits in the category of schemes preserve henselian pairs \cite[Tag 0EM6]{StacksProject}, 
the scheme in \eqref{eq:colimItoDeltaAUhZ2} is henselian with respect to its fiber $(\Delta^n_{B}\times U)_Z$ over $Z$.
Hence \eqref{eq:colimItoDeltaAUhZ2} is an object in $\SmatBcZ$.  
There is a naturally induced closed immersion
\begin{equation}
\label{eq:ItoDeltaAUhZ}
(\partial\Delta^n_{B}\times U)^h_Z\amalg_{(\partial\Delta^n_{B}\times U_Z)} (\Delta^n_B\times U)_Z
\not\hookrightarrow 
(\Delta^n_{B}\times U)^h_Z.
\end{equation}
More generally,
for a morphism of simplicial sets $\mathbf U\colon K\to N(\SmatBcZ)$ to the nerve of $\SmatBcZ$, 
let $\mathbf U_Z$ denote the composite $K\to N(\SmatBcZ)\to N(\SmatZ)$, 
where the second morphism is induced by $X\mapsto X\times_B Z$.
Then any monomorphism of simplicial sets $e\colon L\to K$ induces a monomorphism
\begin{equation}
\label{eq:Cyl(Y)BZcupCylX_Z}
\Cyl^{B,Z}_L \mathbf U e \amalg_{\Cyl^{Z}_L \mathcal U_Z e} (\Cyl^{Z}_K\mathbf U_Z)
\hookrightarrow 
\Cyl^{B,Z}_K\mathbf U
\end{equation}
in the presheaf category $\Pre(\SmatBlZ)$.

\begin{lemma}
\label{lm:liftingsA1homotopy}
For $X\in \Smat_B$, $U\in \SmatBcZ$ and \eqref{eq:ItoDeltaAUhZ}, 
there exists a lifting in all diagrams in $\Sch_{B,Z}$ of the form:
\begin{equation}
\label{eq:trilift-bItXhZ}
\xymatrix{
(\partial\Delta^n_{B}\times U)^h_Z\amalg_{(\partial\Delta^n_{B}\times U_Z)} (\Delta^n_B\times U)_Z
\ar[d]\ar[r] & \XhZ\\
(\Delta^n_{B}\times U)^h_Z \ar@{.>}[ur].
}
\end{equation}
The analogous lifting property holds in $\Fr_+(\Sch_{B,Z})$.

Moreover, 
for morphisms of simplicial sets $L\stackrel{e}{\hookrightarrow} K\xrightarrow{\mathbf U} N(\SmatBcZ)$, 
\eqref{eq:Cyl(Y)BZcupCylX_Z} induces a surjection
\begin{equation}
\label{eq:trilift-ICylXhZ}
\Psi(\Cyl^{B,Z}_K\mathbf U)
\to
\Psi(\Cyl^{B,Z}_L\mathbf U\amalg_{\Cyl^{Z}_L\mathbf U_Ze} \Cyl^{Z}_K\mathbf U_Z).
\end{equation}
Here $\Psi\colon \Pre(\SmatBlZ)\to \Set$ is the functor given by  
$$
\calF
\mapsto
\Pre(\SmatBlZ)(\calF,\SmatBlZ(-,\XhZ))
$$

The analogous result holds in the category $\Fr_+(\SmatBlZ)$. 
\end{lemma}

\begin{proof}
To prove the claim for \eqref{eq:trilift-bItXhZ} we begin by applying the lifting property 
for affine henselian schemes with respect to smooth affine morphisms, 
see \Cref{lm:AffSmHenselianLift}, 
to the diagram in $\Sch_B$:
\begin{equation}
\label{eq:trilift-bItXhZ2}
\xymatrix{
(\partial\Delta^n_{B}\times U)^h_Z\amalg_{(\partial\Delta^n_{B}\times U_Z)} (\Delta^n_B\times U)_Z
\ar[d]\ar[r] & X\\
(\Delta^n_{B}\times U)^h_Z \ar@{.>}[ur].
}
\end{equation}
The indicated lifting in \eqref{eq:trilift-bItXhZ2} exists since \eqref{eq:ItoDeltaAUhZ} is an affine henselian pair.
By applying the functor $(-)^h_Z\colon\Sch_B\to\Sch_{B,Z}$ to \eqref{eq:trilift-bItXhZ2}, 
we deduce the desired lifting in \eqref{eq:trilift-bItXhZ} since $(-)^h_Z$ preserves the morphism 
\eqref{eq:ItoDeltaAUhZ} in $\Sch_{B,Z}$.
The same proof applies in the framed setting by reference to \Cref{lm:liftFr}.

The above argument proves also \eqref{eq:trilift-ICylXhZ} when $K=\Delta^n$, $L=\partial\Delta^n$;
that is, the map 
\begin{equation}\label{eq:Cyllifting-simplexcase}
\Psi((\Delta^n_{B}\times U)^h_Z)
\to
\Psi((\partial\Delta^n_{B}\times U)^h_Z\amalg_{(\partial\Delta^n_{B}\times U_Z)} (\Delta^n_B\times U)_Z)
\end{equation}
is surjective. 
The case of $L=K$ holds tautologically.
Suppose $K=L\cup \alpha$, where $\alpha\in K$ is a simplex such that $\partial \alpha\subset L$.
We write $U=\mathbf U(\alpha)$ and deduce the surjectivity of \eqref{eq:trilift-ICylXhZ} from 
\eqref{eq:Cyllifting-simplexcase} owing to the pullback square:
\begin{equation*}\xymatrix{
\Psi(\Cyl^{B,Z}_K\mathbf U)\ar[d]
\ar[r]&
\Psi(\Cyl^{B,Z}_L\mathbf U e \amalg_{\Cyl^{Z}_L\mathbf U_Z e} \Cyl^{Z}_K\mathbf U_Z)\ar[d]\\
\Psi((\Delta^n_{B}\times U)^h_Z)
\ar[r]&
\Psi((\partial\Delta^n_{B}\times U)^h_Z\amalg_{(\partial\Delta^n_{B}\times U_Z)} (\Delta^n_B\times U)_Z)
}
\end{equation*}
The general case of a monomorphism of simplicial sets $L\to K$ follows now by an induction argument 
on the number of non-degenerate simplices of the complement $K-L$.
\end{proof}

\begin{definition}
An object $\calF\in\HHtriv(\SmBlZ)$ is $\A^{1}$-local if for all $X\in\Sm_{B}$ there are naturally 
induced equivalences of simplicial sets
\[
\calF(X_Z)
\xrightarrow{\simeq}
\calF(\A^{1}\times X_Z),\quad
\calF(\XhZ)
\xrightarrow{\simeq}
\calF((\A^{1}\times X)^h_Z).\]
\end{definition}
The $\A^{1}$-localization endofunctor $\Lrep_{\A^{1}}$ on $\HHtriv(\SmatBlZ)$ 
and $\HHtriv(\Fr_+(\SmatBlZ))$ is given analogous to the discussion in \Cref{section:ReductionSmat}.
The functor $u^{-1}\colon \Smat_{B,Z}\to \Smat_{B*Z}$ from \Cref{lemma:Sm-diagrams-recipe} induces an adjunction:
\begin{equation}
\label{eq:ulsuus}
u^*\colon 
\HHtriv(\Smat_{B,Z})
\xrightleftarrows{1em} 
\HHtriv(\Smat_{B*Z})
\colon u_*
\end{equation}
If $\calF\in \HHtriv(\Smat_{B,Z})$ then $u^*(\calF)\in \HHtriv(\Smat_{B*Z})$ is the left Kan extension of 
$\calF$ along $u^{-1}$.
The right adjoint is given by $u_*\calF(X^h_Z)=\calF(X^h_Z)$, 
where $\calF\in \HHtriv(\Smat_{B*Z})$ and $X^h_Z\in \Smat_{B*Z}$.
Similarly, 
we have the adjunction
\begin{equation}
\label{eq:ulsuusfr}
u^*_\fr\colon \HHtriv(\Fr_+(\Smat_{B,Z}))\xrightleftarrows{1em} 
\HHtriv(\Fr_+(\Smat_{B*Z}))\colon u_*^\fr.
\end{equation}
Note that $u_*$ coincides with the composite of $u_*^\fr$ and the forgetful functor 
$$
\HHtriv(\Fr_+(\Smat_{B*Z}))\to \HHtriv(\Smat_{B*Z}).
$$

\begin{definition}
An object $\calF\in\HHtriv(\SmBlZ)$ is rigid if for all $X\in\Sm_{B}$ there is a naturally induced 
equivalence of simplicial sets
\[
\calF(\XhZ)\xrightarrow{\simeq}\calF(X_Z).
\]
\end{definition}

\Cref{prop:SmoothLift->A1lrig:origin} proves that the $\A^1$-localization of a 
homotopy colimit of presheaves of the form $h(X^h_Z)$ is rigid.
First we need the following notation.

\begin{definition}
\label{def:A1locorextendedtoSmatBcZ}
For $\calF\in \HHtriv(\SmatBcZ)$, 
let $\calF_{\A^{1}},\calF_{\A^{1}}^{[1]}\in \HHtriv(\Smat_{B*Z})$ be shorthand for $\Lrep_{\A^{1}} u^* (\calF)$, 
$\Lrep_{\A^{1}}^{[1]} u^* (\calF)$.
Similarly, 
if $\calF\in\HHtriv(\Fr_+(\SmatBcZ))$ we write $\calF_{\A^{1}}^\fr$, 
$\calF_{\A^{1}}^{\fr,{[1]}}$ for 
$\Lrep_{\A^{1}} u^*_\fr (\calF)$, $\Lrep_{\A^{1}}^{[1]} u^*_\fr (\calF)$.
\end{definition}

\begin{prop}
\label{prop:SmoothLift->A1lrig:origin}
For any $\calF\in \HHtriv(\SmatBcZ)$ the simplicial presheaf $\calF_{\A^{1}}$ is $\A^{1}$-local and rigid. 
Likewise, 
for any $\calF\in \HHtriv(\Fr_+(\SmatBcZ))$, 
the simplicial presheaf $\calF_{\A^{1}}^\fr$ is $\A^{1}$-local and rigid. 
\end{prop}

\begin{proof}
By appeal to \cite[Lemma 2.8]{Du} we may assume $\calF=\SmatBcZ(-,\XhZ)$, 
$X\in \SmatB$, 
by writing any simplicial presheaf as a homotopy colimit of representable ones.
In this case, 
we have 
\[
u^*(\calF)=\SmatBlZ(-,\XhZ),
\]
and since $\SmatBcZ(-,\XhZ)$ is fibrant in $\Spc_s(\SmatBcZ)$,
we have an isomorphism
\begin{equation}\label{eq:calFA1representedbyXhZ}
\calF_{\A^{1}}^{[1]}(-)
\cong
\SmatBlZ((\Delta^\bullet_B \times -)^h_Z, \XhZ).
\end{equation}
We claim that $\calF_{\A^{1}}^{[1]}(U^h_Z)\to \calF_{\A^{1}}^{[1]}(U \times Z)$ is a trivial fibration of 
simplicial sets for all $U\in \SmAff_B$. 

By \Cref{lm:liftingsA1homotopy} there exists a lifting in every diagram in $\Sch_B$ of the form:
\begin{equation}
\label{eq:diag:trilift-ICylXhZ}
\xymatrix{
(\partial\Delta^n_{B}\times U)^h_Z\amalg_{(\partial\Delta^n_{B}\times U_Z)} (\Delta^n_B\times U)_Z
\ar[d]\ar[r]^-{p} & \XhZ\\
(\Delta^n_{B}\times U)^h_Z \ar@{.>}[ur]
}
\end{equation}

Owing to \eqref{eq:calFA1representedbyXhZ} a lifting in \eqref{eq:diag:trilift-ICylXhZ} is equivalent 
to a lifting in the diagram:
\begin{equation}
\label{eq:deltaDeltaDeltaF(UhZ)F(UZ)}
\xymatrix{
\partial\Delta^n\ar[d]\ar[r] & \calF_{\A^{1}}^{[1]}(U^h_Z) \ar[d]\\ 
\Delta^n \ar@{.>}[ru]\ar[r] & \calF_{\A^{1}}^{[1]}(U_Z) 
}
\end{equation}
Indeed, by \eqref{eq:calFA1representedbyXhZ}, we have 
\[
\sSet_*(\Delta^n, \calF_{\A^{1}}^{[1]}(U^h_Z)) 
\cong
\SmBlZ( \Delta^n_{B,Z}\times_{B,Z} U^h_Z, X^h_Z )
\cong
\Sch_B( (\Delta^n_{B}\times U)^h_Z, X^h_Z ),
\]
and commutative squares of the form \eqref{eq:deltaDeltaDeltaF(UhZ)F(UZ)} 
are in bijection with morphisms $p$ in \eqref{eq:diag:trilift-ICylXhZ}.
Since the left vertical morphism in \eqref{eq:deltaDeltaDeltaF(UhZ)F(UZ)} is a 
generating cofibration of simplicial sets, 
the right vertical morphism is a trivial Kan fibration.
It follows that there is a naturally induced equivalence of simplicial sets
\[
\calF_{\A^{1}}^{[1]}(U^h_Z)
\xrightarrow{\simeq} 
\calF_{\A^{1}}^{[1]}(U_Z).
\] 
By applying $\Lrep_{\A^{1}}$ and using that $\Lrep_{\A^{1}}\simeq \Lrep_{\A^{1}} \Lrep_{\A^{1}}^{[1]}$, 
we get the equivalence
\[
\calF_{\A^{1}}(U^h_Z)
\xrightarrow{\simeq} 
\calF_{\A^{1}}(U_Z).
\] 
This shows that $\calF_{\A^{1}}$ is a rigid presheaf on $\SmBlZ$.
We may use the same approach for $\calF_{\A^{1}}^\fr$.
\end{proof}

For a fixed $Y\in \SmatZ$, 
we consider the comma categories 
\[
\begin{array}{lll}
Y/\SmatZ &:=& \{f\colon Y\to U_Z \mid U_Z\in \SmatZ\}, \\
Y/\SmatBcZ &:=& \{f\colon Y\to U^h_Z \mid U^h_Z\in \SmatBcZ\}. 
\end{array}
\]

\begin{prop}
\label{prop:trfhocol}
For any $X\in\SmatB$, $Y\in \SmatZ$ there is a canonically induced trivial fibration of simplicial sets
\begin{equation}\label{eq:hocolimSmatBcZhocolimSmatZ}
\hocolimABZ_{Y/\SmatBcZ}(X^h_Z)\rightarrow\hocolimAZ_{Y/\SmatZ}(X_Z).  
\end{equation}
Here $X^h_Z$ and $X_Z$ are viewed as representable presheaves on $\SmatBcZ$ and $\SmatZ$, 
respectively; see \Cref{def:hocolimA} for unexplained notation.
Analogous results hold for $\Fr_+(\SmatBcZ)$ and $\Fr_+(\SmatZ)$.
\end{prop}
\begin{proof}
We may identify \eqref{eq:hocolimSmatBcZhocolimSmatZ} with the morphism 
\[
\hocolim^{\Delta_{B*Z}}_{Y/\SmatBcZ}(X^h_Z)
\rightarrow
\hocolim^{\Delta_{B*Z}}_{Y/\SmatZ}(X^h_Z),
\]
where, at both sides, $X^h_Z$ denotes the representable presheaf on $\SmatBlZ$.
We claim that there exists a lifting in any diagram of simplicial sets of the form
\begin{equation}
\label{eq:deltaDeltaDeltaF(UhZ)F(UZ)2}
\xymatrix{
\partial\Delta^n\ar[d]\ar[r] & \hocolim^{\Delta_{B*Z}}_{Y/\SmatBcZ}(X^h_Z) \ar[d]\\ 
\Delta^n \ar@{.>}[ru]\ar[r] & \hocolim^{\Delta_{B*Z}}_{Y/\SmatZ}(X^h_Z). 
}
\end{equation}
The lower and upper horizontal maps in \eqref{eq:deltaDeltaDeltaF(UhZ)F(UZ)2} give rise to the respective elements
\[
\beta\in N(Y/\SmatZ)_n=\sSet(\Delta^n,N(Y/\SmatZ)),\; 
\partial\in \sSet(\partial\Delta^n,N(Y/\SmatBcZ)).
\] 
Since the diagram \eqref{eq:deltaDeltaDeltaF(UhZ)F(UZ)2} commutes, 
we have the equality
\begin{equation}
\label{eq:partialbeta}
\partial_Z 
= 
\partial\beta.
\end{equation} 
Here we write
$\partial\gamma$ for the composite of the canonical inclusion $\partial\Delta^n\to \Delta^n$ and some 
$\gamma\colon \Delta^n\to K$.
Likewise, 
$\gamma_Z\colon K\to N(Y/\SmatZ)$ denotes the composite of some $\gamma\colon K\to N(Y/\SmatBcZ)$ with 
the canonical map $N(Y/\SmatBcZ)\to N(Y/\SmatZ)$.
In view of \Cref{lm:hocolimA(K)}, 
we may replace \eqref{eq:deltaDeltaDeltaF(UhZ)F(UZ)2} with the equivalent lifting problem in $\Pre(\SmatBlZ)$
\begin{gather}
\nonumber 
\label{eq:diag:trilift-ICylXhZ_parial_beta}
\xymatrix{
\Cyl^{B,Z}_{\partial \Delta^n}\partial \amalg_{\Cyl^{Z}_{\partial \Delta^n}\partial} \Cyl^{Z}_{\Delta^n}\beta
\ar@{.>}[d]\ar[r] & \XhZ\\
\Cyl^{B,Z}_{\Delta^n}\alpha \ar@{.>}[ur]
,}
\\\label{eq:alphabetadeta}
\alpha\in N(Y/\SmatBcZ)_n=\sSet(\Delta^n,N(Y/\SmatBcZ)), \quad
\partial = \partial\alpha, \quad
\partial\alpha_Z=\beta.
\end{gather}
Here the simplex $\alpha$ and the dotted maps in \eqref{eq:alphabetadeta} correspond uniquely to the lifting in 
\eqref{eq:deltaDeltaDeltaF(UhZ)F(UZ)2}, 
and the equalities in \eqref{eq:alphabetadeta} are equivalent to the commutativity of the triangles in 
\eqref{eq:deltaDeltaDeltaF(UhZ)F(UZ)2}.

By \Cref{lm:LiftSmat} and \Cref{lm:AffSmHenselianLift} it follows that for every sequence  
\[
\beta 
= 
\left( 
(U_0)_Z\xrightarrow{\overline{f}_1}(U_1)_Z\xrightarrow{\overline{f}_2}\cdots\xrightarrow{\overline{f}_n}(U_n)_Z 
\right)
\]
of morphisms in $\SmatZ$, 
there is a corresponding sequence
\[
\alpha 
= 
\left( 
(U_0)^h_Z\xrightarrow{f_1}(U_1)^h_Z\xrightarrow{f_2}\cdots\xrightarrow{f_n}(U_n)^h_Z 
\right)
\]
such that $\alpha_Z=\beta$.
Moreover, 
the uniqueness assertion in \Cref{lm:LiftSmat} implies $\partial\alpha= \partial$ for any $\partial$ 
that satisfies property \eqref{eq:partialbeta}.
Thus, 
for any $\partial$ and $\beta$ as above,
there exists an $\alpha$ as in \eqref{eq:alphabetadeta}.
To conclude, 
the second claim of \Cref{lm:liftingsA1homotopy} shows that for any such sequence $\alpha$ there exists a lifting 
in the diagram
\begin{equation}
\label{eq:diag:trilift-ICylXhZ_partialalpha}
\xymatrix{
\Cyl^{B,Z}_{\partial \Delta^n}\partial\alpha \amalg_{\Cyl^{Z}_{\partial \Delta^n}\partial\alpha_Z} \Cyl^{Z}_{\Delta^n}\alpha_Z
\ar[d]\ar[r] & \XhZ\\
\Cyl^{B,Z}_{\Delta^n}\alpha \ar@{.>}[ur]
}
\end{equation}
in $\Pre(\SmatBlZ)$.
\end{proof}

If $X^h_Z\in \SmatBcZ$, 
we use the shorthand notation 
\[
\Lrep_{\A^{1}}^{B,Z,[1]}(X^h_Z) := \SmatBcZ(-\times_{B,Z}\Delta^\bullet_{B,Z},X^h_Z).
\]
Similar notations will be used for the categories $\Fr_+(\SmatZ)$ and $\Fr_+(\SmatBcZ)$.

In the formulation of the next result, 
we use the adjunctions $u^*\dashv u_*$, $u^*_\fr\dashv u_*^\fr$ in \eqref{eq:ulsuus}, \eqref{eq:ulsuusfr}, 
respectively, 
and furthermore the adjunctions
\[
i^*_{B,Z} 
\colon 
\HHtriv(\SmatBcZ)
\xrightleftarrows{1em}
\HHtriv(\Smat_Z)
\colon
i^{B,Z}_*,
\]
\[
i^*_{B,Z,\fr}
\colon 
\HHtriv(\Fr_*(\SmatBcZ))
\xrightleftarrows{1em}
\HHtriv(\Fr_*(\Smat_Z))
\colon 
i^{B,Z,\fr}_*.
\]

\begin{prop}
\label{prop:LrepA1uuscommute}
The functors $u^*$, $i^*_{B,Z}$, $u^*_\fr$, $i^*_{B,Z,\fr}$ commute with the $\A^{1}$-localization $\Lrep_{\A^{1}}$.
In particular, 
they preserve $\A^{1}$-local objects.
\end{prop}
\begin{proof}
We show that both $u^*$ and $i^*_{B,Z}$ commutes with $\Lrep_{\A^{1}}^{[1]}$ 
($u^*_\fr$, $i^*_{B,Z,\fr}$ can be dealt with similarly).  
If $\calF\in \HHtriv(\SmatBcZ)$, 
then $u^*\calF\big|_{\SmatBcZ}\cong \calF$ and $u^*\calF\big|_{\SmatBcZ}=i^*_{B,Z}\calF$.
Thus it suffices to prove there is a natural isomorphism 
\[
i^*_{B,Z}(\Lrep_{\A^{1}}^{B,Z,[1]}(\calF))
\cong 
\Lrep_{\A^{1}}^{Z,[1]}(i_{B,Z}^*(\calF))
\]
in $\HHtriv(\SmatZ)$.
Here we use the functors $\Lrep_{\A^{1}}^{Z,[1]}\colon \HHtriv(\SmatZ)\to \HHtriv(\SmatZ)$, 
$\Lrep_{\A^{1}}^{B,Z,[1]} \colon \HHtriv(\SmatBcZ)\to \HHtriv(\SmatBcZ)$.
As in the proof of \Cref{prop:SmoothLift->A1lrig:origin}, 
we may assume that $\calF=\SmatBcZ(-,\XhZ)$, 
$X\in \SmatB$. 
The presheaf $\SmatBcZ(-,\XhZ)$ is fibrant in $\Spc_s(\SmatBcZ)$. 
Hence we are reduced to showing a natural isomorphism
\[
i^*_{B,Z}(\SmatBcZ(-\times_{B,Z}\Delta^\bullet_{B,Z},X^h_Z))
\cong
\SmatZ(-\times_{Z}\Delta^\bullet_{Z},X^h_Z).
\]
Working pointwise, 
in the unstable homotopy category $\HH_s$ we have the isomorphisms 
\begin{gather}
\label{eq:Zvalues}
\Lrep_{\A^{1}}^{Z,[1]}(X_Z)(Y) 
\cong
\hocolim_{U_Z\in Y/\SmatB}\Lrep_{\A^{1}}^{Z,[1]}(X_Z)(U_Z),
\\
\label{eq:uusZvalues}
i^*_{B,Z}(\Lrep_{\A^{1}}^{B,Z,[1]}(X^h_Z))(Y)
\cong 
\hocolim_{U^h_Z\in Y/\SmatBcZ}\Lrep_{\A^{1}}^{B,Z,[1]}(X^h_Z)(U^h_Z).
\end{gather}
On the right-hand sides of \eqref{eq:Zvalues}, \eqref{eq:uusZvalues},  
the objects are homotopy colimits taken in the category of simplicial sets.
We note that \eqref{eq:Zvalues} holds because $Y$ is the initial object of $Y/\SmatZ$.
The isomorphism in \eqref{eq:uusZvalues} holds because $i^*_{B,Z}(\calF(-\times \Delta^\bullet_{B,Z}))$ 
is the left Kan extension 
of $\calF(-\times \Delta^\bullet_{B,Z})=\SmatBcZ(-\times_{B,Z}\Delta^\bullet_{B,Z},X^h_Z)$ along $\SmatBcZ\to \SmatZ$.
For any $Y\in \SmatZ$,
we use \Cref{prop:trfhocol} and 
\Cref{lm:hocolimLrepA1SetPresheaf} to conclude that there 
is a natural isomorphism
\[
\hocolim_{U^h_Z\in Y/\SmatBcZ}\Lrep_{\A^{1}}^{B,Z,[1]}
(X^h_Z)(U^h_Z) 
\xrightarrow{\cong}
\hocolim_{U_Z\in Y/\SmatB}\calF(U_Z).
\]
Thus commutativity holds for $\Lrep_{\A^{1}}^{[1]}$.
For $\Lrep_{\A^{1}}^{[l]}$ and $\Lrep_{\A^{1}}$ we may now 
appeal to \eqref{eq:LAl->LA}.
\end{proof}

\begin{lemma}
\label{lm:iusBlZsimeqrds}
The functor $i^*_{B*Z}\colon \HHtriv(\SmatBlZ)\to 
\HHtriv(\SmatZ)$, see \eqref{eq:Sm-diagrams-recipe} in 
\Cref{lemma:Sm-diagrams-recipe}, 
is isomorphic to the restriction 
$r_*$ along the embedding $r^{-1}\colon \SmatZ\to \SmatBlZ$ 
of \Cref{rem:rmorpihsmofsites}.
The same holds in the framed setting.
\end{lemma}
\begin{proof}
We consider $\SmatBlZ$, $\SmatZ$, 
and reduce to representable functors on the category 
$\SmatBlZ$. 
Since the canonical embedding $\SmatBcZ\amalg \SmatZ\to 
\SmatBlZ$ is essentially surjective, there are two cases: 
\begin{itemize}
\item[(1)] $\calF=\SmatBlZ(-,\XhZ)$, 
$X\in \SmatB$,
\item[(2)]
$\calF=\SmatBlZ(-,Y)$, $Y\in \SmatZ$.
\end{itemize}
Since $i^*_{B*Z}$ is a left Kan extension, 
we have 
\[
i^*_{B*Z}(\SmatBlZ(-,\XhZ))
\cong
\SmatZ(-,X_Z),
\quad
i^*_{B*Z}(\SmatBlZ(-,Y))
\cong
\SmatZ(-,Y).
\]
The claim follows now from the isomorphisms 
\[
\SmatZ(-,X_Z)\cong r_*\SmatBlZ(-,\XhZ), 
\quad
\SmatZ(-,Y)\cong r_*\SmatBlZ(-,Y).
\]
\end{proof}

\begin{lemma}
\label{lm:A1loc_id=iusids(Smat)}
For every $\A^{1}$-local object $\calF\in \HHtriv(\SmatBcZ)$ the unit of the adjunction
\[
i^*_{B,Z}
\colon 
\HHtriv(\SmatBcZ)
\xrightleftarrows{1em}
\HHtriv(\Smat_Z)
\colon 
i^{B,Z}_*
\]
induces a natural isomorphism 
\[
\calF
\xrightarrow{\cong} 
i^{B,Z}_* i_{B,Z}^*(\calF).
\]
The same holds for every $\A^{1}$-local object $\calF\in \HHtriv(\Fr_+(\SmatBcZ))$ in the framed setting.
\end{lemma}
\begin{proof}
The functors $i_{B,Z}^*$ and $i^{B,Z}_*$ admit factorizations as in:
\[
\begin{tikzcd} 
\HHtriv(\SmatBcZ)\ar[rr,bend left,"i^*_{B,Z}"]\ar{r}{u^*} &\HHtriv(\SmatBlZ)\ar{r}{i^*_{B*Z}}& \HHtriv(\Smat_Z)\\
\HHtriv(\SmatBcZ)&\HHtriv(\SmatBlZ)\ar{l}[swap]{u_*}& \HHtriv(\Smat_Z)\ar{l}[swap]{i_*^{B*Z}} \ar[ll,bend left,swap,"i_*^{B,Z}"]
\end{tikzcd}
\]
Here $u^*$ and $u_*$ are obtained from the embedding $\SmatBcZ\to \SmatBlZ$,
while $i^*_{B*Z}$, and $i_*^{B*Z}$ are obtained from the functor $\SmatBlZ\to \Smat_Z$
  given by $X^h_Z,Y\mapsto X_Z, Y$ for
   $X\in \Smat_B$, $Y\in \SmatZ$.
By \Cref{lm:iusBlZsimeqrds}, 
we have $i^*_{B*Z}\cong r_*$ obtained from the embedding $\SmatZ\to \SmatBlZ$.
The unit of the adjunction $\calF\to u_*u^*(\calF)$ is an isomorphism since the functor $u^{-1}\colon \SmatBcZ\to\SmatBlZ$ 
is fully faithful.
Since $\calF$ is $\A^{1}$-local, the presheaf $u^*(\calF)$ is rigid according to 
\Cref{prop:SmoothLift->A1lrig:origin,prop:LrepA1uuscommute}.
Hence there are isomorphisms
\[
u^*(\calF) 
\xrightarrow{\cong} 
i_*^{B*Z} r_* u^*(\calF)
\xrightarrow{\cong} 
i_*^{B*Z} i^*_{B*Z} u^*(\calF),
\] 
and we deduce
\[
i_*^{B,Z}i^*_{B,Z}(\calF)
\cong
u_* i_*^{B*Z} i^*_{B*Z} u^*(\calF)
\cong
u_* u^*(\calF)
\cong 
\calF.
\]
The proof for framed correspondences is similar.
\end{proof}

\begin{lemma}
\label{lm:motloc_id=iusids(Sm)}
For every motivic local object $\calF\in \HHtriv(\SmBcZ)$ the unit of the adjunction
\[
i^*_{B,Z}
\colon 
\HHtriv(\SmBcZ)
\xrightleftarrows{1em}
\HHtriv(\Sm_Z)
\colon 
i_*^{B,Z}
\]
induces an isomorphism 
\[
\calF\xrightarrow{\cong} i_*^{B,Z} i^*_{B,Z}(\calF).
\]
The same holds for every motivic local object $\calF\in \HHtriv(\Fr_+(\SmBcZ))$.
\end{lemma}
\begin{proof}
Owing to \Cref{lm:A1loc_id=iusids(Smat)} the unit of the adjunction induces an isomorphism
\begin{equation}
\label{equation:calFrestrictionequivalence}  
\calF\big|_{\SmatBcZ}
\xrightarrow{\cong}
i_*^{B,Z} i^*_{B,Z}(\calF)\big|_{\SmatBcZ}.
\end{equation}
Our claim follows since $\calF$ is Nisnevich local, 
and every scheme in $\SmBcZ$ (resp.~$\Sm_Z$) has a Nisnevich covering in $\SmatBcZ$ (resp.~$\Smat_Z$).
\end{proof}

\begin{lemma}
\label{lm:A1loc_id=iusids(SmAff)}
For every $\A^{1}$-local object $\calF\in \HHtriv(\SmAffBcZ)$ the unit of the adjunction
\[
(i_{B,Z}^\aff)^*
\colon 
\HHtriv(\SmAffBcZ)
\xrightleftarrows{1em}
\HHtriv(\SmAff_Z)
\colon 
(i_\aff^{B,Z})_*
\]
induces an isomorphism 
\[
\calF\xrightarrow{\cong} i_*^{B,Z} i^*_{B,Z}(\calF).
\]
The same holds for every $\A^{1}$-local object $\calF\in \HHtriv(\Fr_+(\SmAffBcZ))$.
\end{lemma}
\begin{proof}
This follows from \eqref{equation:calFrestrictionequivalence} since the canonical embedding $\SmatBcZ\to \SmAffBcZ$ 
induces an equivalence between the subcategories of $\A^{1}$-local objects in $\HHtriv(\SmBcZ)$ and $\HHtriv(\SmatBcZ)$, 
see \Cref{prop:AffSmScZSmat}.
\end{proof}

We write $\PreA(\SmBlZ)$ and $\PreRig(\SmBlZ)$ for the subcategories of $\HHtriv(\SmBlZ)$ spanned by  
$\A^{1}$-local and rigid presheaves, 
respectively.

\begin{lemma}
\label{lm:PresPreSsEquivalence}
The functor $i_{B*Z}^*$ induces an equivalence of categories
\begin{equation}
\label{eq:rusPrerig(BlZ)Pre(Z)}
i^*_\rigid
\colon 
\PreRig(\SmatBlZ)
\xrightarrow{\simeq} 
\HHtriv(\SmatZ).
\end{equation} 
Moreover, 
both $i^*_\rigid$ and its inverse preserves and detects Nisnevich local objects and Nisnevich local equivalences.
\end{lemma}
\begin{proof}
\Cref{lm:LiftSmat} shows there is a fully faithful embedding $r^{-1}\colon \SmatZ\to \SmatBlZ$.
Due to \eqref{eq:Sm-diagrams-recipe} there are functors of categories
\begin{equation}
\label{eq:ifuSmZSmBlZSmBcZ}
\begin{tikzcd}[row sep=1em]
\SmatZ\ar[r,shift left,"r^{-1}"] & \SmatBlZ\ar[l,shift left,"i_{B*Z}^{-1}"]& \SmatBcZ\ar[l, swap,"u^{-1}"] \\
Y\ar[r,mapsto,"r^{-1}"] & Y & \\
Y,X_Z & Y,X_Z^h\ar[l,mapsto,swap,"i_{B*Z}^{-1}"] & X_Z^h.\ar[l,mapsto,swap,"u^{-1}"]
\end{tikzcd}
\end{equation}
\Cref{lemma:Sm-diagrams-recipe} shows that $r^{-1}$ is a right inverse to $i_{B*Z}^{-1}$, 
i.e., 
$\mathrm{id}_{\SmatZ} \cong i_{B*Z}^{-1}\circ r^{-1}$.
\Cref{lm:iusBlZsimeqrds} shows that $i^*_{B*Z}$ is isomorphic to $r_*$.
Hence the identity functor on $\HHtriv(\SmatZ)$ is naturally isomorphic to 
\begin{equation}
\label{eq:rdsidsHHtriv(SmZSmBlZ)}
\HHtriv(\SmatZ) 
\xrightarrow{(i_{B*Z})_*}\HHtriv(\SmatBlZ)
\xrightarrow{(i_{B*Z})^*} \HHtriv(\SmatZ).
\end{equation} 
Owing to the isomorphism of schemes 
\[
(X^h_Z)\times_B Z \cong
X_Z,
\]
the functor $(i_{B*Z})_*$ takes values in $\PreRig(\SmatBlZ)$.
By \eqref{eq:rdsidsHHtriv(SmZSmBlZ)} there are induced functors
\begin{equation}
\label{eq:rdsidsrigidHHtriv(Z)PreRig(BlZ)composite1}
\HHtriv(\SmatZ) 
\xrightarrow{i^\rigid_*}
\PreRig(\SmatBlZ)
\xrightarrow{i_\rigid^*} 
\HHtriv(\SmatZ),
\end{equation}
and the composite in \eqref{eq:rdsidsrigidHHtriv(Z)PreRig(BlZ)composite1} is naturally isomorphic 
to the identity on $\HHtriv(\SmatZ)$.
Moreover, 
the composite  
\begin{equation}
\label{eq:idsrdsrigidHHtriv(Z)PreRig(BlZ)composite2}
\PreRig(\SmatBlZ)
\xrightarrow{i_\rigid^*} 
\HHtriv(\SmatZ) 
\xrightarrow{i^\rigid_*}
\PreRig(\SmatBlZ)
\end{equation}
is naturally isomorphic to the identity on $\PreRig(\SmatBlZ)$,
since for all $\mathcal F\in \PreRig(\SmatBlZ)$ and $X\in\SmatB$ we have
\[
i^\rigid_* i_\rigid^* \calF(X^h_Z)
=
\calF(r(i_{B*Z}(X^h_Z)))
=
\calF(X_Z)
\cong
\calF(X^h_Z).
\]
Moreover, 
for $Y\in \SmatZ$, 
we have
$
i^\rigid_* i_\rigid^* \calF(Y)
=
\calF(r(i_{B*Z}(Y)))
=
\calF(Y)
$.

Since $i_{B*Z}^{-1}$ and $r^{-1}$ in \eqref{eq:ifuSmZSmBlZSmBcZ} preserve Nisnevich coverings and 
fibered products,
both $(i_{B*Z})_*$ and $r_*$, 
and consequently $i^\rigid_*$, $i_\rigid^*$, 
preserve Nisnevich local objects. 
Since $i^\rigid_*$ and $i_\rigid^*$ are inverses, 
the same functors preserve Nisnevich local equivalences.
\end{proof}

The functors $i_{B*Z}^{-1}$ and $u^{-1}$ from \eqref{eq:ifuSmZSmBlZSmBcZ} induce adjunctions
\begin{equation}\label{eq:Htriviusidsuisuds}
\begin{tikzcd}[row sep=-0.25em]
\HHtriv(\SmatBcZ)\ar[r,shift left,"\uus"] & 
\HHtriv(\SmatBlZ)\ar[r,shift left,"(i_{B*Z})^*"]\ar[l,shift left,"\uds"] & 
\HHtriv(\SmatZ) \ar[l,shift left,"(i_{B*Z})_*"].
\end{tikzcd}
\end{equation}
We shall consider the homotopy categories $\SH_{s}(\SmatB)$, $\SH_{s}(\SmatZ)$, $\SH_{s}(\SmatBlZ)$,
and the subcategory $\SHRig(\SmatBlZ)$ of $\SH_{s}(\SmatBlZ)$ spanned by levelwise rigid objects.
The considerations above for simplicial presheaves extend to pointed simplicial presheaves.
By levelwise application of \eqref{eq:Htriviusidsuisuds}, 
there are naturally induced functors 
\[
\begin{array}{llll} 
u_* & \colon \SH_{s}(\SmatBlZ) & \to & \SH_{s}(\SmatBcZ), \\
(i_{B*Z})^* & \colon \SH_{s}(\SmatBlZ) & \to & \SH_{s}(\SmatZ),  \\
(i_{B*Z})_* & \colon \SH_{s}(\SmatBlZ) & \to & \SH_{s}(\SmatZ).
\end{array}
\]

We obtain the following corollaries from \Cref{prop:SmoothLift->A1lrig:origin} and \Cref{lm:PresPreSsEquivalence}.

\begin{corollary}
For every $\calF\in \SH_{s}(\SmatBcZ)$ the object $\calF_{\A^{1}}$ is $\A^{1}$-local and rigid. 
The same holds in the framed setting.
\end{corollary}

\begin{corollary}
\label{lm:PresPreSsEquivalenceSH}
The functors $(i^{B*Z})_*$ and $(i_{B*Z})^*$ induce an equivalence of categories
\[
i^*_\rigid\colon \SHRig(\SmatBlZ)
\xrightleftarrows{1em}
\SH_{s}(\SmatZ)\colon i^{\rig}_*.
\] 
Moreover, 
$i^*_\rigid$ and $i^\rig_*$ preserve and detect Nisnevich local objects and equivalences.
\end{corollary}

\section{The Nisnevich cohomology of generic fibers}
\label{section:tcogf}

The setup in this section concerns a one-dimensional base scheme \( B \) with generic point \( \eta \), 
for which strict \( {\mathbb A}^{1} \)-invariance holds, 
and a framed, quasi-stable, \( \A^{1} \)-local presheaf \( \calF \) on smooth schemes over \( \eta \). 
The key claim in~\eqref{eq:LnisFXhxetasimeqFXhxeta} asserts that the Nisnevich localization of \( \calF \) 
coincides with \( \calF \) itself when evaluated on the generic fiber \( X^h_x \times_B \eta \) 
of any local henselian \( B \)-scheme \( X^h_x \). 
Equivalently,~\eqref{eq:Hnis(((X^h_x)eta)(F))} states that the Nisnevich cohomology of \( \calF \) 
on these generic fibers vanishes in positive degrees. 
Analyzing open complements of smooth divisors in essentially smooth local schemes over \( B \), 
such as \( X^h_x \times_B \eta \), is a crucial step in extending strict \( \A^{1} \)-invariance 
from fields to one-dimensional base schemes. 
This is because the value of \( \calF \) on \( X^h_x \times_B \eta \) is controlled by the stalks of 
inverse images with compact support under closed points in $B$ and also direct images under $\eta\to B$.
We also note that for open subschemes \( U \subset \A^1_k \) over a field \( k \), 
and for products \( U \times X^h_x \), 
analogous isomorphisms play a central role in Voevodsky’s proof of strict \( \A^{1} \)-invariance.
\vspace{0.1in}

Suppose $B$ is a one-dimensional base scheme with generic point $\eta$ such that strict ${\mathbb A}^{1}$-invariance 
holds at ${\eta}$.
Let $\calF$ be an $\A^{1}$-local quasi-stable radditive framed presheaf of $S^{1}$-spectra on $\Sm_{{\eta}}$.
\Cref{th:Hnis(Xhxeta)dimB1} shows that for all $X\in\Sm_B$, $x\in X$, there is an equivalence of $S^{1}$-spectra
\begin{equation}
\label{eq:LnisFXhxetasimeqFXhxeta}
\Lrep_{\nis}(\calF)(X^h_x\times_B \eta)
\simeq 
\calF(X^h_x\times_B \eta),
\end{equation}
where $\Lrep_{\nis}$ is the Nisnevich localization endofunctor.
An equivalent statement,
see \Cref{section:stifabeliangroups}, 
says that for any $\A^{1}$-invariant framed quasi-stable additive presheaf of abelian groups $\calF$ on $\Sm_{{\eta}}$ 
we have
\begin{equation}
\label{eq:Hnis(((X^h_x)eta)(F))}
H^i_\nis(X^h_x\times_B \eta, \calF_\Nis)
\cong
\begin{cases}
\calF(X^h_x\times_B \eta) & i=0\\
0 & i>0.
\end{cases}
\end{equation}
Our idea for proving \eqref{eq:LnisFXhxetasimeqFXhxeta} is to find a framed $\A^{1}$-homotopy between 
$V_\eta=X^h_x\times_B \eta$ and its generic point.
This means that for an essentially smooth local henselian scheme $V=X^h_x$, 
where $X\in \Sm_B$, $x\in X$,
for any closed subscheme $Y\not\hookrightarrow X$ such that $X_\eta-Y_\eta\neq\emptyset$,
there is an \'etale neighborhood $V^\prime$ of $x$ and a commutative diagram in $\ZF_*(\eta)$:
\begin{equation}\label{eq:cZF:A1V->Xeta}
\xymatrix{
V^\prime_\eta\ar[r]&
X_\eta&
X_\eta-Y_\eta\ar[l]
\\
\{1\}\times V^\prime_\eta\ar[u]^{\sigma^n}\ar[r]
&
\A^{1}\times V^\prime_\eta
\ar[u] 
& 
\{0\}\times V^\prime_\eta\ar[u]\ar[l]
.}
\end{equation}
Here we write $V^\prime_\eta=V^\prime\times_B \eta$.
This suffices because the limit of $X_\eta - Y_\eta$ for a given $V$ and all possible choices of $X$ and $Y$
equals the generic point of $V_\eta$,
and the Nisnevich topology on it is trivial, 
while the values of $\Lrep_{\nis}(\calF)$ and $\calF$ at $\eta$ coincide up to equivalence.
Our assumption on the residue fields of the generic points of $B$ is used to ensure that $\Lrep_{\nis}(\calF)$ 
is $\A^{1}$-local. 
We use our assumption $\dim(B)=1$ to construct the contracting framed $\A^{1}$-homotopy \eqref{eq:cZF:A1V->Xeta}.
The $\A^1$-homotopy \eqref{eq:cZF:A1V->Xeta} is like the one from the injectivity theorems in \cite[\S8]{hty-inv}, 
where $\A^1_k$ or $X_{(x)}$ play the role of $V_\eta$ above.
In contrast to the proofs of the injectivity theorems in \cite[\S8]{hty-inv} 
we cannot assume the Picard groups of all closed subschemes of $V_\eta$ are trivial.
\vspace{0.1in}

The Krull dimension of a topological space is defined in \cite[Tag 0055]{StacksProject}.
A scheme is of pure dimension if all its irreducible components have the same Krull dimension.
The codimension of an irreducible closed subset in a topological space is defined in \cite[Tag 02I3]{StacksProject}.
The codimension of a closed subscheme $Y$ in $X$ at $y\in Y$ equals the codimension of the local scheme $Y_{(y)}$ 
of $Y$ at $y$ in $X_{(y)}$.
We say $Y$ has pure codimension $c$ if the codimension of $Y$ in $X$ equals $c$ at every point $y\in Y$.
Moreover, 
$X$ has (pure) relative dimension $d$ over $B$ if the (pure) dimension of all the fibers of the structure morphism 
$f\colon X\to B$ equals $d$.
We say $Y$ has (pure) relative codimension $c$ in $X$ if for each $\sigma\in B$, $\codim_{X_\sigma} (Y_\sigma)= c$.
For a constructible subset $W$ of $X$ \cite[Tag 005G]{StacksProject} we denote its closure by $\overline W$.

\begin{lemma}
\label{prop:onedimbaseEquidim}
Suppose $B$ is a one-dimensional scheme, and let $z\in B$ be a closed point with open complement $U=B-z$.
Let $X$ be an irreducible $B$-scheme of finite type over $B$ such that $X_z=X\times_B z\neq\emptyset$ and 
$X_U=X\times_B U$ is dense in $X$.
If $X_U$ is equidimensional over $U$, then $X$ is equidimensional over $B$.
\end{lemma}

\begin{proof}
Without loss of generality, 
we may assume $X$ is affine, 
and $B$ is local with generic point $U$.
If $\dim X_z=0$, 
then $X$ is quasi-finite over some Zariski neighborhood of $z$ in $B$, 
and the claim follows.
Thus we may assume $\dim X_z=d>0$ and that the claim holds for all $Y$ with $\dim Y_z<d$. 
Since $X_z$ is affine and $d>0$, 
there exists a nonzero regular function $f_z$ on $X_z$ that is not invertible 
on any of the irreducible components of $X_z$. 
Indeed, 
since each irreducible component of $X_z$ has a positive dimension, 
there exist disjoint finite sets of closed points $C_0,C_1\subset X_z$, 
each of which contains at least one point from every irreducible component of $X_z$.
Now, using the Chinese remainder theorem, 
choose $f_z\in \calO_{X_z}(X_z)$ such that $f_z\big|_{C_0}=0$ and $f_z\big|_{C_1}=1$.
Then $Z(f_z)$ is of pure codimension one in $X_z$. 
Since $X$ is affine, 
$f_z$ lifts to a regular function $f$ on $X$. 
Since $Z(f_z)\neq X_z$, it follows that $Z(f)\neq X$.
Since $X$ is irreducible, $Z(f)$ is of pure codimension one in $X$. 
Let $X_1$ be an irreducible component of $Z(f)$. 
This is a closed subscheme in $X$ of pure codimension one, 
and $X_{1,z}:=X_1\times_B z$ is of pure codimension one in $X_z$.
Since $X_U$ is dense in $X$, and $X_1\neq X$, it follows that $X_{1,U}:=X_1\times_B U\neq X_U$. 
Note that $X_U$ is irreducible since it is dense in the irreducible scheme $X$. 
Thus $X_{1,U}$ has positive codimension in $X_U$. 
Further, 
$X_U$ is dense in $X$, so $X_z$ has positive codimension in $X$, 
and $X_{1,z}$ has codimension at least $2$ in $X$. 
Hence $X_1\neq X_{1,z}$ and $X_{1,U}\neq \emptyset$. 
Since $X_1$ is irreducible it follows that $X_{1,U}$ is dense in $X_1$.
Summarizing the above we conclude the vanishing locus 
$$
X_{1,U}=Z(f\big|_{X_U})
$$
is nonempty and of pure codimension one in $X_U$.
Since $X_{1,U}$ is dense in the irreducible scheme $X_1$, $\dim X_{1,z}<\dim X_z$ and $X_{1,z}\neq \emptyset$. 
By the inductive assumption, 
$X_1$ is equidimensional over $B$.  
Using that $\codim_{X_z} X_{1,z}=\codim_{X_U} X_{1,U}=1$ we deduce the same statement for $X$.
\end{proof}

\begin{lemma}
\label{prop:onedimbaseCloseurefiber}
Suppose $B$ is a one-dimensional scheme with closed point $z\in B$ and complement $U=B-z$.
Let $X$ be a scheme over $B$ such that $X_z=X\times_B z$ and $X_U=X\times_B U$ are nonempty. 
Assume $Z_U$ is a closed subscheme in $X_U$ of positive codimension over $U$,
and let $Z$ be the closure of $Z_U$ in $X$. 
Then $Z$ has positive codimension over $z$ in $X$.
\end{lemma}
\begin{proof}
Without loss of generality we may assume $B$ is local, 
$Z\times_B z\neq \emptyset$, 
and $X$, $Z_U$ are irreducible.
\Cref{prop:onedimbaseEquidim} implies $X$ and $Z$ are equidimensional over $B$. 
Since $Z_U$ has positive codimension in $X_U$, 
the same holds for $Z$ in $X$.
\end{proof}

\subsection{Automorphisms and framed correspondences}
\label{subsect:Autoframed}

In the following, 
we introduce framed morphisms and analyze the action of such framed correspondences on Nisnevich local 
$\A^{1}$-local quasi-stable framed presheaves of $S^{1}$-spectra.
We refer to \Cref{sectionApp:FrCor} for our conventions on framed correspondences.
Fix a base scheme $B$ and $X\in \Sm_B$.
If $E\in \mathrm{GL}_n(X)$ we use the same symbol $E\in \Fr_n(X,X)$ to denote the framed correspondence 
$$
(0\times X, E\cdot(t_1,\dots, t_n),\mathrm{pr}\colon (\A^n_X)^h_{0\times X}\to X).
$$
Here $E\cdot(t_1,\dots, t_n)$ is the vector of regular functions on $\A^n_X$ obtained from multiplication by $E$ on the 
coordinates $(t_1,\dots,t_n)$ of $\A^n_X$.
We let $\mathrm{E}_{n}(X)\subset \GL_n(X)$ denote the subgroup generated by all elementary matrices.
In particular, 
the $(n\times n)$-identity matrix $\id_n\in\mathrm{E}_{n}(X)$.

\begin{lemma}
\label{prop:ElementaryMatrxHomotopy}
If $E\in\mathrm{E}_n(X)$ then $E=\id_n\in\overline{\Fr}_n(X,X)$ for all $X\in \Sm_B$,
where $\overline{\Fr}_n(X,X)$ is the set of $\A^1$-homoptopy classes of $\Fr_n(X,X)$.
Thus, 
for every $\A^{1}$-local framed presheaf $\calF$ of $S^{1}$-spectra on $\Sm_B$, 
the morphism $E^*\colon \calF(X)\to \calF(X)$ is a levelwise equivalence of $S^{1}$-spectra. 
\end{lemma}
\begin{proof}
We write $E=E_1\cdots E_m$ for elementary matrices $E_i$ and set
$$
E_i(\lambda)=(1-\lambda)E_i+\lambda\id_n \in \mathrm{Fr}_n(X\times\A^{1},X).
$$
The framed correspondences $E_i(\lambda)\in\Fr_n(X\times\A^{1})$ for $1\leq i\leq m$ yields the $\A^{1}$-homotopies
\[
E
=
E_1\cdots E_m
\sim_{\A^{1}}
E_1\cdots E_{m-1}
\sim_{\A^{1}}
\dots
\sim_{\A^{1}} 
E_1
\sim_{\A^{1}} 
\id_n.
\]
\Cref{prop:A1localhomotopycommutativesquare} implies the claim for $E^*\colon \calF(X)\to \calF(X)$.
\end{proof}

\begin{lemma}
\label{prop:LocalschemeinvqstFr}
Let $\calF$ be an $\A^{1}$-local quasi-stable framed presheaf of $S^{1}$-spectra on $\Sm_B$.
Suppose $X$ is an essentially smooth local scheme over $B$.
If $E\in \GL_n(X)$, 
then the framed correspondence $E\in \Fr_n(X,X)$ induces a levelwise equivalence $E^*\colon \calF(X)\xrightarrow{\simeq} \calF(X)$.
\end{lemma}
\begin{proof}
When $E\in \SL_n(X)$, 
then $E\in \mathrm{E}_n(X)$ since $X$ is local and the claim follows from \Cref{prop:ElementaryMatrxHomotopy}.
Let $(E^{-1},E)\in \GL_{2n}(X)$ denote the block-diagonal matrix 
\[
\left(
\begin{array}{ll} 
E^{-1} & 0 \\ 0 & E 
\end{array}
\right).
\] 
With this definition, 
we have equalities of framed correspondences 
\[
E^{-1}\circ E=(E,E^{-1}), 
E\circ E^{-1}=(E^{-1},E)
\in 
\Fr_{2n}(X,X).
\]
It follows that 
\[
(E^{-1})^*\circ  E^* = (E^{-1},E)^*,
E^*\circ (E^{-1})^* = (E^{-1},E)^*.
\]
Since $(E^{-1},E), (E,E^{-1})\in \SL_{2n}(X)$ we conclude that $E^*$ and $(E^{-1})^*$ are levelwise equivalences.
\end{proof}

\begin{lemma}
\label{prop:LnissigmaEeq}
Let $\calF$ be an $\A^{1}$-local quasi-stable framed presheaf of $S^{1}$-spectra on $\Sm_B$.
Suppose $c=(Z,\varphi_1,\dots, \varphi_n,\mathrm{pr})\in \Fr_n(X,X)$, 
where $\mathrm{pr}$ is induced by the projection $\A^n_X\to X$ and induces an isomorphism $Z\cong X$.
Then $c^*\colon \Rep_\nis(\calF)(X)\to \Rep_\nis(\calF)(X)$ is a levelwise equivalence. 
\end{lemma}
\begin{proof}
First we reduce to the case when $Z=0\times X$.
Since $c$ is a framing of $\id_X$, the canonical projection induces an equivalence $Z\cong X$.
There is an $\A^{1}$-homotopy $c\sim_{\A^{1}} c^\prime$ between $c$ and 
\[
c^\prime 
= 
(0\times X, T^*(\varphi), \mathrm{pr})\in \Fr_n(X,X).
\]
Here $l=(l_1,\dots,l_n)\colon X\cong Z\hookrightarrow \A^n_X$,
and 
\[
T
=
\id_{\A^n_X}+l\mathrm{pr}
\colon 
\A^n_X
\to 
\A^n_X;
(t_1,\dots,t_n)
\mapsto 
(t_1+l_1,\dots,t_n+l_n).
\]
The $\A^{1}$-homotopy is given by the framed correspondence
\[
c_{\lambda} 
= 
(T_{\lambda}^{-1}(Z\times\A^{1}), T_{\lambda}^{*}(\varphi), (\mathrm{\pr}\times\id_{\A^{1}})\circ T_{\lambda})
\in 
\Fr_n(X\times\A^{1}, X), 
\]
where 
\[
T_{\lambda} 
=
\id_{\A^n_X}+\lambda l\mathrm{pr}
\colon \A^n_X\times\A^{1}\to \A^n_X; 
(t_1,\dots ,t_n)\mapsto (t_1+\lambda l_1, \dots, t_n+\lambda l_n).
\]
Note that $(\mathrm{\pr}\times\id_{\A^{1}})=(\mathrm{\pr}\times\id_{\A^{1}})\circ T_{\lambda}$, 
where $(\mathrm{\pr}\times\id_{\A^{1}})\colon \A^n_{X\times\A^{1}}\to X\times\A^{1}$ is the canonical projection.
Then $T_0=\id_{\A^n_X}$, $T_1=T$, and
\[
\begin{array}{lllllll}
c_0&=&(T_{0}^{-1}(Z), T_{0}^{*}(\varphi), (\mathrm{\pr})\circ T_{0})
&=&(Z, \varphi, \mathrm{\pr})&=& c,\\ 
c_1&=&(T_{1}^{-1}(Z), T_{1}^{*}(\varphi), (\mathrm{\pr})\circ T_{1})
&=&(0\times X, T^*(\varphi), \mathrm{\pr})&=&c^\prime.
\end{array}\]
Thus we may assume $c=c^\prime$ and $Z=0\times X$.
\vspace{0.1in}

Let $I_{\A^n\times X}(0\times X)$ denote the vanishing ideal of the closed subscheme $0\times X$ in ${\A^n\times X}$. 
For any polynomial $f=f_0+f_1+\dots+f_d\in \calO_X(X)[t_1,\dots,t_n]$, 
where $f_i$ is homogeneous of degree $i$, 
we have $f-(f_0+f_1)\in I_{\A^n\times X}(0\times X)^2$ for the linear polynomial $f_0+f_1$.
We may choose linear polynomials 
$\varphi_1^\prime,\dots,\varphi_n^\prime\in\calO_X(X)[t_1,\dots,t_n]\cong\calO_{\A^n_X}(\A^n_X)$ 
such that $\varphi^\prime_i-\varphi_i\in I_{\A^n\times X}(0\times X)^2$ for $i=1,\dots,n$.
Then $\varphi^\prime_i\in I_{\A^n\times X}(0\times X)$, 
and hence $(\varphi^\prime_1,\dots, \varphi_n^\prime) \cong E\cdot(t_1,\dots,t_n)$ for some $E\in\GL_n(X)$.
Here $I_{\A^n\times X}(0\times X)$ is the ideal sheaf of the closed immersion $0\times X\not\hookrightarrow \A^n\times X$.
Now the framed correspondence
\[
(0\times X\times \A^{1}, (1-\lambda)\varphi+\lambda\varphi^\prime,(\mathrm{pr}\times\id_{\A^{1}}))
\in 
\Fr_n(X\times\A^{1},X)
\]
gives an $\A^{1}$-homotopy $c\sim_{\A^{1}} c^{\prime\prime}$ between $c$ and 
\[
c^{\prime\prime}
= 
(X\times\A^{1},\varphi^{\prime},\mathrm{pr})
\in 
\Fr_{n}(X,X).
\]
Hence we may assume $c=c^{\prime\prime}$ and $(\varphi_{1},\dots,\varphi_{n})\cong E\cdot(t_{1},\dots,t_{n})$ for some $E\in\GL_{n}(X)$.
\vspace{0.1in}

To each $x\in X$ we can associate the framed correspondence $c_x=E_x\in \Fr_n(X^h_x, X^h_x)$, 
where $E_x\in \GL(X^h_x)$ is the stalk of $E$.
\Cref{prop:LocalschemeinvqstFr} shows that there is an induced equivalence of stalks
\[
c_x^*
\colon 
\calF(X^h_x)
\xrightarrow{\simeq}
\calF(X^h_x).
\] 
That is, 
$c^*\colon \calF(X)\to \calF(X)$ is a Nisnevich local equivalence. 
This finishes the proof.
\end{proof}

\begin{remark}
The analog of \Cref{prop:LnissigmaEeq} 
for presheaves with tangentially framed correspondences is simpler. 
This is because framed correspondences $(Z,\varphi_1,\dots,\varphi_i,\mathrm{pr})$ and 
$(Z^\prime,\varphi_1^\prime,\dots,\varphi_i^\prime,\mathrm{pr})$,  
where $Z=Z^\prime$ and $\phi_{Z(I^2(Z))}=\phi^\prime_{Z(I^2(Z^\prime))}$,
see \Cref{subsubsectionList}(3) for the notation, 
have the same images in $\pi_0\Corrfr(B)$.
Additionally, it implies the form stated above for affine $X$ because the presheaves $\Fr(-,Y)$ and $h^\fr(Y)$ are $\A^1$-equivalent on affine schemes, as shown in \cite[Corollaries 2.2.20, 2.3.35]{five-authors}.
\end{remark}

\begin{corollary}
\label{cor:LnisrEsimeqr}
Suppose $B$ is affine and let  $\calF$ be an 
$\A^{1}$-local quasi-stable framed presheaf of $S^{1}$-spectra on $\Sm_B$.
For every 
$c=(Z,\varphi_1,\dots, \varphi_n,g)\in \Fr_n(X,Y)$ 
such that
the projection
$\pr
\colon 
\A^n_B\times_B X\to X
$
induces an isomorphism
$\pr\big|_Z\colon Z\to X$,
and
the composite 
$$
X
\xrightarrow{(\pr\big|_Z)^{-1}}
Z
\xrightarrow{g}
Y,
$$
where 
$\pr
\colon 
\A^n_B\times_B X\to X
$
is the projection, 
equals $r\colon X\to Y$, 
we have
\[
c^*
=
r^*
\colon
\Lrep_\nis(\calF)(Y)
\to 
\Lrep_\nis(\calF)(X).
\]
\end{corollary}
\begin{proof}
We note that $\id^\varphi = (Z,\varphi_1,\dots,\varphi_n,\mathrm{pr}\colon (\A^n_X)^h_X\to X)\in \Fr_n(X,X)$ 
is a framing of $\id_X$, see \Cref{def:FramedCorr}.
Since $c$ and $g\circ \id^\varphi$ define the same intermediate framed correspondence, 
see \Cref{def:FramedCorrInterm}, 
we conclude using \Cref{prop:LnissigmaEeq} and \Cref{lm:FrSimplFr}.
\end{proof}

\subsection{Contracting framed \texorpdfstring{$\A^{1}$}{A1}-homotopies}
\label{subsection:ContractingHomotopy}

We fix a smooth irreducible scheme $X$ over a one-dimensional base scheme $B$. 
Suppose $x\in X$ maps to some closed point $z\in B$, 
and set $V=X^h_x$, $U=B-z$, $V_U=V\times_B U$, $X_U=X\times_B U$. 
\Cref{prop:CorFretaVetatoXeta-Yeta} shows that for every closed immersion $Y_U\not\hookrightarrow X_U$ 
of relative positive codimension there exists a linear framed $\A^{1}$-homotopy 
\[
c
\colon 
V_U\times\A^{1}\to X_U/(X_U-Y_U)
\]
between 
the composite 
\begin{equation}
\label{eq:(VU)to(XUq(XUmYU))}V_U=(X^h_x)\times_B U\to X_U\to X_U/(X_U-Y_U)
\end{equation} 
and the constant pointed morphism.
In the next subsection, we apply this to prove the triviality of the cohomology classes on $V_U$,
using that the latter framed $\A^1$-homotopy connects any given cohomology class with the trivial one. 

To construct the announced $\A^1$-homotopy, 
we construct diagrams in $\Sch_B$ of the form
\begin{equation}
\label{eq:XxComapctifiedBprimeembedding}
\xymatrix{
&&P\\
X^h_x\ar[r]^{r}&C\ar[ur]^p\ar@{^(->}[r]^{j}&\ovC\ar[u]_{\overline{p}}&C_\infty\ar[l]_i&\\
V_U\ar@{^(->}[u]\ar[rr]^{r_U} && C_U\ar@{^(->}[ul]\ar@{^(->}[u]\ar[r]^{v} & X_U
}
\end{equation} 
subject to the following properties;
see 
\Cref{prop:onedimbaseleadstoPosCodimSubscemes,lm:ChooseSectionsOnedimensionalqfinite,lm:ChooseSectionsOnedimensionalfinite}
for further refinements.
\begin{itemize}
\item[(0)] 
$v\circ r_U\colon V_U\to X_U$ is the canonical morphism $X^h_x\to X$, 
and $P$ is local.

\item[(1)] 
$\overline p$ is a projective equidimensional morphism,
$p$ is affine, 
$p$ is smooth at $r(x)$,
$C_U\in \Schcci_{P_U}$ where 
$C_U\cong C\times_B U$,
$P_U=P\times_B U$. 
\item[(2)]
$i$ is a closed immersion such that $i(C_\infty)$ has positive relative codimension over $P$,
$j$ is an open immersion.
Moreover, 
there exists an ample bundle $\calO(1)$ on $\ovC$ over $P$ with a section  $x_\infty$ such that 
$C_\infty\cong Z(x_\infty)$ and $C\cong \ovC-C_\infty$.
\item[(3)]
The closure $\overline W$ of $W_U=v^{-1}(Y_U)$ in $\ovC$ has positive relative codimension in $\ovC$ over $P$. 
\end{itemize}

The notation indicates that $\ovC$ is a relative compactification of the affine scheme $C_U$ over $P$. 

To construct a factorization of \eqref{eq:(VU)to(XUq(XUmYU))} through an $\A^{1}$-homotopy 
\begin{equation}
\label{eq:VUi0A1VUcXU/(XU-YU)}
V_U\xrightarrow{i_0} V_U\times\A^{1} \xrightarrow{c} X_U/(X_U-Y_U), 
\end{equation}
we firstly construct factorizations through morphisms 
$C\times_P V_U\to X_U$
\begin{equation}
\label{eq:VUrCUvXU/(XU-YU)}
V_U\xrightarrow{r} C_U \xrightarrow{v} X_U \to X_U/(X_U-Y_U).
\end{equation}
The construction consequently produces $\dim_B X+1$ diagrams of the form \eqref{eq:XxComapctifiedBprimeembedding}.
$P$ plays a role of an intermediate base scheme. 
The relative codimension of $C$ over $P$ consequently decreases 
from $\dim_B X$ in 
to 1 in \Cref{lm:ChooseSectionsOnedimensionalfinite},
while $\dim_B P$ increases from 0 to $\dim_B X-1$.
The initial diagram \eqref{eq:XxComapctifiedBprimeembedding} is provided by $C=X$ and $P=B$.
For $\dim_P C=1$, 
$C_U$ and $\ovC$ are respectively an affine and a projective curves over $P_U$ and $P$.
The mentioned above factorisation of \eqref{eq:(VU)to(XUq(XUmYU))} through $C\times_P V_U$
defined by \eqref{eq:VUrCUvXU/(XU-YU)}
is a ``curve''-homotopy between composites in \eqref{eq:VUi0A1VUcXU/(XU-YU)} and \eqref{eq:VUrCUvXU/(XU-YU)}, 
that is a preliminarily step for the construction of the $\A^1$-homotopy.

Next, we define a full subcategory $\Schcci_B\subset\Aff_B$ for any base scheme $B$,
see \Cref{subsection:candn} for the smooth case.
\begin{definition}
\label{def:Schcci}
Let $B$ be a scheme. 
We say $X$ is a $\cci$-scheme over $B$ and write $X\in \Schcci_B$ if $Z(f)=X\amalg X'$, 
where $f=(f_1,\dots,f_m)$ consists of regular functions on $\A^n_S$ and $\codim_{\A^n\times B} X=m$.
\end{definition}

For the desired framed $\A^{1}$-homotopy over $U$ we equip $C_U$ with a framing over $P_U$ such that 
$C_U\in \Schcci_{P_U}$. 
In the remainder of this section, we refine the properties 
(0)-(3) for \eqref{eq:XxComapctifiedBprimeembedding}.

\begin{lemma}
\label{prop:onedimbaseleadstoPosCodimSubscemes}
Suppose $B$ is a one-dimensional scheme and $z\in B$ is a closed point with complement $U=B-z$.
Fix $X\in\Sm_B$, $x\in X\times_B z$, $V$, $V_U$, $X_U$, $Y_U$ as above, 
and let $n=\dim_B X_{(x)}$ be the relative dimension of $X$ at $x$ over $B$, 
and let $X_{(x)}$ be the local scheme of $X$ at $x$.
Then there is a diagram of the form \eqref{eq:XxComapctifiedBprimeembedding} such that $\dim_P \ovC=n=\dim_B X_{(x)}$, 
$P=B_z$, 
and $p_U$ is smooth.
\end{lemma}
\begin{proof}
By shrinking $X$ to an \'etale neighborhood of $x$ and using $X^h_x\cong X^h_x\times_B B^h_z$ we may assume 
$B=B^h_z$ is local, and $X$ is irreducible.
Since $X\in\Sm_B$, 
there exists an open affine neighborhood $X^{\prime}$ of $x$ in $X$ such that the relative tangent bundle of $X^{\prime}$ 
over $B$ is trivial. 
Note that $X^{\prime}\in\Smat_B$.
Since $X^{\prime}$ is affine, 
its schematic closure inside some projective space yields an open immersion $X^{\prime}\hookrightarrow \overline X^\prime$ 
to a projective $B$-scheme $\overline X^\prime$.
\Cref{prop:onedimbaseEquidim} shows $\ovC :=\overline X^\prime$ is equidimensional over $B$.

To construct $C$ we consider $\ovC- X^\prime$ and let $X_\infty^{\prime\prime}$ be the closure of 
$(\ovC- X^\prime)\times_B U$ in $\ovC$.
Then $X_\infty^{\prime\prime}$ has positive codimension over $z$,
see \Cref{prop:onedimbaseCloseurefiber}, 
and $x\in X^\prime\subset \ovC-X^{\prime\prime}_\infty$. 
Let $\calO(1)$ denote the canonical ample invertible sheaf on $\ovC$ over $B$. 
We may choose a finite set $F$ of closed points in $\ovC$ that has a nonempty intersection with each irreducible component of 
$\ovC-X^{\prime\prime}_\infty$. 
Serre's Theorem on sections of ample bundles \cite[Theorem 5.2]{Ha77} implies that for $d\gg 0$ there
exists a section $x_\infty\in \Gamma(\ovC,\calO(d))$ such that $x_\infty\big|_{X^{\prime\prime}_\infty}=0$,  
and $x_\infty\big|_{F\cup x}$ is invertible.
Then $Z(x_\infty)$ has positive relative codimension in $\ovC$ over $z$, 
and we set 
\begin{equation}\label{eq:Cinfty=Z(xinfty)C=C-Cinfty}
C_\infty:=
Z(x_\infty),\quad
C:=
\ovC - C_\infty.
\end{equation}
Using the open immersion $C=\ovC-Z(x_\infty)\hookrightarrow \ovC-X^{\prime\prime}_\infty$ 
and the fiber product $X^{\prime\prime}_\infty\times_B U \cong (\ovC- X^\prime)\times_B U$
we obtain the composite 
\[
v
\colon 
C_U=C\times_B U
\hookrightarrow 
(\ovC-X^{\prime\prime}_\infty)\times_B U
\cong
X^\prime\times_B U
\to 
X_U.
\]
Since $C$ is an open neighborhood of $x$ in $\ovC$ and 
$X^\prime$ is an open neighborhood of $x$ in $\ovC$,
the canonically induced composite 
\[
\overline{r}
\colon 
V=X^h_x
\to 
X^\prime
\to 
\ovC,
\]
coincides with 
\[
V
=
X^h_x
\xrightarrow{r} 
C\to \ovC.
\]
This implies (0) when setting $P=B_z$.
Moreover, 
the $B$-scheme $C$ is smooth over $r(x)$ since the localization of $C$ at $r(x)$ agrees with the localization of $X$ at $x$. 
\Cref{prop:onedimbaseEquidim} implies that the projective morphism $\overline{p}\colon \ovC\to B$ is equidimensional since 
$X^\prime$ and consequently $\ovC$ are irreducible. 
Since $C_\infty=Z(x_\infty)$ is the vanishing divisor of a section of some ample sheaf, 
$C=\ovC-C_\infty$ is affine over $B$. 
In addition, 
$C_\infty$ has positive relative codimension in $\ovC$ by \eqref{eq:Cinfty=Z(xinfty)C=C-Cinfty} 
and the choice of $x_\infty$.
Since $x_\infty\big|_{X^{\prime\prime}_\infty}=0$ and 
$X^{\prime\prime}_\infty\not\hookrightarrow C_\infty$, 
there is an open immersion
\[
(\ovC-C_\infty))\times_B U
\hookrightarrow 
(\ovC-X^{\prime\prime}_\infty)\times_B U 
\cong
X^\prime\times_B U.
\] 
Since $C_U$ is affine over $P$,
it follows that $C_U$ is smooth with a trivial tangent bundle over $P_U=U$.
This completes the proofs of (1) and (2).
Further (1) implies that $W_U$ has positive codimension in $C_U$ since $\ovC$ and consequently, 
$C$ and $C_U$ are irreducible.
Part (3) follows by \Cref{prop:onedimbaseCloseurefiber}.
Finally, by construction, $\dim_B \ovC=\dim_B X_{(x)}$, 
and $p_U$ is smooth since $C_U\in \Smat_B$.
\end{proof}

\begin{lemma}
\label{lm:ChooseSectionsOnedimensionalqfinite}
In the setting of \Cref{prop:onedimbaseleadstoPosCodimSubscemes}, 
assume more generally that $1\leq n\leq \dim_B X_{(x)}$.
Then in \eqref{eq:XxComapctifiedBprimeembedding} we have $\dim_P \ovC=n$ and $P$ is a local henselian scheme.
\end{lemma}
\begin{proof}
We may assume $X$ is irreducible.
The proof uses induction on $n$.
\Cref{prop:onedimbaseleadstoPosCodimSubscemes} verifies the case $n=\dim_B X_{(x)}$ with $C$ and $\ovC$.
We shall construct $C_1$ and $\ovC_1$ for $n-1$.
Since $C$ is smooth at $r(x)$ over $P$ there exists an \'etale morphism $f=(f_1,\dots,f_n)\colon C\to \A^n_P$.
Note that $p(r(x))$ is the closed point of $P$.
We set $T_\infty=Z(x_\infty)\cap (\ovC-r(x))$ and choose a finite set of closed points $F_\infty$ in $T$ that contains at 
least one point on each irreducible component of $T\times_P p(r(x))$. 
Let $\overline W^\prime$ be the union of the irreducible components of $\overline W$ with positive dimension over $P$,
and set $T_W=\overline W^\prime \cap (\ovC-(r(x)\cup F))$.
Let $F_{W,0}$, $F_{W,1}$ be disjoint finite sets of closed points in $\overline W^\prime\times_P p(r(x))$. 
By construction, 
$(F_{W,0}\cup F_{W,1})\cap F_\infty = \emptyset$ and $r(x)\not\in F_\infty, F_{W,0}, F_{W,1}$.
For $d\gg 0$ there exists a section $s\in\Gamma(\ovC,\calO(d))$ such that $s\big|_F$, $s\big|_{F_{W,0}}$ are invertible,
$s\big|_{F_{W,0}}=0$, 
and
\begin{equation}
\label{eq:s|_ZI2(x)}
s\big|_{Z(I^2_{\ovC}(r(x)))}
=
x_\infty^d f_n.
\end{equation}
Here $I_{\ovC}^2(r(x))\subset \calO_{\ovC}$ is the square of 
the ideal sheaf of $r(x)$ contained in the structure sheaf of $\ovC$.
With this definition we have 
\begin{equation}
\label{eq:Z(B)Xinfty}
\dim_P(Z(B)\cap Z(x_\infty) )
\leq 
\dim_P Z(x_\infty)-1
\leq 
\dim_P \ovC-2
=
n-2. 
\end{equation}

We write $[t_0: t_\infty]$ for the coordinates on $\PP^1$ and form the section 
\[
g=s t_\infty + x_\infty^d t_0
\in 
\Gamma(\ovC\times\PP^1, \calO(d,1)),
\]
where $\calO(d,1)$ is the tensor product of the pullback of $\calO(d)$ to $\ovC$ with $\calO(1)$ on $\PP^1$.
Let $P_1$ be shorthand for $(\PP^1_P)^h_{f_n(r(x))}$ where $f_n(r(x))\in \A^{1}_P$ is the image of $r(x)$ 
along $f_n\colon C\to \A^{1}_P$.
We set 
\[
\ovC_1=Z(g)\times_{\PP^1_P} P_1, 
C_{\infty,1}=\ovC_1\times_{\ovC} C_\infty, 
C_1=\ovC_1\times_{\ovC} C, 
r_1=(r,[s:x^d_\infty])\colon V\to C_1.
\]
Moreover, 
we set $C_{1,U}=C_1\times_B U$ and form the composite 
\[
v_1
\colon 
C_{U,1}
\to
C_{U}
\xrightarrow{v} 
X_U.
\]
Since $x_\infty$ is invertible on $C$ we can identify $C_{1}$ with $C\times_{\A^{1}_P} P_1$ under the morphism 
$C\to \A^{1}_P$ given by the regular function $s/x_\infty^d$ on $C$. 
We can identify $C_{\infty,1}$ with $(Z(B)\cap Z(x_\infty))\times_P P_1$.
Note that $C_\infty=Z(x_\infty)$ has positive codimension, 
and it has pure relative codimension one in $\ovC$ over $P$.
The same holds for $C_{\infty,1}$ in $\ovC_1$.
Setting $\mathrm{max} \dim_{T}(\ovC_1):=\mathrm{max}_{c\in \ovC_1} \dim_{T}(\ovC_1)_c$ 
and similarly for $C_{1,\infty}$,
we have 
\[
\mathrm{max} 
\dim_{T}(\ovC_1)
\leq 
\mathrm{max} 
\dim_{T}C_{\infty,1}+1
=
\dim_P (Z(B)\cap Z(x_\infty))
\leq
\dim_B \ovC-2
=
n-2.
\]
In the second inequality, 
we use \eqref{eq:Z(B)Xinfty}.
On the other hand, 
we have 
\[
\mathrm{min}
\dim_{P_1}(\ovC_1)
:=
\mathrm{min}_{c\in \ovC_1} \dim_{P_{1}}(\ovC_1)_c
\geq 
\dim_P \ovC_1 - \dim_P P_1
=
\dim_P \ovC-1
=
n-1.
\]
This shows that $\ovC_1$ has relative dimension $n-1$ over $P_1$.
By \eqref{eq:s|_ZI2(x)}, 
it follows that $\ovC_1$ is smooth at $r_1(x)$ over $P_1$.
Note that $C_1\times_B U$ is the vanishing locus of the regular function $s/x_\infty^d - t_0/t_\infty$ 
on $C_U\times_P P_1$, 
so that $C_{1,U}\times_B U\in \Schcci_U$.
This finishes the proofs of (1) and (2). 
By construction $P_1$ is local henselian 
and $\dim_{P_1}\ovC_1=n-1$ as shown above; 
this implies (0).

To prove (3), 
we may assume $\overline W$ has pure relative codimension one in $\ovC$ and $\overline W=\overline W^\prime$.
Let $W^\prime=\overline W^\prime\cap C$, 
and consider the closed subscheme $W^\prime_1=W^\prime\times_{\PP^1_P}P_1$ of $C_1$.
Since $s\big|_{F_{W,0}}=0$ and $s\big|_{F_{W,0}}$ is invertible, 
the maximal dimension of $W^\prime$ over $P_{1}$ is less than the maximal dimension of $\overline W$ over $P$.
Since $\dim_{P_1} C_1=\dim_P C-1$, 
$W^\prime_1$ has positive relative codimension in $C_1$ over $P_1$.
Consequently, 
the closure $\overline{W^\prime_1}$ and $\overline{W_1}$ are of positive relative codimension in $\ovC_1$ over $P_1$. 
\end{proof}

\begin{lemma}
\label{lm:ChooseSectionsOnedimensionalfinite}
Let $B$, $z\in B$, $U=B-z$, $X\in\Sm_B$, $x\in X\times_B z$, $V$, $V_U$, and $X_U$ be as in 
\Cref{prop:onedimbaseleadstoPosCodimSubscemes}.
Then in \eqref{eq:XxComapctifiedBprimeembedding} we can arrange that  
$\dim_P \ovC=1$, $P\cong V=X^h_x$,
the closure $\overline W$ of $v^{-1}(Y_U)$ in $\ovC$ is finite over $P$,
and $\ovW\cap C_\infty=\emptyset$.
\end{lemma}
\begin{proof}
We apply \Cref{lm:ChooseSectionsOnedimensionalqfinite} for $n=1$.
By base change along $V\to P$ we may assume $P\cong V$.
We shall modify $C$ in such a way that $C_\infty$ and $\overline W$ become finite over $P$ and $\ovW\cap 
C_\infty=\emptyset$.
Note that $C_\infty$ is a closed subscheme of positive relative codimension in $\ovC$ over $P$.
The same holds for $\overline{W}$, 
see \Cref{prop:onedimbaseleadstoPosCodimSubscemes} and property (3) of \eqref{eq:XxComapctifiedBprimeembedding}. 
Hence $C_\infty$ and $\overline W$ are finite over the local henselian scheme $P$, 
i.e., 
both schemes are disjoint unions of local henselian schemes.
We have $\ovW=\ovW^\prime\amalg \ovW_\infty$,
where $r(x) \in\ovW^\prime$, $r(x)\not\in\ovW_\infty$.
Note that $C_\infty$ and $\ovW_\infty$ are of positive codimension in $\ovC$.
Hence there exists a finite set of closed points $F$ in $\ovC-(C_\infty\cup \ovW_\infty)$ that contains at least 
one point on each irreducible component of $\ovC$.
Now choose a section $x_\infty^\prime\in \Gamma(\ovC,\calO(d))$ for $d\gg 0$ such that 
$x_\infty^\prime\big|_{C_\infty\cup \ovW_\infty}=0$ and $x_\infty\big|_{\ovW^\prime\cup F}$ is invertible.

Next we set $\calO_{\mathrm{new}}(1):=\calO(d)$, 
$x_{\infty,\mathrm{new}}\defeq x_\infty^\prime$, 
$C_\mathrm{new}\defeq\ovC-Z(x^\prime_\infty)$, 
and $C_{\infty,\mathrm{new}}\defeq Z(x^\prime_\infty)$. 
Here $C_{\infty,\mathrm{new}}$ is finite over $P$ since $x_\infty\big|_F$ is invertible.
Then $C_{U,\mathrm{new}}\cong C_{\mathrm{new}}\times_B U$, 
and we let $v_\mathrm{new}\colon C_{U,\mathrm{new}}\to X$ be the restriction of $v$.
The closure $\ovW_\mathrm{new}$ of $v_\mathrm{new}^{-1}(Y_U)$ in $\ovC$ equals $\ovW\cap C=\ovW^\prime$. 
It follows that $\ovW_\mathrm{new}\cap C_{\infty,\mathrm{new}}\cong \ovW^\prime\cap Z(x^\prime_\infty)=\emptyset$,
and we are done.
\end{proof}

\begin{prop}
\label{prop:CorFretaVetatoXeta-Yeta}
In the setting of \eqref{eq:XxComapctifiedBprimeembedding} and \Cref{lm:ChooseSectionsOnedimensionalfinite} there exist linear framed correspondences
\begin{equation}
\label{eq:cZF(A1V-Xeta)}
c\in \ZF_{n,U}(\A^{1}\times V_U, X_U),
\tilde c_1\in \ZF_n(V_U,X_U-Y_U), 
\end{equation}
such that for the canonical morphisms $\can \colon V_U\to X_U$ and $e\colon X_U-Y_U\to X_U$ we have
\[
c\circ i_0=\can^\nu, 
c\circ i_1=j\circ \tilde c_1.
\]
Here $\can^\nu$ is a level $n$ framing of  $\can$.
\end{prop}

\begin{proof}
We use the properties of \eqref{eq:XxComapctifiedBprimeembedding} shown in 
\Cref{lm:ChooseSectionsOnedimensionalfinite}.
Since $C_U\in\Schcci_{V_U}$, for some large enough $N\in\mathbb Z_{\geq 0}$, there are regular functions
$\varphi=(\varphi_1,\dots,\varphi_{N-1})\in \calO_{\A^N\times {V_U}}(\A^N_U)$ 
such that $Z(\varphi)\simeq C_U\amalg \widehat C_U$ for some $\widehat C_U$.
Both $\ovW$ and $C_\infty$ are finite over the local scheme $V$ since $\ovW\cap C_\infty=\emptyset$.
The closed subscheme $\Delta=r(V)$ of $\ovC$ is local and $\Delta\cap C_\infty=\emptyset$. 
Moreover, 
since $C_\infty$ and $\Delta\cup \ovW$ are semi-local affine schemes, 
all line bundles on $C_\infty$ and $\Delta\cup \ovW$ are trivial. 
The constant non-zero sections of the trivial bundles on ${C_\infty}$ and $\ovW\cup\Delta$ define invertible sections
\[
u_{\infty}\in \Gamma(C_\infty,\calO(1)),\quad
u_{Y,\Delta}\in \Gamma(\ovW\cup \Delta,\calO(1)).
\]
Since $C\to V$ is smooth over $\Delta$,
there exists an invertible sheaf $\calL(\Delta)$ on $\ovC$ with a section $\delta\in \Gamma(\ovC,\calL(\Delta))$ 
such that $Z(\delta)=\Delta$.
It follows that $\delta\big|_{C_\infty}$ and $x_\infty\big|_\Delta$ are invertible.
Similarly, 
since $\ovW\cup\Delta$ is an affine semi-local scheme, 
there exists an invertible section 
\[
\check{\delta}_\Delta
\in 
\Gamma(\ovW\cup\Delta,\calO(\Delta)).
\]

Recall that $\ovW\cap C_\infty=\emptyset$ and $\Delta\cap C_\infty=\emptyset$.
Serre's theorem on ample bundles implies that for $d\gg 0$, 
there exists sections satisfying the  properties:
\begin{center}
\begin{tabular}{@{}>{$}l<{$}>{$}l<{$}@{}}
\toprule
s_0^+\in \Gamma(\ovC ,\calL(-\Delta)(d)) & s_1\in \Gamma(\ovC ,\calO(d))\\
\midrule 
s_0^+\big|_{C_\infty}=u_\infty^d \delta\big|_{C_\infty}^{-1} &
s_1\big|_{C_\infty}=u_\infty^d\\
s_0^+\big|_{\ovW\cup \Delta}=u_{Y,\Delta}^{d} \check{\delta}_\Delta^{-1}&
s_1\big|_{\ovW}=u_{Y,\Delta}^{d}\big|_{\ovW}
\\\bottomrule
\end{tabular}
\end{center}
Using the above, 
we define the section
\[
s
=
\delta s_0^+ \lambda +s_1(1-\lambda)
\in 
\Gamma(\ovC\times\A^{1},\calO(d)).
\]
Since $s\big|_{\ovC_\infty\times\A^{1}}=u_\infty^d$ is invertible on $\ovX_\infty\times\A^{1}$, 
we see that $Z(s)$ is finite over $\ovC_\infty\times\A^{1}$.
Thus we obtain the framed correspondence
\[
c^\prime 
= 
(Z(s)\times_B U,(\varphi_1,\dots,\varphi_{N-1},s/x_\infty^d),g)
\in 
\Fr_N((X^h_x)_U\times\A^{1}, X_U),
\]
where, 
by \Cref{lm:AffSmHenselianLift}, 
the morphism $g\colon (\A^N_{V_U})^h_{Z(s)\times_B U}\to X_U$ lifts the composite
\[
{Z(s)\times_B U}
\to 
C_U
\xrightarrow{v}
X_U.
\]
The sections $s_0^+\big|_{\ovW}$ and $s_1\big|_{\ovW}$ are invertible.
Thus $Z(s_0^+)$ and $Z(s_1)$ are open in $g^{-1}(X_U-Y_U)$ and there are framed correspondences
\[
\begin{array}{ll}
\tilde c_0^{+}
= 
(Z(s_0^+),(\phi_1,\dots,\phi_{N-1}, (\delta s_0^+)/x_\infty^d),\tilde g^{+}_0)&
\in 
\Fr_N((X^h_x)_U, X_U-Y_U),\\
\tilde c_1
=
(Z(s_1),(\phi_1,\dots,\phi_{N-1}, s_1/x_\infty^d),\tilde g_1)&
\in 
\Fr_N((X^h_x)_U, X_U-Y_U).
\end{array}
\]
Here $\tilde g_0^+$ and $\tilde g_1$ are induced by $g$.
Since $s_0^+\big|_{\Delta}$ is invertible, we deduce the equalities 
\begin{equation}
\label{eq:cprimei0i1}
\begin{array}{ll}
c^\prime\circ i_0 = \can^\nu+ j\circ \tilde c_0^{+},\\
c^\prime\circ i_1 = j\circ \tilde c_1,
\end{array}
\end{equation}
where $\can^\nu = (\Delta,(\phi_1,\dots,\phi_{N-1},(\delta s_0^+)/x_\infty^d)$,
$g_{0,\Delta})\in \Fr_n((X^h_x)_U, X_U)$,
$g_{0,\Delta}\colon (\A^n_{V_U})^h_\Delta\to X_U$ is induced by $g$, 
and $V_U=(X^h_x)_U$.
Setting $c=c^\prime - j\circ \tilde c_0^+$ and using \eqref{eq:cprimei0i1} finishes the proof.
\end{proof}

\subsection{Triviality of the cohomology on \texorpdfstring{$(X^h_x)_U$}{XhxU}}
\label{subsection:Hi(Xhx)_SHIeta}
If $\calF \in \Spt_s(\Sm_{B})$ or $\calF \in \Spt_s(\Fr_+(B))$,  
we will want to replace $\calF$ by the homotopy cofiber (in $\Spt_s(\Sm_{B})$ or $\Spt_s(\Fr_+(B))$)
\[
\tildeRep_\nis(\calF):=\hocofib(\calF\to \Rep_{\nis}(\calF)).
\]
For any closed subscheme $Y$ of $X$, 
we write $\calF_Y(X)=\hofib(\calF(X)\to\calF(X-Y))$.
In particular, 
for $\tildeRep_\nis(\calF)$ we use the abbreviation
\[
\tildeRep_\nis(\calF)_{Y}(X)
=
\hofib(\tildeRep_\nis(\calF)(X)\to \tildeRep_\nis(\calF)(X-Y)).
\]

\begin{theorem}
\label{th:Hnis(Xhxeta)dimB1}
Suppose that $B$ is a one-dimensional scheme such that strict $\A^{1}$-invariance holds over its generic points. 
For a closed point $z\in B$, $X\in \Sm_B$, and $x\in X\times_B z$, 
we set $U=B-z$, $X_U = X\times_B U$, and $(X^h_x)_U=X^h_x\times_B U$.
Then for any $\A^{1}$-local quasi-stable radditive framed presheaf $\calF$ of $S^{1}$-spectra over $U$,
there is a natural equivalence 
\[
\calF((X^h_x)_U)
\xrightarrow{\simeq} 
\Rep_{\nis}(\calF)((X^h_x)_U).
\]
\end{theorem}

\begin{proof}
We may assume that $B$ is a local domain with generic point $U=\eta$.
In this case, 
we show that the homotopy cofiber $\tildeRep_\nis(\calF)((X^h_x)_U)\simeq 0$.
Let $\nu$ denote the generic point of $X^h_x\times_B U$.
The canonical equivalence of stalks $\Rep_\nis(\calF)(\nu)\simeq \calF(\nu)$ implies that 
$\tildeRep_\nis(\calF)(\nu)\simeq 0$.
Thus we obtain the vanishing
\[
\varinjlim_{Y}\tildeRep_\nis(\calF)((X^h_x)_U-Y)
\simeq 
\tildeRep_\nis(\calF)(\nu)\simeq 0.
\] 
Here the indexing scheme $Y$ runs over the filtered system of closed subschemes in $(X^h_x)_U$ 
of positive relative codimension.
Hence there is a canonical equivalence
\begin{equation}
\label{eq:Fnis(XhxU)otinjlimFnistildeY(XhxU)}
\varinjlim_{Y} \tildeRep_\nis(\calF)_{Y}((X^h_x)_U)
\xrightarrow{\simeq}
\tildeRep_\nis(\calF)((X^h_x)_U).
\end{equation}

In the following, 
we apply the moving technique provided by \Cref{prop:CorFretaVetatoXeta-Yeta} to prove that 
the transition maps in the colimit \eqref{eq:Fnis(XhxU)otinjlimFnistildeY(XhxU)} induce
trivial maps on all homotopy groups.
To begin with, 
we note the equivalence
\begin{equation}\label{eq:change-filtering-indices}
\varinjlim_{Y} \tildeRep_\nis(\calF)_{Y}((X^h_x)_U)
\simeq 
\varinjlim_{\tilde Y, \tilde X} \tildeRep_\nis(\calF)_{\tilde Y}(\tilde X_U).
\end{equation}
Here $\tilde X$ runs over the cofiltered system of \'etale neighborhoods of $x$ in $X$,
and $\tilde Y$ runs over the filtered system of closed subschemes of positive relative codimension in $\tilde X_U$.
The claimed equivalence \eqref{eq:change-filtering-indices} follows since for every closed subscheme $Y$ 
of positive codimension in $(X^h_x)_U$, 
the closure $\bar Y$ of the image of $Y$ in some $\tilde X_U$ is a closed subscheme of positive codimension 
such that $Y\subset \bar Y\times_{\tilde X} X^h_x$.

\Cref{prop:CorFretaVetatoXeta-Yeta} furnishes a framed correspondence $c\in \ZF_N((X^h_x)_U)\times\A^{1},\tilde X_U)$
such that for $j\colon \tilde X_U-Y_U\to \tilde X_U$ and the naturally induced sections
$i_0, i_1\colon (X^h_x)_U\to (X^h_x)_U\times\A^{1}$, 
we have 
\begin{equation}
\label{eq:homotopyconnectsections}
c\circ i_0 = c_0=\can^\nu, 
c\circ i_1 = c_1=j\circ \tilde c_1.
\end{equation}

Recall that $\can^\nu\in \Fr_N((X^h_x)_U, X_U)$ is a framing of the canonical morphism 
$\can\colon (X^h_x)_U\to \tilde X_U$
and $c_1\in \Fr_N((X^h_x)_U, X_U-Y_U)$.
There is a naturally induced morphism 
\[
\can^*\colon \tildeRep_\nis(\calF)_{\tilde Y_U}(\tilde X_U)
\to 
\tildeRep_\nis(\calF)((X^h_x)_U).
\]
By strict $\A^{1}$-invariance over $U$, 
both $\calF$ and $\Rep_{\nis}(\calF)$ are $\A^{1}$-local quasi-stable framed presheaves of $S^{1}$-spectra on $\Sm_U$.
Hence, 
by \eqref{eq:homotopyconnectsections}, 
we have
$(\can^\nu)^*=c_0^*=i_0^*c^*=i_1^*c^*=c_1^*=\tilde c_1^*\circ j^*$.
Since $\tildeRep_\nis(\calF)_{\tilde Y_U}(\tilde X_U-\tilde Y_U)\simeq 0$,
it follows that
\[
(\can^\nu)^*
=
\tilde c_1^*\circ j^*=0.
\]
By \Cref{prop:LnissigmaEeq} and \Cref{cor:LnisrEsimeqr} it follows that $\can^*=0$.
Thus for all $\tilde X$ and $\tilde Y$,  
the morphism $\can^*$ induces the trivial map on homotopy groups.
Hence \eqref{eq:Fnis(XhxU)otinjlimFnistildeY(XhxU)} is trivial on homotopy groups, 
and we conclude the vanishing $\tildeRep_\nis(\calF)( (X^h_x)_U )\simeq 0$.
\end{proof}

\begin{corollary}\label{corollary:ident-Nis-coh-loc-sch-corol}
The identification of Nisnevich cohomology in \eqref{eq:Hnis(((X^h_x)eta)(F))} holds for every generic point $\eta\in B$ 
and any quasi-stable radditive framed presheaf of abelian groups $\calF$ on $\Sm_{{\eta}}$.
\end{corollary}

\begin{proof}
Without loss of generality, 
we may assume $B$ is a local irreducible scheme that equals the closure of $\eta$.
Applying \Cref{th:Hnis(Xhxeta)dimB1} to the Eilenberg-MacLane object $\EM(\calF)\in\SH_{s}(B)$ of $\calF$ finishes the proof.
\end{proof}

\section{Nisnevich strict and 
\texorpdfstring{$\tf$}{tf}-Nisnevich strict 
\texorpdfstring{$\A^{1}$}{A1}-invariance}

In \Cref{sect:counterexample}, we construct a counterexample to the verbatim generalization of the claim of Nisnevich strict \(\A^{1}\)-invariance from base fields to positive-dimensional base schemes B. In \Cref{section:tship}, we propose a weaker form of this generalization, 
called \(\tf\)-Nisnevich strict \(\A^{1}\)-invariance, 
which we shall prove for one-dimensional base schemes in \Cref{section:proofmaintheorem}.

\subsection{Counterexamples to Nisnevich strict \texorpdfstring{$\A^{1}$}{A1}-invariance}
\label{sect:counterexample}

Following Voevodsky's \Cref{th:shiCor} one may ask whether the Nisnevich sheafification of any 
$\A^{1}$-invariant presheaf with transfers over $B$ is strictly $\A^{1}$-invariant.
An equivalent statement is that for any such presheaf $\calF$ and essentially smooth local henselian scheme 
$U$ over $B$, 
we have 
\begin{equation}
\label{eq:contrex:Hnis(A1UF)}
H^i_{\nis}(\A^{1}_U,\calF_\Nis)
\cong
\begin{cases}
\calF(U) & i=0, \\
0 & i>0.
\end{cases}
\end{equation}

We give a counterexample to \eqref{eq:contrex:Hnis(A1UF)} for every positive dimensional base scheme $B$.
Owing to \Cref{lemma:reductiontolocalhensel} we may assume $B$ is a local henselian scheme.

\begin{lemma}
\label{lemma:reductiontolocalhensel}
The strict $\A^{1}$-invariance theorem in the form of \eqref{eq:contrex:Hnis(A1UF)} holds over a 
base scheme $B$ if and only if it holds over $B^h_\sigma$ for every $\sigma\in B$. 
\end{lemma}
\begin{proof}
Any essentially smooth local henselian scheme $U$ over $B^h_\sigma$ is an essentially 
smooth local henselian scheme over $B$. 
Conversely, 
any essentially smooth local henselian scheme $U$ over $B$ is a scheme over $B^h_\sigma$, 
where $\sigma$ is the image of the closed point of $U$.
Our claim for \eqref{eq:contrex:Hnis(A1UF)} follows readily from these observations.
\end{proof}

\begin{example}
\label{example:H(A1U)}
Let $B$ be a local scheme.
Assume $\dim B>0$,
and let $f\in \mathcal O(B)$ be a regular function such that the vanishing locus $Z(f)$ is a 
proper closed subscheme of $B$.
Letting $t$ denote the coordinate on $\A^{1}_\Z$, 
we have the rational function 
\[
r=ft-1
\in 
\mathcal O(\A^{1}_\Z\times B), 
\]
and the open complement 
\[
V
=
\A^{1}_B-Z(fr).
\] 
Let $\calE$ be the $\A^{1}$-invariant presheaf with transfers defined as the cokernel
\[
\calE\defeq\Coker\p*{\mathrm{Cor}_B(\A^{1}_B\times_B (-),V)\xrightarrow{i_0^*-i_1^*}\mathrm{Cor}_B(-,V)}
\]
for the canonical sections of the affine line over $B$.
Consider the class $\mathrm{id}\in H^0_{\nis}(V,\calE)$ given by the identity morphism on $V$. 
We claim that $\mathrm{id}\in H^0_{\nis}(V,\calE)$ maps to a nontrivial class in $H^1_{\nis}(\A^{1}_B,\calE)$
under the boundary morphism $\delta$ in the Mayer-Vietoris sequence
\begin{multline*}
H^0_{\nis}((\A^{1}_B-Z(f))\amalg (\A^{1}_B-Z(r)),\calE)\to 
H^0_{\nis}(V,\calE)\xrightarrow{\delta}H^1_{\nis}(\A^{1}_B,\calE)\to \\
\to H^1_{\nis}(\A^{1}_B-Z(f))\amalg (\A^{1}_B-Z(r)),\calE)\to \cdots.
\end{multline*}
First, 
we note that $\Cor_B(\A^{1}_B-Z(r),V)$ is the trivial group, 
as the following argument shows.
By assumption, $Z(f)$ is a nonempty proper closed subscheme of $B$, 
and moreover $Z(f)\times_B V=\emptyset$.
Suppose that $W$ is  a nonempty irreducible closed subscheme of $(\A^{1}_B-Z(r))\times_B V$ which is finite 
and surjective over $\A^{1}_B-Z(r)$;
in other words, assume that $W\in\Cor_B(\A^{1}_B-Z(r),V)$ is nontrivial.
Then the fiber of $W$ over $Z(f)$ would be nonempty. 
On other hand, 
since $Z(f)\times_B V=\emptyset$ and $W\subset (\A^{1}_B-Z(r))\times_B V$,
it follows that $Z(f)\times_B W=\emptyset$.
This shows the vanishing $H^0_{\nis}(\A^{1}_B-Z(r),\calE)=0$.
Thus, if $\delta(\mathrm{id})=0$, 
then the element $\mathrm{id}\in H^0_{\nis}(V,\calE)$ is in the image of $H^0_{\nis}(\A^{1}_B-Z(f),\calE)$. 
But the latter would imply that for every $\A^{1}$-invariant presheaf with transfers $\calF$ over the 
nonempty scheme $B-Z(f)$, 
the canonically induced morphism
\begin{equation}
\label{equation:surjection}
\calF\p*{\A^{1}_B-Z(f)\cong \A^{1}_{B-Z(f)}}
\to 
\calF\p*{V\cong\A^{1}_{B-Z(f)}-Z(r)}  
\end{equation}
is surjective. 
However, 
this fails for the $\A^{1}$-invariant presheaf with 
transfers of global units $\mathcal O^\times$ on $\Sm_{B-Z(f)}$
because in this case 
$r^{-1}\in\mathcal O_{V}(V)^\times$ 
does not belongs to the image of \eqref{equation:surjection}.
\end{example}

For a general base scheme,
\Cref{lemma:withtransfersandframed} shows strict $\A^{1}$-invariance for additive quasi-stable 
framed presheaves of abelian groups imply \SHI for presheaves with transfers. 
By \Cref{th:SHIeqS1Ab1,th:SHIeqS1Ab}, 
\SHI for framed presheaves of abelian groups is equivalent to Nisnevich \SHI in the sense of \Cref{def:SHIthnistriv}.
In summary, 
\Cref{example:H(A1U)} disproves all the three naive variants of \SHI discussed in this section.

\begin{lemma}
\label{lemma:withtransfersandframed}
Every presheaf with transfers of abelian groups on $\Sm_{B}$ has the structure 
of a framed additive presheaf of abelian groups on $\Sm_{B}$. 
Moreover, every such framed presheaf is quasi-stable.
\end{lemma}
\begin{proof}
If $c=(S, V,\varphi,g)\in \Fr_n(X,Y)$ then its (non-reduced) support $S$ is isomorphic to the vanishing locus 
$Z(\varphi)$ in $V$,
where $V$ is an \'etale neighborhood of $S$ in $\A^n_X$.
Note that $S$ is a closed subscheme in $\A^n_X$ and finite over $X$.
It follows that $S$ is either finite surjective over $X$ or empty.
The same holds for the image $\underline{S}$ of $S$ in $X\times Y$ under the canonical projection $S\to X$ and 
$g\colon V\to Y$.
The irreducible components $\underline{S}_i$ of $\underline{S}$ yield the finite correspondence
\begin{equation}\label{eq:ctr=sumci}
c^\mathrm{tr}=\sum_{i} c_i \underline{S}_i
\in
\Cor(X,Y).
\end{equation}
The sum runs over the generic points $\eta_i$ of $S$ and 
$c_i=\dim_{K}\calO_{S_{(\eta_i)}}(S_{(\eta_i)})/\dim_{K}{\calO_{\eta_i}(\eta_i)}$,
where $K$ is the residue field of the image of $\eta_i$ in $X$,
$S_{(\eta_i)}$ is the local scheme of $S$ at $\eta_i$,
and $\underline{S}_i$ is the closure of the image of $\eta_i$ in $X\times Y$.
Now the first claim follows by using the functor 
\begin{equation}
\label{eq:Fr+Cor}
\Fr_+(B)\to \Cor(B);
c\mapsto c^\mathrm{tr}.
\end{equation}
Finally, the equality $(\sigma c)^\mathrm{tr}=\sigma (c^\mathrm{tr})$ implies the second claim.
\end{proof}
\begin{remark}
We note that \eqref{eq:Fr+Cor} preserves the composition by reduction to fields.
Indeed,
it suffices to consider fields since the multiplicities of the cycles in \eqref{eq:ctr=sumci} 
are defined via the fibers over the generic points of $B$.
\end{remark}

\subsection{\texorpdfstring{$\tf$}{tf}-Nisnevich strict \texorpdfstring{$\A^{1}$}{A1}-invariance}
\label{section:tship}

The strict $\A^{1}$-invariance theorems of Garkusha-Panin \cite[\S 17]{hty-inv} and 
Voevodsky \cite[\S3.2]{Voe-hty-inv} concern the effect of the Nisnevich localization functor 
$L_{\nis}$ on $\A^{1}$-local objects.
The verbatim generalization to a base scheme $B$ and $S^{1}$-spectra takes the following form.

\begin{definition}
\label{def:SHIthnistriv}
Let $\mathcal F$ be an $\A^{1}$-local quasi-stable radditive framed presheaf of $S^{1}$-spectra on $\Sm_{B}$. 
We say that Nisnevich strict ${\mathbb A}^{1}$-invariance holds on $\Sm_{B}$ if, 
for every $\calF$ as above, 
$\Lrep_{\nis}(\calF)$ is an $\A^{1}$-local quasi-stable radditive framed presheaf of $S^{1}$-spectra. 
\end{definition}

In \Cref{def:SHIthnistriv} we may replace $\Sm_{B}$ with $\SmBcZ$, $\SmAff_B$, $\SmAffScZ$, $\SmatB$, $\SmatBcZ$.
In \Cref{section:defsmZsmB} we prove that Nisnevich strict ${\mathbb A}^{1}$-invariance holds on $\Sm_{B,Z}$ 
when $Z$ is a zero-dimensional scheme.
In \Cref{sect:counterexample}, 
however,
we show that Nisnevich strict ${\mathbb A}^{1}$-invariance fails on $\Sm_B$ for any positive dimensional 
base scheme.
The same holds for $\Sm_{B,Z}$ when $Z$ is positive dimensional. 
These observations motivate the following definition.

\begin{definition}
\label{def:SHIthnistf}
Let $\mathcal F$ be a quasi-stable radditive framed presheaf of $S^{1}$-spectra on $\Sm_{B}$
such that $\Lrep_\tf(\calF)$ is $\A^{1}$-local.
We say that $\tf$-Nisnevich strict ${\mathbb A}^{1}$-invariance holds on $\Sm_B$ if, 
for every $\calF$ as above, 
$\Lrep_{\nis}(\calF)$ is an $\A^{1}$-local quasi-stable radditive framed presheaf of $S^{1}$-spectra. 
\end{definition}

The property on $B$ stated in \Cref{definition:SHI} implies the one stated in \Cref{def:SHIthnistf}.
In fact, 
these properties are equivalent because restricting the endofunctor 
$\Lrep_\Nis \colon \HH_s(B)\to \HH_s(B)$ to the subcategory of $\tf$-local objects amounts to taking 
the composite 
\[
\HH_\tf(B)\xrightarrow{L^\tf_\Nis} \HH_\Nis(B) \xrightarrow{R^\Nis_\tf} \HH_\tf(B).
\]

\begin{remark}
\Cref{thm:SHIrc} shows that $\tf$-Nisnevich strict ${\mathbb A}^{1}$-invariance implies $\tf$-Nisnevich strict 
${\mathbb A}^{1}$-invariance for presheaves of abelian groups,
see \Cref{def:str_xibase}.
In \Cref{def:SHIthnistf} we can replace $\Sm_{B}$ with $\SmBcZ$, $\SmAff_B$, $\SmAffScZ$, $\SmatB$, $\SmatBcZ$.
\end{remark}

\begin{remark}
\label{rm:SHIdim0NistfNis}
By \Cref{prop:Proptftop} (iii) (resp.~\Cref{lmSmBcsigmatftopologytrivial}), 
the $\tf$-topology is trivial on $\Sm_B$ (resp.~$\Sm_{B,Z}$) if and only if $\dim B=0$ (resp.~$\dim Z=0$).
In this case,
the notions of Nisnevich and $\tf$-Nisnevich strict ${\mathbb A}^{1}$-invariance coincide.
\end{remark}

In \Cref{section:proofmaintheorem} we prove $\tf$-Nisnevich strict ${\mathbb A}^{1}$-invariance for 
one-dimensional schemes.
Next, 
we discuss some reduction steps and establish the equivalence between different forms of Nisnevich and
$\tf$-Nisnevich strict ${\mathbb A}^{1}$-invariance.

\begin{lemma}
\label{lm:baselocality}
Suppose $\tf$-Nisnevich strict ${\mathbb A}^{1}$-invariance holds on $\Sm_{B_\sigma}$ for all $\sigma\in B$, 
where $B_\sigma$ denotes the local scheme at $\sigma$.
Then $\tf$-Nisnevich strict ${\mathbb A}^{1}$-invariance holds on $\SmB$.
\end{lemma}
\begin{proof}
\Cref{lm:BaseChangeCommutes} shows that the base change morphisms 
$f^*\colon \SH_{s}(B)\to \prod_{\sigma\in B}\SH_{s}(B_\sigma)$ and 
$f^*\colon \SH_{s}(\Fr_+(\Sm_{B}))\to \prod_{\sigma\in B}\SH_{s}(\Fr_+(\Sm_{B_\sigma}))$
along $\coprod_{\sigma\in B} B_\sigma\to B$ commute with $L_{\nis}$, $L_\tf$, and $L_{\A^{1}}$.
In particular, 
$f^*$ preserves $\A^{1}$-local objects. 
Since $f^*$ detects $\A^{1}$-local objects on the subcategory of $\tf$-local objects, 
see \Cref{lm:infZarCoveringBaseChange}, 
we are done.
\end{proof}

\begin{lemma}
\label{cor:AffSmScZSmat}
Suppose $B$ is an affine scheme and $Z\not\hookrightarrow B$ is a closed immersion.
Then $\tf$-Nisnevich strict $\A^{1}$-invariance holds on $\SmatBcZ$ if and only if it 
holds on $\SmAff_{B,Z}$.
\end{lemma}
\begin{proof}
By \Cref{prop:AffSmScZSmat} the restriction functor $r^\cci\colon\HHtriv(\SmAff_{B,Z})\to\HHtriv(\SmatBcZ)$ 
preserves and detects $\A^{1}$-local objects, 
and it commutes with the localization endofunctors $\Lrep_\nis$ and $\Lrep_\tf$. 
The desired stable result follows by arguing levelwise.
\end{proof}

\begin{lemma}
\label{cor:SmScZSmatA1locpreserve}
Suppose $B$ is an affine scheme and $Z\not\hookrightarrow B$ is a closed immersion.
Then Nisnevich strict ${\mathbb A}^{1}$-invariance holds on $\SmBcZ$ if and only if it holds on $\SmatBcZ$.
\end{lemma}
\begin{proof}
\Cref{prop:AffSmScZSmat} allows us to replace $\Smat_{B,Z}$ with $\SmAff_{B,Z}$.
Moreover,
via \Cref{lm:SmSmaffA1surjectiveandpreserves} and \Cref{prop:SmAffSmniseqivalence} we arrive at $\SmBcZ$.
\end{proof}

\section{\texorpdfstring{$\tf$}{tf}-localization by closed subschemes}
\label{sect:tfLoc}
In this section, 
for a fixed closed immersion $Z\not\hookrightarrow B$,
we prove 
our localization or recollement theorem 
in $\SHstf(B)$ discussed in \Cref{subsection:tfmlt}
without $\A^1$-localization, i.e., the claim
regarding \eqref{eq:locpriShs1nis(Sms)}. 
This allows us to reduce the problem of $\tf$-Nisnevich strict $\A^{1}$-invariance 
for $\SH_{s,\tf}(\Sm_{B})$ 
to $\SHstf(B,Z)$ and $\SH_{s,\tf}(\Sm_{B-Z})$.

\subsection{\texorpdfstring{$\tf$}{tf}-localization}
Pertinent to localization is the geometric observation that every $X\in \Sm_{B}$ allows an ``infinitesimal'' $\tf$-covering by 
the essentially smooth $B$-scheme $X^h_Z$ 
and smooth $B$-scheme $X\times_B(B-Z)$. 
If $B$ is one-dimensional, 
then with the exception of $X$ all the terms in the $\tf$-square 
\[
\xymatrix{ 
(X^h_Z)\times_B(B-Z) \ar[d]\ar[r]& X\times_B(B-Z)\ar[d] \\ 
X^h_Z\ar[r] & X
}
\] 
admit only trivial $\tf$-coverings.
The second critically important ingredient is the triviality of the Nisnevich cohomology of $X\times_B(B-Z)$ when 
$X$ is an essentially smooth local henselian scheme.
In contrast to the localization theorem for $\SH(B)$, 
we work with $\Sm_{B,Z}$ instead of $\Sm_{Z}$.
This allows us to extend the localization theorem for motivic equivalences to the level of Nisnevich local equivalences.

The pointed homotopy categories we study are related via the functors:
\begin{equation}
\label{eq:LocNisZSquare}
\xymatrix{
& \HHtrivptd(\SmB) & \\
\HHtrivptd(\SmBcZ)\ar[ru]^{\tids} & 
&\HHtrivptd(\SmBmZ)\ar[lu]_{\jds} \\
& \HHtrivptd(\SmB) \ar[lu]_{\tif}\ar[ur]^{\jus} &
}
\end{equation}
For $X\in \Sm_{B}$, $X^h_Z\in \SmBcZ$, $V\in \Sm_{B-Z}$, 
the functors $\tif$, $\jus$, $\tids$, $\jds$ are given by
\begin{equation}
\label{eq:tifjustidsjdscoPresentable}
\begin{array}{ll}
\tif(\calF)(X^h_Z)=\hofib(\calF(X^h_Z)\to \calF(X^h_Z-X_Z)), 
& \calF\in \HHtrivptd(\SmB), \\
\jus(\calF)(V)=\calF(V), &  \calF\in \HHtrivptd(\SmB), \\
\tids(\calF)(X) = \calF(X^h_Z), & \calF\in \HHtrivptd(\SmBcZ), \\
\jds(\calF)(X) = \calF(X-X_Z), & \calF\in \HHtrivptd(\SmBmZ).
\end{array}
\end{equation}
The $\tf$-localization of \eqref{eq:LocNisZSquare} induces the same adjunctions as in 
\eqref{eq:locpriShs1nis(Sms)} for $\tau=\tf$ and the pointed homotopy category $\HH_{s,\bullet}(\Sm_B^\tf)$.
We shall use similar notation for $\SH_{s}(\SmB)$, $\SH_{s}(\SmBcZ)$, and $\SH_{s}(\SmBmZ)$.
In this context, 
the motivic localization endofunctor $\Lrep_{\mot}$ preserves the adjunctions in \eqref{eq:locpriShs1nis(Sms)} 
and their properties.

\begin{prop}
\label{prop:filtr-iFj:unstable}
For every $\tf$-local object $\calF\in \HHtrivptd(\Sm_B)$ there is a homotopy pullback square
\begin{equation}
\label{eq:homotopypullback_tidsfFjdus}
\xymatrix{
\tids\tif(\calF)\ar[r]\ar[d]& \calF\ar[d] \\ 
\ast \ar[r]&\jds\jus(\calF).
}
\end{equation}
The same holds for every $\tf$-local object $\calF\in \HHtrivptd(\SmAff_B)$.
\end{prop}

\begin{remark}
We would like to emphasize that, in contrast to the usual $\A^1$-version of the localization (or recollement) theorem appearing in \Cref{prop:filtr-iFj:unstable}, our formulation is carried out in the category $\SmBcZ$ and involves the functor $\tilde{i}^!$, rather than the category $\Sm_Z$ and the functor $i^!$. 
This distinction is essential: working in $\SmBcZ$ with $\tilde{i}^!$ allows us to obtain the desired statement without passing to the $\A^1$-localized setting. In particular, no $\A^1$-localization is required in the proof or formulation of \Cref{prop:filtr-iFj:unstable}.
\end{remark}

\begin{proof}
Let $\mathcal H$ be shorthand for the homotopy fiber $\hofib(\calF\to \jds\jus(\calF))$. 
By the definitions of $\tids$, $\tif$, $\jds$, $\jus$, 
there are isomorphisms
\begin{align*} 
\tids\tif(\calF)(X)&\cong \hofib(\calF(X^h_Z)\to \calF(X^h_Z\times_B(B-Z))),\\
\mathcal H(X)&\cong \hofib(\calF(X)\to\calF(X\times_B(B-Z))).
\end{align*}
For every $\tf$-local simplicial presheaf $\calF$, 
there is a homotopy pullback square
\[
\begin{tikzcd}
\calF(X)\ar{r}\ar{d} &\calF(X\times_B(B-Z))\ar{d}\\
\calF(X^h_Z)\ar{r} & \calF(X^h_Z\times_B(B-Z)).
\end{tikzcd}
\]
Hence there is a canonically induced isomorphisms
\begin{equation}
\label{eq:cofib=tildeidsf}
\mathcal H(X)
\xrightarrow{\cong} 
\tids\tif(\calF)(X).
\end{equation}
In more detail, let $\calF$ be a $\tf$-local simplicial presheaf on $\Sm_B$, then it follows that  
for any {\'e}tale neighborhood $\widetilde X$ of $Y=X\times_B Z$ in $X$, 
there is an equivalence
\[
\hofib(\calF(X)\to\calF(X\times_B(B-Z)))
\stackrel{\simeq}{\to} 
\hofib(\calF(\widetilde X)
\to
\calF(\widetilde X\times_B(B-Z))).
\]
Thus for the henselization $X^h_Z$, 
there is an equivalence
\[
\hofib(\calF(X)\to\calF(X\times_B(B-Z)))
\stackrel{\simeq}{\to} 
\hofib(\calF(X^h_Z)\to\calF(X^h_Z\times_B(B-Z))).
\]
Clearly \eqref{eq:cofib=tildeidsf} implies \eqref{eq:homotopypullback_tidsfFjdus}.
\end{proof}

\begin{corollary}
\label{cor:filtr-iFj:stable}
For every $\tf$-local object $\calF\in \SH_{s}(\Sm_B)$ there is a distinguished triangle
\begin{equation}
\label{eq:triangle_tidsfFjdus}
\tids\tif(\calF)\to \calF\to \jds\jus(\calF)\to \tids\tif(\calF)[1].  
\end{equation}
\end{corollary}

Over a one-dimensional local base scheme, 
the distinguished triangle \eqref{eq:triangle_tidsfFjdus} allows us to describe the 
$\tf$-localization endofunctor on $\SH_{s}(\Sm_B)$. 

\begin{prop}
\label{prop:B1dimlocLtfhocofib}
Suppose $B$ is a one-dimensional local base scheme with closed point $z$.  
Then for every $\calF\in \SH_s(\Sm_B)$ there is a canonical isomorphism
\begin{equation}\label{eq:onedimB:LtfF}
\Lrep_\tf(\calF) 
\cong
\hocofib(\jds\jus(\calF)[-1]\to \tids\tif(\calF)).    
\end{equation}
That is, for every $X\in \Sm_B$, there is an isomorphism
\[
\Lrep_\tf(\calF)(X)
\cong 
\hofib(\calF(X\times_B (B-z))
\oplus 
\calF(X^h_z)
\to 
\calF(X^h_z\times_B (B-z))).
\]
\end{prop}
\begin{proof}
By \Cref{cor:filtr-iFj:stable} there is the distinguished triangle 
\begin{equation}\label{eq:tidstifLtfFLtfFjdsjusLtfF}
\tids\tif(\Lrep_\tf(\calF)) \to \Lrep_\tf(\calF)  \to \jds\jus(\Lrep_\tf(\calF)).
\end{equation}
By \Cref{prop:tifjdstfloceq,prop:tifNislocobj,prop:tidsjusjdsLnisLtf}
proven below, we have 
$\tif(\Lrep_\tf(\calF))\simeq\Lrep_\tf(\tif(\calF))$
and
$\jus(\Lrep_\tf(\calF))\simeq\Lrep_\tf(\jus(\calF))$.
Since the $\tf$-topologies on $\Sm_{B,z}$ and $\Sm_{B-z}$ are trivial by \Cref{prop:Proptftop}(iii),
the outer terms in \eqref{eq:tidstifLtfFLtfFjdsjusLtfF} are equivalent to
$\tids\tif(\calF)$
and
$\jds\jus(\calF)$.
Thus there is the distinguished triangle 
\[
\tids\tif(\calF) \to \Lrep_\tf(\calF) \to \jds\jus(\calF).
\]
The equivalence \eqref{eq:onedimB:LtfF} follows.
\end{proof}

\begin{remark}
When $B$ is one-dimensional, the $\tf$-topologies on $\Sm_Z$ and $\Sm_{B-Z}$ are trivial. 
Thus we do not need to $\tf$-sheafify $\jds\jus(\calF)$ or $\tids\tif(\calF)$ in the proof 
of \Cref{prop:B1dimlocLtfhocofib}.
\end{remark}

\subsection{Localization for motivic local objects}
\label{subsection:lfmlo}

\begin{theorem}
\label{th:HA1tflocalizationrecollement}
For every motivic local object $\calF\in \HHtrivptd(\Sm_B)$ 
(resp.~every $\tf$-motivic local object $\calF\in \HHtrivptd(\SmAff_B)$) there is a homotopy pullback square
\begin{equation}
\label{eq:pullback_idsfFjdus}
\xymatrix{
\ids\iuf(\calF)\ar[r]\ar[d]& \calF\ar[d] \\ \ast \ar[r]&\jds\jus(\calF).  
}
\end{equation}
The same holds for every motivic local object $\calF\in \HHtrivptd(\Fr_+(B))$ 
(resp.~every $\tf$-motivic local object $\calF\in \HHtrivptd(\Fr_+(\SmAff_B))$).
\end{theorem}
\begin{proof}
We note there exist factorizations: 
\[\begin{tikzcd}
\HHtrivptd(\Sm_Z)\ar{r}{i_*^{B,Z}}\ar[rr,bend left,"i_*"] 
& \HHtrivptd(\SmBcZ) \ar{r}{\wt i_*} & \HHtrivptd(\Sm_B)\\
\HHtrivptd(\Sm_B)\ar{r}{\wt i^!}\ar[rr,bend right,"i^!"] 
& \HHtrivptd(\SmBcZ)\ar{r}{i^*_{B,Z}} & \HHtrivptd(\Sm_Z)
\end{tikzcd}\]
\Cref{prop:filtr-iFj:unstable} implies that for every motivic local object $\calF\in\HHtrivptd(\Sm_B)$, 
and similarly for $\HHtrivptd(\SmAff_B)$, 
there is a homotopy pullback square
\begin{equation*}
\label{eq:pullback_tidsfFjdus}
\xymatrix{
\tids\tif(\calF)\ar[r]\ar[d]& \calF\ar[d] \\ \ast \ar[r]&\jds\jus(\calF).
}
\end{equation*}
To conclude the proof we appeal to \Cref{lm:motloc_id=iusids(Sm)}, 
and \Cref{lm:A1loc_id=iusids(SmAff)} for $\HHtrivptd(\SmAff_B)$, 
which shows that there is an equivalence
\[
\tids\tif(\calF) 
\xrightarrow{\cong}
i_*^{B,Z} i^*_{B,Z} \tids\tif(\calF). 
\]
\end{proof}

\subsection{Framed \texorpdfstring{$\tf$}{tf}-localization}
To show $\tif$, $\tids$, $\jus$, $\jds$ preserve quasi-stable radditive framed presheaves we work with $\Spc_s(\Fr_+(B))$ 
and $\HHtrivptd(\Fr_+(B))$, 
and similarly for $\Fr_+(B-Z)$ and $\Fr_{+}(B,Z)$.
Due to \Cref{lm:HensBasechangeFrCor}, 
see also \Cref{rm:FrCorSchScZ}, 
there are well defined functors
\begin{equation}
\label{eq:tildeijFr_+}
\begin{array}{l}
\Fr_+(\EssSm_B)\to \Fr_+(\EssSm_B); X\mapsto X^h_Z,\\
\Fr_+(\EssSm_B)\to \Fr_+(\EssSm_{B-Z}); X\mapsto X\times_B (B-Z).
\end{array}
\end{equation}
As in \eqref{eq:tifjustidsjdscoPresentable} there exist base change functors
\[
\begin{array}{lll}
\tif\colon\HHtrivptd(\Fr_+(\EssSm_{B}))&\to& \HHtrivptd(\Fr_+(\EssSm_{B,Z})) , \\
\tids\colon\HHtrivptd(\Fr_+(\EssSm_{B,Z}))&\to& \HHtrivptd(\Fr_+(\EssSm_{B})), \\ 
\jus\colon\HHtrivptd(\Fr_+(\EssSm_{B}))&\to& \HHtrivptd(\Fr_+(\EssSm_{B-Z})) , \\
\jds\colon \HHtrivptd(\Fr_+(\EssSm_{B-Z}))&\to&\HHtrivptd(\Fr_+(\EssSm_{B}).
\end{array}\]
For simplicity we let $\Sm\in\{\SmBcZ,\SmB,\SmBmZ\}$ and $\Fr_+\in\{\Fr_+(B,Z),\Fr_+(B),\Fr_+(B-Z)\}$.
An object $\calF\in \HHtrivptd(\Fr_+)$ is Nisnevich local (resp.~$\tf$-local) if 
it maps to a Nisnevich local (resp.~$\tf$-local) object under the forgetful functor 
\begin{equation}
\label{eq:SHsfr->SHs}
\HHtrivptd(\Fr_+)\to \HHtrivptd(\Sm).
\end{equation}

\begin{prop}
\label{prop:StFrRad}
The functors $\tif$, $\tids$, $\jus$, $\jds$ commute with the forgetful functor in \eqref{eq:SHsfr->SHs}
and preserve quasi-stable radditive framed presheaves.
\end{prop}

\begin{proof}
The first claim follows from the definitions of $\tif$, $\tids$, $\jus$, $\jds$. 
The functors in \eqref{eq:tildeijFr_+} preserve coproducts of schemes,
and for every $X$,
the framed correspondence $\sigma_X$ maps to $\sigma_{X^h_Z}$ in $\EssSm_{B}$ and $\sigma_{X_{B-Z}}$ in $\EssSm_{B-Z}$.
\end{proof}

Similarly to \Cref{prop:filtr-iFj:unstable}, \Cref{cor:filtr-iFj:stable}, and \Cref{prop:B1dimlocLtfhocofib} we have:
\begin{prop}
\label{prop:filtr-iFjFr}
For every $\tf$-local object $\calF\in \HHtrivptd(\Fr_+)$ there is a homotopy pullback square
\begin{equation*}
\label{eq:homotopypullback_tidsfFjdus:Fr}
\xymatrix{
\tids\tif(\calF)\ar[r]\ar[d]& \calF\ar[d] \\ \ast \ar[r]&\jds\jus(\calF).  
}
\end{equation*}
Thus for every $\tf$-local object $\calF\in \SH_{s}(\Fr_+)$ there is a distinguished triangle
\begin{equation*}
\label{eq:triangle_tidsfFjdus:Fr}
\tids\tif(\calF)\to \calF\to \jds\jus(\calF)\to \tids\tif(\calF)[1].  
\end{equation*}
\end{prop}

\begin{definition}
Suppose $B$ is a one-dimensional local base scheme. 
Let $z$ be the closed point of $B$.
We define the functor 
\[
\Lrep_{\tf}^\fr
\colon 
\SH_s(\Fr_+(B))
\to 
\SH_s(\Fr_+(B))
\]
by setting
\begin{equation}\label{eq:diag:LfrtfLtf}
\Lrep^\fr_\tf(\calF)
:=
\hocofib(\jds\jus(\calF)[-1]\to \tids\tif(\calF)).
\end{equation}
That is, 
for each $X\in \Sm_B$, 
we have
\[
\Lrep^\fr_\tf(\calF)(X)
\simeq
\hofib(\calF(X\times_B (B-z))\oplus \calF(X^h_z)\to \calF(X^h_z\times_B (B-z))).
\]
\end{definition}

\begin{corollary}
\label{cor:onedimLnisLnisfr}
Suppose $B$ is a one-dimensional local base scheme. 
The functor 
$\Lrep_{\tf}^\fr$ 
preserves quasi-stable radditive framed presheaves of $S^1$-spectra.
Moreover, 
the diagram
\begin{equation}
\label{equation:Lrepnisfr}
\begin{tikzcd}
\SH_s(\Fr_+(B))\ar[r,"\Lrep_{\tf}^\fr"]\ar{d} & \SH_s(\Fr_+(B))\ar{d} \\
\SH_s(B)\ar[r,"\Lrep_{\tf}"] & \SH_s(B)
\end{tikzcd}
\end{equation}
commutes.
It follows that $\Lrep_{\tf}$ preserves quasi-stable radditive framed presheaves of $S^1$-spectra.
\end{corollary}
\begin{proof}
Let $z$ be the closed point of $B$.
We note that $\Lrep^\fr_\tf$ preserves quasi-stable radditive framed presheaves of $S^1$-spectra:
This follows since the functors $\SmB\to \Sm_{B-Z}$; $X\mapsto X\times_B (B-z)$, 
$\Sm_B\to \Sm_{B,Z}$; $X\mapsto X^h_z$, 
are well defined on the level of framed correspondences,
as in \Cref{lm:HensBasechangeFrCor}, 
and preserve coproducts and framed correspondences of the form $\sigma_X$.
\Cref{prop:B1dimlocLtfhocofib} implies the commutativity of \eqref{equation:Lrepnisfr}.
\end{proof}

\subsection{\texorpdfstring{$\A^{1}$}{A1}-, Nisnevich- and \texorpdfstring{$\tf$}{tf}-localization}

In the following, 
we show the functors $\tif$, $\tids$, $\jus$, and $\jds$ preserve the various locality conditions on 
radditive framed presheaves.

\begin{prop}
\label{prop:A1local}
The functors $\tids$, $\jus$, $\jds$ preserve $\A^{1}$-local objects. 
The same holds for $\tif$ on the subcategory of $\tf$-local objects.
\end{prop}
\begin{proof}
If $\calF$ is $\A^{1}$-local and $\tf$-local there are equivalences of $S^{1}$-spectra
\begin{align*} 
\tif \calF\p*{(\A^{1}\times X)^h_Z}
&\stackrel{(1)}{\simeq} \calF\p*{(\A^{1}\times X)^h_Z/((\A^{1}\times X)^h_Z\times(B-Z))} \\ 
&\stackrel{(2)}{\simeq} \calF\p*{(\A^{1}\times X)/(\A^{1}\times X\times(B-Z))}\\
&\stackrel{(3)}{\simeq} \calF\p*{X/(X\times(B-Z))}\\
&\stackrel{(4)}{\simeq} \calF\p*{\XhZ/(\XhZ\times(B-Z))}\\
&\stackrel{(5)}{\simeq} \tif \calF\p*{\XhZ},
\end{align*}
Here we write $\calF\p*{X/U}=\fib(\calF(X)\to\calF(U))$ for a morphism $U\to X$.
In the above,  
(1) and (5) hold by definition,
(2) and (4) use $\tf$-localness and the respective $\tf$-Nisnevich squares,
while (3) holds by $\A^1$-localness. 
The claims for $\tids$, $\jus$, $\jds$ follow because, 
in suggestive notation,
the functors 
\[
\begin{array}{llllllll}
\SmB&\to& \SmBcZ;& X&\mapsto& W&=&X^h_Z,\\
\SmBmZ&\to& \SmB;& V&\mapsto& X&=&V,\\
\SmB&\to&\SmBmZ; &X&\mapsto& V&=&X\times_B(B-Z), 
\end{array}
\]
commute with the corresponding endofunctors
\[\begin{array}{llllll}
\SmBcZ&\to&\SmBcZ;& W&\mapsto& W\times_{B,Z}\A^{1}_{B,Z},\\
\SmB&\to&\SmB;& X&\mapsto& X\times_B\A^{1}_B,\\ 
\SmBmZ&\to&\SmBmZ;& V&\mapsto& V\times_{B-Z}\A^{1}_{B-Z}.
\end{array}\]
\end{proof}

\begin{prop}
\label{prop:tidsjusjdsLnisLtf}
The functors $\tids$, $\jus$ preserve Nisnevich local objects and Nisnevich local equivalences.
The functor $\jds$ preserves Nisnevich local objects.
The same results hold for $\tf$-local objects and $\tf$-local equivalences.
\end{prop}
\begin{proof}
\Cref{lm:topologically-XtoXhs,lm:tf-topologically-XtoXhs,lm:LocEqCovPoints,lm:schemespointssites} 
imply the claim for $\tids$.
The case of $\jus$ follows similarly 
because the functor $\Sm_{B-Z}\rightarrow \Sm_{B}$ 
preserves fiber products, Nisnevich coverings, and points.
The latter also holds for $\tf$-coverings and points
(for $\jds$, 
we use the first part of \Cref{lm:LocEqCovPoints}).
\end{proof}

\begin{prop}
\label{prop:tifNislocobj}
The functor $\tif$ preserves Nisnevich local objects and $\tf$-local objects.
\end{prop}
\begin{proof}
Suppose $\calF\in \HHtriv(\Sm_B)$ is Nisnevich local (resp.~$\tf$-local).
By \Cref{definition:NiswNis,definition:tf}, 
any Nisnevich covering (resp.~$\tf$-covering) $\widetilde X^h_Z\to X^h_Z$ in $\SmBcZ$ is induced by 
a Nisnevich covering (resp.~$\tf$-covering) $\widetilde X\to X$ in $\Sm_B$. 
For $\check C_{\widetilde X}(X,\calF)$, 
see \eqref{eq:Check(widetilde X)},
we have equivalences
\begin{align*}
\check{C}_{\widetilde X^h_Z}(X^h_Z, \tif\calF)
&\simeq
\hofib\p*{ \check{C}_{\widetilde X}(X, \calF)\to 
\check{C}_{\widetilde X\times_B (B-Z)}(X\times_B (B-Z), \calF)}\\
&\simeq\hofib\p*{ \calF(X)\to \calF(X\times_B (B-Z)) }\\
&\simeq\hofib\p*{ \calF(X^h_Z)\to \calF(X^h_Z\times_B (B-Z)) }\\
&\simeq\tif\calF(X^h_Z).
\end{align*}
The first and fourth equivalences follow from the definition of $\tif\calF$ in \eqref{eq:tifjustidsjdscoPresentable},
while the second and third equivalences follow since $\calF$ is Nisnevich local (resp.~$\tf$-local).
\end{proof}

\begin{prop}
\label{prop:tifjdstfloceq}
If $B$ is a one-dimensional base scheme then $\tif$ and $\jds$ preserve $\tf$-equivalences on $\SH_{s}(\Sm_B)$. 
\end{prop}

\begin{proof}
By \Cref{prop:Proptftop}(vii),
the claim for $\jds$ follows because the functor $\SmBcZ\to \EssSm_B$; $X^h_Z\mapsto X^h_Z\times_{B}(B-Z)$,
preserves $\tf$-points when $B$ is one-dimensional.
We note that $X^h_Z$ is a $\tf$-point.
The claim for $\tif$ follows since $\SmBcZ\to \EssSm_B$; $X^h_Z\mapsto X^h_Z$, 
preserves $\tf$-points for any $B$.
Furthermore, 
$X^h_Z\times_{B}(B-Z)$ is a $\tf$-point since 
\[
X^h_Z\times_{B}(B-Z)
\cong
X^h_Z\times_B \coprod_{\eta\in (B-Z)^{(1)}} \eta,
\] 
where $\eta$ runs over the set of generic points of $B-Z$. 
To conclude, 
we note that the $\tf$-topology on $\EssSm_\eta$ is trivial.
\end{proof}

\begin{corollary}
\label{cor:ijcommutLnis}
The functors $\tids$, $\jus$ commute with the localization endofunctors $\Lrep_{\nis}$ and $\Lrep_{\tf}$. 
Moreover,
if $\dim B=1$ with closed point 
$z$, 
then $\tif$ and $\jds$ commute with $\Lrep_\tf$.
\end{corollary}
\begin{proof}
\Cref{prop:tidsjusjdsLnisLtf} implies the claims for $\tids$, and $\jus$.
The isomorphism $\tif \Lrep_\tf\cong  \Lrep_\tf \tif$ follows from \Cref{prop:tifNislocobj,prop:tifjdstfloceq}, 
and $\jds \Lrep_\tf\cong  \Lrep_\tf \jds$ follows from \Cref{prop:tidsjusjdsLnisLtf,prop:tifjdstfloceq}.
\end{proof}

\begin{prop}
\label{prop:tifjdsLNis}
Suppose $B$ is a one-dimensional base scheme with closed point $z$.
If $\calF$ is an $\A^{1}$-local quasi-stable radditive framed presheaf of $S^1$-spectra on $Sm_{B}$, 
there are equivalences 
\[
\Lrep_\nis(\tif(\calF))\simeq \tif(\Lrep_\nis(\calF)),\quad
\Lrep_\nis(\jds(\calF))\simeq \jds(\Lrep_\nis(\calF)).
\]
\end{prop}
\begin{proof}
\Cref{prop:tifNislocobj,prop:tidsjusjdsLnisLtf} show $\jds(\Lrep_\nis(\calF))$ and $\tif(\Lrep_\nis(\calF))$ are 
Nisnevich local, 
and the same hold for $\Lrep_\nis(\tif(\calF))(X^h_x)$ and $\Lrep_\nis(\jds(\calF))$ by construction.
If $X^h_x$ is an essentially smooth local scheme in $\EssSm_{B,z}$, 
then there are equivalences 
\begin{align*}
\Lrep_\nis(\tif(\calF))(X^h_x)\simeq\tif(\calF)(X^h_x)
&\simeq\hofib( \calF(X^h_x) \to \calF(X^h_x-X_{(x)}) )\\
&\simeq\hofib( \Lrep_\nis(\calF)(X^h_x) \to \Lrep_\nis(\calF)(X^h_x-X_{(x)}) )\\
&\simeq\tif(\Lrep_\nis(\calF))(X^h_x).
\end{align*}
For the middle equivalence we appeal to \Cref{th:Hnis(Xhxeta)dimB1}.
It follows that $\Lrep_\nis(\tif(\calF))\simeq \tif(\Lrep_\nis(\calF))$, 
since both objects are Nisnevich local.

For every $X\in \Sm_B$, 
$x\in X\times_B z$, 
we can similarly use \Cref{th:Hnis(Xhxeta)dimB1} to show
\[
\ids(\Lrep_\nis(\calF))(X^h_x)\simeq \Lrep_\nis(\jds(\calF))(X^h_x).
\]
Moreover, 
if $x\in X-X\times_B z$,
there are equivalences
\[
\jds(\Lrep_\nis(\calF))(X^h_x)
\simeq\Lrep_\nis(\calF)(X^h_x)
\simeq\calF(X^h_x)
\simeq\jds(\calF)(X^h_x)
\simeq\Lrep_\nis(\jds(\calF))(X^h_x).
\]
This concludes the proof because $\jds(\Lrep_\nis(\calF))$ is Nisnevich local.
\end{proof}

\subsection{Localization and \texorpdfstring{$\tf$}{tf}-Nisnevich strict \texorpdfstring{$\A^{1}$}{A1}-invariance}

\begin{theorem}
\label{th:LocNisZSHI}
Suppose $B$ is a one-dimensional base scheme.
Let $Z$ be a closed subscheme of dimension zero with open complement $B-Z$.
Then $\tf$-Nisnevich strict ${\mathbb A}^{1}$-invariance on $\SmBcZ$ and $\SmBmZ$ 
implies $\tf$-Nisnevich strict ${\mathbb A}^{1}$-invariance on $\SmB$. 
\end{theorem}
\begin{proof}
Let $\calF\in \SH_{s}(\Fr_+({B}))$ be a quasi-stable radditive framed presheaf such that $\Lrep_\tf(\calF)$ is $\A^{1}$-local.
Owing to \Cref{cor:filtr-iFj:stable} there is a distinguished triangle
\begin{equation}
\label{equation:triangleforLnisF}
\tids\tif(\Lrep_{\nis}\calF)
\to 
\Lrep_{\nis}(\calF)
\to 
\jds\jus(\Lrep_{\nis}\calF).   
\end{equation}
\Cref{prop:StFrRad} and \Cref{cor:ijcommutLnis} show that $\tif(\calF)$ and $\jus(\calF)$ 
are quasi-stable radditive framed presheaves. 
Since $\Lrep_\tf(\calF)$ is $\A^{1}$-local, 
\Cref{prop:A1local} and \Cref{cor:ijcommutLnis} imply that $\Lrep_\tf(\tif(\calF))$ and 
$\Lrep_\tf(\jus(\calF))$ are $\A^{1}$-local.
By strict $\A^{1}$-invariance, 
both $\Lrep_\nis(\tif(\calF))$ and $\Lrep_\nis(\jus(\calF))$ are $\A^{1}$-local quasi-stable radditive framed presheaves.
Thus, 
by \Cref{prop:tifjdsLNis,prop:A1local},
the outer terms in \eqref{equation:triangleforLnisF} are $\A^{1}$-local.
Moreover, 
since all the terms in \eqref{equation:triangleforLnisF} are $\tf$-local, 
it follows that $\Lrep_{\nis}(\calF)$ is $\A^{1}$-local by viewing \eqref{equation:triangleforLnisF} 
as a distinguished triangle in $\SH_s(\Sm_{B})$.
\end{proof}

\begin{remark}\label{rem:generalitySection:sect:tfLoc}
All the results in this section hold verbatim for the categories $\Sm_B$, $\SmAff_B$, and $\Smat_B$,
with the exception of \Cref{th:HA1tflocalizationrecollement} for $\tf$-motivic categories. 
The said result holds for $\SmAff_B$ and $\Smat_B$.
\end{remark}

\section{Reduction from \texorpdfstring{$\SmBcZ$}{} to \texorpdfstring{$\SmZ$}{}}
\label{section:defsmZsmB}

Next we reduce the problem of $\tf$-Nisnevich strict $\A^{1}$-invariance over $\Sm_{B,Z}$ to the same problem over $\Sm_Z$.
The critical geometric input we use is that smooth morphisms admit liftings along henselian pairs 
$X_Z\not\hookrightarrow X^h_Z$ in the category of schemes. 
In \Cref{lm:liftFr} we show a similar property for framed correspondences. 

Recall from \Cref{subsection:candn} the categories $\SmatBlZ$, $\SmatBcZ$, 
and from \eqref{eq:rdsidsrigidHHtriv(Z)PreRig(BlZ)composite1}, \eqref{eq:Htriviusidsuisuds} the functors
\begin{equation}
\label{equation:irigdiagram}
\begin{tikzcd}[row sep=-0.25em]
\HHtriv(\SmatZ) 
& \HH_{\triv,\rig}(\SmatBlZ)\lar[swap,"i_\rigid^*"] \ar{r}{\Irig} 
& \HHtriv(\SmatBlZ)\ar{r}{\uds} & \HHtriv(\SmatBcZ) \\ 
Y\times Z & Y \ar[l,mapsto] &  
\XhZ & \XhZ. \lar[mapsto]
\end{tikzcd}
\end{equation}
Here $\Irig$ is the canonical embedding so that $i^*_\rigid\simeq (i_{B*Z})^* \circ \Irig$, 
where $(i_{B*Z})^*\colon \HHtriv(\SmatBlZ)\to \HHtriv(\SmatZ)$ is equivalent to the restriction 
along the canonical functor $r\colon \SmatZ\to \SmatBlZ$ by \Cref{lm:iusBlZsimeqrds}.
We will use similar notation in the setting of framed correspondences.

\begin{prop}
\label{prop:DiffSmatSlZ}
Suppose for each $\A^{1}$-local 
quasi-stable radditive framed presheaf of $S^{1}$-spectra $\calF$ on $\SmatZ$
the Nisnevich localization endofunctor $\Lrep_{\nis}(\calF)$ is $\A^{1}$-local. 
Then the same holds for $\A^{1}$-local quasi-stable radditive framed presheaves of $S^{1}$-spectra on $\SmatBcZ$.
\end{prop}

\begin{proof}
Let $\calF$ be an $\A^{1}$-local quasi-stable radditive framed presheaf of $S^{1}$-spectra on $\SmBcZ$. 
We need to show that $\Lrep_{\nis}(\calF)$ is an $\A^{1}$-local.
To that end we employ \eqref{equation:irigdiagram} and break the proof into the following steps.
\begin{itemize}
\item[(i)]
By \Cref{lm:deform:uus-circ-Irig}(1) there is an $\A^{1}$-local rigid framed presheaf of $S^{1}$-spectra $\calF^\rig$ on 
$\SmatBlZ$ such that $\calF\simeq \uds \Irig(\calF^\rig)$.
\item[(ii)]
By \Cref{lm:deform:uus-circ-Irig}(2), 
$\calF^\rig$ is a quasi-stable radditive framed presheaf of $S^{1}$-spectra.
\item[(iii)]
Consider the object $\calF_Z\simeq i^*_\rigid(\calF^\rig)\in \HHtriv(\SmatZ)$.
By \Cref{lm:deform:idsrig}(2), 
it is $\A^{1}$-local, 
and by \Cref{lm:deform:idsrig}(1), 
it is a quasi-stable radditive framed presheaf of $S^{1}$-spectra. 
\item[(iv)]
If the endofunctor $\Lrep_\nis$ on $\HHtriv(\Fr_+(\SmatZ))$ takes 
$\A^{1}$-local quasi-stable radditive objects 
to $\A^{1}$-local ones,
then $\Lrep_{\nis}(\calF_Z)$ is $\A^{1}$-local.
According to \Cref{lm:deform:idsrig}(2) we have $i^*_\rigid(\Lrep_{\nis}(\calF^\rig))\cong  \Lrep_{\nis}(\calF_Z)$,  
and hence it is an $\A^{1}$-local by \Cref{lm:deform:idsrig}(1).  
\item[(v)]
\Cref{lm:deform:uus}(2) and \Cref{lm:deform:Irig}(2) imply the equivalence
$\Lrep_{\nis}(\calF)\cong  \uds \Irig(\Lrep_{\nis}(\calF^\rig))$.
\Cref{lm:deform:uus}(1) and \Cref{lm:deform:Irig}(1) show it is an $\A^{1}$-local. 
\end{itemize}
\end{proof}

We will use the following reformulation of \Cref{lm:liftFr} and \Cref{lm:AffSmHenselianLift}.

\begin{lemma}
\label{lm:diffFrCor}
For $X_i\in \SmatBcZ$ we set $Y_i=X_i\times_B Z$ for $i=0,1$.
Then for the canonical closed immersions $Y_i\not\hookrightarrow X_i$ and every morphism $Y_0\to Y_1$ in $\Sch_B$
there exists a commutative diagram 
\[
\xymatrix{
Y_0\ar[d] \ar[r] & X_0\ar@{.>}[d] \\
Y_1\ar[r] & X_1.
}
\]
The same result holds for morphisms in $\Fr_+(\Sch_B)$.
\end{lemma}

\begin{lemma}
\label{lm:deform:uus-circ-Irig}
The following hold for the composite functor $\uds \circ \Irig$ in \eqref{equation:irigdiagram}.
\begin{enumerate}
\item
For any $\A^{1}$-local framed presheaf of $S^1$-spectra $\calF$ on $\SmatBcZ$ there is an $\A^{1}$-local 
rigid framed presheaf $\calF^\rig$ on $\SmatBlZ$ such that $\uds(\Irig (\calF^\rig))\simeq \calF$.
\item
It detects quasi-stable radditive framed presheaves of $S^{1}$-spectra on the subcategory of framed presheaves 
of $S^1$-spectra.
\item
It is conservative.
\end{enumerate}
\end{lemma}

\begin{proof}
Part (1) follows by taking $\calF^\rig=\calF_{\A^{1}}^\fr$ in view of 
\Cref{prop:SmoothLift->A1lrig:origin}.
Suppose $\uds\Irig(\calF)$ is quasi-stable.
Then $\sigma_{X^h_Z}$ induces an auto-equivalence on $\calF(X^h_Z)$ for all $X\in \Smat_B$, 
hence $\sigma_{X_Z}$ induces an auto-equivalence on $\calF(X_Z)$,
since $\calF(X^h_Z)\cong \calF(X_Z)$ by rigidity.
Moreover, 
if $\uds\Irig(\calF)$ is radditive, 
then $\calF$ is radditive by \Cref{cor:henselianpairsplitting} and rigidity as above.
This proves part (2).
To prove (3), 
suppose $f\in\calF\to \mathcal G$ is a morphism in $\HH_\rig(\SmatBlZ)$ such that 
$u_*R_\rig(\calF)\simeq u_*R_\rig(\mathcal G)$.
Any $Y$ in $\SmatBlZ$ is either of the form $X_Z$ or $X^h_Z$
for some
$X^h_Z\in \SmatBlZ$.
For any 
$X^h_Z\in \SmatBlZ$,
we have the equivalences 
\[
\calF(X_Z)\simeq 
\calF(X^h_Z)\simeq 
\uds R_\rig\calF(X^h_Z)\simeq 
\uds R_\rig\mathcal G(X^h_Z)
\mathcal G(X^h_Z).
\]
\end{proof}

\begin{lemma}
\label{lm:deform:uus}
The following hold for the functor $\uds$ in \eqref{equation:irigdiagram}.
\begin{enumerate}
\item
It preserves $\A^{1}$-local and quasi-stable radditive framed presheaves of $S^{1}$-spectra. 
\item 
It commutes with $\Lrep_{\nis}$.
\end{enumerate}
\end{lemma}
\begin{proof}
To conclude (1) we use that $\SmatBcZ\to \SmatBlZ; \XhZ\mapsto \XhZ$, preserves 
morphisms of the form $(\A^{1}\times X)^h_Z\to \XhZ$.
To conclude (2), 
we use that the same functor preserves Nisnevich coverings, fiber products, and Nisnevich points owing to 
\Cref{lm:LocEqCovPoints,lm:schemespointssites}.
\end{proof}

\begin{lemma}
\label{lm:Lnis(rigid)}
The endofunctor $\Lrep_\nis$ on $\HHtriv(\SmatBcZ)$ preserves rigid presheaves of $S^{1}$-spectra. 
\end{lemma}
\begin{proof}
This follows since every Nisnevich covering in the category $\SmatBlZ$ of a henselian pair $X_Z\to X^h_Z$ is a 
henselian pair of the form $\widetilde X_Z\to \widetilde X^h_Z$ for some Nisnevich covering $\widetilde X\to X$.
\end{proof}

\begin{corollary}
\label{cor:Lnis(rigid)}
The functor $\Lrep_\nis$ restricts to an endofunctor on $\HHtrivrigid(\SmatBcZ)$.
\end{corollary}

\begin{lemma}
\label{lm:deform:Irig}
The following hold for $\Irig\colon \HH_{\triv,\rig}(\SmatBlZ)\to\HHtriv(\SmatBlZ)$ in \eqref{equation:irigdiagram}.
\begin{enumerate}
\item
It preserves and detects $\A^{1}$-local quasi-stable radditive framed presheaves of $S^{1}$-spectra.
\item 
It commutes with $\Lrep_{\nis}$.
\end{enumerate}
\end{lemma}
\begin{proof}
Part (1) follows because $\Irig$ is the embedding of the subcategory spanned by rigid presheaves 
of $S^{1}$-spectra in $\HHtriv(\SmatBlZ)$.
For part (2), 
\Cref{lm:Lnis(rigid)} implies that the endofunctor $\Lrep_\nis$ on $\HHtrivrigid(\SmatBcZ)$ coincides with 
the restriction of the endofunctor $\Lrep_\nis$ on $\HHtriv(\SmatBcZ)$.
\end{proof}

\begin{lemma}
\label{lm:deform:ids}
The following hold for the functor $i_{B*Z}^*$ in \Cref{lm:iusBlZsimeqrds}. 
\begin{enumerate}
\item 
It preserves $\A^{1}$-local quasi-stable radditive framed presheaves of $S^{1}$-spectra. 
\item 
It commutes with $\Lrep_{\nis}$.
\end{enumerate}
\end{lemma}
\begin{proof}
By \Cref{lm:iusBlZsimeqrds} the functor $i_{B*Z}^*$ is induced by restriction along the canonical embedding  
\begin{equation}
\label{eq:iSmatZSmatBlZ:proofids}
\SmatZ\to \SmatBlZ;
Y\mapsto Y.
\end{equation}
Part (1) follows since \eqref{eq:iSmatZSmatBlZ:proofids} gives rise to the embedding 
$\Fr_+(\SmatZ)\to \Fr_+(\SmatBlZ)$ given by $Y\mapsto Y$.
This functor preserves coproducts and morphisms of the form $\A^{1}\times Y\to Y$ and $\sigma_{Y}$.

Part (2) is equivalent to the statement that $i_{B*Z}^*$ preserves Nisnevich local objects and Nisnevich local equivalences.
This follows by \Cref{lm:LocEqCovPoints,lm:schemespointssites} because the embedding in \eqref{eq:iSmatZSmatBlZ:proofids} 
preserves coverings and points in the Nisnevich topology and preserves fiber products.
\end{proof}

\begin{lemma}
\label{lm:deform:idsrig}
The following hold for the functor $i^*_\rigid$ in \eqref{equation:irigdiagram}.
\begin{enumerate}
\item 
It preserves and detects $\A^{1}$-local and quasi-stable radditive framed presheaves of $S^{1}$-spectra. 
\item 
It commutes with $\Lrep_{\nis}$.
\end{enumerate}
\end{lemma}

\begin{proof}
The first claim in (1) follows by \Cref{lm:deform:Irig}(1) and \Cref{lm:deform:ids}(1).
For the second claim we consider $i_{B,Z}\colon \SmatBlZ\to \SmatZ; X^h_Z\mapsto X_Z$.
The restriction $(i_{B,Z})_*\colon \HHtriv(\SmatZ)\to\HHtriv(\SmatBlZ)$ takes values in rigid objects.
Hence we obtain $i^\rigid_*\colon \HHtriv(\SmatZ)\to\HHtrivrigid(\SmatBlZ)$.
\Cref{lm:PresPreSsEquivalence} shows $i^*_\rigid$ and $i^\rigid_*$ are inverses,
and \Cref{lm:iusBlZsimeqrds} shows $i^*_\rigid$ is isomorphic to the functor $r_*$ obtained from restriction 
along $r\colon \SmatZ\to \SmatBlZ$.
Moreover, 
$i_{B,Z}$ and $r$ lift to functors between categories of framed correspondences
\[
\begin{array}{rl}
(i_{B,Z})^\fr\colon \Fr_+(\SmatBlZ)\to \Fr_+(\SmatZ);& X^h_Z\mapsto X_Z, \\
r^\fr\colon \Fr_+(\SmatZ)\to \Fr_+(\SmatBlZ);& X_Z\mapsto X^h_Z.
\end{array}
\]
To conclude the proof for (1) we use that $(i_{B,Z})^\fr$ and $r^\fr$ preserve coproducts and morphisms 
of the form $\A^{1}\times V\to V$ and $\sigma_V$ for any scheme $V\in\SmatBlZ$.
Part (2) follows from \Cref{lm:PresPreSsEquivalence}, \Cref{lm:deform:Irig}(2), \Cref{lm:deform:ids}(2).
\end{proof}

\begin{theorem}
\label{th:DiffSmScZ}
Let $Z$ be a closed subscheme of $B$.
If Nisnevich strict ${\mathbb A}^{1}$-invariance holds on $\Sm_Z$ then it holds on $\SmBcZ$.
\end{theorem}
\begin{proof}
Owing to \Cref{lm:baselocality} we may assume $B$ is affine, 
and owing to \Cref{cor:SmScZSmatA1locpreserve} we may replace $\SmZ$ and $\SmBcZ$ with $\SmatZ$ and $\SmatBcZ$, 
respectively.
We conclude by appealing to \Cref{prop:DiffSmatSlZ}.
\end{proof}

\section{\texorpdfstring{$\tf$}{tf}-Nisnevich strict homotopy invariance}
\label{section:proofmaintheorem}

In this section, 
we finish the proof of our main result concerning $\tf$-Nisnevich strict $\A^{1}$-invariance,
see \Cref{def:SHIthnistf}.

\begin{prop}
\label{prop:SHIT_BmS-<SmZSmU}
Suppose $Z\not\hookrightarrow B$ is a closed subscheme of dimension zero and the base scheme $B$ is one-dimensional. 
If $\tf$-Nisnevich \SHI holds on $\Sm_{Z}$ and $\Sm_{B-Z}$, then $\tf$-Nisnevich \SHI holds on $\Sm_{B}$.
\end{prop}
\begin{proof}
In view of \Cref{rm:SHIdim0NistfNis}, 
since $\tf$-Nisnevich strict ${\mathbb A}^{1}$-invariance holds on $\Sm_Z$ by assumption, 
\Cref{th:DiffSmScZ} implies $\tf$-Nisnevich strict $\A^{1}$-invariance on $\SmBcZ$.
Using that $\tf$-Nisnevich strict ${\mathbb A}^{1}$-invariance holds on $\Sm_{B-Z}$ by assumption, 
we conclude $\tf$-Nisnevich strict ${\mathbb A}^{1}$-invariance on $\SmB$ by appealing to \Cref{th:LocNisZSHI}.
\end{proof}

We are ready to prove our final permanence result about strict $\A^{1}$-invariance.
In the proof we use that Nisnevich \SHI holds over every residue field of the base scheme.

\begin{theorem}
\label{thm:SHIrc}
Let $B$ be a base scheme of dimension one.
Then $\tf$-Nisnevich \SHI holds on $\Sm_{B}$.
\end{theorem}
\begin{proof}
Owing to \Cref{lm:baselocality} we may assume $B$ is a local scheme.
If $z\in B$ is the unique closed point, 
our assumption shows that Nisnevich \SHI holds on $\Sm_{z}$ and $\Sm_{B-z}$.
Since $z$ and $B-z$ are zero-dimensional the $\tf$-topologies on $\Sm_z$, $\Sm_{B,z}$, 
$\Sm_{B-z}$ are trivial by \Cref{prop:Proptftop}(iii).
Thus $\tf$-Nisnevich strict ${\mathbb A}^{1}$-invariance is equivalent to Nisnevich \SHI for these categories.
To conclude the proof, 
we use \Cref{prop:SHIT_BmS-<SmZSmU}.
\end{proof}

\begin{corollary}
\label{cor:SHIrcAb}
Suppose $B$ is a one-dimensional scheme.  
Then $B$ satisfies $\tf$-Nisnevich \SHI for framed presheaves of abelian groups in the sense of \Cref{def:str_xibase}.
\end{corollary}
\begin{proof}
Follows from \Cref{thm:SHIrc,th:SHIeqS1Ab1,th:SHIeqS1Ab}
(\SHI holds for quasi-stable radditive framed presheaves of abelian groups over every residue field of $B$).
\end{proof}

\begin{corollary}
\label{cor:ptobasedefinition:SHI}
Suppose $B$ is a one-dimensional scheme.
Then $B$ satisfies \SHI in the sense of \Cref{definition:SHI}, 
i.e., 
the endofunctor $\Lrep^\tf_\nis$ on $\SHstf(\Sm_{B})$ preserves $\A^{1}$-local quasi-stable framed 
presheaves of $S^1$-spectra.
\end{corollary}
\begin{proof}
The proof makes use of the fact that Nisnevich strict $\A^{1}$-invariance holds over all the residue fields of $B$.
\Cref{thm:SHIrc} shows that $\tf$-Nisnevich \SHI holds on $\Sm_{B}$.
An $\A^{1}$-local quasi-stable framed object of $\SHstf(\Sm_{B_\sigma})$ is 
synonymous with an $\A^{1}$-local $\tf$-local quasi-stable object $\calF\in \Spt_s(\Fr_+(B))$.
Then $\Lrep_\tf(\calF)\cong \calF$ is an $\A^{1}$-local quasi-stable framed presheaf of $S^1$-spectra,
and $\Lrep^\tf_\nis(\calF)\cong \Lrep_\nis(\calF)$. 
By $\tf$-Nisnevich \SHI over $B$, 
$\Lrep_\nis(\calF)$ is an $\A^{1}$-local quasi-stable framed presheaf of $S^1$-spectra. 
Thus $\Lrep^\tf_\nis$ preserves $\A^{1}$-local quasi-stable framed presheaves of $S^1$-spectra.
\end{proof}

\begin{remark}
\label{rm:subsection:squareAtimesV}
Let us elaborate on the discussion in \Cref{subsection:squareAtimesV} by outlining an alternate approach 
to \Cref{thm:SHIrc}.
If $V$ is an essentially smooth local henselian scheme over a scheme $B$, 
it suffices to show every Nisnevich covering $\widetilde U\to \A^{1}\times V$ admits a refinement 
$\widetilde U^\prime\to \A^{1}\times V$, 
in the sense of \cite[Tag 00VT]{StacksProject}, 
obtained from $\tf$-squares and $(\A^{1},\ZF)$-contractible Nisnevich squares, 
see \eqref{eq:NissqA1FrContr}.

Over fields, 
an analysis of the proof for Nisnevich strict $\A^{1}$-invariance in \cite[\S 17]{hty-inv} shows every Nisnevich 
covering of $\A^{1}\times V$ admits a refinement obtained from $(\A^{1},\ZF)$-contractible Nisnevich squares. 
Owing to \Cref{lm:A1loc_id=iusids(SmAff)} and \Cref{lm:PresPreSsEquivalenceSH} the same holds in $\Sm_{B,z}$.
The proof of \Cref{th:Hnis(Xhxeta)dimB1} shows every Nisnevich covering of $(X^h_z)\times_B \eta$,
where $B$ is one-dimensional and $z\in B^{(1)}$, $\eta\in B^{(0)}$, 
admits a refinement obtained from $\tf$-squares and $(\A^{1},\ZF)$-contractible Nisnevich squares.
In the general case, we can use $\tf$-squares to reduce to the case of $\Sm_{B,z}$, 
for various $z\in B$,
and proceed as above to find $(\A^{1},\ZF)$-contractible squares for each $z$.
However, 
controlling the steps in this process is technically more demanding than simply incorporating the localization 
theorem in \Cref{sect:tfLoc}.
\end{remark}

\section{\texorpdfstring{$\tf$}{tf}-Nisnevich strict \texorpdfstring{$\A^{1}$}{A1}-invariance for abelian groups} 
\label{section:stifabeliangroups}

We deduce \Cref{thm:main} on $\tf$-Nisnevich strict $\A^{1}$-invariance for abelian groups from the following 
comparison result.

\begin{theorem}
\label{th:SHIeqS1Ab1}
If $\tf$-Nisnevich strict $\A^{1}$-invariance for presheaves of $S^{1}$-spectra holds on $\Sm_{B}$ 
in the sense of \Cref{def:SHIthnistf}, 
then $B$ satisfies $\tf$-Nisnevich strict $\A^{1}$-invariance for presheaves of abelian groups 
in the sense of \Cref{def:str_xibase}.
\end{theorem}

\begin{proof}
We write $\EM(\calF)\in\SH_{s}(B)$ for the Eilenberg-MacLane object associated to a presheaf of abelian groups 
$\calF$ on $\Sm_{B}$.
By \cite[Theorem 8.26]{Jardine-local} there are canonical isomorphisms
\begin{equation}
\label{eq:piEM}
\pi_n\Lrep_\nis\EM(\calF)\cong H_\nis^n(-,\calF_\nis),\quad
\pi_n\Lrep_\tf\EM(\calF)\cong H_\tf^n(-,\calF_\tf).
\end{equation}
Assume $\tf$-Nisnevich strict $\A^{1}$-invariance holds for presheaves of $S^{1}$-spectra.
Let $\calF\colon \Fr_+(B)\to \Ab$ be a $\tf$-strictly $\A^{1}$-invariant quasi-stable radditive framed presheaf. 
The equivalences $\sigma^*\colon \calF(X)\simeq \calF(X)$ and $\calF(X_1\amalg X_2)\simeq\calF(X_1)\oplus \calF(X_2)$ 
imply similar ones for $\EM(\calF)$, i.e., 
it is a quasi-stable radditive framed presheaf of $S^{1}$-spectra. 
By \eqref{eq:piEM} it follows that $\Lrep_\tf\EM(\calF)$ is $\A^{1}$-local.
By assumption $L_{\nis}\EM(\calF)$ is $\A^{1}$-local.
Thus $H_{\nis}^n(-,\calF)\cong \pi_{-n}\Lrep_{\nis}(\calF)$ is $\A^{1}$-invariant for all $n\geq 0$ by \eqref{eq:piEM}.
\end{proof}

\begin{theorem}
\label{th:SHIeqS1Ab}
If the field $k$ satisfies $\tf$-Nisnevich strict $\A^{1}$-invariance for presheaves of abelian groups 
in the sense of \Cref{def:str_xibase}, 
then $\tf$-Nisnevich strict $\A^{1}$-invariance for presheaves of $S^{1}$-spectra holds on $\Sm_{k}$ 
in the sense of \Cref{def:SHIthnistf}.
\end{theorem}

\begin{proof}
Let $\calF$ be an $\A^{1}$-local quasi-stable radditive framed presheaf of $S^{1}$-spectra.
Note that $\pi_l(\calF)$, 
$l\in \mathbb Z$, 
is an $\A^{1}$-invariant quasi-stable radditive framed presheaf of abelian groups.
Thus $H^n_\nis(-,\calF)$ is $\A^{1}$-local by the assumption on $k$.
By \eqref{eq:piEM} also $\Lrep_{\nis}\EM(\calF)$ is $\A^{1}$-local, 
i.e., 
there is an equivalence
\begin{equation}
\label{eq:LA1LniscalEMpiF=LniscalEMpiF}
\Lrep_{\A^{1}}\Lrep_{\nis}\EM(\pi_l\calF)
\simeq 
\Lrep_{\nis}\EM(\pi_l\calF).
\end{equation}

We refer to \cite[Section 10.6]{Jardine-local} for the yoga of Postnikov towers with respect to $t$-structures. 
Let $\SH_{s,\geq l}(\Fr_+(B))$ be the subcategory of $\SH_{s}(\Fr_+(B))$ spanned by objects $\calF$ with trivial 
homotopy presheaf $\pi_n(\calF)$ is the range $n<l$.
Likewise, 
we write $\SH_{s,\leq l}(\Fr_+(B))$ for the subcategory spanned by objects $\calF$ for which $\pi_n(\calF)=0$ when $n>l$.
Let $\calF_{\geq l}$ and $\calF_{\leq l}$ be the truncations of $\calF$ in $\SH_{s,\geq l}(\Fr_+(B))$ and 
$\SH_{s,\leq l}(\Fr_+(B))$, respectively.
With these definitions there are 
isomorphisms
\begin{equation}
\label{eq:colomFgeq0limFleq0}
\hocolim_l \calF_{\geq l} 
\cong 
\calF,
\holim_l \calF_{\leq l} 
\cong 
\calF.
\end{equation}

For $l,b\in \mathbb Z$ we can form the truncation
\[
\calF_{\geq l,\leq b}
\in 
\SH_{s,\geq l}(\Fr_+(B))\cap \SH_{s,\leq b}(\Fr_+(B)),
\] 
and for $b\geq l$ there is an
isomorphism
\begin{equation}
\label{eq:ConeFlFl-1EM(pil)}
\mathrm{hofib}(\calF_{\geq l,\leq b}
\to 
\calF_{\geq l,\leq b-1})
\cong
\EM(\pi_b(\calF)).
\end{equation}
By \eqref{eq:colomFgeq0limFleq0} and induction we conclude that for all $b\geq l\in \mathbb Z$ there is an
isomorphism
\begin{equation}
\label{eq:LA1LniscalFlb=LniscalFlb}
\Lrep_{\A^{1}}L_{\nis}(\calF_{\geq l,\leq b})
\cong 
\Lrep_{\nis}(\calF_{\geq l,\leq b}).
\end{equation}
This shows $\pi_n(\Lrep_\nis\calF_{\geq l,\leq b})(X)$ is $\A^{1}$-invariant for every $n\in \mathbb Z$.

If $X\in\Sm_B$ is $d$-dimension and $b>n+d$, 
\Cref{lm:dimXvanishingpiLnis} implies $\pi_n(\Lrep_\nis\calF_{\geq b})(X)=0$.
Hence, 
for all $l\in \mathbb Z$, we have 
\[
\pi_n(\Lrep_\nis\calF_{\geq l})(X)
\cong 
\pi_n(\Lrep_\nis\calF_{\geq l,\leq n+d})(X),
\]
and 
 \begin{equation}
\label{eq:piFl(X)=(XtimesA1)}
\pi_n(\Lrep_\nis\calF_{\geq l})(X\times\A^{1})
\cong 
\pi_n(\Lrep_\nis\calF_{\geq l})(X).
\end{equation}
Since \eqref{eq:piFl(X)=(XtimesA1)} holds all $X\in\Sm_B$ and $n\in \mathbb{Z}$ there is an isomorphism 
\begin{equation}
\label{eq:LA1LniscalFl=LniscalFl}
\Lrep_{\A^{1}}L_{\nis}(\calF_l)
\cong 
\Lrep_{\nis}(\calF_l)
\end{equation}
for every $l\in \mathbb Z$.
Owing to \Cref{lm:homotopylimhomotopycolomLnisLA1} we deduce the isomorphism
\begin{equation}
\label{eq:LA1LniscalF=LniscalF}
\Lrep_{\A^{1}}\Lrep_{\nis}(\calF)
\cong 
\Lrep_{\nis}(\calF),
\end{equation}
and thus $\Lrep_{\nis}\Lrep_\nis(\calF)$ is $\A^{1}$-local.
\end{proof}

\begin{lemma}
\label{lm:homotopylimhomotopycolomLnisLA1}
If $\calF\simeq \hocolim_{l}\calF_l$ in $\SH_s(\Fr_+(B))$ then there are isomorphisms
\[
\Lrep_\nis(\calF)\simeq \hocolim_{l} \Lrep_\nis(\calF_l),
\Lrep_{\A^{1}}(\calF)\simeq \hocolim_{l} \Lrep_{\A^{1}}(\calF_l).
\]
\end{lemma}
\begin{proof}
We may assume $\calF\in \Sm_{B}(\Sm_B)$ since the forgetful functor $\SH_s(\Fr_+(B))\to \SH_s(\Sm_B)$ 
preserves homotopy colimits.
To prove the first claim, 
we show that the sequential homotopy colimit in question preserves
Nisnevich equivalences and local objects in $\SH_s(\Sm_B)$.

Suppose $\calF_\bullet\to \mathcal G_\bullet$ is a morphism of sequential diagrams such that 
$\calF_l\to \mathcal G_l$ is an isomorphism in $\SH_s(\Sm_B)$ for all $l\geq 0$.
Then, 
for all $X\in \Sm_B$ and $x\in X$, 
there are isomorphisms
\[
(\hocolim_{l}\calF_l)(X^h_x)
\cong 
\hocolim_{l}\calF_l(X^h_x)
\cong 
\hocolim_{l}\calG_l(X^h_x) 
\cong 
(\hocolim_{l}\calG_l)(X^h_x).
\]
Moreover, 
if $\calF_\bullet$ is comprised of Nisnevich local objects, 
then for any Nisnevich square given by an open immersion $U\to X$ and \'etale morphism $X^\prime\to X$ we have isomorphisms
\begin{equation*}
\calF(X^\prime\times_X U)
\cong 
\hocolim_{l}\hocoeq ( \calF_l(X)\rightrightarrows \calF_l(U\amalg X^\prime) )\cong  \\
\hocoeq (\calF(X)\rightrightarrows \calF(U\amalg X^\prime)).
\end{equation*}
Here we use the assumption $\calF\cong \hocolim_{l}\calF_l$.
To finish the proof we use the isomorphisms 
\[
\Lrep_{\A^{1}}^{[1]}(\calF)
\cong  
\calF(\Delta^\bullet\times -)
\cong  
\hocolim_{l}\calF_l(\Delta^\bullet\times -)
\cong  
\hocolim_{l}\Lrep_{\A^{1}}^{[1]}\calF_l.
\]
The claim follows because $\Lrep_{\A^{1}}\cong \hocolim_r \Lrep_{\A^{1}}^{[r]}$.
\end{proof}

\begin{lemma}
\label{lm:dimXvanishingpiLnis}
Let $\calF$ be a presheaf of $S^{1}$-spectra on $\Sm_B$ and suppose $X\in \Sm_B$ is $d$-dimensional.
Then for integers $l, n\in \mathbb Z$ such that $n<l-d$, 
we have $\pi_n(\Lrep_\nis\calF_{\geq l})(X)=0$.
\end{lemma}
\begin{proof}
Note that $\calF$ extends continuously to a presheaf of $S^{1}$-spectra on $\EssSm_B$.
By forming $(\calF_{\geq l})[-l]$ we may assume that $l=0$ and $\calF \in \SH_{s,\geq 0}(\Sm_B)$.
Let $U$ be an open subscheme of $X\in\EssSm_B$. 
In the range $n<-d$ we will show the vanishing
\begin{equation}
\label{equation:vanishingUX}
[X/U\wedge S^n,\Lrep_\nis\calF]_{\SH_s(B)}
=
0.
\end{equation}
We proceed by induction on $d=\dim X$ and $c=\dim (X-U)$. 
The claim for $d=0$ holds since there are no nontrivial Nisnevich coverings of a zero-dimensional scheme.

Assume $c=0$, $d\geq 0$, and set $Y=X- U$. 
Then, 
since $\Lrep_\nis\calF$ is Nisnevich local, 
we have
\[
[X/U\wedge S^n,\Lrep_\nis\calF]_{\SH_s(B)}
\cong
[X^h_Y/(X^h_Y\times_X U)\wedge S^n,\Lrep_\nis\calF]_{\SH_s(B)}.
\] 
Here, 
$X^h_Y$ is a disjoint union of local henselian schemes because $c=0$.
Moreover, 
since $\calF$ is radditive, 
we have $\calF\in \SH_{s,\geq 0}(\Sm_B)$.
Recall that Nisnevich coverings on a local henselian scheme are trivial;
for $n<0$, it follows that 
\[
[X^h_Y\wedge S^n,\Lrep_\nis\calF]_{\SH_s(B)}
\cong
[X^h_Y\wedge S^n,\calF]_{\SH_s(B)}=0.
\]
Thus, 
in the range $n<-d$,
the inequalities $\dim (X^h_Y\times_X U)<\dim X^h_Y\leq\dim X$ imply
\[
[(X^h_Y\times_X U)\wedge S^n,\Lrep_\nis\calF]_{\SH_s(B)}=0.
\]
For $n<-d$, 
we deduce the vanishing
\[
[X^h_Y/(X^h_Y\times_X U)\wedge S^n,\Lrep_\nis\calF]_{\SH_s(B)}=0.
\]

Assume \eqref{equation:vanishingUX} holds for all pairs $X^\prime,U^\prime$ such that $\dim X^\prime<d$, 
or $\dim X^\prime=d$, $\dim (X^\prime- U^\prime)<c$.
Let $\nu$ be the union of the generic points of $Y=X- U$. 
For all $n<-d$ the inductive assumption reads 
\[
[X_\nu/(X_\nu\times_X U)\wedge S^n,\Lrep_\nis\calF]_{\SH_s(B)}=0.
\]  
Hence for every $e\in [X/U\wedge S^n,L_{\nis}\calF]_{\SH_s(B)}$ there exists an open subscheme $V\subset X$ such that 
the closure of $V$ contains $Y$ and the image of $e$ in $[V/(V\times_X U)\wedge S^n,L_{\nis}\calF]_{\SH_s(B)}=0$ is trivial.
Since $\dim (X- (U\cup V))<\dim Y$, it follows by the inductive assumption that 
\[
[X/(V\cup U)\wedge S^n,\Lrep_\nis\calF]_{\SH_s(B)}=0
\] for all $n<-d$.  
Using the naturally induced exact sequence
\[
[X/(V\cup U)\wedge S^n,\Lrep_\nis\calF]_{\SH_s(B)}
\to 
[X/U\wedge S^n,\Lrep_\nis\calF]_{\SH_s(B)}
\to 
[V/(V\cap U)\wedge S^n,\Lrep_\nis\calF]_{\SH_s(B)}
\]
we conclude $e=0$.
\end{proof}

\begin{lemma}
\label{lm:tfcohomologyonedimB}
Let $B$ be a one-dimensional local irreducible scheme with closed point $z\in B$ and generic point $\eta$.
If $\calF$ is a presheaf of abelian groups on $\Sm_B$, 
then for $X\in \Sm_B$ we have 
\[
H^*_\tf(X,\calF)
\cong
\begin{cases}
\ker(\calF(X^h_z)\oplus \calF(X_\eta)\to \calF((X^h_z)_\eta)) & i=0,\\
\coker(\calF(X^h_z)\oplus \calF(X_\eta)\to \calF((X^h_z)_\eta)) & i=1,\\
0 & i\geq 2.
\end{cases}\]
\end{lemma}
\begin{proof}
Consider the Eilenberg-MacLane object $\mathbf{H}\calF$ on $\Sm_B$.
\Cref{prop:B1dimlocLtfhocofib} implies
\[
\Lrep_{\tf}(\mathbf{H}\calF)(X)
\cong
\hofib(\mathbf{H}\calF(X^h_z)\oplus \mathbf{H}\calF(X_\eta)\to \mathbf{H}\calF((X^h_z)_\eta))
\cong
\mathbf{H}(\Lrep_\tf(\calF)),
\]
where $\Lrep_\tf(\calF)$ denotes the complex
\begin{equation}
\label{eq:LtfFadditivecomplex}
[
\calF (X^h_z)\oplus \calF(X_\eta)\to \calF((X^h_z)_\eta)].
\end{equation}
Thus the $\tf$-cohomology of $\calF$ agrees with the cohomology of the complex \eqref{eq:LtfFadditivecomplex}.
\end{proof}

\section{Base change and localization for the \texorpdfstring{$\tf$}{tf}-topology}
\label{section:SplittingdiagramLocalization}

In this section, we use $\infty$-categories to revise and 
strengthen the claims concerning the $\tf$-motivic localization or recollement results presented in \Cref{subsection:tfmlt}. Specifically, we establish the \emph{commutativity} of Nisnevich localization and $\A^{1}$-localization, together with the $\Gm$-loop functor, with the functors appearing in the diagram of $\infty$-categories~\eqref{eq:Smaf(Z)A1Smaf(SZ)nisSmaf(B)(U)}.

\subsection{}
If $Z$ is a closed subscheme of the base scheme $B$, 
we have the functors
\begin{equation}
\label{eq:SmZScZSmZ}
\begin{array}{lclclcl}
\Sm_Z &\xleftarrow{\overarrow{i}}& \Sm_{B,Z} &\xleftarrow{\tilde i}& \Sm_{B} &\xrightarrow{j}& \Sm_{B,Z}\\
X_Z &\mapsfrom& X^h_Z &\mapsfrom& X &\mapsto& X\times_{S}(B-Z)
\end{array}
\end{equation}
There are naturally induced functors between $\infty$-categories of framed presheaves:
\begin{equation}
\label{eq:Smaf(Z)A1Smaf(SZ)nisSmaf(B)(U)}
\xymatrix{
& & \cPre^\fr(\Sm_B)\ar[dl]_{\tilde i^!}\ar[dr]^{j^*} & \\
\cPre^\fr(\Sm_Z)\ar@<1ex>[r]^{\overarrow{i}_*} &\cPre^\fr(\Sm_{B,Z})
\ar@<1ex>[l]^{\overarrow{i}^*}\ar[rd]_{\tilde i_*}&  &\cPre^\fr(\Sm_{B-Z})\ar[ld]^{j_*} \\
& & \cPre^\fr(\Sm_B) &
}
\end{equation}
Explicitly, 
we have the formulas
\begin{gather}
\label{eq:PreScZiSjB-Z}
\begin{array}{ll}
\tilde i^!(F)(X^h_Z) = \fib( F(X^h_Z)\to F(X^h_Z-X_Z) ); &X^h_Z\in \Sm_{B,Z},\\
\tilde i_*F(X)=F(X^h_Z); &X\in\Sm_B,\\
j^*F(V)=F(V); &V\in \Sm_{B-Z},\\
j_*F(X)=F(X-X_Z); &X\in \Sm_B,\\
\overarrow{i}_*F(X)=F(X_Z); & X\in\Sm_{B,Z};\\
\overarrow{i}^*(h^\fr(Y^h_Z))=h^\fr(Y_Z), & Y\in \Sm_B.
\end{array}\end{gather}
Moreover, 
$\overarrow{i}^*(F)$ is the left Kan extension of $F\in \cPre^\fr(\Sm_{B,Z})$ along $\overarrow{i}$. 
\vspace{0.1in}

For $V\in \Sm_{B,Z}$ we consider the endofunctor $F\mapsto F^V$ on $\cPre^\fr(\Sm_{B,Z})$, 
where 
$$
F^V(-)=F(-\times_{B,Z} V).
$$

\begin{lemma}
\label{lm:(Vtimes-)taulocA1inv}
For $V\in \Sm_{B,Z}$ the endofunctor $(-)^V$ on $\cPre^\fr(\Sm_{B,Z})$ preserves both $\tau$-local and 
$\A^{1}$-invariant objects.
\end{lemma}
\begin{proof}
The claim follows since the endofunctor $\Sm_{B,Z}\to \Sm_{B,Z};$ $U\mapsto U\times_{B,Z} V$ preserves 
$\tau$-coverings and morphisms of the form $X\times_{B,Z}\A^{1}\to X$.
\end{proof}

Next, 
we state and outline the proofs of the two main results in this section relating to diagram 
\eqref{eq:Smaf(Z)A1Smaf(SZ)nisSmaf(B)(U)}.
Auxiliary results are deferred to \Cref{sect:DeformationZSZ}.

\begin{theorem}
\label{th:LocA1tfnisstructuresShLocsquare}
\begin{itemize}
\item[(1)]
The functors $\tilde i^!$, $\tilde i_*$, $j^*$, $j_*$ preserve Nisnevich and $\tf$-local objects. 
There is a canonical equivalence
\[
i_*i^!(F)\simeq \cofib(F\to j_*j^*F)
\]
for the naturally induced functors
\[
\Sh_{\tf}^\fr(\Sm_{B,Z})
\rightleftarrows 
\Sh_\tf^\fr(\Sm_{B}) 
\rightleftarrows 
\Sh_\tf^\fr(\Sm_{B-Z}).
\]
\item[(2)]
For $V\in \Sm_B$, $V_{B-Z}=V\times_S(B-Z)$, 
and any framed $\tf$-sheaf $F$ on $\Sm_B$, $\Sm_{B,Z}$, or $\Sm_{B-Z}$, 
there are equivalence
\[\begin{array}{ll}
\tilde i^! F^V\simeq (\tilde i^!F)^{V^h_Z},& 
\tilde i_* F^{V^h_Z}\simeq (\tilde i_*F)^{V},\\ 
j^* F^{V_{B-Z}}\simeq (j^*F)^V,& 
j_* F^{V_{B-Z}}\simeq (j_*F)^V.
\end{array}
\]
In particular, 
$\tilde i^!$, $\tilde i_*$, $j^*$, $j_*$ commute with $\Omega^l_{\Gm}\Lrep_{\A^{1}}$ for all $l\geq 0$.
\item[(3)]
When restricted to $\A^1$-invariant framed presheaves on $\Sm_{B}$, 
there is an equivalence 
\[
\tilde i^! \Lrep_{\nis} 
\simeq 
\Lrep_{\nis} \tilde i^!
\colon
\cPre_{\A^{1}}(\Sm_B)\to \cPre(\Sm_{B,Z}).
\]
Similar equivalences hold for $\tilde i_*$, $j^*$, and $j_*$.
\end{itemize}

\end{theorem}
\begin{proof}
By \Cref{lm:wtaulocaobjectsShLocsquare},  
the functors $\tilde i^!$, $\tilde i_*$, $j^*$, $j_*$ preserve both Nisnevich and $\tf$-local objects.
For any $F\in \cPre_\tau^\fr(\Sm_B)$ and $X\in \Sm_B$,
there is a canonical fiber sequence
\[
\fib( F(X) \to F(X\times_S (B-Z)) ) \to F(X)\to F(X\times_S (B-Z))
\]
and an equnvalence
\[
F(X\times_S (B-Z))
\simeq 
j_* j^* F(X).
\]
Now for any $F\in \Sh_\tf^\fr(\Sm_{B})$,
there are canonical equivalence
\[
\tilde i_* \tilde i^! F(X)\simeq 
\fib( F(X^h_Z) \to F(X^h_Z\times_S (B-Z)) )\simeq  
\fib( F(X) \to F(X\times_S (B-Z))), 
\]
and thus, 
there is a natural fiber sequence
\[
\tilde i_* \tilde i^! F(X)\to F(X)\to j_* j^* F(X).
\]
The above proves Part (1).

Concerning Part (2), 
\Cref{lm:F(Vtimes-)ShLocsquare} implies the claims for $(-)^V$.
The claim for $\Omega_{\Gm}\Lrep_{\A^{1}}$ follows from the equvalence 
$$
\Lrep_{\A^{1}} F(X)\simeq  F(\Delta_{-}^\bullet\times X),
$$
where $\Delta_{-}^\bullet=\Delta_{B}^\bullet$, $\Delta_{B,Z}^\bullet$, $\Delta_{B-Z}^\bullet$,
and 
$$
\Omega_{\Gm} F(X)\simeq \fib ( F(\Gm\times X)\to F(\{1\}\times X) ).
$$

Part (3) follows by appealing to \Cref{lm:tiusjustiufjdsLnis} below.
\end{proof}

We say that a functor is exact with respect to \( \A^{1} \)-equivalences 
if it preserves \(\A^{1}\)-equivalences. 
Similarly, a functor is said to be conservative with respect to \( \A^{1} \)-equivalences 
if every morphism whose image is an \( \A^{1} \)-equivalence 
was already an \( \A^{1} \)-equivalence. 
Analogous terminology applies for Nisnevich equivalences and other types of equivalences.

\begin{theorem}
\label{th:LocA1tfnisstructuresDeformation}
\begin{itemize}
\item[(1)]
The functors $\overarrow{i}^*$ and $\overarrow{i}_*$, 
see \eqref{eq:Smaf(Z)A1Smaf(SZ)nisSmaf(B)(U)} and 
\eqref{eq:PreScZiSjB-Z},
preserve 
$\A^{1}$-invariant objects, 
$\tf$-local objects 
and 
Nisnevich local objects; 
moreover, 
there is a naturally induced equivalence
\begin{equation}
\label{lm:overlineistauA1}
\overarrow{i}_*^{\A^{1},\tf}\colon 
\cPre^\fr_{\A^{1},\tf}(\SmAff_{Z})
\simeq \cPre^\fr_{\A^{1},\tf}(\SmAff_{B,Z})
\colon \overarrow{i}^*_{\A^{1},\tf}
\end{equation}
\item[(2)] 
$\overarrow{i}^*$ is exact and conservative with respect to $\A^{1}$-equivalences.
\item[(3)]
$\overarrow{i}_*$ is exact and conservative with respect to Nisnevich local equivalences.
\item[(4)]
$\overarrow{i}^*$ is Nisnevich exact
and conservative with respect to Nisnevich local equivalences on the subcategory 
$\cPre^\fr_{\A^{1}}(\SmAff_{B,Z})$.
\item[(5)]
Both $\overline{i}_*$ and $\overarrow{i}^*$ commutes with the endofunctor $\Omega_{\Gm}\Lrep_{\A^{1}}$ on 
the $\infty$-categories $\cPre^\fr(\SmAff_{Z})$ and $\cPre^\fr(\SmAff_{B,Z})$.
\end{itemize}
\end{theorem}
\begin{proof}
The claim on $\A^{1}$-invariant objects in (1) follows by \Cref{prop:iAGm}(1).
Let $\tau$ denote either the $\tf$- or Nisnevich topology.
Next, 
we discuss $\tau$-local objects. 
The claim holds for $\overarrow{i}_*$ because the functor $\SmAff_{B,Z}\to \SmAff_Z;$ $X^h_Z\mapsto X_Z$
preserves $\tau$-coverings.
By definition, 
for any Nisnevich covering in $\SmAff_Z$ there is a refinement $\widetilde X\to X$ and 
an Nisnevich covering $\widetilde Y\to Y$ such that $\widetilde X=\widetilde Y_Z$, $X=Y_Z$. 
Thus $\overarrow{i}^*$ preserves $\tau$-local objects.
The equivalence \eqref{lm:overlineistauA1} follows because by \Cref{prop:iAGm}(1), 
the adjunction \eqref{eq:ovSBcZSZ} is an equivalence.

Since $i_*$ commutes with $\Lrep_{\A^{1}}$, 
it is exact, 
i.e., 
given an $\A^1$-equivalence $F\to G$,
$i_*(F)\to i_*(G)$ is an $\A^1$-equivalence.
Moreover, 
it induces equivalences on subcategories of $\A^{1}$-invariant objects and it is conservative with 
respect to $\A^{1}$-equivalences. 
This completes the proof of (2).
For (3) and (4) we use \Cref{prop:ovidsus(simeqNis)}. 
Part (5) is \Cref{prop:iAGm}(2).
\end{proof}

\subsection{Deformation to closed subschemes}
\label{sect:DeformationZSZ}
Let 
$\mathcal S_{*}$ denote $\Smat_{*}$ or $\SmAff_{*}$, where $*$ is $B$, $Z$, or $(B,Z)$.

\begin{lemma}
\label{lm:BcZDeltalifttrivfib}
For all $X,U\in\mathcal S_B$, 
there are equivalences
\[
\Lrep_{\A^{1}}h^\fr(X^h_Z)(U^h_Z)
\xrightarrow{\simeq}
\Lrep_{\A^{1}}h^\fr(X^h_Z)(U_Z)
\xrightarrow{\simeq}
\Lrep_{\A^{1}}h^\fr(X_Z)(U_Z) 
\]
in the $\infty$-category of pointed spaces $\cSpc_*$.
\end{lemma}
\begin{proof}
The second equivalence holds because $h^\fr(X^h_Z)(U_Z)\simeq h^\fr(X_Z)(U_Z)$. 
To prove the first one, 
we let $X$ be a smooth affine $B$-scheme and $l\in \mathbb Z_{\geq 0}$.
Note that
\Cref{lm:etneigh:productoverBcZsublm}
provide isomorphisms
\[\begin{array}{lcl}
(\partial \Delta^l_{B,Z}\coprod_{\partial \Delta^l_Z}\Delta^l_Z)\times_{B,Z} U^h_Z&\simeq& 
((\partial \Delta^l_{B}\coprod_{\partial \Delta^l_Z}\Delta^l_Z)\times_{B} U^h_Z
)^h_Z
\\
\Delta^l_{B,Z}\times_{B,Z} U^h_Z&\simeq& (\Delta^l_{B}\times_{B} U)^h_Z
\end{array}\]

Applying \Cref{lm:liftFr} to the henselian pair given by the 
closed immersions
\[
(\partial \Delta^l_{B,Z}\coprod_{\partial \Delta^l_Z}\Delta^l_Z)\times_{B,Z} U^h_Z\not\hookrightarrow \Delta^l_{B,Z}\times_{B,Z} U^h_Z
\]
we obtain the surjection 
\[
\Fr_+^{B,Z}(U^h_Z\times_{B,Z}\Delta^l_{B,Z},X^h_Z)
\to 
\Fr_+^Z(U_Z\times_Z\Delta^l_Z,X^h_Z)\times_{\Fr_+^{Z}(U_Z\times_{Z}\partial\Delta^l_{Z},X^h_Z)} 
\Fr_+^{B,Z}(U^h_Z\times_{B,Z}\partial\Delta^l_{B,Z},X^h_Z).
\]
The latter surjection implies that the morphism 
\begin{equation}
\label{eq:FrBZAtoFrZA}
\Fr_+^{B,Z}(U^h_Z\times_{B,Z}\Delta^\bullet_{B,Z},X^h_Z)
\to 
\Fr_+^Z(U_Z\times_Z\Delta^\bullet_Z,X_Z^h)
\end{equation}
has the right lifting property with respect to the class of morphisms $\partial \Delta^l\to \Delta^l$.
Here we use that $\Fr_+^Z(-,X_Z^h)$ and $\Fr_+^{B,Z}(-,X_Z^h)$ satisfy closed gluing in the sense of 
\cite[\S A.2]{five-authors}. 
Thus \eqref{eq:FrBZAtoFrZA} is a trivial fibration.
Applying \Cref{lm:liftFr} and iterating the functor $\Lrep_{\A^{1}}^{[1]}$,
we obtain the trivial fibration
\[
\Lrep^{[l]}_{\A^{1}}F^h_Z(U^h_Z)
\to 
\Lrep^{[l]}_{\A^{1}}F_Z(U_Z)
\]
for $F^h_Z(-) = \Fr_+^{B,Z}(-,X^h_Z)$ and $F_Z(-) = \Fr_+^Z(-,X_Z^h)$.
Hence there is an equivalence 
\[
\Lrep_{\A^{1}}F^h_Z(U^h_Z)
\stackrel{\simeq}{\to} 
\Lrep_{\A^{1}}F_Z(U_Z).
\]
This finishes the proof on account of the $\A^1$-equivalence 
\begin{equation}
\label{A1equivalence}
\Fr_+^{B,Z}(-,X^h_Z)
\xrightarrow{\sim \A^1} 
\Corr^{\fr}_{B,Z}(-,X^h_Z).
\end{equation}
We note that \eqref{A1equivalence} follows by applying the arguments in 
\cite[Corollary 2.2.20, Corollary 2.3.25]{five-authors}
to $X^h_Z$ and $\Smat_{B,Z}$.
\end{proof}

\begin{lemma}\label{lm:iusA1invariant}
The functor $\overline i^*\colon \cPre^\fr(\mathcal S_{B,Z})\to \cPre^\fr(\mathcal S_Z)$
preserves $\A^{1}$-invariant objects.
\end{lemma}
\begin{proof}
Consider the commutative diagram:
\[\xymatrix{
\cPre^\fr(\mathcal S_{B,Z})\ar[d]\ar[r]^{\overline i^{*}} & \cPre^\fr(\mathcal S_Z)\ar[d] \\
\cPre(\Fr_+(\mathcal S_{B,Z}))\ar[d]\ar[r]^{\overline i^{*}_{\Fr_{+}}} & \cPre(\Fr_+(\mathcal S_Z))\ar[d]\\
\cPre(\mathcal S_{B,Z})\ar[r] & \cPre(\mathcal S_Z)
}\]
By definition of $\A^1$-invariance, 
the vertical maps preserve and detect $\A^1$-invariant objects.
The claim follows because $\overline i^*_{\Fr_+}$ preserves $\A^1$-invariant objects by \Cref{prop:LrepA1uuscommute}.
\end{proof}

\begin{lemma}
\label{lm:LrepA1ovupperstar}
The functor $\overarrow{i}^*\colon \cPre^\fr(\mathcal S_{B,Z})\to \cPre^\fr(\mathcal S_Z)$ 
admits the natural equivalence
\begin{equation}\label{eq:LA1ovovLA1}\Lrep_{\A^{1}}\overarrow{i}^*\simeq \overarrow{i}^*\Lrep_{\A^{1}}.\end{equation}
\end{lemma}
\begin{proof}
Let $(X\times\A^1)^h_Z\to X^h_Z$ be 
the naturally induced map in $\SmatBcZ$. 
Under the Yoneda embedding, 
it induces a morphism between representable presheaves.
The functor $\overarrow{i}^*$ sends the latter to the morphism induced by $(X\times\A^1)_Z\to X_Z$.
It follows that $\overarrow{i}^*$ preserves $\A^1$-equivalences.
\Cref{lm:iusA1invariant} shows that $\overarrow{i}^*$ preserves $\A^1$-invariant objects.
Hence \eqref{eq:LA1ovovLA1} follows.
\end{proof}

\begin{lemma}
\label{lm:ovA1}
There is a naturally induced adjunction 
\begin{equation}
\label{eq:ovSBcZSZ}
\overarrow{i}^*\colon 
\cPrefr_{\A^{1}}(\mathcal S_{B,Z})
\rightleftarrows 
\cPrefr_{\A^{1}}(\mathcal S_Z)
\colon 
\overarrow{i}_*.
\end{equation}
\end{lemma}
\begin{proof}
Since $\overarrow{i}^*$ and $\overarrow{i}_*$ preserve $\A^{1}$-invariant objects by 
\Cref{lm:iusA1invariant} and \Cref{lm:LrepA1ovupperstar}, 
respectively,
they restrict to functors on the subcategories of $\A^1$-invariant objects.
Hence \eqref{eq:ovSBcZSZ} follows from the adjunction
$\overarrow{i}^*\colon 
\cPrefr(\mathcal S_{B,Z})
\rightleftarrows 
\cPrefr(\mathcal S_Z)
\colon 
\overarrow{i}_*.
$
\end{proof}

\begin{lemma}\label{lm:generatorsLA1representable}
The category $\cPrefr_{\A^{1}}(\mathcal S_Z)$ is generated via colimits by the objects
$
\Lrep_{\A^{1}}h^\fr(Y), 
$
where $h^\fr(Y)\in \cPrefr(\mathcal S_Z)$ is the presheaf represented by $Y\in \mathcal S_Z$.
Likewise, 
$\cPrefr_{\A^{1}}(\mathcal S_{B,Z})$ is generated via colimits by the objects 
$
\Lrep_{\A^{1}}h^\fr(X^h_Z), 
$  
where $h^\fr(X^h_Z)\in \cPrefr(\mathcal S_{B,Z})$ is the presheaf represented by $X^h_Z\in \mathcal S_{B,Z}$. 
\end{lemma}
\begin{proof}
Recall the localization functor $L_{\A^{1}}\colon \cPrefr(\mathcal S_{Z})\to\cPrefr_{\A^{1}}(\mathcal S_Z)$.
The claim follows because $\cPrefr(\mathcal S_Z)$ is generated via colimits by representable functors, 
and $L_{\A^{1}}$ preserves colimits.
The proof for $\cPrefr(\mathcal S_{B,Z})$ is similar.
\end{proof}

\begin{lemma}
\label{lm:ovcolimits}
The functors $\overline{i}_*$ and $\overline{i}^*$ preserve colimits.
\end{lemma}
\begin{proof}
For $X\in \mathcal S_B$ and any diagram of presheaves $F_\alpha$ in $\cPrefr(\mathcal S_Z)$,
there are 
equivalences
\[
\overline{i}_*(\varinjlim_{\alpha}F_\alpha)(X^h_Z)
\simeq 
(\varinjlim_{\alpha}F_\alpha)(X_Z) 
\simeq 
\varinjlim_{\alpha}(F_\alpha(X_Z)) 
\simeq 
\varinjlim_{\alpha}(\overline{i}_*F_\alpha(X^h_Z)),
\]
where the outer equivalences hold by the definition of $\overline{i}_*$,
and the middle one is because colimits in the category of presheaves are formed schemewise.
The claim for $\overline{i}^*$ follows because it is a left adjoint functor.
\end{proof}

\begin{lemma}\label{lm:unovusds}
The unit of the adjunction \eqref{eq:ovSBcZSZ} is an equivalence
\begin{equation}
\label{eq:unovusds}
\mathrm{Id}_{\cPrefrA(\mathcal S_{B,Z})}
\xrightarrow{\simeq}
\overline{i}_*\overline{i}^*.
\end{equation}
\end{lemma}
\begin{proof}
For all $X,U\in\mathcal S_B$, there are equivalences
\[\begin{array}{lcl}
\overline{i}_*\overline{i}^*\Lrep_{\A^{1}}(h^\fr(X^h_Z))(U^h_Z) 
& \simeq & \Lrep_{\A^{1}}\overline{i}_*\overline{i}^*(h^\fr(X^h_Z))(U^h_Z) \\
& \simeq & \Lrep_{\A^{1}}\overline{i}_*(h^\fr(X_Z))(U^h_Z) \\
& \simeq & \Lrep_{\A^{1}}h^\fr(X_Z)(U_Z) \\
& \simeq & \Lrep_{\A^{1}}h^\fr(X^h_Z)(U^h_Z).
\end{array}\]
In the 
fourth 
equivalence we appeal to \Cref{lm:BcZDeltalifttrivfib}.
Thus $\overline{i}_*\overline{i}^*\Lrep_{\A^{1}}(h^\fr(X^h_Z))\simeq \Lrep_{\A^{1}}h^\fr(X^h_Z)$, 
and the claim follows from \Cref{lm:generatorsLA1representable} since both 
$\overline{i}_*$ and $\overline{i}^*$ preserve colimits by \Cref{lm:ovcolimits}.
\end{proof}

\begin{lemma}\label{lm:ovA1usesssurj}
The functor $\overarrow{i}^*$ in \eqref{eq:ovSBcZSZ} is essentially surjective.
\end{lemma}
\begin{proof}
\Cref{lm:LiftSmat} shows that for any $Y\in \mathcal S_Z$ there exists $X\in \mathcal S_B$ such that
\[
\Lrep_{\A^{1}}h^\fr(Y) 
\simeq 
\Lrep_{\A^{1}}\overarrow{i}^*(h^\fr(X^h_Z))
\simeq 
\overarrow{i}^*(\Lrep_{\A^{1}}h^\fr(X^h_Z)).
\]
Since $\overarrow{i}^*$ preserves colimits, 
\Cref{lm:generatorsLA1representable} implies that $\overarrow{i}^*$ in \eqref{eq:ovSBcZSZ} is essentially surjective.
\end{proof}

\begin{prop}\label{prop:iAGm}
The following hold for the adjunction 
\[
\overarrow{i}^*
\colon 
\cPre^\fr(\mathcal S_{B,Z})
\rightleftarrows 
\cPre^\fr(\mathcal S_{Z})
\colon 
\overarrow{i}_*.
\]
\begin{itemize}
\item[(1)]
$\overarrow{i}^*$ and $\overarrow{i}_*$ preserve $\A^{1}$-invariant objects, 
and the adjunction \eqref{eq:ovSBcZSZ} is an equivalence.
\item[(2)]
$\overarrow{i}^*$ and $\overarrow{i}_*$ commute with $\Lrep_{\A^{1}}$ and 
$\Omega^j_{\Gm}\Lrep_{\A^{1}}$, for $j\geq 0 $,
\end{itemize}
\end{prop}
\begin{proof}
We note that $\overarrow{i}_*$ commutes $\Omega^j_{\Gm}\Lrep_{\A^{1}}$, 
$j\geq 0$,
because $(\Gm\times\Delta_B^n)^h_Z\times_B Z=(\Gm\times\Delta_B^n)_Z$.
Moreover, 
$\overarrow{i}^*$ and $\overarrow{i}_*$ preserve $\A^{1}$-invariant objects by 
\Cref{lm:LrepA1ovupperstar,lm:iusA1invariant}, 
and induce the adjunction \eqref{eq:ovSBcZSZ} by \Cref{lm:ovA1}.
By \Cref{lm:unovusds} the unit of the adjunction \eqref{eq:ovSBcZSZ} is an equivalence. 
By \Cref{lm:ovA1usesssurj} the functor $\overarrow{i}^*$ in \eqref{eq:ovSBcZSZ} is essentially surjective.
Hence the adjunction \eqref{eq:ovSBcZSZ} is an equivalence.

To complete the proof of (2), 
we note that $\overarrow{i}_*$ commutes with $\Omega_{\Gm}$. 
Thus its inverse $\overarrow{i}^*$ in \eqref{eq:ovSBcZSZ} commutes with $\Omega_{\Gm}$ too.
It follows that $\overarrow{i}^*$ commutes with $\Omega_{\Gm}\Lrep_{\A^{1}}$.
\end{proof}

Recall that $\mathcal S_{*}$ denotes $\Smat_{*}$ or $\SmAff_{*}$.
\begin{prop}
\label{prop:ovidsus(simeqNis)}
The functor $\overarrow{i}_*\colon \cPre^\fr(\mathcal S_Z)\to \cPre^\fr(\mathcal S_{B,Z})$
preserves and detects Nisnevich local equivalences.
Furthermore, 
the same holds for 
$\overarrow{i}^*\colon \cPre^\fr(\mathcal S_{B,Z})\to \cPre^\fr(\mathcal S_{Z})$
on the subcategory of $\A^{1}$-invariant objects. 
\end{prop}
\begin{proof}
Suppose that $F\rightarrow G$ is a Nisnevich local equivalence in $\cPre^\fr(\mathcal S_Z)$. 
If $X\in \Sm_{B}$, 
the morphism $\overarrow{i}_*F(X^h_x)\to \overarrow{i}_*G(X^h_x)$ 
or equivalently $F((X^h_x)_Z)\to G((X^h_x)_Z)$ is then an equivalence since $(X^h_x)_Z\cong (X_Z)^h_x$ 
is an essentially smooth local henselian scheme.

Assume that $F\to G$ induces a Nisnevich local equivalence $\overarrow{i}_*F\simeq_{\nis} \overarrow{i}_*G$ and 
let $X^h_x$ be an essentially smooth local henselian scheme for $X\in \mathcal S_{Z}$.
\Cref{lm:LiftSmat} implies there exists a scheme $\widetilde X\in \mathcal S_{B}$ and a morphism 
$x\to \widetilde X$ such that $(\widetilde X_{(x)})_Z\cong X_{(x)}$.
Hence $(\widetilde X^h_x)_Z\cong X^h_x$, 
and thus we can identify $F(X^h_x)\to G(X^h_x)$ with 
$\overarrow{i}_*F(\widetilde X^h_x)\simeq \overarrow{i}_*G(\widetilde X^h_x)$.
This proves the claim for $\overarrow{i}_*$.

Since $\overarrow{i}\colon \calS_{B,Z}\to \calS_{Z}$ preserves Nisnevich squares, 
$\overarrow{i}_*$ preserves Nisnevich local objects and $\overarrow{i}^*$ preserves Nisnevich local equivalences.
\Cref{prop:iAGm}(1) implies our final claim since $\overarrow{i}_*$ preserves Nisnevich local equivalences.
\end{proof}

\begin{lemma}
There are canonical equivalences
\[\cPre_{\A^{1}}(\SmAff_{B,Z})\simeq \cPre_{\A^{1}}(\Smat_{B,Z}), \; \cPre^\fr_{\A^{1}}(\SmAff_{B,Z})\simeq \cPre^\fr_{\A^{1}}(\Smat_{B,Z}).\]
\end{lemma}
\begin{proof}
For any $X^h_Z\in \SmAff_{B,Z}$ there exists an $\A^{1}$-equivalent scheme $(X^\prime)^h_Z\in \Smat_{B,Z}$.
Namely, 
$X^\prime$ is the total space of the vector bundle $N$ over $X$ such that $N\oplus T_X$ is trivial; 
in this case, 
$T_{X^\prime}$ is the inverse image of $N\oplus T_X$ along $X^\prime\to X$.
Hence the functors
\begin{equation}\label{eq:HSmAfftoHSmat}\cPre_{\A^{1}}(\SmAff_{B,Z})\to \cPre_{\A^{1}}(\Smat_{B,Z}), \cPre^\fr_{\A^{1}}(\SmAff_{B,Z})\to \cPre^\fr_{\A^{1}}(\Smat_{B,Z})\end{equation}
are conservative.
Since $\Smat_{B,Z}\to \SmAff_{B,Z}$ is fully faithful,
the composite functors
\[
\begin{array}{lclcl}
\cPre_{\A^{1}}(\Smat_{B,Z})&\to& \cPre_{\A^{1}}(\SmAff_{B,Z})&\to& \cPre_{\A^{1}}(\Smat_{B,Z}),\\ 
\cPre^\fr_{\A^{1}}(\Smat_{B,Z})&\to& \cPre^\fr_{\A^{1}}(\SmAff_{B,Z})&\to& \cPre^\fr_{\A^{1}}(\Smat_{B,Z}),
\end{array}
\]
are equivalent to the identity.
Therefore, 
since \eqref{eq:HSmAfftoHSmat} is conservative, 
the composite functors
$\cPre_{\A^{1}}(\SmAff_{B,Z})\to \cPre_{\A^{1}}(\Smat_{B,Z})\to \cPre_{\A^{1}}(\SmAff_{B,Z})$,
$\cPre^\fr_{\A^{1}}(\SmAff_{B,Z})\to \cPre^\fr_{\A^{1}}(\Smat_{B,Z})\to \cPre^\fr_{\A^{1}}(\SmAff_{B,Z})$
are equivalent to the identity too.
\end{proof}

\subsection{Localization theorem for $\tf$-sheaves}

For any closed immersion $Z\not\hookrightarrow B$ and a topology $\tau$ on the category $\Sm_B$, 
we defined the topology $\tau$ on $\Sm_{B,Z}$ as the weakest topology on $\Sm_{B,Z}$ such that the functor 
$\Sm_{B}\to \Sm_{B,Z}$; $X\mapsto X^h_Z$, 
is continuous.
Throughout this section $\tau$ is shorthand for the $\tf$-topology or the Nisnevich topology. 

\begin{lemma}
\label{lm:wtaulocaobjectsShLocsquare}
The functors $\tilde i^!$, $\tilde i_*$, $j^*$, $j_*$
on $\cPre^\fr(\Sm_{B})$, $\cPre^\fr(\Sm_{B-Z})$, and $\cPre^\fr(\Sm_{B,Z})$ in \eqref{eq:PreScZiSjB-Z}
preserve Nisnevich and $\tf$-local objects.
\end{lemma}
\begin{proof}
The claims for $\tilde i_*$, $j^*$ and $j_*$ follow since the functors 
\begin{equation}\label{eq:XhZX(B-Z)X}\begin{array}{ll}
\Sm_{B}\to \Sm_{B,Z}; &X\mapsto X^h_Z,\\
\Sm_{B}\to \Sm_{B-Z}; &X\mapsto X_{B-Z},\\
\Sm_{B-Z}\to \Sm_{B}; &X\mapsto X,\\
\end{array}\end{equation}
preserve Nisnevich and $\tf$-coverings.

By definition, 
the $\tf$-topology on $\Sm_{B,Z}$ is the weakest topology such that the functor $\Sm_B\to \Sm_{B,Z}$ is continuous,
then any $\tf$-covering of a scheme $X^h_Z\in\Sm_{B,Z}$ admits a refinement $\widetilde X^h_Z\to X^h_Z$ induced by 
some $\tf$-covering $\widetilde X\to X$.
Moreover, 
since the base change functor $\Sm_{B}\to\Sm_{B-Z}$ is continuous,
$\widetilde X_{B-Z}\to X_{B-Z}$ is a $\tf$-covering too.
Thus, 
for any $F\in \Sh^\fr_\tf(\Sm_{B})$, 
there are equivalences
\begin{align*}
\check{C}_{\widetilde X^h_Z}(X^h_Z,\tilde i^! F)&\simeq
\fib( \check{C}_{\widetilde X^h_Z}(X^h_Z,F)\to \check{C}_{\widetilde X^h_Z-X_Z}(X^h_Z-X_Z,F) )\\& \simeq 
\fib( \check{C}_{\widetilde X}(X,F)\to \check{C}_{\widetilde X-X_Z}(X-X_Z,F) )\\& \simeq 
\fib( F(X)\to F(X-X_Z) )\\& \simeq 
\fib( F(X^h_Z)\to F(X^h_Z-X_Z) )\\& \simeq 
\tilde i^! F(X^h_Z).
\end{align*}
This concludes the proof for the $\tf$-topology.
Since the Nisnevich topology is stronger then the $\tf$-topology,
a similar argument applies to the former.
\end{proof}

\begin{lemma}\label{lm:tiusjustiufjdsLnis}
For all $E\in \cPre^\fr_{\A^{1}}(\Sm_{B,Z})$, $F\in \cPre^\fr_{\A^{1}}(\Sm_{B})$, 
and $G\in \cPre^\fr_{\A^{1}}(\Sm_{B-Z})$ there are naturally induced equivalences
\begin{equation}
\label{eq:B,Z_B_Z_Lnis}
\begin{array}{lcl}
\tilde i_* \Lrep_{\nis} E\simeq \Lrep_{\nis} \tilde i_* E,
&
j^* \Lrep_{\nis} F\simeq \Lrep_{\nis} j^* F,
\\
\tilde i^! \Lrep_{\nis} F\simeq \Lrep_{\nis} \tilde i^! F,
&
j_* \Lrep_{\nis} G\simeq \Lrep_{\nis} j_* G.
\end{array}
\end{equation}
\end{lemma}
\begin{proof}
Since by \Cref{lm:wtaulocaobjectsShLocsquare}, both $\tilde i_*$ and $j^*$ preserve Nisnevich local objects;
thus to prove the first row equivalences in \eqref{eq:B,Z_B_Z_Lnis} it remains to show preservation of 
Nisnevich local equivalences.
This follows because the first and last functors in \eqref{eq:XhZX(B-Z)X} preserve essentially 
smooth local henselian schemes.

Since
$\tilde i^!$ and $j_*$ preserve Nisnevich local objects by \Cref{lm:wtaulocaobjectsShLocsquare}, 
for any $F$ and $G$ as above 
the presheaves $\tilde i^! \Lrep_\nis F$ and $j_* \Lrep_\nis G$ are Nisnevich local.
We are reduced to prove the morphisms 
\begin{equation}
\label{eq:tsLnis}
\tilde i^! \Lrep_\nis F\to \Lrep_\nis \tilde i^! \Lrep_\nis F
\leftarrow 
\Lrep_\nis \tilde i^! F
\end{equation} 
and
\begin{equation}
\label{eq:LNisG}
j_* \Lrep_\nis G\to \Lrep_\nis j_* \Lrep_\nis G
\leftarrow 
\Lrep_\nis j_* G
\end{equation}
induce equivalence of the stalks on the Nisnevich topology points.
The claim for \eqref{eq:tsLnis} follows since, 
for all $X\in \Sm_{B}$, $x\in X_Z$,
there are equivalences
\begin{gather*}
\tilde i^! F(X^h_x)\simeq
\fib ( F(X^h_x)\to F((X^h_x)_{B-Z}) ),\\
\tilde i^! \Lrep_\nis F(X^h_x)\simeq
\fib ( \Lrep_\nis F(X^h_x)\to \Lrep_\nis F((X^h_x)_{B-Z}) )\stackrel{Th \ref{th:Hnis(Xhxeta)dimB1}}{\simeq}
\fib ( F(X^h_x)\to F((X^h_x)_{B-Z}) ).
\end{gather*}
The right side equivalence in the second row holds because by \Cref{th:Hnis(Xhxeta)dimB1} $\Lrep_\nis F((X^h_x)_{B-Z})\simeq F((X^h_x)_{B-Z})$ for any $F\in \cPre^\fr_{\Sigma,\A^1}(\Sm_B)$ and consequently for any $F\in \cPre^\fr_{\A^1}(\Sm_B)$ since $(X^h_x)_{B-Z}$ is irreducible. 
The claim for \eqref{eq:LNisG} follows because, 
for any $X\in \Sm_{B}$, $x\in X_Z$,
\begin{gather*}
\Lrep_\nis j_* G(X^h_x)
\simeq
j_* G(X^h_x)\simeq
G((X^h_x)_{B-Z}),\\
j_* \Lrep_\nis G(X^h_x)\simeq
\Lrep_\nis G((X^h_x)_{B-Z}) \stackrel{Th \ref{th:Hnis(Xhxeta)dimB1}}{\simeq}
G((X^h_x)_{B-Z}), 
\end{gather*},
where the second row uses \Cref{th:Hnis(Xhxeta)dimB1} as explained above,
and when $x\in X_{B-Z}$, we have 
\[
j_* F(X^h_x)\simeq F(X^h_x), 
\,\, 
j_* \Lrep_\nis F(X^h_x)\simeq \Lrep_\nis F(X^h_x)\simeq F(X^h_x).
\]
\end{proof}

\begin{lemma}
\label{lm:F(Vtimes-)ShLocsquare}
When restricted to subcategories of $\tf$-local objects, 
the functors $\tilde i^!$, $\tilde i_*$, $j^*$, $j_*$ in \eqref{eq:PreScZiSjB-Z} commute with $(-)^V$, 
for all $V\in \Sm_{B}$, $\Sm_{B,Z}$, $\Sm_{B-Z}$.
\end{lemma}
\begin{proof}
The claims for $\tilde i_*$, $j^*$, $j_*$ follow because
the functors \eqref{eq:XhZX(B-Z)X} commute with the functor $X\mapsto X\times V$.
The case of $\tilde i^!$ follows because
\begin{align*}
\tilde i^!F( (X\times V)^h_Z )&\simeq
\fib( F( (X\times V)^h_Z )\to F( (X\times V)^h_Z-(X\times V)_Z ) )\\ &\simeq
\fib( F( X\times V )\to F( (X\times V)-(X\times V)_Z ) )\\& \simeq
\tilde i^!F^V( X^h_Z ).
\end{align*}
\end{proof}

\section{Stable motivic localization for framed trivial fiber sheaves} 
\label{sect:stablemotiviclocalization}

The main results of this section are concerned with explicit formulas for the motivic infinite loop space 
\begin{equation}
\label{equation:Nmils}
\OmegaSigma^{\tf,\fr}_{\A^{1},\nis}
\simeq
\Lrep_\nis\colim\Omega^l_{\Gm}\Lrep^{\fr,\tf}_{\A^{1}}(\Sigma^l_{\Gm}F^{\gp}),
\end{equation}
and the stable motivic localization 
\begin{equation}
\label{equation:smloc}
\catLrep^{\tf,\fr}_\nissmot(G)
\simeq  
\catLrep^\fr_\nis \catLrep^\fr_{\Gm}\catLrep^{\fr,\tf}_{\A^{1}}G^\gp.
\end{equation}
The expression \eqref{equation:Nmils} is taken in $\cPrefrtf(B)$, 
while \eqref{equation:smloc} is a natural equivalence in $\PSpt_\tf^{s,t,\fr}(B)$.
These results are recorded in \Cref{prop:OmegaSigmafrtf_SGA1} and \Cref{th:LrepsmotfrPSpttf}, 
respectively.
When $B$ is a field, 
\eqref{equation:Nmils} is equivalent to the strict homotopy invariance theorem together with property (3) 
below for the $\Gm$-loop functor $\Omega_{\Gm}$ (a weaker statement than the cancellation theorem). 

To prove \eqref{equation:Nmils} and \eqref{equation:smloc}, 
we reformulate the strict homotopy invariance theorem and the said property of $\Omega_{\Gm}$. 
In the process,
we upgrade to base schemes and $\infty$-categories the formal viewpoint for fields and 
triangulated categories developed in \cite{DruDMGWeff}.
By \Cref{lm:LreptLprepcpreserveexact}, 
the $\tf$-Nisnevich strict homotopy invariance theorem over $B$,
see \Cref{def:SHIthnistf}, 
can be formulated in the following equivalent ways:
\begin{itemize}
\item[(1)] The endofunctor $\Lrep_\nis$ on $\cSptsfrtf(B)$ preserves $\A^{1}$-invariant spectra, and 
\item[(2)] The endofunctor $\Lrep_{\A^{1}}$ on $\cSptsfrtf(B)$ preserves Nisnevich local equivalences.
\end{itemize}
By combining (2) with 
\begin{itemize}
\item[(3)] The endofunctor $\Omega_{\Gm}$ on $\cSptsfrAtf(B)$ preserves Nisnevich local equivalences, 
\end{itemize}
we obtain
\begin{itemize}
\item[(4)] The endofunctor $\Omega_{\Gm}\Lrep_{\A^{1}}$ on $\cSptsfrtf(B)$ preserves Nisnevich local equivalences.
\end{itemize}
Instead of proving (2) and (3) separately, 
we prove (4) directly. 
To that end, 
we apply the results of \Cref{section:SplittingdiagramLocalization} and study the functor 
$\Omega^\infty_{\Gm}\colon \cSptstfrtf(B)\to \cSptsfrtf(B)$ through the commutative diagram
\begin{equation*}
\label{eq:catPre(B)Shvnis(B)A1Gm}\xymatrix{
\cSptsfrtf(B)\ar[r]\ar[d]& \cSptsfrtf(B)[(\Omega_{\Gm}\Lrep_{\A^{1}})^{-1}]\ar[r]^-{\simeq}\ar[d]& 
\cSptsfrAtf(B)[\Gm^{\wedge -1}]\ar[d] \\
\cSptsfrnis(B)\ar[r]& \cSptsfrnis(B)[(\Omega_{\Gm}\Lrep_{\A^{1}}^{-1})]\ar[r]^-{\simeq} & 
\cSptsfrnis(B)[\Gm^{\wedge -1}]\ar[r]^-{\simeq}
& \catSH^{\fr}(B).
}
\end{equation*}
Here for a presentably symmetric monoidal $\infty$-category $\mathcal C$ and an object $X\in\mathcal C$, 
we denote by $\mathcal C[X^{-1}]$ the initial presentably symmetric monoidal $\infty$-category under 
$\mathcal C$ in which $X$ becomes invertible \cite[\S 2.1]{RobalnoncomKtheorybridge}.
More generally, 
we adopt the same notation for the stabilization of $\mathcal C$ with respect to an adjoint functor pair 
as in \cite[\S 2.2]{RobalnoncomKtheorybridge}.

\subsection{Formal strict homotopy invariance and motivic infinite loop spaces}
\label{sect:reformlemma}
Throughout this section, 
we fix a presentable $\infty$-category $\mathcal C$ with reflective subcategories 
$\mathcal C_{\alpha}$, 
$\mathcal C_{\beta}$.
Let \begin{equation}\label{eq:LCCEalpha}
L_{\alpha}\colon \mathcal C\rightleftarrows \mathcal  C_{\alpha}\colon E_\alpha, \quad L_\alpha\dashv E_\alpha,
\end{equation} 
be the adjunction of 
the localization functor $L_\alpha$ and the canonical embedding $E_{\alpha}$. 
$E_{\alpha}\colon \mathcal C_{\alpha}\to \mathcal C$
An $\alpha$-equivalence in $\mathcal C$ is a morphism that maps to an equivalence under $L_{\alpha}$. 
In related notation,
let $\Lrep_{\alpha}$ denote the corresponding localization endofunctor on 
$\mathcal C$ taking values in $\mathcal C_{\alpha}$, respectively.
That is, 
$\Lrep_{\alpha}$ is the composite 
$E_{\alpha}L_{\alpha}\colon \mathcal C\to \mathcal C_{\alpha}\to \mathcal C$.
We refer to \cite[Subsections  5.2.7, 5.5.4]{LurieHTT} for background on localizations of $\infty$-categories.
Furthermore, 
we assume that the intersection $\mathcal C_{\beta\alpha}=\mathcal C_\alpha\cap\mathcal C_\beta$ 
is a reflective subcategory of $\mathcal C$ and denote 
the corresponding 
endofunctor on $\mathcal C$ by $\Lrep_{\beta\alpha}$.
This extra assumption holds if $C_{\alpha}$ and $\mathcal C_{\beta}$ are accessible localizations of $\mathcal C$ 
\cite[Section 5.4]{LurieHTT}.

The next result is a formal version of strict homotopy invariance.
\begin{lemma}
\label{lm:LreptLprepcpreserveexact}
The following conditions are equivalent:
\begin{itemize}
\item[(1)] 
The canonical morphism $\Lrep_{\beta}\Lrep_{\alpha}\to\Lrep_{\beta\alpha}$ is an equivalence.
\item[(2)] $\Lrep_{\beta}$ preserves $\mathcal C_{\alpha}$.
\item[(3)] $\Lrep_{\alpha}$ preserves $\beta$-equivalences.
\end{itemize}
\end{lemma}
\begin{proof}
Let us show that (1) implies (2).
Since $\Lrep_{\beta\alpha}$ takes values in $\mathcal C_{\alpha}$, 
it preserves $\mathcal C_{\alpha}$. 
Then by assumption, 
$\Lrep_{\beta}\Lrep_{\alpha}$ preserves $\mathcal C_{\alpha}$.
Conversely, 
since $\Lrep_{\beta}$ preserves $\mathcal C_{\alpha}$ and $\Lrep_{\alpha}$ lands in $\mathcal C_{\alpha}$, 
it follows that $\Lrep_{\beta}\Lrep_{\alpha}$ lands in $\mathcal C_{\alpha}$.
Therefore, 
$\Lrep_{\beta}\Lrep_{\alpha}$ takes values in $\mathcal C_{\beta\alpha}$,
and there is a canonical equivalence
\[
\Lrep_{\beta}\Lrep_{\alpha}
\stackrel{\simeq}{\to} 
\Lrep_{\beta\alpha}\Lrep_{\beta}\Lrep_{\alpha}.
\] 
Applying $\Lrep_{\beta\alpha}$ to the composite 
\[
\Id
\to 
\Lrep_{\alpha} 
\to 
\Lrep_{\beta} 
\Lrep_{\alpha}
\]
yields an equivalence,
since $\Lrep_{\beta\alpha}$ is isomorphic to 
$\Lrep_{\beta\alpha}\Lrep_{\beta}$ and $\Lrep_{\beta\alpha}\Lrep_{\alpha}$.
Thus, 
we deduce the canonical equivalence in (1) by using the commutative diagram:
\[
\xymatrix{
\Id\ar[r]\ar[d] & \Lrep_{\beta}\Lrep_{\alpha} \ar[d]^-{\simeq}\\ 
\Lrep_{\beta\alpha} \ar[r]^-{\simeq}& 
\Lrep_{\alpha\beta} \Lrep_{\beta}\Lrep_{\alpha} } 
\]

To show that (1) implies (3), 
assume $F\to G$ is a $\beta$-equivalence so that $\Lrep_{\beta}F\simeq \Lrep_{\beta}G$ and 
$\Lrep_{\beta\alpha}F\simeq \Lrep_{\beta\alpha}G$.
Hence, 
by assumption, 
$\Lrep_{\beta} (\Lrep_{\alpha} F)\simeq \Lrep_{\beta} (\Lrep_{\alpha} G)$ and 
$\Lrep_{\alpha} F\to \Lrep_{\alpha} G$ is a $\beta$-equivalence.
Conversely, 
since the functor $\Lrep_\alpha$ preserves $\beta$-equivalences,
the functor $\Lrep_\beta\Lrep_\alpha$ sends $\beta$-equivalences to equivalences.
At the same time, 
$\Lrep_\beta\Lrep_\alpha$ sends $\alpha$-equivalences to equivalences. 
Hence since the $\beta\alpha$-equivalences are generated by $\alpha$-equivalences and $\beta$-equivalences, 
$\Lrep_\beta\Lrep_\alpha$ sends $\beta\alpha$-equivalences to equivalences.
In particular, 
it takes 
\begin{equation}
\label{eq:LbetaLalpha} 
\Id\to \Lrep_{\beta} \Lrep_{\alpha}
\colon 
\mathcal C\to \mathcal C
\end{equation}
to an equivalence.
Let $\mathcal C^\prime$ denote the full subcategory of $\mathcal C$ spanned by objects $G$ such that 
the canonical morphism $G\to \Lrep_\beta\Lrep_\alpha(G)$ is an equivalence.
Hence for any $F\in \mathcal C$ and $G\in\mathcal C^\prime$ the morphism
\[
\Map_{\mathcal C}(F,G)\to 
\Map_{\mathcal C}(\Lrep_{\beta}\Lrep_{\alpha} F,\Lrep_{\beta}\Lrep_{\alpha} G)
\simeq 
\Map_{\mathcal C}(\Lrep_{\beta}\Lrep_{\alpha} F,G)
\]
is an inverse to the canonical morphism 
\begin{equation}
\label{eq:compo-refl} 
\Map_{\mathcal C}(\Lrep_{\beta}\Lrep_{\alpha} F,G)
\to 
\Map_{\mathcal C}(F,G).
\end{equation}
Thus, 
since \eqref{eq:compo-refl} is an equivalence, 
it follows that $\Lrep_{\beta}\Lrep_{\alpha}$ induces an adjunction $\mathcal C\rightleftarrows \mathcal C^\prime$. 
Hence $\mathcal C^\prime$ is a reflective subcategory of $\mathcal C$.
Moreover, 
since $\Lrep_{\beta}\Lrep_{\alpha}$ takes $\alpha$- and $\beta$-equivalences to equivalences 
and \eqref{eq:compo-refl} is an equivalence, 
we deduce that $\Lrep_{\beta}\Lrep_{\alpha}$ takes values in the subcategory of $\alpha$- and 
$\beta$-local objects.
It follows that $\mathcal C^\prime\subset \mathcal C_{\beta\alpha}$.
Finally, 
since the canonical morphism \eqref{eq:LbetaLalpha} is an $\beta\alpha$-equivalence, 
we conclude that $\Lrep_{\beta}\Lrep_{\alpha} \simeq \Lrep_{\beta\alpha}$ and 
$\mathcal C^\prime=\mathcal C_{\beta\alpha}$.
\end{proof}

\begin{lemma}
\label{lm:LocsubsetStab}
There is a canonical equivalence $\mathcal C[\Lrep_{\alpha}^{-1}]\simeq \mathcal C_{\alpha}$. 
\end{lemma}
\begin{proof}
By definition, 
the localization $\mathcal C[\Lrep_{\alpha}^{-1}]$ is universally equipped with a functor 
$s\colon \mathcal C\to \mathcal C[\Lrep_{\alpha}^{-1}]$ and endofunctors $\Lrep_{\alpha}$,  
$\Lrep_{\alpha}^{-1}$ such that 
$\Lrep_{\alpha}^{-1}\Lrep_{\alpha}\simeq \mathrm{Id}_{\mathcal C[\Lrep_{\alpha}^{-1}]}$ 
and $s\Lrep_{\alpha}\simeq \Lrep_{\alpha}s$.
There exists a canonical functor $\mathcal C[\Lrep_{\alpha}^{-1}]\to \mathcal C_{\alpha}$ since
$\mathcal C_{\alpha}$ and $L_{\alpha}\colon \mathcal C\to \mathcal C_{\alpha}$ satisfy the said
properties.
Moreover,
$\Lrep_{\alpha}$ is idempotent. 
Hence $s$ takes $\alpha$-equivalences to equivalences, 
and there is the canonical functor 
$\mathcal C_{\alpha}\simeq \mathcal C[\Lrep_{\alpha}^{-1}]$.
In the commutative diagram
\[
\xymatrix{
\mathcal C\ar[r]\ar[dr]& \mathcal C_{\alpha}\ar@<1ex>[d]\\
& \mathcal C[\Lrep_{\alpha}^{-1}]\ar@<1ex>[u] 
}
\]
the composite endomorphisms are equivalences according to the universal property.
\end{proof}

Let $\mathcal C$ be a presentable symmetric monoidal $\infty$-category. 
Assume there exists an adjunction $\Sigma\dashv \Omega$ on $\mathcal C$ and define endofunctors on 
$\mathcal C_{\alpha}$ by 
\begin{equation}\label{eq:SigmaOmegaalpha}
\Omega_{\alpha}=L_{\alpha} \Omega E_{\alpha},\;
\Sigma_{\alpha}=L_{\alpha} \Sigma E_{\alpha},
\end{equation}
see \eqref{eq:LCCEalpha}.
We consider the categories $\mathcal C[\Omega^{-1}]$, $\mathcal C_\alpha[\Omega^{-1}_\alpha]$ as in 
\cite[\S 2.2]{RobalnoncomKtheorybridge} and \Cref{section:StablelocalizationDefinitionsNotations}.

\begin{lemma}
\label{lm:LalphaOmega}
Suppose $\Sigma$ and $\Omega$ preserve $\alpha$-equivalences.
Then there are naturally induced equivalences
\begin{equation}\label{eq:SLaLaSOLaLaO}
\Sigma_\alpha L_\alpha 
\simeq 
L_\alpha \Sigma, \;\Omega_\alpha L_\alpha 
\simeq 
L_\alpha \Omega,
\end{equation}
and
\begin{equation}\label{eq:OaSaLLaOaSa}
\Omega_{\alpha}^\infty\Sigma^\infty_{\alpha}L_{\alpha}
\simeq 
L_{\alpha}\Omega_{}^\infty\Sigma^\infty_{}
\colon
\mathcal C\to \mathcal C_\alpha.
\end{equation}
Moreover, 
the localization $L_{\alpha}\colon \mathcal C\to \mathcal C_{\alpha}$ induces a functor 
$L_{\alpha}^\Omega\colon\mathcal C[\Omega^{-1}]\to \mathcal C_{\alpha}[\Omega_{\alpha}^{-1}]$ such that 
\begin{equation}
\label{eq:SLalphaLalphaOmegaS}
\Sigma^\infty_\alpha L_\alpha \simeq 
L_\alpha^\Omega \Sigma^\infty\colon 
\mathcal C\to 
\mathcal C_\alpha[\Omega^{-1}_\alpha].
\end{equation}
\end{lemma}
\begin{proof}
Since $\Sigma$ and $\Omega$ are exact with respect to $\alpha$-equivalences, 
$\Sigma$ and $\Omega$ induce endofunctors $\Sigma_{\alpha}$ and $\Omega_{\alpha}$ on 
$\mathcal C_\alpha$, as defined in \eqref{eq:SigmaOmegaalpha},
such that \eqref{eq:SLaLaSOLaLaO} hold. Then since $L_\alpha E_\alpha$ is equivalent to the identity on $\mathcal C_\alpha$ then
\eqref{eq:SigmaOmegaalpha} hold for such $\Sigma_\alpha$ and $\Omega_\alpha$.
Moreover, since $\Sigma$ is exact with respect to $\alpha$-equivalences,
$\Omega$ preserves $\alpha$-local objects, 
and we have 
$
\Omega E_\alpha 
\simeq 
E_\alpha \Omega_\alpha.
$
Then there is an induced adjunction $\Sigma_{\alpha}\dashv \Omega_{\alpha}$ owing to 
the equivalences of mapping spaces 
\[
\begin{split}
\Map_{\mathcal C_\alpha}(\Sigma_\alpha L_\alpha X, L_\alpha Y) &\simeq 
\Map_{\mathcal C_\alpha}(L_\alpha \Sigma X, L_\alpha Y) \\ &\simeq 
\Map_{\mathcal C}(\Sigma X, E_\alpha L_\alpha Y) \\ &\simeq 
\Map_{\mathcal C}(X, \Omega E_\alpha L_\alpha Y) \\ &\simeq 
\Map_{\mathcal C}(X, E_\alpha \Omega_\alpha L_\alpha Y) \\ &\simeq 
\Map_{\mathcal C_\alpha}(L_\alpha X, \Omega_\alpha L_\alpha Y).
\end{split}
\]
By the above, 
we extract four commutative squares
\[\xymatrix{
\mathcal C \ar@<0.5ex>[r]^{\Sigma}\ar[d]_{\Lrep_\alpha} & 
\mathcal C\ar@<0.5ex>[l]^{\Omega}\ar[d]^{\Lrep_\alpha} 
\\ 
\mathcal C_\alpha\ar@<0.5ex>[r]^{\Sigma_\alpha} & \mathcal C_\alpha \ar@<0.5ex>[l]^{\Omega_\alpha}}
\xymatrix{\mathcal C\ar@<0.5ex>[r]^{\Sigma^\infty}\ar[d]_{\Lrep_\alpha}& \mathcal C[\Omega^{-1}]\ar@<0.5ex>[l]^{\Omega^\infty}\ar[d]^{\Lrep_\alpha^\Omega} \\ \mathcal C_\alpha \ar@<0.5ex>[r]^{\Sigma^\infty_\alpha}& \mathcal C_\alpha[\Omega^{-1}_{\alpha}]\ar@<0.5ex>[l]^{\Omega^\infty_\alpha}} \]
The left side squares provide \eqref{eq:SLaLaSOLaLaO}, 
and the commutativity of the right side squares 
implies \eqref{eq:SLalphaLalphaOmegaS} and \eqref{eq:OaSaLLaOaSa}.
\end{proof}

Next, 
we give a refined formal formulation of strict homotopy invariance related to loop spaces.
We refer to \Cref{def:PreSpectra} for a discussion of the $\infty$-category of prespectra
$\PSpt^\Omega(\mathcal C)$.

\begin{lemma}
\label{lm:ReformulationOmega}
Suppose $\Omega$ preserves $\alpha$-local objects.
Then the following conditions are equivalent.
\begin{itemize}
\item[(1)] 
The natural transformation $L_{\alpha}\Omega^\infty\to\Omega^\infty_{\alpha}L^\Omega_{\alpha}\colon 
\mathcal C[\Omega^{-1}]\to \mathcal C_\alpha$ is an equivalence.
\item[(2)] 
The natural transformation 
$\Lrep_{\alpha}\Lrep_{\Omega}\to\Lrep_{\Omega,\alpha}\colon\PSpt^\Omega(\mathcal C)
\to\PSpt^\Omega(\mathcal C)$ is an equivalence.
\item[(3)] $\Omega$ preserves $\alpha$-equivalences, i.e., 
if $f\colon F\to G$ is an $\alpha$-equivalence, 
then $\Omega(f)\colon \Omega(F)\to\Omega(G)$ is an $\alpha$-equivalence. 
\end{itemize}
\end{lemma}
\begin{proof}
Note that $\Omega$ preserves $\alpha$-equivalences if and only if 
\begin{equation}
\label{eq:LtnOmega(t)}
L_{\alpha} \Omega 
\simeq 
\Omega_{\alpha} L_{\alpha}.
\end{equation}
Using \eqref{eq:LtnOmega(t)}, 
it follows that $L_{\alpha}\Omega^\infty\simeq\Omega^\infty_{\alpha}L^\Omega_{\alpha}$.
Thus (3) implies (1) 
--- this part will be applied in \Cref{prop:Pretffr(SmSz)SGA1nis}.
\Cref{lm:LreptLprepcpreserveexact} shows that (2) is equivalent to asserting that applying 
$L_{\alpha}$ levelwise takes $\Omega$-objects in $\PSpt^\Omega(\mathcal C)$ to $\Omega_{\alpha}$-objects 
in $\PSpt^\Omega(\mathcal C)$, 
which is equivalent to \eqref{eq:LtnOmega(t)}.
Thus (2) and (3) are both equivalent to \eqref{eq:LtnOmega(t)} and imply (1).
To conclude the proof, 
we show that (1) implies (2). 
For $F\in \PSpt_{\Omega}(\mathcal C)$ the terms of $\Lrep_{\Omega,\alpha} F$ are given by
$E_{\alpha}\Omega_{\alpha}^\infty L^\Omega_{\alpha}[l] F$.
Here $[l]$ denotes the shift functor on $\PSpt^\Omega(C)$.
By (1), 
the terms of $\Lrep_{\alpha} L_{\Omega} F$ are given by 
$E_{\alpha}\Omega_{\alpha}^\infty L^\Omega_{\alpha}[l] F 
\simeq  
E_{\alpha} L_{\alpha}\Omega^\infty [l] F=\Lrep_{\alpha} \Omega^\infty [l] F$.
This implies part (2).
\end{proof}

\subsection{Properties of the endofunctor $\Omega_{\Gm}\Lrep_{\A^{1}}$.}
In what follows, 
we study the endofunctor $\Omega_{\Gm}^l\Lrep_{\A^{1}}$ on $\cSptsfr(\Sm_{B,z})$ for a base scheme $B$ 
and a point $z\in B$.
Our first goal is to show that $\Omega_{\Gm}$ preserves Nisnevich local equivalences in the sense that there 
is a Nisnevich local stable weak equivalence $\Omega_{\Gm}F\to \Omega_{\Gm}\Lrep_\nis F$ for all 
$F\in\cSptsfrA(\Sm_{B,z})$ when $B=z=\Spec k$.

 
\begin{lemma}
\label{lm:OmegaGmsimeqOmegaGmLrepnis}
Let $k$ be a field.
For every $F\in \cSptsfrA(\Sm_k)$ and open subscheme $V$ in $\A^{1}_k$
there is a Nisnevich local stable equivalence of presheaves of $S^1$-spectra 
\begin{equation}\label{eq:OmegaGm(Lnis)}
F^V\simeq_{S^1,\nis} (\Lrep_\nis F)^V.
\end{equation}
\end{lemma}
\begin{proof}
Since $F$ is $\A^{1}$-invariant, 
the strict homotopy invariance statement in \Cref{sect:AssumptionsBase}(1) shows that $\Lrep_\nis(F)$ is 
$\A^{1}$-invariant too. 
We note that $\overline F=\cofib(F\to \Lrep_\nis(F)$ and $\overline{F}^V$ are $\A^{1}$-invariant. 
If $X\in \Sm_k$, $x\in X$, 
we let $\eta$ be the generic point of the local scheme $\Xlocx$ of $X$ at $x$, 
and let $\iota$ denote the generic point of $V\times\eta$. 
Using \cite[Theorem 3.15(1,3)]{hty-inv} we conclude that, 
for all $l\in \mathbb Z$, 
there are injections of abelian groups
\begin{equation}
\label{eq:F(VXlocxotVetaotiota)}
\pi_l \overline{F}^V(X^\mathrm{loc}_{x}) 
\to 
\pi_l \overline{F}^V(\eta)\to \pi_l \overline{F}(\iota).
\end{equation}
Since $F(\iota)\simeq L_{\nis} F(\iota)$, the target group in \eqref{eq:F(VXlocxotVetaotiota)} is trivial.
Consequently, 
$\pi_l \overline{F}^V(X^\mathrm{loc}_{x})=0$ and hence $\overline{F}^V\simeq 0$.
Our claim follows by appealing to the equivalence
\[
\overline{F}^V
\simeq  
\cofib (F^V\to (\Lrep_\nis F)^V ).
\]
\end{proof}

\begin{lemma}
\label{lm:PregpfrkOmegaLrepNisexSHI}
Let $k$ be a field and $l\in \mathbb Z$.
As an endofunctor of $\cSptsfr(k)$ and $\cPrefr(k)^\gp$, 
the composite $\Omega_{\Gm}^l\Lrep_{\A^{1}}$ preserves Nisnevich local equivalences.
\end{lemma}
\begin{proof}
Suppose that $F_0\to F_1$ is a Nisnevich local equivalence in $\cSptsfr(k)$. 
For its fiber, 
$F\defeq \fib (F_0\to F_1)\simeq_{\nis} 0$, 
we show $\Omega_{\Gm}^l\Lrep_{\A^{1}}F\simeq_\nis 0$, or equivalently that 
$\Omega_{\Gm}^l\Lrep_{\A^{1}} F_0\simeq_\nis \Omega_{\Gm}^l\Lrep_{\A^{1}}F_1$.
By strict homotopy invariance over $k$, 
the endofunctor $\Lrep_\nis$ on $\cSptsfr(\Sm_k)$ preserves $\A^{1}$-invariant objects.
By \Cref{lm:LreptLprepcpreserveexact}, 
$\Lrep_{\A^{1}}$ preserves Nisnevich local equivalences;
thus, 
we obtain
\begin{equation}
\label{eq:LA1nisexact}
\Lrep_{\A^{1}}F
\simeq_{\nis,S^1} 0.
\end{equation}
Thus, 
by \Cref{lm:OmegaGmsimeqOmegaGmLrepnis} and \eqref{eq:LA1nisexact}, 
we deduce the equivalences
\[
\Omega_{\Gm}^l\Lrep_{\A^{1}} F
\simeq  
\Omega_{\Gm}^l\Lrep_\nis\Lrep_{\A^{1}} F
\simeq 
0.
\]

Since any additive presheaf on $\Corr^\fr(k)$ have a natural schemewise structure of an $\mathbb{E}_\infty$-monoid by \cite[]{five-authors},
delooping group-like simplicial spaces \cite{SegalammaSp} determines a functor 
$E\colon \cPrefr(\Sm_k)^{\gp}\to \cSptsfr(\Sm_k)$ such that 
$\Omega^\infty_{s} E(F)\simeq  F$ for every $F\in \cPrefr(\Sm_k)^{\gp}$.

So our claim for $\cPrefr(k)^{\gp}$ follows from the case of $\cSptsfr(k)$
because the functor $E$ preserves Nisnevich local equivalences.
\end{proof}

\begin{lemma}
\label{lm:OmegaGmLrepA1(simeq_nis)ShtfScz}
For a base scheme $B$ and point $z\in B$,
the endofunctor $\Omega_{\Gm}^l\Lrep_{\A^{1}}$ on $\cPrefrtf(\Sm_{B,z})^{\gp}$ preserves Nisnevich local equivalences.
\end{lemma}
\begin{proof}
Since the $\tf$-topology is trivial on $\Sm_{B,z}$, 
see \Cref{lmSmBcsigmatftopologytrivial}, 
there is an equivalence
\[
\cPrefrtf(\Sm_{B,z})^{\gp}\simeq\cPrefr(\Sm_{B,z})^{\gp}.
\]

Parts (1) and (2) of \Cref{th:LocA1tfnisstructuresDeformation} and part (4) of \Cref{th:LocA1tfnisstructuresDeformation} 
show that, 
for all $F\in\cPrefr(\Sm_{B,z})^{\gp}$, 
there are natural equivalences 
\begin{equation}
\label{eq:OmegaGmLA1overarrowi}
\Omega_{\Gm}^l\Lrep_{\A^{1}} F
\simeq 
\overarrow{i}_*\overarrow{i}^*( \Omega_{\Gm}^l\Lrep_{\A^{1}} F ) 
\simeq 
\overarrow{i}_*(\Omega_{\Gm})^l\Lrep_{\A^{1}} (\overarrow{i}^*F)).
\end{equation}

Suppose that $F\to G$ is a Nisnevich local equivalence in $\cPrefr(\Sm_{B,z})^{\gp}$.
\Cref{th:LocA1tfnisstructuresDeformation}(3) yields $\overarrow{i}^*F\simeq_{\nis}\overarrow{i}^*G$; 
\Cref{lm:PregpfrkOmegaLrepNisexSHI} yields 
$\Omega_{\Gm}^l\Lrep_{\A^{1}} (\overarrow{i}^*F)\simeq_{\nis} \Omega_{\Gm}^l\Lrep_{\A^{1}} (\overarrow{i}^*G)$, 
and hence
\[
\overarrow{i}_*(\Omega_{\Gm}^l\Lrep_{\A^{1}} (\overarrow{i}^*F))
\simeq_{\nis}
\overarrow{i}_*(\Omega_{\Gm}^l\Lrep_{\A^{1}} (\overarrow{i}^*G)).
\]
Thus the claim follows from \eqref{eq:OmegaGmLA1overarrowi}. 
\end{proof}

\subsection{Motivic infinite loop spaces in $\cPrefrtf(B)$.}

In this subsection, 
we compute the unit $\OmegaSigmatffrAnis$ of the suspension-loop adjunction 
\[
\Sigma^\infty_{}\colon \cPrefrtf(B) 
\leftrightarrows 
\SptstfrANis(B) 
\colon \Omega^\infty_{}
\]
in terms of the unstable motivic localization endofunctor $\Lrep^{\fr,\tf}_{\A^{1}}$ on $\cPrefrtf(B)$
(and similarly in the stable setting of $S^1$-spectra).
Denote by $\OmegaSigmatffrAtf$ the unit of the adjunction 
\[
\Sigma^\infty_{\A^{1},\tf}
\colon 
\cPrefrtf(B) 
\leftrightarrows 
\SptstfrAtf(B)
\colon 
\Omega^\infty_{\A^{1},\tf}.
\]

\begin{prop}
\label{prop:Pretffr(SmSz)SGA1nis}
There is an equivalence 
\[
\OmegaSigmatffrAnis 
\simeq  
\Lrep_{\nis}\OmegaSigmatffrAtf
\] 
in the $\infty$-category of endofunctors on $\cPrefrtf(\Sm_{B,z})$ and $\cSptsfrtf(\Sm_{B,z})$.
\end{prop}

\begin{proof}
By \Cref{lm:LocsubsetStab} there is an equivalence
\[
\SptstfrAtf(\Sm_{B,z})
\simeq 
\catPre_{\A^{1},\tf}^\fr(\Sm_{B,z})[(\Omega_{\Gm}\Lrep_{\A^{1}})^{-1}].
\]
Using Nisnevich localization, 
we deduce the equivalence
\[
\SptstfrANis(\Sm_{B,z})
\simeq 
\catPre_{\A^{1},\nis}^\fr(\Sm_{B,z})[(\Omega_{s}\Omega_{t}\Lrep_{\A^{1}})^{-1}].
\]
To conclude, 
we apply \Cref{lm:ReformulationOmega} and invoke \Cref{lm:OmegaGmLrepA1(simeq_nis)ShtfScz}, 
which shows that $\Omega_{s}\Omega_{t}\Lrep_{\A^{1}}$ preserves Nisnevich local equivalences 
on $\cPrefrtf(\Sm_{B,z})$.
The stable case of $S^1$-spectra is similar.
\end{proof}

\begin{lemma}
\label{lm:CommuteSSz}
For any base scheme $B$ and $z\in B$, 
the base change functor 
\begin{equation}
\label{eq:cPrefrSmBSmBcz}
\cPrefr(\Sm_{B})\to \cPrefr(\Sm_{B_z})
\end{equation} 
commutes with the endofunctors
$\Lrep_{\A^{1}}$, 
$\Lrep_{\tf}$, 
$\Lrep_{\nis}$, 
$\OmegaSigmatffrAtf$,
and $\OmegaSigmatffrAnis$
on $\cPrefr(-)$. 
The same holds for the case of $S^1$-spectra.
\end{lemma}

\begin{proof}
If $V\in \Sm_{B}$, 
then since the base change functor \eqref{eq:cPrefrSmBSmBcz} commutes with the endofunctor 
$F\mapsto F^V$ on $\cPrefr(-)$,
it commutes with $\Lrep_{\A^{1}}$ and $\Omega_{\Gm}$ as well.
Since the functor 
$
-\times_B B_z\colon \Sm_B\to\Sm_{B_z}
$
preserves $\tf$-squares and Nisnevich squares, 
the functor \eqref{eq:cPrefrSmBSmBcz} preserves Nisnevich and $\tf$-local equivalences.
Moreover, 
since any $\tf$-square or Nisnevich square over $B_z$ comes from a $\tf$- or Nisnevich square 
along the functor $-\times_B B_z$, 
we conclude that \eqref{eq:cPrefrSmBSmBcz} preserves $\tf$-local and Nisnevich local objects.
Hence the functor \eqref{eq:cPrefrSmBSmBcz} commutes with
$\Lrep_{\tf}$ and $\Lrep_{\nis}$.
The claim for $\OmegaSigmatffrAtf$
follows from the cases of 
$\Lrep_{\A^{1}}$,
$\Lrep_{\tf}$, 
and $\Lrep_{\nis}$.
\end{proof}

\begin{prop}
\label{prop:Pretffr(B)SGA1nis}
There are equivalences of endofunctors
\[
\OmegaSigmatffrAnis 
\simeq  
\Lrep_{\nis}\OmegaSigmatffrAtf 
\]
on the $\infty$-categories $\cPrefrtf(B)$ and $\cSptsfrtf(B)$. 
\end{prop}
\begin{proof}
The proof proceeds by induction on $\dim B$. 
The base of induction is $B=\emptyset$. 
Assume the claim holds for all schemes of dimension less than $\dim B$.

Owing to \Cref{lm:CommuteSSz}, and the conservativity of the base change functor
\[\cPrefr(\Sm_{B})\to \prod_{z\in B}\cPrefr(\Sm_{B_z}),\]
we may assume that $B$ is a local scheme with closed point $z\in B$.
\Cref{prop:Pretffr(SmSz)SGA1nis} shows that the equivalence holds on $\cPrefrtf(\Sm_{B,z})$.
By the inductive hypothesis, 
the equivalence holds for $\cPrefrtf(\Sm_{B-z})$.
Parts (2) and (3) of \Cref{th:LocA1tfnisstructuresShLocsquare} imply the equivalences
\[
\tilde i_*\tilde i^!\OmegaSigmatffrAnis
\simeq 
\tilde i_*\tilde i^!\Lrep_{\nis}\OmegaSigmatffrAtf,
\,\,
j_*j^*\OmegaSigmatffrAnis
\simeq  
j_*j^*\Lrep_{\nis}\OmegaSigmatffrAtf
\]
in $\catSpt^{\Gm,\fr}_\tf(\Sm_{B})$.
The claim follows now from part (1) of \Cref{th:LocA1tfnisstructuresShLocsquare}.
The case of $S^1$-spectra is similar.
\end{proof}

We complete this section by giving explicit formulas for the functors 
$\OmegaSigmatffrAnis$ and $\OmegaSigmatffrAtf$.

\begin{prop}
\label{prop:OmegaSigmafrtf_SGA1}
For every $F\in \cPrefrtf(B)$ there are natural equivalences
\begin{align*}
&
\OmegaSigmatffrAtf(F)\simeq  
\colim\Omega^l_{\Gm}\Lrep^{\fr,\tf}_{\A^{1}}(\Sigma^l_{\Gm}F^{\gp}),\\[5pt]
&\OmegaSigmatffrAnis(F)\simeq 
\Lrep^\tf_\nis\colim\Omega^l_{\Gm}\Lrep^{\fr,\tf}_{\A^{1}}(\Sigma^l_{\Gm}F^{\gp}).
\end{align*}
The analogous formulas hold in the stable case of $S^1$-spectra.
\end{prop}

\begin{proof}
Since the endofunctors $-\times \Gm$ and $-\times\A^1$ on $\Sm_B$ commute and the subcategory 
$\cPrefr_{\A^{1},\tf}(B)$ of $\cPrefrtf(B)$ is closed under colimits,
taking ${\Gm}$-loops $\Omega^l_{\Gm}$ yields an endofunctor on the former category.
Our second claim follows by \Cref{prop:Pretffr(B)SGA1nis}.
The case of $S^1$-spectra is similar.
\end{proof}

\begin{theorem}
\label{th:LrepsmotfrPSpttf}
For all $F\in \cPSptsfrtf(B)$ and $G\in \cPSptstfrtf(B)$, 
there are natural equivalences
\begin{align*}
\catLrep^{s,\tf,\fr}_\nissmot F 
& \simeq   
\phantom{\catLrep^\fr_{\text{nis}}}\catLrep^\fr_{\Gm}\catLrep^{\fr,\tf}_{\A^{1}}F^\gp, \\ 
\catLrep^{\tf,\fr}_\nissmot G  
& \simeq   
\catLrep^\fr_\nis \catLrep^\fr_{\Gm}\catLrep^{\fr,\tf}_{\A^{1}}G^\gp. 
\end{align*}
\end{theorem}

\begin{proof}
This follows by the second equivalence in \Cref{prop:OmegaSigmafrtf_SGA1} because the terms of 
$\catLrep^{\tf,\fr}_\nissmot F$ are given by  
$\OmegaSigmatffrAtf \Sigma_{\Gm}^l \Sigma^{n}_{s}F$.
\end{proof}

\section{Stable motivic localization over one-dimensional base schemes}
\label{sect:ReusltFormulations}

Following the preparatory results in \Cref{sect:stablemotiviclocalization}, 
we are ready to formulate our main results.
Throughout \Cref{section:stablemotloc,subsection:smlmfp1ss,subsection:motivicinfiniteloopspaces}, 
we let $B$ be a one-dimensional base scheme.

\subsection{The stable motivic localization functor}
\label{section:stablemotloc}

Consider the adjunction of $\infty$-categories of $(s,t)$-biprespectra
\[
\gamma^*\colon 
\cPSptst(B)=\cPSptst(\Sm_B)
\rightleftarrows 
\cPSptst(\Corr^\fr(B))=\cPSptstfr(B)\colon \gamma_*
\] 
The localization endofunctors $\Lrep_\nissmot$ and $\Lrep^\fr_\nissmot$ with respect to stable motivic equivalences 
are given by the composite functors
\begin{align*}
\Lrep_\nissmot&\colon 
\cPSptst(B)\to 
\cSptstAnis(B)\to\cPSptst(B),
\\
\Lrep^\fr_\nissmot&\colon
\cPSptstfr(B)\to
\cSptstfrAnis(B)\to\cPSptstfr(B).
\end{align*}
We write $\Lrep^{s}_\nissmot$ (resp.~$\Lrep^{s,\fr}_\nissmot$) for the analogous endofunctor of
$\cPSpts(B)$ (resp.~$\cPSptsfr(B)$).

\begin{theorem}
\label{th:LrepsmotfrPSpt}
For every $F\in \cPSptsfr(B)$ there are natural equivalences
\begin{align*}
\catLrep^{s,\fr}_\nissmot F&\simeq  \catLrep^\fr_\nis \catLrep^\fr_{\A^1,\tf}F^\gp,\\[5pt] 
\catLrep^\fr_\nissmot F&\simeq \catLrep^\fr_\nis \catLrep^\fr_{\Gm}\catLrep^\fr_{\A^1,\tf}F^\gp.
\end{align*}
The analogous equivalences hold for every $F\in \cPSptstfr(B)$.
\end{theorem}
\begin{proof}
Consider the canonical embedding functor $\nu_*\colon \cPrefrtf(B)\to \cPrefr(B)$.
By appeal to the canonical equivalence $\nu_*\Lrep^{\fr,\tf}_{\A^{1}}\simeq\Lrep^\fr_{\A^{1},\tf} \nu_*$,
the result is a reformulation of \Cref{th:LrepsmotfrPSpttf}.
\end{proof}

\begin{corollary}
\label{th:LrepsmotPSpt}
For every $F\in \cPSptst(B)$ there is a natural equivalence 
\[\begin{array}{lcl}
\Lrep_\nissmot F &\simeq & \Lrep_\nis \Lrep_{\Gm}\Lrep_{\A^1,\tf} (\gamma_*\gamma^* F)^\gp.
\end{array}\]
In particular, 
for every $X\in \Sm_B$, 
we have equivalences 
\[
\begin{array}{lcl}
\Lrep_\nissmot \Sigma^\infty_{s}X_+ &\simeq & 
\{\Lrep_\nis (\Lrep_{\A^1,\tf} \Fr(-, X_{+}\wedge S^n))^{\gp}\}_{n\geq 0},\\
\Lrep_\nissmot \Sigma^\infty_{s,t}X_+ &\simeq & 
\{\Lrep_\nis (\Lrep_{\Gm}\Lrep_{\A^1,\tf} \Fr(-, X_{+}\wedge S^n\wedge \Gm^{\wedge l}))^{\gp}\}_{n,l\geq 0}.
\end{array}\]
\end{corollary}
\begin{proof}
We consider biprespectra and note that, 
as in \Cref{subsect:LrepA1notdef}, 
the endofunctor $\catLrep_{\A^{1}}$ on the $\infty$-category $\cPSptst(B)$ is equivalent to the colimit of the 
functors $\catLrep_{\A^{1}}^{[l]}=(\catLrep_{\A^{1}}^{[1]})^l$, 
where $\catLrep_{\A^{1}}^{[1]} F = F(\Delta^\bullet_B\times -)$.  
Similarly, 
see \Cref{lm:SptS(C)}, 
the endofunctor $\catLrep_{\Gm}$ on $\cPSptst(B)$ is given by
\begin{equation}
\label{eq:Gmstab}
F\mapsto \colim\Omega_{\Gm}^l F(l).
\end{equation}
Here $F(1)$ denotes the $\Gm$-shift of the $(s,t)$-bispectrum $F$. 
With these descriptions, 
we deduce the natural equivalences
\[
\catLrep_{\A^{1}}\gamma_*\simeq  \gamma_*\catLrep^\fr_{\A^{1}}, \quad\catLrep_{\Gm}\gamma_*\simeq  
\gamma_*\catLrep^\fr_{\Gm}.
\]
Next, 
by \Cref{th:restrfrLoctfniscommute}, 
there are natural equivalences 
\[
\catLrep_\nis\gamma_*
\simeq
\gamma_*\catLrep^\fr_\nis,\quad \catLrep_\tf\gamma_*
\simeq \gamma_*\catLrep_\tf.
\]
\Cref{th:LrepsmotfrPSpt} finishes the proof of the first claim.
To conclude, we use the motivic equivalences 
$\gamma_*\gamma^*(\Sigma_{s}^\infty X_+)\simeq 
\{\catLrep_{\A^{1}}\Fr(-, X_{+}\wedge S^n)^{\gp}\}_{n\geq 0}$,
$\gamma_*\gamma^*(\Sigma_{s,t}^\infty X_+)\simeq 
\{\catLrep_{\A^{1}}\Fr(-, X_{+}\wedge S^n\wedge \Gm^{\wedge n})^{\gp}\}_{n\geq 0}$, 
see, e.g., 
\cite[Corollary 2.3.25]{five-authors}.
The argument for $S^1$-prespectra is similar.
\end{proof}

Now we formulate our main computational application of \Cref{th:LrepsmotfrPSpt}.

\begin{corollary}\label{cor:piisomorphism}
There is a canonically induced commutative diagram
\[
\xymatrix{
\SptstfrAtf(B)
\ar[r] \ar[d]&
\SptstfrANis(B)
\ar[r] \ar[d]&
\SptstfrAtf(B) \ar[d]
\\
\Sptst(B)
\ar[r] & 
\SptstNis(B)
\ar[r] & 
\Sptst(B).
}
\]
Thus, 
for any $Y\in \Sm_B$ and essentially smooth local henselian scheme $U$, 
we have
\begin{equation}\label{eq:piSHAonenisBYUcongpiSHfrAonetfBYU}
\pi_{*,*}^{\SH_{\A^{1},\nis}(B)}(Y)(U)
\cong
\pi_{*,*}^{\SH^\fr_{\A^{1},\tf}(B)}(Y)(U).
\end{equation}
\end{corollary}
\begin{proof}
\Cref{th:LrepsmotfrPSpt} and the equivalence
$\gamma_*\catLrep^\fr_\nis\simeq\catLrep_\nis\gamma_*$ in \Cref{th:restrfrLoctfniscommute} 
imply the diagram commutes.
The commutativity of the outer square implies the isomorphism 
$$
\pi_{*,*}^{\SH^\fr_{\A^{1},\nis}(B)}(Y)(U)
\cong
\pi_{*,*}^{\SH^\fr_{\A^{1},\tf}(B)}(Y)(U).
$$
Now the isomorphism \eqref{eq:piSHAonenisBYUcongpiSHfrAonetfBYU} follows 
from the equivalence $\SH^\fr(B)\simeq \SH(B)$ in \cite{Hoyois-framed-loc}.
\end{proof}

Let $\Rep_{\nissmot}$ (resp.~$\Rep^{s}_{\nissmot}$) denote the fibrant replacement endofunctor 
on the category of 
$(s,t)$-bispectra (resp.~$S^1$-spectra) of presheaves on $\Sm_{B}$ equipped with the stable 
motivic model structure 
\cite[\S2.3]{Nordfjordeid}, \cite{Jardine-spt}.
We write 
$\Rep_{\tfsmot}$ for the stable $\tf$-motivic localization,
$\Rep_{\nis}$ for the levelwise Nisnevich local injective model structure, and 
$\Rep_{\Gm}$ for the endofunctor given by the formula \eqref{eq:Gmstab} on the category of $(s,t)$-bispectra.

\begin{theorem}
\label{th:Lsmotspecrtaofframedpresheaves}
For every $S^1$-spectrum or $(s,t)$-bispectrum $\calF$ of radditive quasi-stable framed presheaves,
there are canonical levelwise schemewise equivalences
\begin{equation}\label{eq:LnismotdecompostionLnisLtfsmotSptstB}\begin{array}{lcl}
\Rep^{s}_{\nissmot}(\calF)&\simeq& 
\Rep_\nis(\Rep_{\A^{1}, \tf}(\calF))^{\gp},\\
\Rep_{\nissmot}(\calF)&\simeq& 
\Rep_\nis(\Rep_{\Gm}\Rep_{\A^{1}, \tf}(\calF))^{\gp}.
\end{array}\end{equation}
In particular, for every $X\in \Sm_B$, we have 
\begin{equation}
\label{eq:smotNis(A1tf)FrSigmastX}
\begin{array}{lcl}
\Rep^{s}_{\nissmot}(\Sigma^\infty_{s}X_+)&\simeq& 
\Rep_\nis(\Rep_{ \A^{1}, \tf}\Fr(-,\Sigma^\infty_{s}X_+))^{\gp},\\
\Rep_{\nissmot}(\Sigma^\infty_{s,t}X_+)&\simeq& 
\Rep_\nis\Rep_{\Gm}(\Rep_{ \A^{1}, \tf}\Fr(-,\Sigma^\infty_{s,t}X_+))^{\gp}.
\end{array}
\end{equation}
\end{theorem}
\begin{proof}
Let $N(-)$ denote the nerve functor and set $\mathcal S=\SmAff_B$.
We consider the diagram of $\infty$-categories
\[
\xymatrix{
\cPrefrAtf(\mathcal S)\ar[r]\ar[d]^{\catLrep^\fr_{\A^1,\tf}} 
&\cPreAtf(\mathcal S)\ar[r]\ar[d]^{\catLrep_{\A^1,\tf}} 
& N(\Ho(\cPreAtf(\mathcal S)))\ar@{=}[r]\ar[d]^{\Lrep_{\A^1,\tf}} 
& N(\Ho(\SpcAtf(\mathcal S)))\ar[d]^{\Lrep_{\A^1,\tf}} \\
\cPrefr(\mathcal S)\ar[r]& \cPre(\mathcal S)\ar[r] 
& N(\Ho(\cPre(\mathcal S)))\ar@{=}[r]&N(\Ho(\Spc_s(\mathcal S))).
}
\]
The functor $\Fr_+(B)\to \Corr^\fr(B)$ 
induces for all $X\in \SmAff_B$ an $\A^1$-equivalence $\Fr_{B}(-,X)\to h^\fr(X)$ of 
simplicial presheaves on $\SmAff_B$ by \cite[Corollary 2.3.25]{five-authors}.
Thus the essential image of $\cPrefrAtf(\mathcal S)$ in the nerve of $\Ho(\Spc_s(\mathcal S))$ 
coincides with the $\A^1$-invariant $\tf$-local quasi-stable framed presheaves on $\mathcal S$.
Hence, for any $\calF\in \Ho(\Spc_s(\mathcal S))$ 
given by an $\A^1$-invariant $\tf$-local quasi-stable framed presheaf on $\mathcal S$, 
\Cref{th:LrepsmotfrPSpt} shows there is an equivalence
\[
\Lrep_{\nissmot}(\calF)\simeq  
\Lrep_\nis\Lrep_{\A^{1},\tf}(\calF)^{\gp},
\]
and for any $\calG\in \Ho(\Spt_\text{st}(\mathcal S))$ there is an equivalence
\[
\Lrep_{\nissmot}(\calG)
\simeq  
\Lrep_\nis\Lrep_{\Gm}\Lrep_{\A^{1},\tf}(\calG)^{\gp}.
\]
Consequently, 
the equivalences in \eqref{eq:LnismotdecompostionLnisLtfsmotSptstB} hold in the category $\SmAff_B$.
Since both sides of \eqref{eq:LnismotdecompostionLnisLtfsmotSptstB} are Nisnevich local, 
the equivalence extends to $\Sm_B$.

Finally, 
\eqref{eq:smotNis(A1tf)FrSigmastX} is equivalent to the second part of \Cref{th:LrepsmotPSpt}.
\end{proof}

\subsection{Stably motivic local models for $\PP^1$-suspension spectra}
\label{subsection:smlmfp1ss}
Let $\catSH(B)$ denote the $\infty$-category of motivic spectra, 
a.k.a., $\PP^1$-spectra of $\A^1$-invariant Nisnevich sheaves defined as the limit 
\begin{equation}
\label{eq:prolimPre}
\catSH(B)
=
\catSpt^{\PP^1}_{\A^1,\nis}(B)
=
\varprojlim( \cdots 
\xrightarrow{\Omega_{\PP^1}} \cHAnis(B) \cdots 
\xrightarrow{\Omega_{\PP^1}} \cHAnis(B) 
\xrightarrow{\Omega_{\PP^1}} \cHAnis(B))
\end{equation}
taken in the $\infty$-category of $\infty$-categories.
Here $\cHAnis(B)$ is the $\infty$-category of $\A^1$-invariant Nisnevich sheaves on $\Sm_B$. 
Since the endofunctors $\Omega_{\PP^1}$ and $\Omega_{S^1\wedge \Gm}$ on $\cHAnis(B)$ are equivalent,
which can be seen using the standard open covering of the projective line, 
there is an equivalence of $\infty$-categories 
\begin{equation}
\label{equation:P1stequivalence}
\cSptPAnis(B)\simeq \cSptswtAnis(B).
\end{equation}
Here $\cSptP(B)$ and $\cSptswt(B)$ denote the $\infty$-categories of $\PP^1$-spectra and 
$S^1\wedge \Gm$-spectra in the sense of \Cref{def:catSpt} with respect to the endofunctors 
$\Omega_{\PP^1}$ and $\Omega_{S^1\wedge \Gm}$ on $\cPre(B)$.
We write $\cSptPAnis(B)$ and $\cSptswtAnis(B)$ for the subcategories of $\A^1$-invariant 
Nisnevich local objects.

\begin{theorem}
\label{th:projlimSHSuspens}
The $\PP^1$-suspension functor $\Sigma^{\infty}_{\PP^1}\colon\Sm_B^\op\to \catSH(B)$
admits a description in terms of framed motives given by 
\begin{equation} 
\label{eq:CorSigmaYFrSigmaY22}
X_{+}
\mapsto
\{\catLrep_{\nis} \catLrep_{\Gm} (\catLrep_{\A^1,\tf} 
\Fr(-\times\Delta^{\bullet}_B, X_{+}\wedge S^n\wedge \Gm^{\wedge n}))^{\gp}\}_{n\geq 0}.
\end{equation}
If $n>0$, 
the $\A^1$-invariant Nisnevich sheaf
$\catLrep_{\nis} \catLrep_{\Gm} \catLrep_{\A^1,\tf}\Fr(-, X_{+}\wedge S^n\wedge \Gm^{\wedge n})$ is group-like.
\end{theorem}
\begin{proof}
When forming \eqref{eq:CorSigmaYFrSigmaY22} we use the naturally induced maps
\begin{equation*}
\catLrep_{\nis}\catLrep_{\Gm} \catLrep_{\A^1,\tf}\Fr(-, X_{+}\wedge S^n\wedge \Gm^{\wedge n}) 
\to
\Omega_{\PP^1}(\catLrep_{\nis}\catLrep_{\Gm} \catLrep_{\A^1,\tf}\Fr(-, X_{+}\wedge 
S^{n+1}\wedge \Gm^{\wedge n+1})).
\end{equation*}
Moreover, 
we appeal to the equivalence \eqref{equation:P1stequivalence}.
The unit of the adjunction 
$$
\cPSptswt(B) \rightleftarrows \cSptswtAnis(B)
$$ 
yields for every $X\in \Sm_B$ a canonical morphism
\begin{equation}
\label{eq:SusptoLSwGAN}
\Sigma^\infty_{\swt}X_+\to 
\catLrep_{S^1\wedge \Gm} \catLrep_{\A^1,\nis}(\Sigma^\infty_{\swt}X_+).
\end{equation}
The functor $\cPSptst(B)\to \cPSptswt(B)$ taking an $(s,t)$-biprespectrum to its diagonal $S^1\wedge \Gm$-prespectrum
preserves stable motivic equivalences and local objects.
By \Cref{th:LrepsmotPSpt} it follows that \eqref{eq:SusptoLSwGAN} is equivalent to the natural morphism 
\[
\Sigma^\infty_{\swt} X_+\to 
\catLrep_\nis\catLrep_{\Gm} \catLrep_{\A^1,\tf} \{\Fr(-\times\Delta^{\bullet}_{k}, 
X_{+}\wedge S^n\wedge \Gm^{\wedge n})\}^\gp_{n\geq 0}.
\]
\end{proof}

\subsection{Motivic infinite loop spaces} 
\label{subsection:motivicinfiniteloopspaces}

Let $\Shv_*(\Sm_B)$ denote the category of pointed simplicial Nisnevich sheaves on $\Sm_B$. 
Every $Y\in \Sm_B$ admits an associated $\PP^1$-suspension spectrum that is the Nisnevich 
sheafification of the presheaf spectrum 
$\Sigma^\infty_{\PP^1} Y_+$.
Let $\Omega^\infty_{\mathrm{mot}}\Sigma^\infty_{\mathrm{mot}} Y\in \Shv_*(\Sm_B)$ denote the 
zeroth term of the fibrant replacement of the $\PP^1$-suspension spectrum of $Y$ with respect 
to the stable motivic model structure.

\begin{theorem}
\label{th:Lacalizationfunctordecomposition}
In $\Shv_*(\Sm_B)$, 
there is a Nisnevich local equivalence 
\begin{equation}
\label{eq:OmegastTGSAOmegastTTinftyLmot(Y)2}
\Omega^\infty_{\mathrm{mot}}\Sigma^\infty_{\mathrm{mot}} Y
\xrightarrow{\simeq} 
\varinjlim\limits_{l}\Omega^l_{\Gm^{\wedge 1}}(\Lrep_{\A^{1},\tf}\Fr(Y_+\wedge \Gm^{\wedge l}))^{\gp}.
\end{equation}
Here $\Fr(Y_+\wedge \Gm^{\wedge l})=\Fr(-,Y_+\wedge \Gm^{\wedge l})$ is the presheaf of 
framed correspondences in \cite{Framed}.
\end{theorem}
\begin{proof}
Here $\Omega^\infty_{\mathrm{mot}}\Sigma^\infty_{\mathrm{mot}} Y$ is 
equivalent to the zeroth term of the $(s,t)$-bispectrum $\Rep_{\nissmot}(\Sigma^\infty_{s,t}Y_+)$.
Thus the claim follows from \Cref{th:Lsmotspecrtaofframedpresheaves} because the right hand side of 
\eqref{eq:OmegastTGSAOmegastTTinftyLmot(Y)2} is Nisnevich local equivalent to the zeroth term of 
the right hand side of \eqref{eq:smotNis(A1tf)FrSigmastX}.
\end{proof}

\subsection{Shifting the homotopy $t$-structure}
\label{sect:shifthomtstr}

In this subsection, 
$B$ is an arbitrary base scheme unless otherwise specified.
We write 
\[
\PSptt\catSpt^s(B), \quad
\PSptt\catSpt^{s,\fr}(B)
\]
for the $\infty$-categories of $\Gm$-prespectra of presheaves of $S^1$-spectra on $\Sm_B$, $\Corrfr_B$,
respectively.
We will make use of various $\tf$-motivic and Nisnevich motivic localization functors on $\infty$-categories
\begin{equation}\label{eq:SptsPresstfr_AtfAnis}\begin{array}{lclclcl}
\cSpts(B)&\to& \cSHsAtf(B)&\to& \cSHsAnis(B),&&\\
\PSptt\catSpt^s(B)&\to& \cSHstAtf(B)&\to& \cSHstAnis(B),&&\\ 
\cSptsfr(B)&\to& \cSHsfrAtf(B)&\to& \cSHsfrAnis(B),&&\\
\PSptt\catSpt^{s,\fr}(B)&\to& \cSHstfrAtf(B)&\to& \cSHstfrAnis(B)&\simeq& \cSHstAnis(B).
\end{array}\end{equation}

\begin{definition}
\label{definition:schemwiseconnective}
An object $\calF\in \cSpts(B)$ or $\calF\in \cSptsfr(B)$ is schemewise $n$-connective if 
its stable homotopy presheaf $\pi_i(\calF)$ is trivial for $i<n$.
Likewise, 
$\calF\in \PSptt\catSpt^s(B)$ or $\calF\in \PSptt\catSpt^{s,\fr}(B)$ is schemewise 
$n$-connective if it is a $\Gm$-prespectrum of schemewise $n$-connective $S^1$-spectra.
\end{definition}

Let $\mathcal C(B)$ denote one of the following $\infty$-categories
\[
\cSpts(B),\; 
\cSptsfr(B),\; 
\PSptt\catSpt^s(B),\; 
\PSptt\catSpt^{s,\fr}(B).
\]

\begin{lemma}
\label{lm:S1LA1(tfloc;>0)}
The endofunctor $\catLrep_{\A^{1}}$ on $\mathcal C(B)$ preserves $\tf$-sheaves and 
schemewise $n$-connective objects.
\end{lemma}

\begin{proof} 
The $S^1$-stable $\A^1$-localization admits the description
\begin{equation}
\label{eq:LA1S1OmegaS1ShiftS1}
\begin{array}{lcl}
\catLrep_{\A^{1}}\calF
&\simeq & 
\colim_{l} \catLrep^{(l)}_{\A^{1}} \calF\\ &\simeq &  
\colim_{l} (\calF[l]^{(\Delta^1_B/\partial\Delta^1_B)^{\wedge l}}), 
\end{array}
\end{equation}
where 
\[
\catLrep^{(1)}_{\A^{1}} \calF(-)=\colim\left(\calF(\Delta^1_B\times -)\rightrightarrows \calF(-)\right).
\]
Here the maps are induced by the $0$- and $1$-sections;
see \cite[\S 4.2]{Morel-connectivity}, \cite[Proposition 3.19]{ConnBase} for details on
\eqref{eq:LA1S1OmegaS1ShiftS1}.
Hence $\catLrep^{\mathcal C}_{\A^{1}}$ preserves $\tf$-sheaves because sequential colimits 
preserves $\tf$-sheaves, 
and likewise for the endofunctors on $\mathcal C(B)$ given by 
$\calF\mapsto \Omega_{\Delta^1_B/\partial\Delta^1_B}^l\calF$ and $\calF\mapsto \calF[l]$.

To prove the claim for schemewise $n$-connective objects, 
we use the equivalence
\[
\catLrep^{\mathcal C}_{\A^{1}}\calF
\simeq 
\colim_{[l]\in \Delta^\mathrm{op}} \calF^{\Delta^l_B}. 
\]
Here, 
the endofunctor $\calF\mapsto \calF^{\Delta^l_B}=\calF(-\times_B \Delta^l_B)$ 
preserves schemewise $n$-connective objects in $\mathcal C(B)$, 
and likewise for the colimit.
\end{proof}

\begin{lemma}
\label{lm:conndim}
If $B$ has finite Krull dimension $d$ and $\calF\in \mathcal C(B)$ is schemewise connective, 
then $\catLrep_\tf\calF$ is schemewise $(-d)$-connective.
\end{lemma}
\begin{proof}
By \Cref{prop:cdstructurecompleteregularbounded} and \cite[Theorem 2.26]{VV:cd} the 
$\tf$-cohomological dimension of any scheme $X\in \Sch_B$ is $\leq d$.
Thus the claim follows as in \cite[\S 4.3]{Morel-connectivity} --- 
with a reference to \Cref{th:restrfrLoctfniscommute} in the case of $\cSptsfr(B)$.
\end{proof}

\begin{prop}
\label{prop:S1stabLrepA1tf_cor:shifttstrA1tf}
On $\mathcal C(B)$, 
there is a canonical equivalence of endofunctors 
\[
\catLrep_{\A^{1},\tf} 
\simeq  
\catLrep_{\A^{1}}\catLrep_\tf.
\]
If $B$ has finite Krull dimension $d$ and $\calF$ is schemewise connective, 
then $\catLrep_{\A^{1},\tf}\calF$ is schemewise $(-d)$-connective.
\end{prop}
\begin{proof}
For $\calF\in\mathcal C$, 
the first statement in \Cref{lm:S1LA1(tfloc;>0)} shows  
$\catLrep_{\A^{1}}\catLrep_{\tf}\calF$ is $\A^{1}$-local and $\tf$-local.
Hence the canonical morphism $\catLrep_{\A^{1}}\catLrep_\tf \calF\to\catLrep_{\A^{1},\tf}\calF$ is an equivalence. 
The second part follows since $\catLrep_\tf\calF$ is schemewise $(-d)$-connective, 
see \Cref{lm:conndim}, 
and $\catLrep_{\A^{1}}$ preserves schemewise $(-d)$-connective objects, 
see \Cref{lm:S1LA1(tfloc;>0)}.
\end{proof}

\Cref{th:LrepsmotfrPSpt} and \Cref{prop:S1stabLrepA1tf_cor:shifttstrA1tf} have the following consequences.

\begin{theorem}
\label{theorem:connectivitysmotsheaves2}
\par(1) For any base scheme $B$ of Krull dimension $d$, 
the leftmost functors in \eqref{eq:SptsPresstfr_AtfAnis}
take schemewise connective objects to schemewise $(-d)$-connective objects.
\par(2) 
The rightmost functors in \eqref{eq:SptsPresstfr_AtfAnis} preserve Nisnevich locally connective objects.
\par(3) 
The composite functors in \eqref{eq:SptsPresstfr_AtfAnis}
take Nisnevich locally connective objects to $(-d)$-Nisnevich locally connective ones.
\end{theorem}
\begin{proof}
Part (1) follows from \Cref{prop:S1stabLrepA1tf_cor:shifttstrA1tf}.
To prove (2), 
\Cref{th:LrepsmotfrPSpt} and \Cref{th:LrepsmotPSpt} show there are natural equivalences
\[
\begin{array}{lcl}
\Lrep^{\fr,s}_\nissmot& \simeq & \Lrep_\nis \Lrep_{\Gm}\Lrep_{\A^1,\tf},\\
\Lrep_\nissmot  &\simeq & \Lrep_\nis \Lrep_{\Gm}\Lrep_{\A^1,\tf} \gamma^*. 
\end{array}
\]
Here $\gamma^*\colon \PSpt(B)\to \PSpt^\fr(B)$ is the canonical functor and we consider the localization functors 
\[
\begin{array}{lclclcl}
\Lrep^{\fr,s}_\nissmot&\colon& \cPSptsfr(B)&\to& \cSHsfrAnis(B)&\to&\cPSptsfr(B), \\
\Lrep_\nissmot&\colon& \cPSptst(B)&\to& \cSHAnis(B)&\to&\cPSptst(B).
\end{array}
\]
Part (2) follows because 
for any $\calF\in \cPSpts(B)$ 
and essentially smooth local henselian scheme $U$,
the morphism of presheaves
\begin{equation}
\calF(U)
\to
\Lrep_{\nis}\calF(U)
\end{equation}
is an equivalence.
Consequently, 
the functor $\Lrep_\nis$ is 
preserves 
Nisnevich locally connective objects.

Part (3) follows now by (1) and (2).
\end{proof}

\appendix

\section{Essentially smooth schemes and henselian pairs}
\label{section:essahp}

Let $X_\alpha$ be a directed inverse system of schemes with affine transition maps $X_\alpha\to X_\beta$.
The cofiltered limit $\varprojlim_{\alpha} X_\alpha$ is a scheme \cite[Tag 01YX]{StacksProject}.
If $U\to X$ is an open immersion and each $X_\alpha$ is an $X$-scheme, 
we set $U_\alpha=X_\alpha\times_X U$.
Then the canonically induced morphism $\varprojlim_{\alpha} U_\alpha\to\varprojlim_{\alpha} X_\alpha$ is also 
an open immersion.

\begin{definition}
\label{def:EssentiallySmooth}
An essentially smooth $B$-scheme $X$ is a cofiltered limit 
\[
X
=
\varprojlim_{\alpha} X_\alpha
\]
of smooth $B$-schemes with affine and dominant transition maps.
\end{definition}

\begin{remark}\label{rem:productonEssSm}
The monoidal product on $B$-schemes restricts to the subcategory $\EssSm_B$ because 
affine morphisms and dominant morphisms of smooth $B$-schemes are closed under products by 
\cite[Tags 01SC, 01SD]{StacksProject} 
and 
\cite[Tag 0H3F]{StacksProject}, 
respectively.
\end{remark}

\begin{lemma}
\label{lm:EssSmContNistfcov}
The canonical functor $\SH_s(\Sm_B)\to \SH_s(\EssSm_B)$ preserves Nisnevich local and $\tf$-local objects, 
and it preserves Nisnevich local equivalences and $\tf$-local equivalences.
\end{lemma}
\begin{proof}
Suppose $X\in \EssSm_B$ and $Y$ is a closed subscheme of $X$.
Then $X\cong \varprojlim_\alpha X_\alpha$ for a filtering system of the schemes $X_\alpha\in \Sm_B$
with affine transition maps, 
and $Y\cong \varprojlim_\alpha Y_\alpha$, where $Y_\alpha$ is the closure of $Y$ in $X_\alpha$.
If $\calF\in \SH_s(\Sm_B)$ is Nisnevich local, 
then for the continuation of $\calF$ to essentially smooth schemes we have  
\begin{align*}
\calF^\mathrm{cont}(X)
:=
\hocolim_\alpha \calF(X_\alpha) 
&\cong  
\hocolim_\alpha\hofib( \calF((X_\alpha)^h_{Y_\alpha})\vee \calF^\mathrm{cont}(X_\alpha-Y_\alpha)\\
& \to \calF((X_\alpha)^h_{Y_\alpha}-Y_\alpha) )\\
&\cong  
\hofib(\hocolim_\alpha\calF((X_\alpha)^h_{Y_\alpha})\vee \hocolim_\alpha\calF^\mathrm{cont}(X_\alpha-Y_\alpha)\\
&\to \hocolim_\alpha\calF((X_\alpha)^h_{Y_\alpha}-Y_\alpha)),
\end{align*}
or equivalently
$$
\calF^\mathrm{cont}(X)\cong \hofib(\calF^\mathrm{cont}(X^h_Y)\vee \calF^\mathrm{cont}(X-Y) \to \calF^\mathrm{cont}(X^h_Y-Y)).
$$
This shows that $\calF^\mathrm{cont}\in\SH_s(\EssSm_B)$ is Nisnevich local.
A similar argument applies in the $\tf$-topology since $X_Y$ is isomorphic to $\varprojlim_\alpha (X_\alpha\times_B Y)$.
The claim for the local equivalences follows since a point in the Nisnevich topology (resp.~$\tf$-topology) is an 
essentially smooth local henselian scheme $X^h_x$, $x\in X$ (resp.~$X^h_\sigma$, $\sigma\in B$, by 
\Cref{prop:Proptftop}(vii)).
\end{proof}

For further reference, 
we record two base change results for simplicial presheaves. 
We refer to \cite[Tags 00X1, 00YK, 05V1]{StacksProject} for our standard terminology on sites.

\begin{lemma}
\label{lm:LocEqCovPoints}
Suppose $(\mathcal S^\prime,{\tau^\prime})\to (\mathcal S,\tau)$ is a morphism of sites given by a 
continuous morphism $f\colon \mathcal S\to \mathcal S^\prime$.
Then the base change functor $f_*\colon s\Pre(\mathcal S^\prime)\to s\Pre(\mathcal S)$ preserves local objects.
If $(\mathcal S,\tau)$, $(\mathcal S^\prime,{\tau^\prime})$ have enough points,  
then $f_*$ preserves local equivalences.
\end{lemma}

\begin{proof}
The first claim follows by \cite[Corollary 5.24]{Jardine-local}
because $f_*$ is a right Quillen adjoint for the local injective model structure.
Indeed, 
if $\calF\in s\Pre(\mathcal S^\prime)$ is $\tau^\prime$-local and $\widetilde X\to X$ is a 
$\tau$-covering in $\mathcal S$, 
then for the $\check{\text{C}}$ech construction $\check C_{f(\widetilde X)}(f(X),\calF)$, 
see \eqref{eq:Check(widetilde X)}, 
there are isomorphism  
\[
\check C_{\widetilde X}(X,f_*\calF)
\cong   
\check C_{f(\widetilde X)}(f(X),\calF)
\cong  
\calF(f(X))
\cong  
f_*\calF(X).
\]
For preservation of local equivalences we note the canonically induced morphism $f_*\to \mathbf{R}f_*$ 
is an equivalence, 
see also \cite[Proposition 1.27, p.105]{Morel-Voevodsky}.
If $\calF\to \calG$ is a local equivalence in $s\Pre(\mathcal S^\prime)$,
then for any $\tau^\prime$-point $V^\prime$, 
there is an equivalence $\calF(V^\prime)\simeq \calG(V^\prime)$.
Since $f$ is continuous, it preserves points, 
hence for any $\tau$-point $V$,
we have $f_*\calF(V)\simeq f_*\calG(V)$.
Then by assumption $f_*\calF\to f_*\calG$ is a local equivalence.
\end{proof}

\begin{lemma}
\label{lm:schemespointssites}
Let $(\mathcal C,\tau)\to(\mathcal C^\prime,\tau^\prime)$ be a morphism of sites with underlying functor 
$f\colon \mathcal C^\prime\to \mathcal C$, 
where $\mathcal C$, $\mathcal C^\prime$ are subcategories of $\Sch_B$.
Suppose $(\mathcal C,\tau)$, $(\mathcal C^\prime,\tau^\prime)$ have enough points given by objects in 
$\mathcal C$ and $\mathcal C^\prime$, 
respectively. 
Suppose also that for any $\tau^\prime$-point $U$ in $\mathcal C^\prime$,
$f(U)$ is a $\tau$-point in $\mathcal C$.
Then $f$ is a continuous morphism of sites, 
and on simplicial presheaves, 
$f_*$ takes $\tau^\prime$-local equivalences to $\tau$-local equivalences.
\end{lemma}

\begin{proof}
Let $\calF\to \mathcal G$ be a $\tau$-local equivalence between simplicial presheaves on $\mathcal C$.
Then $\calF(V)\simeq \mathcal G(V)$ for each $\tau$-point $V$ in $\mathcal C$.
For any $\tau^\prime$-point $U$ in $\mathcal C^\prime$,
$f(U)$ is a $\tau$-point by assumption, 
so that $f_*\calF(U)=\mathcal F(f(U))\simeq \mathcal G(f(U))=f_*\mathcal G(U)$.
Thus $f_*\calF\to f_*\mathcal G$ is a $\tau^\prime$-local equivalence.
\end{proof}

Recall from \cite[Tag 09XD]{StacksProject} the notion of an affine henselian pair.
Next we recall the notion of henselization of pairs used in \Cref{section:SmScZ},
see \cite[Tag 0A02, 0EM7]{StacksProject}.

\begin{lemma}
\label{lm:AffHens}
The henselization of an affine scheme $X$ along a closed subscheme $Y$ is the cofiltered limit 
\[
X^h_Y
\cong
\varprojlim_{\alpha} X_\alpha
\]
indexed by affine {\'e}tale morphisms $X_\alpha\to X$ that admit a lifting $Y\to X_\alpha$. 
The henselization $X^h_Y$ is an affine scheme. 
If $X$ is noetherian, then so is $X^h_Y$.
\end{lemma}

More generally, 
for a closed immersion $Y\not\hookrightarrow X$ of $B$-schemes, 
we form the cofiltered category of affine {\'e}tale morphisms $X_\alpha\to X$ equipped with liftings 
\begin{equation}
\label{equation:henselnbds}
\begin{tikzcd}
& X_\alpha\ar{d}\\
Y\ar[ur,dotted]\ar[r,hook] & X.
\end{tikzcd}
\end{equation}
Using the diagrams \eqref{equation:henselnbds} we form the cofiltered limit and $B$-scheme
\begin{equation}\label{eq:proafhenselizationXhY}
X^h_Y
\defeq
\varprojlim X_\alpha.
\end{equation}
The canonically induced closed immersion $Y\not\hookrightarrow X^h_Y$ lifts the morphism 
$Y\not\hookrightarrow X$ in $\Sch$.

\begin{remark}\label{rem:etaleneighborhoods}
    Note that 
\begin{equation}\label{eq:proafhenselizationXhYfinalsufystem}
    X^h_Y
    \defeq
    \varprojlim X_\alpha, 
    \end{equation}
    where the limit is taken over the cofiltered system of \'etale morphisms $X^\prime\to X$
    such that $X_\alpha\times_B Y\cong Y$, 
    because the latter family is cofinal in the family \eqref{equation:henselnbds}:
    the functor $X_\alpha\mapsto X_\alpha-(X_\alpha\times_X Y - Y)$ turns each \'etale neighborhood \eqref{equation:henselnbds} to one such that $X_\alpha\times_B Y\cong Y$.
\end{remark}
\vspace{0.05in}

In what follows, 
a pro-{\'e}tale morphism refers to a cofiltered limit 
\[
\varphi
\colon 
\widetilde X
= \varprojlim_\alpha X_\alpha
\to 
X
\]
obtained from affine {\'e}tale morphisms $\varphi_\alpha\colon X_\alpha\to X$ and 
affine {\'e}tale transition maps $X_\alpha\to X_\beta$. 
We note that the morphism $\varphi$ is affine according to 
\cite[\href{https://stacks.math.columbia.edu/tag/01YX}{Tag 01YX}]{StacksProject}.

Pro-{\'e}tale schemes over non-affine schemes are only used in our discussion of the category $\Sm_{B,Z}$, 
see \Cref{section:defsmZsmB,section:ReductionSmat}.
In the proof of our results, 
however, 
it suffices to consider henselizations of affine schemes because we reduce to this case.
We refer to \Cref{subsection:candn} and the discussion below \eqref{equation:categories} for more details.

\begin{lemma}
\label{lm:pro-etale:lift}
Suppose that $\varphi\colon \widetilde X\to X$ is a pro-{\'e}tale morphism and let $i\colon Y\not\hookrightarrow X$ 
be a closed immersion that admits a lifting $\widetilde i\colon Y\not\hookrightarrow\widetilde X$. 
There is a canonically induced isomorphism $\widetilde X^h_Y\simeq X^h_Y$. 
Consequently, $i$ lifts to a morphism $X^h_Y\to \widetilde X$.
\end{lemma}

\begin{proof}
Write $X^h_Y\cong\varprojlim X_\alpha$, $f_\alpha\colon X_\alpha\to X$, and $s_\alpha\colon Y\to X_\alpha$ 
following \eqref{equation:henselnbds}.
Similarly, 
we have $\widetilde X^h_Y\cong\varprojlim \widetilde X_\alpha$, 
$\widetilde f_\alpha\colon \widetilde X_\alpha\to X$, 
and $\widetilde s_\alpha\colon Y\to \widetilde X_\alpha$.
The morphism in question is induced by 
\begin{equation}
\label{eq:embeding_widetildeXalpha-s-into-Xalpha-s}
(\widetilde f_\alpha\colon 
\widetilde X_\alpha\to \widetilde X, \widetilde s_\alpha\colon Y\to \widetilde X_\alpha)
\mapsto 
(\varphi \circ f_\alpha\colon \widetilde X_\alpha\to X, s_\alpha\colon Y\to \widetilde X_\alpha).
\end{equation} 
To construct an inverse, 
we note that any \'etale neighborhood of the form 
\[
(f_\alpha\colon X_\alpha \to X, s_\alpha\colon Y\to X_\alpha),
\]
yields an \'etale neighborhood
\[
(\widetilde f_\alpha\colon \widetilde X_\alpha \defeq \widetilde X\times_X X_\alpha \to \widetilde X, 
\widetilde s_\alpha \defeq (\widetilde i,s_\alpha) \colon Y\to \widetilde X_\alpha).
\]
\end{proof}

\begin{corollary}
Suppose $X\in \Sm_B$ and $i\colon Y\not\hookrightarrow X$ is a closed immersion. 
Then $X^h_Y$ is essentially smooth over $B$. 
\end{corollary}
\begin{proof}
Let $\widetilde X$ be the union of the connected components of $X$ that have a nonempty intersection with $Y$.
\Cref{lm:pro-etale:lift} shows that the open immersion $\widetilde X\to X$ induces an isomorphism 
$\widetilde X^h_Y\cong X^h_Y$. 
For any \'etale neighborhood $X_\alpha$ of $Y$ in $\widetilde X$, 
the morphism $X\to \widetilde X$ is dominant. 
Hence $\widetilde X^h_Y$ is an essentially smooth scheme.
\end{proof}

\begin{corollary}
\label{cor:pro-etale:henseiso}
Suppose $i\colon Y\not\hookrightarrow X$ and $i^\prime\colon Y\not\hookrightarrow X^\prime$ are closed immersions such that 
$X\cong X^h_Y$, $X^\prime\cong (X^\prime)^h_Y$.
If $\varphi\colon X^\prime\to X$ is a pro-{\'e}tale morphism and $\varphi\circ i^\prime=i$, 
then $\varphi$ is an isomorphism.
\end{corollary}
\begin{proof}
The claim follows since $X^\prime \cong (X^\prime)^\prime_Y\cong X^h_Y$, 
where the last isomorphism follows from \Cref{lm:pro-etale:lift}.
\end{proof}

\begin{corollary}
\label{cor:henselianpairsplitting}
Let $Y\not\hookrightarrow X$ be a closed immersion such that $X\cong X^h_Y$.
If $Y=Y_1\amalg Y_2$, then $X\cong X_1\amalg X_2$, where $X_1=X^h_{Y_1}$ and $X_2=X^h_{Y_2}$. 
\end{corollary}
\begin{proof}
The pro-\'etale morphism $t\colon X^h_{Y_1}\amalg X^h_{Y_2}\to X^h_Y\cong X$ restricts to the 
decomposition $Y_1\amalg Y_2$.
Hence, 
by \Cref{cor:pro-etale:henseiso}, 
it follows that $t$ is an isomorphism.
\end{proof}

The following statement is a variation of results on the lifting property for affine henselian schemes 
with respect to smooth affine morphisms,
see \cite[Theorem I.8]{Gruson1972} and \cite{Elkik1973SolutionsD}.

\begin{lemma}
\label{lm:AffSmHenselianLift}
Suppose $B$ is affine and let $X\in \SmAff_B$, $U\in \AffSch_B$.
If $T\not\hookrightarrow U$ is a closed immersion, 
there is a naturally induced surjection 
\[
\SmAff_B(U^h_T,X)\to\SmAff_B(T, X).
\]
\end{lemma}
\begin{proof}
Since $X$ is smooth affine, 
there exists a vector bundle $\xi$ on $X$ such that $\xi\oplus T_X$ is trivial. 
Note that $X$ is a retract of the total space $\widetilde X$ of $\xi$ and the tangent bundle of the 
$X$-scheme $\widetilde X$ is trivial.
Hence we may assume $X$ has a trivial tangent bundle.

Choose a closed immersion $X\not\hookrightarrow \A^n_B$ for some $n\gg 0$. 
Since the tangent bundle of $X$ is trivial,
the associated normal bundle $N_{X/\A^n_B}$ of $X$ in $\A^n_B$ is stably trivial. 
By increasing $n$ we may assume $N_{X/\A^n_B}$ is trivial.
Hence there exist polynomials $f_1,\dots, f_{l}\in \mathcal O(\A^n_B)$ such that the vanishing locus 
$Z(f_1,\dots,f_l)$ splits as $X\amalg X^\prime$ for some $X^\prime$.
We use the same notation for the pullbacks of the $f_i$'s in $\mathcal O(\A^n_U)$ and $\mathcal O(\A^n_{T})$.

The graph of any scheme morphism $T\to X$ yields a closed immersion 
$\Gamma\not\hookrightarrow X\times_B T\subset \A^n_{T}$ 
such that the canonical projection induces an isomorphism $\Gamma\simeq T$.
Since the cotangent bundle of $X$ is trivial there exist polynomials $f_{l+1},\dots, f_{n}\in \mathcal O(\A^n_U)$ 
such that the differentials of the $f_i$'s, $i=l,\dots,n$, 
furnish a trivialization of the conormal bundle of $\Gamma$ in $X\times_B T$. 
We choose liftings of $f_{l+1},\dots,f_n$ to regular functions $\widetilde f_{l+1},\dots,\widetilde f_n\in \mathcal O(\A^n_{U^h_T})$
and use the trivialization of the conormal bundle of $\Gamma$ in $\A^n_T$ to conclude the vanishing locus 
$$
V
=
Z(f_1,\dots,f_l,\widetilde f_{l+1},\dots,\widetilde f_n)\subset (X\times_B U)\subset \A^n_U
$$ 
is {\'e}tale over $\Gamma$.
Now, 
since $T\not\hookrightarrow U^h_T$ is a henselian pair, 
\Cref{lm:pro-etale:lift} shows that the morphism $T\simeq \Gamma\to V$ lifts to $U^h_T\to V$. 
The desired lifting $U^h_T\to X$ of $T\to X$ is given by the composite 
\[
U^h_T \to V\to X\times_B U \to X.
\]
\end{proof}

\section{Framed correspondences}
\label{sectionApp:FrCor}

In this appendix, 
we collect definitions and results from the theory of framed correspondences which is used 
in the main body of the paper.

\begin{definition}
\label{def:FramedCorr}
Let $X$ and $Y$ be schemes over a base scheme $B$.
A framed correspondence of level $n$ from $X$ to $Y$ is a triple $c=(S,\phi,g)$, 
where $S\not\hookrightarrow \A^n_X$ is a closed immersion, 
$V=(\A^n_X)^h_S$ is the scheme defined in \eqref{eq:proafhenselizationXhY}, 
$\phi\colon V\to \A^n_B$ and $g\colon V\to Y$ are 
morphisms such that $S\cong V\times_{\phi,\A^n_B,i} (0\times B)$.
Here $i\colon 0\times B\to \A^n_B$ is the $0$-section.
We write $\Fr_n(X,Y)$ for the set of framed correspondences of level $n$ from $X$ to $Y$, 
and refer to $\Supp(c)\defeq S$ as the support of $c$.

If the projection $\A^n_X\to X$ induces an isomorphism $Z\cong X$,
we say that $c$ is a framing of the morphism 
$$X\cong S\xrightarrow{g\big|_S} Y.$$
\end{definition}

\begin{remark}
In \Cref{def:FramedCorr} we use the scheme $(\A^n_X)^h_S$ instead of \'etale neighborhoods of $B$ as in 
\cite[Definition 2.1]{Framed}.
These two definitions of framed correspondences agree if $X$ is affine. 
Consequently, 
the notions of strict $\A^{1}$-invariance in the sense of \Cref{def:SHIthnistriv,def:SHIthnistf} 
are equivalent for both types of framed correspondences.
\end{remark}

\begin{definition} 
\label{example:quasistable}
If $X\in \Sch_{B}$ we set $\sigma_X:=(0\times X, \mathrm{id}_{\A^{1}_X},\mathrm{pr}\colon \A^{1}_X\to X)\in\Fr_1(X,X)$.
\end{definition}

\begin{definition}
\label{def:FrCorrComposition}
The composition of two framed correspondences 
$(S, \phi, g)\in \Fr_n(X,X^\prime)$ and $(S^\prime, \phi^\prime, g^\prime)\in\Fr_m(X^\prime,X^{\prime\prime})$
is the level $m+n$ framed correspondence from $X$ to $X^{\prime\prime}$ 
given by $(S^{\prime\prime}, \phi^{\prime\prime}, g^{\prime\prime})$, 
where 
\begin{itemize}
\item[(1)] 
$S^{\prime\prime} =S\times_{X^\prime} S^\prime\subset \A^{n+m}_{X}
\cong 
\A^n_{X}\times_{g,X^\prime,\mathrm{pr}_{X^\prime}} \A^m_{X^\prime}$, 
$\mathrm{pr}_{X^\prime}\colon \A^m_{X^\prime}\to X^\prime$.
\item[(2)] 
$V^{\prime\prime}\to \A^{n+m}_{X}$ is the henselization of $\A^{n+m}_{X}$ at $S^{\prime\prime}$.
There is a canonical morphism $V^{\prime\prime}\to V$ (resp.~$V^{\prime\prime}\to V^\prime$)
to the henselization of $\A^n_X$ at $B$ (resp.~$\A^m_Y$ at $S^\prime$). 
\item[(3)]
$\phi^{\prime\prime}\colon V^{\prime\prime}\to \A^{n+m}_B$ is given by the composites
$V^{\prime\prime}\to V\xrightarrow{\phi} \A^n_B$, $V^{\prime\prime}\to V^\prime\xrightarrow{\phi^\prime} \A^m_B$.
\item[(4)]
$g^{\prime\prime}\colon V^{\prime\prime}\to X^{\prime\prime}$ is the composite of 
the morphism $V^{\prime\prime}\to V^\prime$, 
induced by the base change $(v^\prime)^*(g)$ of $g$ along $v^\prime\colon V^\prime\to X^\prime$ and 
$g^{\prime}\colon V^\prime\to X^{\prime\prime}$.
\end{itemize}
\end{definition}

If $Z\not\hookrightarrow B$ is a closed immersion, 
recall from \Cref{subsection:candn} that $X^h_Z$ denotes $X^h_{X\times_B Z}$.
We write $\Sch_{B,Z}$ for the full subcategory of $\Sch_B$ spanned by $X^h_Z$ for $X\in \Sch_B$.

\begin{lemma}
\label{lm:etneigh:productoverBcZsublm}
For any $X,Y\in\Sch_B$, 
the natural morphism of $B$-schemes
\[(X^h_Z\times_B Y)^h_Z\to (X\times_B Y)^h_Z\]
induced by
the natural morphism of $B$-schemes $X^h_Z\to X$
is an isomorphism.
\end{lemma}

\begin{proof}
Recall that $X^h_Z\cong\varlim_{X^\prime}X^\prime$, see \Cref{rem:etaleneighborhoods},
where $X^\prime$ runes over the cofiltered category of
\'etale neighborhoods $X^\prime\to X$ of $Z$ in $X$ such that $X^\prime\times_B Z\cong X\times_B Z$.
Given such an \'etale neighborhood $X^\prime\to X$, 
the morphism
$(X^\prime\times_B Y)\to(X\times_B Y)$ is \'etale and $(X^\prime\times_B Y)^h_Z\cong(X\times_B Y)^h_Z$.
By commuting limits we obtain
\[
(X^h_Z\times_B Y)^h_Z\cong(\varlim_{X^\prime}X^\prime\times_B Y)^h_Z\cong
\varlim_{X^\prime}(X^\prime\times_B Y)^h_Z\cong(X\times_B Y)^h_.
\]
\end{proof}

\begin{lemma}
\label{lm:productoverBcZ}
For any $X,Y\in\Sch_B$, 
the natural morphism of $B$-schemes
\[(X^h_Z\times_B Y^h_Z)^h_Z\to (X\times_B Y)^h_Z\]
induced by
the natural morphisms of $B$-schemes 
$X^h_Z\to X$ 
and
$Y^h_Z\to Y$
is an isomorphism.
\end{lemma}
\begin{proof}
By \Cref{lm:etneigh:productoverBcZsublm} we have
\[(X^h_Z\times_B Y^h_Z)^h_Z\cong 
(X^h_Z\times_B Y)^h_Z\cong 
(X\times_B Y)^h_Z.\]
\end{proof}

\begin{definition}
\label{def:productoverBcZ}
The monoidal product \index{Categories of $B$-schemes!$\times_{B,Z}$} $\Sch_{B,Z}\times\Sch_{B,Z}\to \Sch_{B,Z}$ is given by 
\[
(X^h_Z,Y^h_Z)\mapsto 
X^h_Z\times_{B,Z}Y^h_Z\stackrel{\text{def}}{=}(X^h_Z\times_B Y^h_Z)^h_Z.
\]
\end{definition}

Note that, by \Cref{lm:productoverBcZ}, there is a naturally induced isomorphism
\begin{equation}\label{eq:XhztimesBcZYhZcongXtimesBYhZ}
X^h_Z\times_{B,Z}Y^h_Z
\cong
(X\times_B Y)^h_Z.
\end{equation}
If $X\in\Sch_{B,Z}$ we write $\A^n_X$ for the fiber product $\A^n\times_{B,Z} X\in \Sch_{B,Z}$,
and note that it is isomorphic to $(\A^n\times_B X)^h_Z$.

\begin{definition}\label{def:fibredproductSmBcZ}
For morphisms
$X^h_Z\to S^h_Z$, $Y^h_Z\to S^h_Z$ in $\SmBcZ$,
we define $X^h_Z\times_{S^h_Z}Y^h_Z\in\SmBcZ$
as the image of $X^h_Z\times_{S^h_Z,S_Z}Y^h_Z\in\Sm_{S^h_Z,S_Z}$ 
under the functor $\Sm_{S^h_Z,S_Z}\to\Sm_{B,Z}$,  
where $S_Z=S^h_Z\times_B Z$.
\end{definition}
We have the isomorphism
\begin{equation}\label{eq:productXhZShZYhZ}X^h_Z\times_{S^h_Z}Y^h_Z\cong(X^h_Z\times_{S^h_Z} Y^h_Z)^h_Z.\end{equation}

\begin{lemma}\label{lm:prodSmAffSmatBcZ}
    The subcategories $\SmAffBcZ$ and $\SmatBcZ$ in $\SmBcZ$
    are closed with respect to the fibre products in $\SmBcZ$ from \Cref{def:fibredproductSmBcZ}.
\end{lemma}
\begin{proof}
    The claim follows by \eqref{eq:productXhZShZYhZ} 
    because 
    the subcategories $\SmAff_B$ and $\Smat_B$ in $\Sm_B$
    are closed with respect to the fibre product in $\Sm_B$.
\end{proof}
\vspace{0.1in}

Next we define framed correspondences in $\Sch_{B,Z}$ by tweaking \Cref{def:FramedCorr}.
\begin{definition}
\label{def:FrCorrSchScZ}
For $X,Y\in \Sch_{B,Z}$,
a framed correspondence of level $n$ from $X$ to $Y$ is a triple $(S,\phi,g)$, 
where $S\not\hookrightarrow (\A^n_X)^h_Z$ is a closed subscheme, 
$V=((\A^n_X)^h_Z)^h_S \cong (\A^n_X)^h_S$ is given by \eqref{eq:proafhenselizationXhY},
$\phi\colon V\to \A^n_B$ and $g\colon V\to Y$ are morphisms such that $S\cong V\times_{\phi,\A^n_B,i} (0\times B)$.
Here $i\colon 0\times B\to \A^n_B$ is the $0$-section.
\end{definition}

Framed correspondences in $\Sch_{B,Z}$ admit a composition given as in \Cref{def:FrCorrComposition}. 
We write $\Fr_+(\Sch_B)$ for the category with objects $X\in\Sch_{B}$ and morphisms given by
\[
\Fr_+(X,Y)\defeq \bigvee_{n}\Fr_n(X,Y).
\] 
The smooth $B$-schemes span a full subcategory $\Fr_+(B)=\Fr_+(\Sm_{B})$ of $\Fr_+(\Sch_B)$.
Similarly, 
the subcategory $\Sm_{B,Z}\subset\Sch_{B,Z}$ spanned by $X^h_Z$, $X\in \Sm_{B}$, 
gives rise to a full subcategory $\Fr_+(B,Z)$ of $\Fr_+(\Sch_{B,Z})$.

\begin{lemma}
\label{lm:HensIdempotent}
If $Z\not\hookrightarrow W\not\hookrightarrow X$ are closed immersions then $X^h_Z \cong (X^h_W)^h_Z$.
If $Z, W\not\hookrightarrow X$ are closed immersions such that the natural map $Z^h_{Z\cap W}\to Z$ 
is an isomorphism, 
then 
$$
X^h_Z \cong (X^h_W)^h_{Z^h_{Z\cap W}} \cong (X^h_W)^h_Z.
$$  
The latter isomorphism uses the closed immersion $Z\cong Z^h_{Z\cap W}\not\hookrightarrow X^h_W$.
\end{lemma}
\begin{proof}
The first claim follows from the universal property of the cofiltered limit in \eqref{eq:proafhenselizationXhY}.
The second follows from the first because $Z\cong  Z^h_{Z\cap W}$ so that 
$(X^h_W)^h_{Z\cap W}\cong (X^h_W)^h_Z$,
and
$X^h_{Z\cap W} \cong X^h_Z$,
and consequently we obtain 
\[
X^h_Z \cong 
X^h_{Z\cap W} \cong 
(X^h_W)^h_{Z\cap W}\cong 
(X^h_W)^h_Z.
\]
\end{proof}

\begin{lemma}
\label{rm:FrCorSchScZ}
Given a closed immersion $Z\not\hookrightarrow B$, and $X,Y\in \Sch_B$, 
there is an isomorphism of sets
\[\Fr^{\Sch_{B,Z}}_n(X^h_Z, Y^h_Z)\cong \Fr^{\Sch_{B}}_n(X^h_Z,Y^h_Z)\]
where the left-hand side is as in \Cref{def:FrCorrSchScZ} for 
$X^h_Z,Y^h_Z\in \Sch_{B,Z}$, and
the right-hand side is as in \Cref{def:FramedCorr} for $X^h_Z,Y^h_Z\in \Sch_{B}$.
\end{lemma}
\begin{proof}
For any $X\in\Sch_B$, by \Cref{lm:etneigh:productoverBcZsublm} we have
\[(\A^n_{X^h_Z})^h_Z=(\A^n_{X^h_{X\times Z}})^h_{\A^n_{X\times_B Z}}\cong(\A^n_{X})^h_{\A^n_{X\times_B Z}}\cong(\A^n_X)^h_Z.\]
Given a closed immersion $S\not\hookrightarrow \A^n_{X^h_Z}$, 
$S$ is finite over $X^h_Z$, and consequently $S^h_Z\cong S$.
So there is the closed immersion 
\[S\not\hookrightarrow (\A^n_{X^h_Z})^h_Z\cong (\A^n_X)^h_Z.\]
Applying \Cref{lm:HensIdempotent} to the immersions 
$S,\A^1_{X\times_B Z}\hookrightarrow \A^n_{X^h_Z}$
shows there is an isomorphism 
\[((\A^n_X)^h_Z)^h_S\cong (\A^n_{X^h_Z})^h_S.\]
\end{proof}

\begin{lemma}
\label{lm:HensBasechangeFrCor}
If $Z\not\hookrightarrow B$ is a closed immersion, 
then the functors 
\begin{equation}
\label{eq:Sch(B)Sch(BcZ)Sch(B-Z)}
\begin{array}{lll}
\Sch_B\to \Sch_B&\colon& X\mapsto X^h_Z, X\mapsto X\times_B(B-Z),\\
\Sch_B\to \Sch_{B,Z}&\colon& X\mapsto X^h_Z,\\
\Sch_B\to \Sch_{B-Z}&\colon& X\mapsto X\times_B(B-Z).\\
\end{array}
\end{equation}
extend to framed correspondences, 
see \Cref{def:FramedCorr}, 
so there are induced functors
\[\begin{array}{lll}
\Fr_+(\Sch_B)\to \Fr_+(\Sch_B)&\colon& X\mapsto X^h_Z, X\mapsto X\times_B(B-Z),\\
\Fr_+(\Sch_B)\to \Fr_+(\Sch_{B,Z})&\colon& X\mapsto X^h_Z,\\
\Fr_+(\Sch_B)\to \Fr_+(\Sch_{B-Z})&\colon& X\mapsto X\times_B(B-Z).\\
\end{array}\]
\end{lemma}
\begin{proof}
The data of an explicit framed correspondence over $B$ is encoded by diagrams of the form
\begin{equation}
\label{diagram:frameddefinition}.    
\xymatrix{
\A^n_X\ar[rdd]& V\ar[l]\ar[dr]^{(\varphi,g)}\\
& S\ar[u]\ar[lu]\ar[d]\ar[dr] & \A^n_Y\\
&X &Y\ar[u]_{\{0\}}.
}
\end{equation}
It remains to note that the functors in \eqref{eq:Sch(B)Sch(BcZ)Sch(B-Z)} 
preserve diagrams of the form \eqref{diagram:frameddefinition}.
\end{proof}

\begin{lemma}
\label{lm:SmccisubsetSchcci}
There is a commutative triangle 
\[\xymatrix{
\Smat_B\ar[dr]\ar[rr] && \Schcci_B\ar[dl]\\
& \Sch_B & 
}
\]
The fully faithful functors $\Smat_B\to\Sch_B$ and $\Schcci_B\to\Sch_B$ appear in
\Cref{convnotations:schemescategories}. That is, 
for every $X\in \Smat_B$ there exists a closed immersion $X\not\hookrightarrow \A^n_B$ and 
regular functions $e_1,\dots,e_n\in \calO(\A^n_B)$ such that $Z(e_1,\dots,e_n)=X\amalg\wh X$.
\end{lemma}
\begin{proof}
By assumption $X\in \SmAff_B$ and the relative tangent bundle of $X$ over $B$ is stably trivial. 
Hence there is a closed immersion $X\not\hookrightarrow\A^{n_1}_X$ and 
$T_{X/B}\oplus \mathcal O^{\oplus l}_X\cong \mathcal O^{\oplus n_2}_X$.
By increasing $n_1$ and $n_2$ we may assume $n=n_1=n_2$.
We claim there exists a trivialization of the normal bundle $N_{X/\A^n_B}$.
In effect, 
choose regular functions $e_1,\dots,e_l$ on $\A^n_B$ such that $e_i\big|_X=0$, $i=1,\dots,l$. 
Then the differential of $(e_1,\dots ,e_l)$ yields the desired trivialization of $N_{X/\A^n_B}$.
\end{proof}

\begin{lemma}
\label{lm:cciequationsciequations}
For every $X\in \Smat_B$ or $X\in \Schcci_B$ there exists a closed immersion $X\not\hookrightarrow \A^n_B$ 
and regular functions $e_1,\dots, e_l\in \calO(\A^n_B)$ such that $Z(e_1,\dots,e_l)=X$.
\end{lemma}
\begin{proof}
Owing to \Cref{lm:SmccisubsetSchcci} we may assume $X\in \Schcci_B$ is a clopen subscheme of the vanishing locus 
$Z(e^\prime_1,\dots,e^\prime_{l^\prime})$ for some regular functions $e^\prime_i\in \calO(\A^{n^\prime}_B)$ such that 
$\dim_B Z(e_1^\prime,\dots,e_{l^\prime}^\prime)=n-l$.
It follow that $Z(e^\prime_1,\dots,e^\prime_n)=X\amalg\wh X$.
Let $s$ be a regular function on $\A^{n^\prime}_B$ such that $s\big|_X=0$, $s\big|_{\wh X}=1$, 
and set $n=n^\prime+1$, $l=l^\prime+1$. 
Define $e_1,\dots,e_{l^\prime}\in \calO_{\A^n_B}(\A^n_B)$ as
the pullbacks of $e^\prime_{i}$ along the projection $\A^n_B\to \A^{n^\prime}_B$.
Next we set $e_l:=x_n-\tilde{s}\in\calO_{\A^n_B}(\A^n_B)$, 
where $x_n$ is the coordinate on $\A^n$ such that $Z(x_n)=\A^{n^\prime}_B$ and $\tilde{s}$ is the pullback of $s$.  
Then the subscheme $X\times 0$ of $\A^n_B$ agrees with the vanishing locus $Z(e_1,\dots,e_n)$.
\end{proof}

\begin{lemma}
\label{lm:liftFr}
 Suppose $B$ is affine and let $X\in \SmAff_B$, $U\in \AffSch_B$.
If $Y\not\hookrightarrow U$ is a closed immersion, 
there is a naturally induced surjection 
\[
\Fr_+(U^h_Y,X)\to\Fr_+(Y, X).
\]
\end{lemma}
\begin{proof}
Let $c=(S,V,\phi,g)\in \mathrm{Fr}_n(Y,X)$ for some $S\subset \A^n_{Y}$ and $\phi_i\in \mathcal O(V)$.
Let $T$ be the first-order thickening of $Z(\phi)\subset \A^n_{Y}$,
see \cite[Tag 04EW]{StacksProject}, 
with restrictions $\phi_i\big|_T\in \mathcal O(T)$ for $1\leq i\leq n$.
Let $\overline \phi_i\in \mathcal O(\A^n_{U})$ be a lifting of $\phi_i\big|_T$ along the closed immersion of 
affine schemes $T\not\hookrightarrow \A^n_{U}$, 
and set $\widetilde S=Z(\overline \phi)$.
By shrinking $V$ we may assume there is a closed immersion $V\not\hookrightarrow\A^{l}_{\A^n\times Y}$, 
for some $l$,
such that the image of $B$ equals $0_S=0\times S$.
Moreover,
since $V$ is \'etale over $\A^n\times Y\cong \A^n_{Y}$, 
\Cref{lm:cciequationsciequations} lets us assume $V=Z(e_1,\dots,e_l)$ for regular functions 
$e_j\in \calO(\A^l_{\A^n\times Y})$, $1\leq j\leq l$.
Now choose a lifting $\widetilde e_j\in \mathcal O(\A^l_{\A^n\times U})$ of $e_j$ such that 
$e_j\big|_{O\times\widetilde S}=0$ and set $\widetilde V=Z(\widetilde e_1,\dots, \widetilde e_l)$;
this is a subscheme of $\A^l_{\A^n\times U}$. 
The first order thickening $\widetilde T$ of $\widetilde S$ in $\A^n_U$ and the closed immersions
\[
V\not\hookrightarrow 
\A^l_{\A^n\times Y}
\not\hookrightarrow 
\A^l_{\A^n\times U}, 
\widetilde S
\not\hookrightarrow
\A^n_U\cong 0\times \A^n_U
\not\hookrightarrow
\A^l\times \A^n_U
\cong 
\A^l_{\A^n\times U}
\]
yield $V\amalg_{T}\widetilde T\not\hookrightarrow \A^l_{\A^n\times U}$.
Let $\widetilde\phi_i\in \mathcal O(\widetilde V)$ be a lift of the regular function 
$\phi_i\amalg \overline{\phi_i}\in \mathcal O(V\amalg_T \widetilde T)$, $1\leq i\leq n$.
Note that $\widetilde S$ is finite over $U$;
hence $S\hookrightarrow \widetilde S$ is a henselian pair since the same holds for $Y\hookrightarrow U$.
Thus the henselization of $B$ in $\A^n_{Y}$ equals the henselization of $\widetilde S$ in $\A^n_U$.
Since $X$ is smooth over $B$, 
\Cref{lm:AffSmHenselianLift} shows that there exists a lifting $\widetilde g\colon (\A^n_{U})^h_{\widetilde S}\to X$ 
of $(\A^n_{Y})^h_R\xrightarrow{g} X$. 
The framed correspondence $\widetilde c= (\widetilde S,\widetilde V, \widetilde \phi_i,\widetilde g)$ is the 
desired lifting of $c$.
\end{proof}

\begin{definition}
\label{def:overlineF}
For $c_0,c_1\in \Fr_n(X,Y)$, $X,Y\in \Sm_B$, 
we write $c_1\sim_{\A^{1}}c_2$ if there exists a framed correspondence $c\in\Fr_n(X\times\A^{1},Y)$ such that 
for the canonical sections $i_0,i_1\colon X\to X\times\A^{1}$ we have $c\circ i_0=c_0$ and $c\circ i_1=c_1$.
Let $\ovFr_n(X,Y)$ denote the quotient set of $\Fr_n$ with respect to the equivalence relation generated 
by $\sim_{\A^{1}}$.
\end{definition}

\begin{definition}
\label{def:ZFandoverlineZF}
For $Y\in \Sm_B$ let $\mathbb Z\Fr_n(-,Y)$ 
denote the presheaf of free abelian groups associated to $\Fr_n(-,Y)$.
Following the notations of \cite[Definition 2.11]{hty-inv}, 
we write $\ZF_n(-,Y)$ for the quotient of $\ZFr_n(-,Y)$ with respect to the relations
\[
\begin{array}{lll} 
c=c_1+c_2, &c=(S,\phi,g)\in \Fr_+(X,Y),&X\in \Sm_B,\\
S=S_1\amalg S_2 & c_1=(S_1,\phi,g)\in \Fr_+(X,Y), &\\
& c_2=(S_2,\phi,g)\in \Fr_+(X,Y). & 
\end{array}
\]
For $X\in\Sm_B$ and the canonical sections $i_0, i_1\colon X\to \A^{1}\times X$, 
we define $\overline \ZF_n(-,Y)$ by
\[
\overline \ZF_n(X,Y)
:=
\Coker(\ZF_n(X\times\A^{1},Y)\xrightarrow{i^*_0-i^*_1}\ZF_n(X,Y)).
\]
\end{definition}

\begin{definition}
\label{def:FrZFrZFovZF}
If $X,Y\in\Sm_B$ consider the pointed sets and abelian groups
\[\begin{array}{llllll}
\Fr_+(X, Y) &:=& \bigvee_{n\geq 0}\Fr_n(X, Y),&\quad \Fr(X, Y) &:=& \varinjlim_{n}\Fr_n(X,Y),\\
\ZFr_*(X, Y) &:=& \bigoplus_{n\geq 0}\ZFr_n(X, Y),&\quad \ZFr(X,Y)&:=&\varinjlim_{n}\ZFr_n(X,Y),\\
\ZF_*(X, Y) &:=& \bigoplus_{n\geq 0}\ZF_n(X, Y),&\quad \ZF(X,Y)&:=&\varinjlim_{n}\ZF_n(X,Y),\\
\ovZF_*(X, Y) &:=& \bigoplus_{n\geq 0}\ovZF_n(X, Y),&\quad \ovZF(X,Y)&:=&\varinjlim_{n}\ovZF_n(X,Y).
\end{array}\]
Precomposition with $\sigma_{Y}$ yields the transition maps in the colimits, 
see \Cref{example:quasistable}.
Let $\ZFr_*(B)$, $\ZF_*(B)$, and $\ovZF_*(B)$ denote the categories with objects smooth $B$-schemes and 
with morphisms as above.
Let $\Fr(Y)$, $\ZFr(Y)$, $\ZF(Y)$, $\ovZF(Y)$ denote the quasi-stable framed presheaves on $\Sm_B$ defined
sectionwise as above.

The above definitions extend to simplicial smooth schemes via left Kan extension 
along the Yoneda embedding.
\end{definition}

Any framed presheaf of $S^1$-spectra $\calF\in \Spt_s(\Fr_+(B))$ defines a functor 
$\calF\colon \Fr_+(B)\to \SH$, 
which we denote by the same symbol.
Since the stable homotopy category $\SH$ is additive, 
any radditive functor $\calF\colon \Fr_+(B)\to \SH$ induces a functor
\begin{equation}
\label{eq:ZFrpairSH}
\ZF_*(B)\to \SH.
\end{equation}

\begin{prop}
\label{prop:A1localhomotopycommutativesquare}
Suppose $\calF\colon \Fr_+(B)\to \SH$ is radditive and $\A^{1}$-local.
For $c_1, c_2\in \Fr_n(X,Y)$ assume there exists an element $c\in \ZF_n(X\times\A^{1}, Y)$ such that 
\[
c_1=c\circ i_0,c_2=c\circ i_1\in\ZF_n(X,Y)
\]
for the canonical sections $i_0,i_1\colon X\to X\times\A^{1}$.
Then we have the equality
\[
c_1^*
=
c_2^*
\colon 
\calF(Y)
\to
\calF(X).
\]
In particular, the above applies to any radditive object $\calF\in \Spt_s(\Fr_+(B))$.
\end{prop}

\begin{proof}
From \eqref{eq:ZFrpairSH} we obtain the diagram
\[
\begin{tikzcd}
\calF(Y)\ar{r}{c^*}\ar[equal]{d} & 
\calF(X\times\A^{1})\ar[d,shift right,swap,"i_0^*"]\ar[d,shift left,"i_1^*"]\\
\calF(Y)\ar[r,shift left,"c_1^*"]\ar[r,shift right,swap,"c_2^*"] & \calF(X).
\end{tikzcd}
\]
Here we have $i_0^* \circ c^*=c^*_2$ and $i_1^* \circ c^*=c^*_1$.
Since $\calF$ is $\A^{1}$-local the projection $p\colon X\times\A^{1}\to X$ induces an inverse 
$p^*\colon \calF(X)\to \calF(\A^{1}\times X)$ to $i_0^*$ and $i_1^*$ in $\SH$. 
Hence the classes of $i_0^*$ and $i_1^*$, 
and thus also the classes of $c_1^*$ and $c_2^*$, 
agree in $\SH$.
\end{proof}

\begin{corollary}
\label{cor:ovZFrpairSH}
Every $\A^{1}$-local radditive functor $\calF\colon \Fr_+(B)\to \SH$ induces canonically a functor $\ovZF_*(B)\to \SH$.
\end{corollary}

Next, 
we introduce the notion of intermediate framed correspondences.
Such correspondences are somewhat easier to construct than framed correspondences, 
see \Cref{def:FramedCorr}.

\begin{definition}
\label{def:FramedCorrInterm}
If $X,Y\in\Sch_{B}$, 
an intermediate framed correspondence of level $n$ from $X$ to $Y$ is a triple $c=(S,\phi,g)$, 
where
\begin{itemize}
\item[(1)] 
$S\not\hookrightarrow \A^n_X$ is a closed immersion, 
\item[(2)]
$\phi\colon V\to \A^n_B$ is a morphism such that $S\cong V\times_{\phi,\A^n_B,i} (0\times B)$, 
where $V=(\A^n_X)^h_S$ and $i\colon 0\times B\to \A^n_B$ is the $0$-section,
\item[(3)]
$g\colon S\to Y$ is a morphism.
\end{itemize}
We refer to $B$ as the support of $c$ and write $\Supp(c)\defeq S$.
\end{definition}

\begin{lemma}
\label{lm:FrSimplFr}
Suppose $B$ is an affine scheme, $X\in \AffSch_B$, and $Y\in \SmAff_B$.
Then for every intermediate framed correspondence $(S,\varphi,g)$ there exists a framed correspondence 
$(S,\varphi,\widetilde g)$ such that $g\colon S\to Y$ and $\widetilde g\colon (\A^n_X)^h_S\to Y$ lifts $g$.
\end{lemma}
\begin{proof}
This follows from \Cref{lm:AffSmHenselianLift}.
\end{proof}

\begin{remark}
Intermediate framed correspondences mediate between framed correspondences and 
normally framed correspondences, 
see \cite[\S4]{framed-MW}, \cite{five-authors}.
For $Y\in\Sm_{B}$ there are canonical motivic equivalences 
$\Fr(Y)\to \Fr^\mathrm{im}(Y)\to \Fr^\mathrm{nr}(Y)$ for the said types of framed correspondences.
\end{remark}

\begin{theorem}
\label{th:Injectivity}
Suppose $\calF$ is an $\A^{1}$-invariant quasi-stable radditive framed presheaf of $S^{1}$-spectra over a field $k$.
For $X\in \Sm_k$ and $x\in X$, 
assume that $\calF(\eta)\simeq 0$, 
where $\eta\in X_{(x)}$ is the generic point of the local scheme $X_{(x)}$ of $X$ at $x$.
Then we have $\calF(X_{(x)})\simeq 0$.
\end{theorem}

\begin{proof}
The presheaf of homotopy groups $\pi_{i}\calF\colon \Sm_k\to \Ab$, $i\geq 0$, 
equals the composite 
\[
\Sm_k
\to 
\Fr_+(k)\xrightarrow{\calF} 
\SH\xrightarrow{\pi_i} 
\Ab.
\] 
Thus $\pi_i\calF$ is an $\A^{1}$-invariant quasi-stable radditive framed presheaf of abelian groups
(the same holds over any base scheme when replacing the $\A^{1}$-invariance condition with $\A^{1}$-locality, 
see \Cref{section:shiitoms1s}). 
We are done since the proof of \cite[Theorem 3.9]{hty-inv} works for all fields. 
\end{proof}

\begin{prop}
\label{prop:SHI(k)&InjimpliesZareqNis}
Suppose \SHI holds on $\Sm_{k}$ and $\calF$ is an $\A^{1}$-local radditive framed presheaf of $S^{1}$-spectra on 
$\Sm_k$ in the sense of \Cref{def:SHIthnistriv}.
Then there is a canonical equivalence 
\[
\Lrep_\Zar(\calF)
\xrightarrow{\simeq} 
\Lrep_\nis(\calF).
\]
\end{prop}

\begin{proof}
Strict $\A^{1}$-invariance implies $\Lrep_\nis(\calF)$ is $\A^{1}$-local, quasi-stable, and radditive.
\Cref{th:Injectivity} implies, 
for any $X\in \Sm_k$ and $x\in X$, 
an equivalence of $S^{1}$-spectra
\[
\Lrep_\nis(\calF)(X_{(x)})
\xrightarrow{\simeq}
\calF(X_{(x)}).
\]
Moreover, there is a naturally induced equivalence
\[
\Lrep_\Zar(\calF)(X_{(x)})
\xrightarrow{\simeq}
\calF(X_{(x)}).
\]
Hence the canonical morphism $\Lrep_\Zar(\calF)\to \Lrep_\nis(\calF)$ is a Zariski local equivalence.
Since both $\Lrep_\nis(\calF)$ and $\Lrep_\Zar(\calF)$ are Zariski local,
this ends the proof.
\end{proof}

\begin{definition}\label{def:angleuanglrlepsion}
For an invertible function $u\in \calO^\times_B(B)$ we define
\[
\langle u\rangle
:=
(0\times B,\A^{1}\times B, ut,\pr\colon \A^{1}\times B\to B)
\in 
\Fr_1(B,B),
\]
where $t$ denotes the coordinate function on $\A^{1}$.
Moreover, 
for $n\geq 1$, 
we set 
\[
n_\varepsilon \defeq 
\sum\limits_{i=0}^{n-1} 
\langle (-1)^i\rangle
\in 
\ZF_1(B,B).
\]
\end{definition}
For units $u_1,u_2\in \calO^\times_B(B)$ we have the equality
\[
[\langle u_1\rangle\circ \langle u_2\rangle] 
=
[\langle u_1 u_2\rangle]\in \ovZF_2(B,B),
\]
and for $n,l\geq 1$ we have
\begin{equation}
\label{eq:nlvarepsilon=nvarepsilonlvarepsilon}
[\sigma (nl)_\varepsilon] 
= 
[n_\varepsilon \circ l_\varepsilon] 
\in 
\ovZF_2(B,B).
\end{equation}

\begin{prop}
\label{prop:FiniteCoprimeExtensionsLiftofmorphism}
For $U\in \Sm_B$ suppose the regular functions $f_1,f_2\in \calO(\A^{1}_U)$ are 
defined by monic separable polynomials in $\calO_U(U)[t]\cong \calO(\A^{1}_U)$ of coprime degrees $l_1$ and $l_2$, 
respectively.
Then for every $\A^{1}$-local quasi-stable radditive framed presheaf of $S^{1}$-spectra $\calF$ on $\Sm_B$, 
the identity $\mathrm{id}_{\calF(U)}$ factors through the morphism $\calF(U)\to\calF(Z(f_1)\amalg Z(f_2))$ induced by 
$Z(f_1)\amalg Z(f_2)\to U$.
\end{prop}
\begin{proof}
The vanishing loci $Z(f_1)$ and $Z(f_2)$ are closed subschemes in $\A^{1}_U$ that are finite and \'etale over $U$.
Since $l_1$ and $l_2$ are coprime, there exist non-zero integers $n_1$, $n_2$ such that 
\begin{equation}
\label{eq:n1l1n2l2=1}
n_1 l_1 - n_2 l_2=1.
\end{equation}
We may assume $n_1,n_2>0$.
\Cref{prop:FiniteExtension} shows that the endomorphism on $\calF(U)$ induced by $(l_j)_\varepsilon$ admits a factorization
\[
\begin{array}{lllll}
(l_j)_\varepsilon\colon \calF(U)&\xrightarrow{\pr_j^*}&\calF(Z(f_j))&\xrightarrow{c_j^*}&\calF(U),\\
\end{array}
\]
for some morphisms $c_j^*$, 
$j=1,2$.
We set 
\[
c^*
:=
(n_1)_\varepsilon^*\circ c_1^* +(n_2)_\varepsilon^*\circ c_2^*
\colon 
\calF(Z(f_1)\amalg Z(f_2))\to \calF(U).
\]
Using \eqref{eq:nlvarepsilon=nvarepsilonlvarepsilon} and \eqref{eq:n1l1n2l2=1} we deduce 
\[
(l_1)_\varepsilon\circ (n_1)_\varepsilon + (l_2)_\varepsilon\circ (n_2)_\varepsilon 
= 
[\id_U]\in \ovZF(U, U).
\]
Since $\calF$ is quasi-stable, 
\Cref{cor:ovZFrpairSH} implies the composite morphism
\begin{align*}
c^*\circ \pr^*&=( (n_1)^*_\varepsilon\circ 
c_1^*\circ \pr^*)+ ( (n_2)^*_\varepsilon\circ c_2^*\circ \pr^* )\\ 
& = (n_1)^*_\varepsilon\circ (l_1)^*_\varepsilon + (n_2)^*_\varepsilon\circ (l_2)^*_\varepsilon
\label{eq:prcirccFr(U,U)}
\end{align*}
is an auto-equivalence of $\calF(U)$.
Hence there is a lift $\calF(U)\to \calF(Z(f_1)\amalg Z(f_2))$ of the identity morphism on 
$\calF(U)$ along $\pr^*$.
\end{proof}

The argument for our next result, 
which we include for the convenience of the reader, 
can also be found in \cite[Appendix B]{five-authors} and \cite[Section 1]{DrKyllfinFrpi00} 
(the initial source is \cite[Section 5]{SmAffOpPairsv5}).
 
\begin{corollary}
\label{cor:FiniteExtensions}
Suppose $B$ is a local scheme whose closed point maps to $\gamma\in \Spec(\Z)$, 
and
\[
f
=
x^l+b_{l-1}x^{l-1}+\dots+b_0
\in
\mathbb{Z}[t]
\]
is separable over the residue field of $\gamma$. 
For every $U\in \Sm_B$ we set $U^\prime_{f}=Z(f_U)$, where $f_U\in\calO(\A^{1}_U)$ is obtained from $f$.
Suppose $\calF$ is an $\A^{1}$-local radditive framed presheaf of $S^{1}$-spectra on $\Sm_B$ such that 
$\calF(U^\prime)\simeq 0$ for $l\gg 0$.
Then we have $\calF(U)\simeq 0$.
\end{corollary}
\begin{proof}
Assume that $\calF(U^\prime)\simeq 0$ for all $l>n$.
Choose a closed point $p$ in $\Spec(\Z)$ that is contained in the closure of $\gamma$;
its residue field is $\mathbb F_p$.
For $j=1,2$ we may choose coprime integers $l_j\in\mathbb Z$, $l_j>n$,
and a separable monic polynomial $r_j$ over $k$ of degree $l_j$.
Such a polynomial exists for $l_j\gg 0$.
Let $f_j$ be a monic polynomial with integral coefficients of degrees $l_j$ that lifts $r_j$.
Then $f_j$ is separable over $\mathbb F_p$, and consequently over the residue field of $\gamma$.
We denote by the same symbol the pullback of $f_j$ in $\calO_{\A^{1}\times U}(\A^{1}\times U)$.
Note that $U^\prime_j\defeq Z(f_j)$ is finite and separable over $U$.
\Cref{prop:FiniteCoprimeExtensionsLiftofmorphism} shows $\mathrm{id}_{\calF(U)}$ admits a factorization
\[
\calF(U)\to \calF(U^\prime_1\amalg U^\prime_2)\to \calF(U).
\]
We have $\calF(U^\prime_1\amalg U^\prime_2)\simeq 0$ since $l_1,l_2>n$. 
It follows that $\calF(U)\simeq 0$.
\end{proof}

\begin{corollary}
\label{prop:FiniteExtensions}
Let $l_{1,i}$, $l_{2,i}$, for $i\geq 0$, be integers such that $(l_{1,i},l_{2,j})=1$ for all $i,j\geq 0$.
For $U\in \Sm_B$, 
suppose the finite surjective \'etale morphisms in the diagrams 
\[
\begin{array}{ccc}
U^\prime_{1,i}\to U^\prime_{1,i-1}\to\dots\to U^\prime_{1,0}=U\\
U^\prime_{2,i}\to U^\prime_{2,i-1}\to\dots\to U^\prime_{2,0}=U
\end{array}
\]
are defined by separable monic polynomials in $\calO_U(U)$ of degrees $l_{1,i}$ and $l_{2,i}$, 
respectively.
Suppose $\calF$ is an $\A^{1}$-local quasi-stable radditive framed presheaf of $S^{1}$-spectra on $\Sm_B$ such that 
\[
\varinjlim_{i\geq 0} \calF(U^\prime_{1,i})
\cong 
0
\text{ and }
\varinjlim_{i\geq 0}\calF(U^\prime_{2,i})
\cong 
0.\]
Then we have $\calF(U)= 0$.
\end{corollary}
\begin{proof}
Using the projections $\pr_{1,i}\colon U^\prime_{1,i}\to U$ and $\pr_{2,i}\colon U^\prime_{2,i}\to U$ we set 
\[ 
F_{i}
:=
\hofib(\calF(U)\to \calF(U^\prime_{1,i}\amalg U^\prime_{2,i})).
\]
\Cref{prop:FiniteCoprimeExtensionsLiftofmorphism} shows $F_{i}\to \calF(U)$ is trivial, 
and by the assumption on $\calF$ we have
\[
\varinjlim_i F_i
\simeq 
\calF(U).
\]
\end{proof}

\begin{lemma}
\label{lm:infirrpolynomials}
For every finite field $k$ there exists a sequence of irreducible polynomials $(f_i)_{\geq 0}$ with coprime degrees.
\end{lemma}
\begin{proof}
We prove the claim by induction. 
Let $f_1,\dots,f_d\in k[t]$ be irreducible polynomials of coprime degrees $l_i>1$.
For $N\gg 0$ and $L=N l_1\cdots l_d +1$, 
there exists a polynomial $f$ in $k[t]$ of degree $L$ such that $f(\alpha)=1$ for every $\alpha\in k$.
Since $L$ is coprime to $l_i$ for all $i=1,\dots, d$, 
there is an irreducible polynomial $f_{d+1}$ that divides $f$ and its degree $l_{d+1}$ is coprime to $l_i$ for all 
$i=1,\dots,d$.
Since $l_{d+1}\neq 1$ we are done.
\end{proof}

\begin{theorem}
\label{theorem:shifinitefields}
Let $k$ be a finite field.
Then \SHI holds on $\Sm_{k}$ in the sense of \Cref{def:SHIthnistriv}.
\end{theorem}
\begin{proof}
Let $\calF$ be an $\A^{1}$-local quasi-stable radditive framed presheaf of $S^{1}$-spectra on $\Sm_k$.
For $X\in \Sm_k$ and $x\in X$ we set $U=X^h_x$.
Let $\widetilde V\to V=\A^{1}_U$ be a Nisnevich covering.
We use the $\check{\text{C}}$ech construction in 
\eqref{eq:Check(widetilde X)} to define the presheaf $\mathcal Fib(-)$ of $S^{1}$-spectra on $\Fr_+(V)$
by setting
\[
\mathcal Fib(W) 
:= 
\hofib( \check C_{\widetilde V\times_V W}(W,\calF)\to\calF(W) ).
\]
Let $K$ be an infinite perfect field over $k$.
By strict $\A^{1}$-invariance for $K$ \cite[\S 17]{hty-inv} the presheaf $L_{\nis}(\calF_K)$ is $\A^{1}$-local, 
where $\calF_K$ denotes the base change to $K$. 
Hence there are equivalences
$$
L_{\nis}(\calF_K)(\A^{1}\times U\times \Spec K)
\simeq 
\calF_K(\A^{1}\times U\times \Spec K),
$$
and 
$$
\check C_{\widetilde V\times_k \Spec K}(V\times_k \Spec K,\calF)
\simeq 
\calF_K(V\times_k \Spec K).
$$
This yields the equivalence 
\begin{equation}
\label{eq:calFib(K)}
\mathcal Fib(V_K)
\simeq 
0.
\end{equation}
Furthermore we claim that 
\begin{equation}
\label{eq:calFib(A1U)simeq0}
\mathcal Fib(\A^{1}_U)
\simeq 0.
\end{equation}
Note that \eqref{eq:calFib(A1U)simeq0} implies $L_{\nis}(\calF)(\A^{1}_U)\simeq \calF(\A^{1}_U)$, 
and hence $L_{\nis}(\calF)(\A^{1}\times-)\simeq L_{\nis}(\calF)(-)$.
This shows that $L_{\nis}(\calF)$ is $\A^{1}$-local.
It remains to prove \eqref{eq:calFib(A1U)simeq0}.
Since $k$ is finite, every finite extension of $k$ is separable. 
By \Cref{lm:infirrpolynomials} there exist separable extensions $k(\alpha_{1,i})$ and $k(\alpha_{2,i})$, 
$i\geq 0$,
of coprime extension degrees over $k$.
For the infinite perfect field $K_j=\varinjlim_{i\geq 0} k(\alpha_{j,i})$,
\eqref{eq:calFib(K)} implies $\mathcal Fib(V_{K_j})\simeq 0$ for $j=1,2$.
\Cref{prop:FiniteExtensions} implies that $\mathcal Fib(V)\simeq 0$.
\end{proof}

\begin{prop}
\label{prop:FiniteExtension}
Let $U\in \Sm_B$ and suppose 
\[
f=x^l+b_{l-1}x^{l-1}+\dots+b_0
\in
\calO_U(U)
\]
is a separable monic polynomial over $\calO_U(U)$ with vanishing locus 
$Z(f)\not\hookrightarrow \A^{1}_U$.
Let $\pr\colon Z(f)\to U$ be the canonical \'etale morphism,
and note that $Z(f)\in \Sm_B$.
Then for every $\A^{1}$-local radditive framed presheaf of $S^{1}$-spectra $\calF$ on $\Sm_B$, 
the endomorphism $l_\varepsilon^*$ on $\calF(U)$ induced by the framed correspondence $l_\varepsilon\in \ZF_1(U,U)$, 
see \Cref{def:angleuanglrlepsion}, 
admits a factorization of the form
\[
\calF(U)
\xrightarrow{\pr^*}
\calF(Z(f))
\to \calF(U).
\]
\end{prop}
\begin{proof}
Note that $c=(Z(f),\A^{1}_U,g)\in \Fr_1(U, Z(f))$,
where $g\colon (\A^{1}_U)^h_{Z(f)}\to Z(f)$ exists by \Cref{lm:AffSmHenselianLift}. 
We set $\tilde f=f(1-\lambda)+x^{l}\lambda\in \calO_{\A^{1}\times U\times\A^{1}}(\A^{1}\times U\times\A^{1})$,
for the coordinate $\lambda$ on the second copy of $\A^{1}$.
Now consider the framed correspondence $\tilde c=(Z(\tilde f),\A^{1}_{U\times\A^{1}},\tilde g)$,
where $\tilde g\colon (\A^{1}_U)^h_{Z(\tilde f)}\to U$ is the canonical projection.
The $\A^{1}$-homotopy given by $\tilde c$ implies the equality
\begin{equation}
\label{eq:prcirccFr(U,U)} 
[\pr\circ c]
=
[l_\varepsilon]
\in 
\ovZF_1(U, U).
\end{equation}

Since $\calF$ is an $\A^{1}$-local quasi-stable radditive framed presheaf,
\Cref{prop:A1localhomotopycommutativesquare} and \eqref{eq:prcirccFr(U,U)} imply the endomorphism
\[
(\pr\circ c)^*
\colon 
\calF(U)
\to 
\calF(Z(f))
\to
\calF(U)
\]
defined by the composite framed correspondence $\pr\circ c$ coincides in $\SH$ with the composition
\[
\calF(U)
\xrightarrow{l_\varepsilon^*}
\calF(U)
\xrightarrow{\tau}
\calF(U)
\]
for some auto-equivalence $\tau$.
That is, 
$l_\varepsilon^*=\tau^{-1}\circ c^*\circ \pr^*\colon \calF(U)\to \calF(U)$,
and the claim follows.
\end{proof}

For $\calF \in \Spt_s(\Sm_{B})$, 
$U\in \Sm_{B}$,
and a morphism $\widetilde X\to X$ in $\Sm_{B}$, 
the $\check{\text{C}}$ech construction $\check C_{\widetilde X}(U,\calF)$ is the homotopy limit 
of the diagram of $S^{1}$-spectra
\begin{equation}
\label{eq:Check(widetilde X)}
\calF(\widetilde X\times_X U)
\begin{smallmatrix}
\longrightarrow\\
\longrightarrow 
\end{smallmatrix} 
\calF(\widetilde X\times_X\widetilde X\times_X U)
\begin{smallmatrix}
\longrightarrow  \\ 
\longrightarrow \\
\longrightarrow 
\end{smallmatrix}
\calF(\widetilde X\times_X \widetilde X\times_X \widetilde X\times_X U)
\begin{smallmatrix}
\longrightarrow  \\ 
\longrightarrow  \\
\longrightarrow  \\
\longrightarrow  
\end{smallmatrix}
\cdots.
\end{equation}
The same construction applies in other settings such as radditive framed presheaves of 
$S^{1}$-spectra.

\section{Nisnevich and \texorpdfstring{$\tf$}{tf} localization of framed presheaves}\label{sect:FrSheafNistf}
The purpose of this appendix is prove \Cref{th:restrfrLoctfniscommute} on Nisnevich and $\tf$-sheafification 
of framed presheaves (this result generalizes \cite[Proposition 3.2.14]{five-authors}).
If $Z$ is a closed subscheme of $B$, 
there is an adjunction between $\infty$-categories of presheaves
\begin{equation}
\label{equation:prefradjunction}
\gamma^*\colon \catPre(\catS)\rightleftarrows \catPre^\fr(\catS)\colon \gamma_*.
\end{equation}
Here $\catS$ is shorthand for $\Sm_{B,Z}$, $\SmAff_{B,Z}$, or $\Smat_{B,Z}$.
\Cref{th:restrfrLoctfniscommute} identifies in a precise way the Nisnevich and $\tf$-localizations of
$\catPre(\catS)$ and $\catPre^\fr(\catS)$.

\begin{definition}\label{def:PrefrNislocObj}
An object $F\in \catPre^\fr(\catS)$ is $\tf$-local (resp.~Nisnevich local) if $\gamma_*F$ is $\tf$-local 
(resp.~Nisnevich local).
\end{definition}

\begin{lemma}\label{lm:gamma-d-*Nisloc}
An object $F\in \catPre^\fr(\catS)$ is $\tf$-local (resp.~Nisnevich local) if and only if 
for any $\tf$-local (resp.~Nisnevich local) equivalence $\mathcal X\to \mathcal Y\in \catPre(\catS)$,
the induced morphism on mapping spaces 
$\mathrm{Map}_{\catPre^\fr(\catS)}(\gamma^*(\mathcal Y),F)
\to
\mathrm{Map}_{\catPre^\fr(\catS)}(\gamma^*(\mathcal X),F)$ 
is an equivalence.
\end{lemma}
\begin{proof}
The claim follows from the adjunction $\gamma^*\dashv\gamma_*$ in \eqref{equation:prefradjunction}.
\end{proof}

\begin{lemma}
The subcategories $\catPre^\fr_\tf(\catS)$ and $\catPre^\fr_\nis(\catS)$ of $\tf$-local and Nisnevich local objects 
are reflective subcategories of $\catPre^\fr(\catS)$.
\end{lemma}
\begin{proof}
The claim follows because by \Cref{lm:gamma-d-*Nisloc} the $\infty$-categories $\catPre^\fr_\tf(\catS)$ and $\catPre^\fr_\nis(\catS)$ 
are equivalent to the localizations of $\catPre^\fr(\catS)$ with respect to the class of  morphisms of the form 
$\gamma^*(f)$ for $\tf$- or Nisnevich local equivalences in $\catPre(\catS)$.
\end{proof}

We shall write $L^\fr_\tau$ (resp. $\Lrep^\fr_\tau$) for the corresponding $\tau$-localization (resp. endofunctor).
A morphism $f$ in $\catPre^\fr(\catS)$ is called a $\tau$-local equivalence if $L^\fr_{\tau}(f)$ is an equivalence. 
We write $\catPre^\fr_\tau(\catS)$ for the $\tau$-localization and $F\simeq_\tau G$ for a $\tau$-local equivalence 
$F\to G$ in $\catPre^\fr(\Sm_{B})$.

For a map $\widetilde X\to X$ in $\catS$, 
let $h_{\widetilde X}(X)$ denote the subpresheaf of the representable presheaf $h(X)$ in $\catPre(\catS)$
that is determined by the image of the canonical morphism $h(\widetilde X)\to h(X)$.
Likewise, 
we let $h_{\widetilde X}^\fr(X)$ denote the image of $h^\fr(\widetilde X)$ in $h^\fr(X)$.

\begin{lemma}
\label{lm:spanslifting}
Suppose $S,U,X\in \Sch_B$, $z\in B$, 
and $c\colon\widetilde X\to X$ is a $\tf$-covering.
For any morphism $g\colon S \to X$ and finite morphism $f\colon S\to U^h_z$, 
there exists a filler $\tilde g$ in the diagram 
\begin{equation}
\label{eq:spansl}
\xymatrix{
U^h_z & S\ar[l]_-{f} \ar@{.>}[r]^-{\tilde g} & \widetilde X\ar[d]^-{c}\\
U^h_z\ar@{=}[u] & S\ar[l]_-{f}\ar@{=}[u]\ar[r]^-{g} & X.
}
\end{equation}
\end{lemma}
\begin{proof}
Since $f$ is finite, 
and the closed immersion $U_z\not\hookrightarrow U^h_z$ is a henselian pair, 
the closed immersion $S\times_B z\not\hookrightarrow S$ is a henselian pair
by \cite[\href{https://stacks.math.columbia.edu/tag/09XI}{Tag 09XI}]{StacksProject}
and thus $S=S^h_z$.
\Cref{prop:Proptftop}(iv)-(vii) shows that $S$ is a $\tf$-point.
Then since $\widetilde X\times_X S\to S$ is a $\tf$-covering, and $S$ is a $\tf$-point, 
there exists a left inverse $S\to \widetilde X\times_X S$. 
By composing with the canonical projection to $\widetilde X$, 
we obtain the required lifting $\tilde g$. 
\end{proof}

\begin{lemma}
\label{lm:gammadstfloc}
For any $\tf$-covering $\widetilde X\to X$,
the canonical morphism $h^\fr_{\widetilde X}(X)\to h^\fr(X)$ maps to a $\tf$-local equivalence in $\catPre(\catS)$ 
under the functor $\gamma_*$ in \eqref{equation:prefradjunction}.
\end{lemma}

\begin{proof}
We check there is an equivalence $h^\fr_{\widetilde X}(X)(U^h_z)\to h^\fr(X)(U^h_z)$ for all $U\in \Sm_{B}$, 
$z\in Z$. 
The definition of tangentially framed correspondences 
\cite[Definition 2.3.4]{five-authors} 
implies isomorphisms
\[
h^\fr(X)(U^h_z)=\colim_{G} \Map_{K(S)}(0,\mathcal L_f),
\quad 
h^\fr_{\widetilde X}(X)(U^h_z)=\colim_{G^\prime} \Map_{K(S)}(0,\mathcal L_f).
\]
Here $K(S)$ is the $K$-theory space of $S$ and $G$ is the groupoid of spans of the form
$$
U^h_z\xleftarrow{f} S\xrightarrow{g} X.
$$
Moreover, 
$f$ is a finite syntomic map, 
and $G^\prime$ is the subgroupoid of $G$ comprised of spans that admit a lifting as in \eqref{eq:spansl}.
The canonical injection of groupoids $G^\prime\hookrightarrow G$ is a surjection by \Cref{lm:spanslifting}. 
\end{proof}

\begin{theorem}
\label{th:restrfrLoctfniscommute}
There are canonical isomorphism of functors 
$\gamma_* \Lrep^\fr_\tf\cong \Lrep_\tf \gamma_*$, $\gamma_* \Lrep^\fr_\nis\cong \Lrep_\nis \gamma_*$ 
from $\catPre^\fr(\catS)$ to $\catPre(\catS)$.
\end{theorem}
\begin{proof}
Denote by $v_\tf$ the class of morphisms in $\catPre(B)$ of the form $h_{\widetilde X}(X)\to h(X)$
for $\tf$-coverings $\widetilde X\to X$
(and likewise for the class $v_\tf^\fr$ of morphisms in $\catPre^\fr(\catS)$ of the form 
$h^\fr_{\widetilde X}(X)\to h^\fr(X)$).
Note that $\catPre^\fr_\tf(\catS)$ is the $\infty$-subcategory of $\catPre^\fr(\catS)$ spanned 
by objects whose image under $\gamma_*$ belongs to $\catPre_\tf(\catS)$.
Hence $\gamma_*$ preserves $\tf$-local objects. 
Moreover, 
it follows that there is an equivalence of $\infty$-categories 
\begin{equation}\label{eq:PrefrtfSsimeqPrefrSv}
\catPre^{\fr}_{\tf}(\catS)\simeq \catPre^\fr(\catS)[(v_{\tf}^{\fr})^{-1}].
\end{equation}
Thus the class of $\tf$-local equivalences $w_\tf^\fr$ in $\catPre^\fr(\catS)$ coincides with the 
morphisms whose images along the functor 
$\catPre^\fr(\catS)\to \catPre^\fr(\catS)[(v_\tf^\fr)^{-1}]$ are equivalences.
By \Cref{lm:gammadstfloc}, 
we have $\gamma_*(v^\fr_\tf)\subset v_{\tf}$.
The universal property of the localization implies that $\gamma_*(w^\fr_\tf)\subset w_{\tf}$, 
where $w_\tf$ denotes the class of $\tf$-local equivalences in $\catPre(\catS)$ 
(this follows because $w_\tf$ coincides with the class of morphisms whose images along the 
functor $\catPre(B)\to\catPre(B)[(v_\tf)^{-1}]$ are equivalences).
To summarize, 
$\gamma_*$ commutes with $\tf$-localization since it preserves $\tf$-local objects and 
$\tf$-local equivalences.
The proof for the Nisnevich topology is similar.
\end{proof}

\begin{remark}
We note that $\catPre^{\fr}_{\tf}(\catS)$ is a reflective subcategory of $\catPre^{\fr}(\catS)$, 
see \eqref{eq:PrefrtfSsimeqPrefrSv}.
Moreover, 
there is a commutative diagram of $\infty$-categories
\[
\xymatrix{
\catPre^\fr(\catS)\ar[r] \ar[d]_{\gamma^*} & 
\catPre^\fr(\catS)[(v_\tf^\fr)^{-1}]\ar[r]^{\simeq}\ar[d]&
\catPre^\fr(\catS)[(w_\tf^\fr)^{-1}]\ar[r]^>>>>>>{\simeq}\ar[d] &
\catPre^{\fr}_{\tf}(\catS)\ar[d]
\\
\catPre(B)\ar[r] &
\catPre(B)[(v_\tf^\fr)^{-1}]\ar[r]^{\simeq}&
\catPre(B)[(w_\tf^\fr)^{-1}]\ar[r]^>>>>>>>{\simeq} &
\catPre_{\tf}(B)
}
\]
Our proof of \Cref{th:restrfrLoctfniscommute} uses the horizontal equivalences and the left commutative square.
\end{remark}

\begin{corollary}\label{cor:gamma-d-sNiscons}
 The functor $\gamma_*\colon \catPre^\fr(\catS)\to \catPre(\catS)$ 
is exact and conservative with respect to $\tf$-local and Nisnevich local equivalences.
\end{corollary}
\begin{proof}
Let $F\to G$ be a $\tf$-local or Nisnevich local equivalence in $\catPre^\fr(\Sm_{B})$ so that 
$L_\tau(F)\cong  L_\tau(G)$.
By \Cref{th:restrfrLoctfniscommute} we have $L_\tau\gamma_*(F)\cong  L_\tau\gamma_*(G)$, 
and therefore $\gamma_*(F)\cong \tau \gamma_*(G)$.
Now let $F\to G$ be a morphism in $\catPre^\fr(\Sm_{B})$ such that $\gamma_*(F)\simeq_\tf\gamma_*(G)$.
Then we have $\Lrep_\tf(\gamma_*F)\cong  \Lrep_\tf(\gamma_*G)$. 
\Cref{th:restrfrLoctfniscommute} yields the isomorphism $\gamma_*(\Lrep^\fr_\tf F)\cong\gamma_*(\Lrep^\fr_\nis G)$.
Since $\gamma$ is conservative, 
we deduce $\Lrep^\fr_\tf F\cong  \Lrep^\fr_\tf G$ and thus $F\simeq_\tf G$.
The proof for $\Lrep_\nis$ is similar.
\end{proof}

\section{Stabilization}
\label{section:StablelocalizationDefinitionsNotations}

This section aims to introduce an $\infty$-category of prespectra suitable for our setting.
We do not claim any originality; 
this material goes back to work by Boardman. 
We follow Lurie's work \cite{LurieDAG1} and refer to Hovey \cite{zbMATH01698557} 
for spectra in model categories.
Throughout, 
$\mathcal C$ is a presentable symmetric monoidal $\infty$-category for which there exists an 
adjunction 
$\Sigma \colon \mathcal C\rightleftarrows \mathcal C\colon\Omega$.
In particular, 
$\mathcal C$ is cocomplete and $\Omega$ preserves limits.
In addition, 
assume $\mathcal C$ is closed under filtered colimits, 
and $\Omega$ preserves filtered colimits.
We write $\mathrm{Cat}_\infty$ for the $\infty$-category of $\infty$-categories.
Let $\mathbb N$ denote the category with objects non-negative integers 
and for each $m\leq n$, a unique morphism $m\to n$.
For $\mathcal C$ there is a functor $\mathbb N^\mathrm{op}\to \mathrm{Cat}_\infty$ given by
\begin{equation}
\label{eq:seqPre(C)Omega}
\mathcal C\xleftarrow{\Omega}\mathcal C\xleftarrow{\Omega}\mathcal C\xleftarrow{\Omega}\cdots.
\end{equation}
Applying Lurie's $\infty$-categorical Grothendieck construction to \eqref{eq:seqPre(C)Omega}, 
see, e.g., \cite{zbMATH06810410}, 
we extract an $\infty$-category $\mathcal C_\Omega$ with a Cartesian fibration to $\mathbb N$.

\begin{definition}
\label{def:PreSpectra}
The $\infty$-category of prespectra $\PSpt^\Omega(\mathcal C)$ is the $\infty$-category of 
sections of the Cartesian fibration $\mathcal C_\Omega\to \mathbb N$.
\end{definition}

See \cite[\S 4.1]{ELSO} for an example of prespectra for $\Omega=\Omega_{\Gm}$.
The shift functor yields a canonical natural transformation 
$\nu\colon \Id_{\PSpt^\Omega(\mathcal C)}\to \Omega(-)[1]$. 
We write $\Omega^l(-)[l]\colon \PSpt^\Omega(\mathcal C)\to \PSpt^\Omega(\mathcal C)$ 
for the $l$-th iteration of $\Omega(-)[1]$ (see \cite[\S4]{hovey:ss} for classical spectra).
Next we recall the stabilization of $\mathcal C$ with respect to $\Omega$ 
(see \cite[\S 2.2]{RobalnoncomKtheorybridge}).

\begin{definition}
\label{def:catSpt}
For $\mathcal C$ and $\Omega$ as above, 
the $\infty$-category of spectra $\catSpt^\Omega(\mathcal C):=\mathcal C[\Omega^{-1}]$ 
is the limit of \eqref{eq:seqPre(C)Omega}.
\end{definition}

\begin{lemma}
\label{lm:SptS(C)}
\begin{itemize}
\item[(1)] 
The subcategory of $\PSpt^\Omega(\mathcal C)$ spanned by objects $F$ such that $\nu\colon F\simeq\Omega F[1]$ 
is reflective.
Its associated localization endofunctor on $\PSpt^\Omega(\mathcal C)$ sends $F$ to $\colim \Omega^l F[l]$.
\item[(2)] 
The reflective subcategory in (1) is equivalent to $\catSpt^\Omega(\mathcal C)$.
\end{itemize}
\end{lemma}
\begin{proof}
(1) 
The functor $F\mapsto \colim \Omega^l F[l]$ is idempotent on account of the equivalences
\[
\Omega(\colim \Omega^l F[l])[1]\simeq 
\colim \Omega^{l+1} F[l+1]\simeq 
\colim \Omega^l F[l]. 
\]
Moreover, 
it sends the canonical natural transformation $F\to \colim \Omega^l F[l]$ to the identity.
Hence by the equivalence of parts (2) and (3) of \cite[Proposition 5.2.7.4]{LurieHTT}, 
the functor is a localization.
The claims follow as in \Cref{lm:LreptLprepcpreserveexact}.
For (2), 
we start with the definition
\[
\mathcal C[\Omega^{-1}]
=
\operatorname{lim}
\left(
\mathcal C\xleftarrow{\Omega}\mathcal C\xleftarrow{\Omega}\mathcal C\xleftarrow{\Omega}\cdots
\right)
\]
from \cite[\S 2.2]{RobalnoncomKtheorybridge}.
By \cite[Corollary 3.3.3.2]{LurieHTT} the above limit of $\infty$-categories is equivalent to the category of 
Cartesian sections of the Cartesian fibration $\mathcal C_\Omega\to \mathbb N$. 
The latter subcategory of the $\infty$-category of all sections is equivalent to the reflective subcategory in (1).
\end{proof}

\begin{notation}\label{notat:PSptSptLSRS}
\Cref{lm:SptS(C)} shows there exists a localization functor
\[
L_\Omega
\colon 
\PSpt^\Omega(\mathcal C)
\to 
\catSpt^\Omega(\mathcal C) 
\] 
with right adjoint embedding $R_\Omega\colon \catSpt^\Omega(\mathcal C)\to\PSpt^\Omega(\mathcal C)$.
We write $\mathcal L_\Omega\colon \PSpt^\Omega(\mathcal C)\to \PSpt^\Omega(\mathcal C)$ for the composite $R_\Omega L_\Omega$.
\end{notation}

\section{Telescopes and homotopy colimits}
\label{sectapp:Homotopycolimits}
In this appendix,  
we record some homotopical constructions needed in \Cref{subsect:Rigid} for scheme categories 
such as $\SmatZ$, $\SmatBcZ$, and $\SmatBlZ$ defined in \Cref{subsection:candn}.
Our example of primary interest is $\SmatBcZ$.
Nevertheless, 
the same constructions can be carried out for any category with a coproduct.
The main result on homotopy colimits, 
\Cref{lm:hocolimLrepA1SetPresheaf}, 
is used in \Cref{prop:LrepA1uuscommute}.

Let $\Delta^\mathrm{op}\SmatBcZ$ denote the category of simplicial objects in $\SmatBcZ$.

\begin{definition}\label{def:geomretricrealisation}
The geometric realization functor
\[
|-|_{B,Z}
\colon
\Delta^\mathrm{op}\SmatBcZ\to \Pre(\SmatBcZ)
\] 
is the left Kan extension of the
functor 
\[\Delta\times \SmatBcZ\to \Pre(\SmatBcZ); 
([n],X^h_Z)\mapsto \Delta_{B,Z}^n\times_{B,Z} X^h_Z
\]
along 
\[
\Delta\times \SmatBcZ\to \Delta^\mathrm{op}\SmatBcZ; 
([n],X^h_Z)\mapsto \Delta^n\times X^h_Z.
\]
\end{definition}

\begin{definition}
\label{def:affinespacessimplcialsiagram}
Let $\mathbf X\colon K\rightarrow N(\SmatBcZ)$ be a morphism of simplicial sets. 
Define the simplicial scheme $\mathbf X_{K}\in\Delta^\mathrm{op}\SmatBcZ$ by 
\begin{equation}
[n]
\mapsto 
(\mathbf X_{K})_n=
\coprod\limits_{\alpha\in K_{n}} 
\mathbf X(v_0(\alpha)).
\end{equation}
Here $v_0(\alpha)$ is the zeroth vertex of the simplex $\alpha$,
and for a map $\nu\colon [n^\prime]\to [n]\in \Delta$ we define 
$(\mathbf X_{K})_n\to (\mathbf X_{K})_{n^\prime}$ by taking the coproduct of the maps
\[
\mathbf X(v_0(\alpha))\to \mathbf X(v_0(\nu^*(\alpha)))
\]
induced by $v_0(\alpha)\to v_0(\nu^*(\alpha))$, 
where $\nu^*\colon K_n\to K_{n^\prime}$.
\end{definition}

Next we define $\Delta_{B,Z}$-telescopes.

\begin{definition}
\label{def:CylBcZX}
Let $\mathbf X\colon K\to N(\SmatBcZ)$ be a morphism of simplicial sets. 
The $\Delta_{B,Z}$-telescope of $\mathbf{X}$ is the geometric realization
\[
\Cyl^{B,Z}_K \mathbf{X} 
:= 
|\mathbf X_{K}|_{B,Z}.
\]
\end{definition}

Let $\mathcal C$ be a small category.
\begin{definition}\label{def:CylBcZcalCX}
For a diagram $\mathbf X\colon \mathcal C\to \SmatBcZ$ and morphism of simplicial sets $d\colon K\to N(\mathcal C)$, 
let \[
\Cyl^{B,Z}_{\mathcal C} \mathbf{X}(K) = \Cyl^{B,Z}_{K} N(\mathbf{X})d
\]
denote the $\Delta_{B,Z}$-telescope of 
$
K
\xrightarrow{d} 
N(\mathcal C)
\xrightarrow{N(\mathbf X)}
N(\SmatBcZ).
$
\end{definition}

We let $\Cyl^{Z}_{-}$ be shorthand for $\Cyl^{Z,Z}_{-}$.
If $\calF\colon(\SmatBcZ)^\mathrm{op}\to\Set$ is a functor, 
we use the same notation for its left Kan extension 
\begin{equation}
\label{eq:calFPre(S)toSet}
\calF
\colon 
\Pre(\SmatBcZ)^\mathrm{op}
\to 
\Set 
\end{equation} 
along the embedding $(\SmatBcZ)^\mathrm{op}\to \Pre(\SmatBcZ)^\mathrm{op}$ opposite to the Yoneda embedding.

\begin{definition}
\label{def:hocolimA}
Let $\mathcal C$ be a small category. 
For a diagram $\mathbf X\colon \mathcal C\to \SmatBcZ$, 
we define the simplicial set
\[
\hocolimABZ_{\mathcal C,\mathbf X}\calF
\colon 
\Delta^\mathrm{op}
\to 
\Set
\]
by setting
\begin{equation}
\label{eq:hocolimABcZX2Deltan}
(\hocolimABZ_{\mathcal C,\mathbf X}\calF)_n
:=
\coprod_{\alpha\in N(\mathcal C)_n}\calF(\Cyl^{B,Z}_{\mathcal C}\mathbf X(\alpha)).
\end{equation}
Here we apply $\mathbf X$ term-wise to the simplices $N(\mathcal C)_n = \sSet(\Delta^n,N(\mathcal C))$.
Moreover, 
the nerve induces the simplicial structure on $\hocolimABZ_{\mathcal C,\mathbf X}\calF$.
When the meaning of $\mathbf X$ is clear,  
we write $\hocolimABZ_{\mathcal C}\calF$ for this construction, 
see, e.g.,
\eqref{eq:hocolimSmatBcZhocolimSmatZ}.
If $K$ is a simplicial set, 
we generalize \eqref{eq:hocolimABcZX2Deltan} by setting 
\begin{equation}
\label{eq:hocolimABcZX2}
K^{\Delta_{B,Z}}_{\mathcal C,\mathbf X}\calF
:=
\coprod_{\sSet(K,N(\mathcal C))}\calF(\Cyl^{B,Z}_{\mathcal C}\mathbf X(K)).
\end{equation}
Here $\mathbf X(K)\in \sSet(K,N(\SmatBcZ))$ is shorthand for the composite 
$K\to N(\mathcal C)\to N(\SmatBcZ)$.
\end{definition}

\begin{remark}
The homotopy colimit notation in \Cref{def:hocolimA} will be justified in \Cref{lm:hocolimLrepA1SetPresheaf}.
\end{remark}

\begin{lemma}
\label{lm:hocolimA(K)}
There is a bijection of sets
\[
\sSet(K,\hocolimABZ_{\mathcal C,\mathbf X}\calF)
\cong 
K^{\Delta_{B,Z}}_{\mathcal C,\mathbf X}\calF.
\] 
\end{lemma}
\begin{proof}
Use that \eqref{eq:calFPre(S)toSet} and the functor $\sSet^\mathrm{op} \to \Set$ defined by \eqref{eq:hocolimABcZX2}
preserve colimits.
\end{proof}

Recall the endofunctor $\Rep_{\A^{1}}^{[1]}$ on $\Spc_{s}(\Smat_{B,Z})$ from \Cref{subsect:LrepA1notdef}.

\begin{lemma}
\label{lm:hocolimLrepA1SetPresheaf}
There is a naturally induced weak equivalence of simplicial sets
\[
\hocolim_{s\in \mathcal C}\Rep_{\A^{1}}^{[1]}\calF(\mathbf X(s))
\xrightarrow{\simeq} 
\hocolimABZ_{\mathcal C,\mathbf X}\calF. 
\] 
\end{lemma}
\begin{proof}
Consider the functor \[ \mathcal E\colon \mathcal C\to \sSet; s\mapsto \Rep_{\A^{1}}^{[1]}\calF(\mathbf X(s)),\]
and the $\mathcal E$-weighted nerve $N^{\mathcal E}(\mathcal C)$ of $\mathcal C$ from 
\cite[Tag 025X]{KerodonProject}.
For each $n\in \mathbb Z_{\geq 0}$, 
there is an isomorphism of sets
\[
N^{\mathcal E}(\mathcal C)_n\cong
\coprod_{\alpha\in N(\mathcal C)_n} {} \calF(\Cyl^{B,Z}_{\mathcal C}\mathbf X(\alpha)), 
\]
and consequently, 
an isomorphism of simplicial sets
\[
N^{\mathcal E}(\mathcal C)
\cong
\hocolimABZ_{\mathcal C,\mathbf X}\calF.
\]
The claim follows from the weak equivalence
\[
\hocolim_{s\in \mathcal C}\Rep_{\A^{1}}^{[1]}\calF(\mathbf X(s))
\xrightarrow{\simeq}
N^{\mathcal E}(\mathcal C)
\] 
provided by \cite[Tag 037S]{KerodonProject}.
\end{proof}

\section{Notation}

\subsubsection{Categories of schemes}\label{convnotations:schemescategories}
\index{Categories of $B$-schemes!$\SmatB$|(}
\index{Categories of $B$-schemes!$\EssSm_{B,Z}$|(}
\index{Categories of $B$-schemes!$\Sm_{B,Z}$|(}
\index{Categories of $B$-schemes!$\Sm_{B*Z}$|(}
\index{Categories of $B$-schemes!$\SmAff_B$|(}
\index{Categories of $B$-schemes!$\SmAff_{B,Z}$|(}
\index{Categories of $B$-schemes!$\SmAff_{B*Z}$|(}
We recall that $\Sm_{B}$ denotes 
the category of smooth separated finite type $B$-schemes,
$\SmAff_B$ denotes the full subcategory of $\Sm_B$ spanned by schemes that admit a closed immersion into 
some finite dimensional affine space $\A^n_B$,
and $\EssSm_B$ denotes essentially smooth $B$-schemes, 
see \Cref{def:EssentiallySmooth}.
Moreover, 
we use of the following full subcategories of the category of $B$-schemes:
\begin{center}
\begin{tabular}{l|l}
$\Sm_{B,Z}$ & $\ip{X_Z^h\mid X\in\Sm_{B}}$\\
$\SmBlZ$ & $\ip{X_Z,X_Z^h\mid X\in\Sm_{B}}$\\
$\SmAff_{B,Z}$ & $\ip*{X_Z^h\mid X\in\SmAff_B}$\\
$\SmAff_{B*Z}$ & $\ip*{X_Z,X_Z^h\mid X\in\SmAff_B}$\\
$\Sm_{B}^\cci$ & $\ip{X\mid X\in\SmAff_B, T_X\cong\calO_X^n\text{ for some }n}$\\
$\SmatBcZ$ & $\ip{X_Z^h\mid X\in\Smat_{B}}$\\
$\SmatBlZ$ & $\ip{X_Z,X_Z^h\mid X\in\Smat_{B}}$\\
\end{tabular}
\end{center}

For example, 
the category $\Sm_{B}^\cci$ is spanned by all $X\in\SmAff_B$ with trivial tangent bundle 
of constant relative dimension over $B$.
The superscript is shorthand for component-wise complete intersection.
We note that if $X\in\SmAff_{B,Z}$, 
then also $(\A^1_{B})^{h}_{Z}\times_{B,Z}X$ is in $\SmAff_{B,Z}$.
Here, 
the fiber product is defined in \eqref{eq:timesBZ}.

\begin{lemma}
\label{lm:SmatBcZdefinitionsequivalence}
The embedding of categories 
\[
\SmatBcZ\to\ip{X\mid X\in\SmAff_{B,Z}, T_X\cong\calO_X^n\text{ for some }n}
\] 
is an equivalence.
\end{lemma}
\begin{proof}
Suppose that $X\in\SmAff_{B,Z}$ and $T_X\cong\calO_X^n$. 
By the definition of $\SmAffBcZ$, 
there exists $\widetilde X^\prime\in\SmAff_B$ such that 
$X\cong {({\widetilde{X}}^\prime)}^h_Z$.
Choose an open neighborhood $\widetilde X$ of $X\times_B Z$ in $\widetilde X^\prime$ such that 
$T_{\widetilde X}\cong\calO_{\widetilde X}^n$.
Then we have $\widetilde X\in\SmatBcZ$ and $X\cong \widetilde X^h_Z$.
\end{proof}

Our categories are related via evident functors, 
where "$\hookrightarrow$" denotes a fully faithful embedding, 
$U=B-Z$, and $X_{U}=X\times_B U$:
\begin{equation}
\label{equation:categories}
\begin{tikzcd} 
\Sm_{U}^\cci\ar[d,hook,shift left]\arrow[r,hook] & \SmAff_{U}\arrow[r,hook]\arrow[d,hook,shift left] & 
\Sm_{U}\arrow[d,hook,shift left]\arrow[r,hook] & \EssSm_{U}\arrow[d,hook,shift left]\\
\Sm_{B}^\cci\ar[d]\ar[u,shift left]\arrow[r,hook] & \SmAff_B\arrow[r,hook]\ar[d]\ar[u,shift left] & 
\Sm_{B}\ar{d}{X\mapsto X^h_Z}\arrow[r,hook]\ar[u,shift left,"X\mapsto X_{U}"] & \EssSm_B\arrow[d,equal]\ar[u,shift left]\\
\Sm_{B,Z}^\cci\arrow[d,hook]\ar[r,hook] & \SmAff_{B,Z}\arrow[r,hook]\ar[d,hook] & 
\Sm_{B,Z}\ar[d,hook]\arrow[r,hook] & \EssSm_B\ar{d} \\
\SmBlZ^\cci \ar[d,shift left]\ar[r,hook] & \SmAff_{B\ast Z}\ar[d,shift left]\arrow[r,hook] & 
\SmBlZ\ar[d,shift left,"X^h_Z\mapsto X_Z"]\arrow[r,hook] & \Sch_B \ar[d,shift left]\\
\Sm_{Z}^\cci \ar[u,hook,shift left]\ar[r,hook] & \SmAff_{Z}\arrow[r,hook]\ar[u,hook,shift left] & 
\Sm_{Z}\arrow[r,hook]\ar[u,hook,shift left] & \Sch_Z. \ar[u,hook,shift left]
\end{tikzcd}
\end{equation}

We consider $\Sm_{B,Z}$ and $\SmBlZ$ mostly for comparison with framed correspondences 
between smooth (not necessarily affine) schemes in \Cref{sect:tfLoc}. 
\Cref{lm:baselocality} reduces the proof of our main result to the case of an affine base scheme.
\Cref{cor:SmScZSmatA1locpreserve,cor:AffSmScZSmat} show that when $B$ is affine, 
we may replace $\SmBcZ$ with $\SmAffBcZ$ or $\SmatBcZ$.
For the proofs in \Cref{section:defsmZsmB} it is essential to work with $\SmatBcZ$.

\index{Categories of $B$-schemes!$\SmatB$|)}
\index{Categories of $B$-schemes!$\EssSm_{B,Z}$|)}
\index{Categories of $B$-schemes!$\Sm_{B,Z}$|)}
\index{Categories of $B$-schemes!$\Sm_{B*Z}$|)}
\index{Categories of $B$-schemes!$\SmAff_B$|)}
\index{Categories of $B$-schemes!$\SmAff_{B,Z}$|)}
\index{Categories of $B$-schemes!$\SmAff_{B*Z}$|)}

\subsubsection{Model categories}
\label{subsubsect:model-cat} 
Let $\mathcal S$ be any of the categories in \eqref{equation:categories}. 
We fix a bounded, complete, regular $cd$-structure $\tau$ on $\mathcal S$ (see \cite[§2]{VV:cd} for details on $cd$-structures).
We write $\mathrm{Spc}(\mathcal S)$ (resp.~$\mathrm{Spt}_s(\mathcal S)$) for the category of 
presheaves of simplicial sets 
(resp.~$S^1$-spectra) on $\mathcal S$.
In the setting of $\mathrm{Spc}(\mathcal S)$ there exist injective model structures as in \cite[Theorem 5.16]{Nordfjordeid}
\begin{equation}
\label{eq:Spc(B)}
\Spc_s(\mathcal S),\quad \Spc_\tau(\mathcal S),\quad \Spc^{\A^{1}}_\tau(\mathcal S).
\end{equation} 
The weak equivalences are schemewise weak equivalences, 
$\tau$-local equivalences,
and motivic $\tau$-local equivalences, 
respectively.
Similarly, 
in the $S^{1}$-stable setting, 
there are 
levelwise schemewise, 
levelwise $\tau$-local, 
and levelwise motivic $\tau$-local model structures \cite[§5.3]{Nordfjordeid}
\begin{equation}
\label{eq:Spts(B)}
\Spt_{s}(\mathcal S),\quad 
\Spt_{{s},\tau}(\mathcal S),\quad 
\Spt^{\A^{1}}_{{s},\tau}(\mathcal S),
\end{equation}
and the stable model structures on 
$\mathrm{Spt}_{s}(\mathcal S)$
\begin{equation}
\label{eq:Sptst(B)}
\Spt_\mathrm{st}(\mathcal S),\quad 
\Spt_{\mathrm{st},\tau}(\mathcal S),\quad 
\Spt^{\A^{1}}_{\mathrm{st},\tau}(\mathcal S).
\end{equation}
We denote the homotopy categories associated to the model structures in \eqref{eq:Spc(B)} and \eqref{eq:Sptst(B)} by:
\index{Homotopy categories!$\HHtriv(\mathcal S)$} 
\index{Homotopy categories!$\HH_\tau(\mathcal S)$} 
\index{Homotopy categories!$\HH^{\A^{1}}_\tau(\mathcal S)$} 
\index{Homotopy categories!$\SH_{s}(\mathcal S)$} 
\index{Homotopy categories!$\SH_{s,\tau}(\mathcal S)$} 
\index{Homotopy categories!$\SH^{\A^{1}}_{s,\tau}(\mathcal S)$} 
\begin{equation}\label{eq:hCatms}
\begin{array}{lll}
\HHtriv(\mathcal S), &\HH_\tau(\mathcal S), &\HH^{\A^{1}}_\tau(\mathcal S),\\
\SH_{s}(\mathcal S), &\SH_{s,\tau}(\mathcal S), &\SH^{\A^{1}}_{s,\tau}(\mathcal S).
\end{array}\end{equation}

\index{Localization functors!$L_\Nis$} 
\index{Localization functors!$\Lrep_\Nis$} 
In the $\tau$-local setting, 
we consider the $\tau$-localization functor, 
the $\tau$-localization endofunctor, 
and the $\tau$-fibrant replacement functor:
\begin{equation}\label{eq:notation:LLrepRep}\begin{array}{lll}
\Loc_{\tau}\colon \HH_{s}(\mathcal S)&\to& \HH_{\tau}(\mathcal S), \\
\Lrep_{\tau}\colon \HH_{s}(\mathcal S)&\to& \HH_{s}(\mathcal S),\\
\Rep_{\tau}\colon \Spc_s(\mathcal S)&\to&\Spc_s(\mathcal S).
\end{array} 
\end{equation}
We use similar notation for the corresponding functors on the $\tau$-localization on $\SH_{s}(\mathcal S)$ and
the levelwise $\tau$-fibrant replacement on $\Spt_{s,\tau}(\mathcal S)$.

\subsubsection{$\infty$-categories} 
We freely use the language of $\infty$-categories \cite{LurieHTT,LurieHA}.
If $\calC$ is an $\infty$-category, 
let $\cPre(\calC)=\Fun(\calC^\op,\cSpc_*)$ denote the $\infty$-category of presheaves on $\calC$ with values 
in the $\infty$-category of pointed anima or spaces $\cSpc_*$, which means pointed $\infty$-groupoids.
If $\calC=\Sm_B$, 
we abbreviate \index{$\infty$-categories!$\cPre(B)$} $\cPre(\Sm_B)$ to $\cPre(B)$. 
Similarly, 
we set \index{$\infty$-categories!$\cPrefr(B)$} $\cPrefr(B)=\cPre(\Corr^\fr(B))$.
For a Grothendieck topology $\tau$, 
we let \index{$\infty$-categories!$\cPre_\tau(B)$} $\cPre_\tau(B)$ denote the subcategory of $\cPre(B)$ spanned by $\tau$-sheaves. 
Finally, we let \index{$\infty$-categories!$\cPre_{\A^1}(B)$} $\cPre_{\A^{1}}(B)$ denote the subcategory of $\cPre(B)$ spanned by $\A^1$-invariant presheaves, and
\index{$\infty$-categories!$\cPre_{\A^1,\tau}(B)$} $\cPre_{\A^1,\tau}(B) = \cPreA(B)\cap \cPre_\tau(B)$. 

These $\infty$-categories are related through adjunctions and localization functors as in the diagram:
\[\begin{tikzcd}
\cPre(B) \arrow[r, pos=0.7,"\gamma^*", shift left=0.5ex]\arrow[rrd, pos=0.8,"L_{\A^{1}}"]\arrow[d,swap,"L_\tf"] & 
\cPrefr(B)\arrow[rrd, "L_{\A^{1}}"]\arrow[d, swap,pos=0.8,"L^\fr_\tf", crossing over]\arrow[l, pos=0.3,"\gamma_*", shift left=0.5ex] &
 &\\
\cPretf(B) \arrow[r, pos=0.7,"\gamma^*_\tf", shift left=0.5ex]\arrow[rrd, pos=0.8, "L_{\A^{1}}^\tf"]\arrow[d, "L^{\tf}_\nis"] & 
\cPrefrtf(B)\arrow[rrd, pos=0.7,"L_{\A^{1}}^{\tf,\fr}"]\arrow[d, crossing over]\arrow[l, pos=0.3,"\gamma_*", shift left=0.5ex] &
 \cPreA(B) \arrow[r,pos=0.2,"\gamma^*_{\A^{1}}", shift left=0.5ex]\arrow[d,crossing over] & 
 \cPrefrA(B)\arrow[d]\arrow[l,pos=0.5,"\gamma_*", shift left=0.5ex] \\
\cPrenis(B) \arrow[r, pos=0.7,"\gamma^*_\nis", shift left=0.5ex]\arrow[rrd, pos=0.7,"L_{\A^{1}}^{\nis}"] & 
\cPrefrnis(B)\arrow[l, pos=0.3,"\gamma_*", shift left=0.5ex]\arrow[rrd, pos=0.7,"L_{\A^{1}}^{\nis,\fr}"] &
 \cPreAtf(B) \arrow[r,pos=0.1, "\gamma^*_{\A^{1},\tf}", shift left=0.5ex]\arrow[d,crossing over] & \cPrefrAtf(B)\arrow[d]\arrow[l,pos=0.5, "\gamma_*", shift left=0.5ex]\\
 & &
 \cPreAnis(B) \arrow[r, pos=0.1,"\gamma^*_{\nis,\A^1}", shift left=0.5ex] & \cPrefrAnis(B)\arrow[l,pos=0.4, "\gamma_*", shift left=0.5ex]
\end{tikzcd}
\]
\index{Localization functors!$L_\Nis$} 
\index{Localization functors!$L^\tf_\Nis$} 
\index{Localization functors!$\Lrep_\Nis$} 
\index{Localization functors!$\Lrep^\tf_\Nis$} 
If $L^{d}_{a}\colon \mathcal P\to \mathcal P^\prime$ is a left adjoint functor as above, 
with right adjoint $R^{d}_{a}\colon \mathcal P^\prime\to \mathcal P$, 
then we write $\catLrep^{d}_{a}$ for the composite $R^{d}_{a}L^{d}_{a}\colon \mathcal P\to \mathcal P$.
In particular, 
$\catLrep_{\nis}$ is the composite \[\catLrep_{\nis}\colon \cPre(B)\to \cPrenis(B)\to \cPre(B).\]

We denote the corresponding homotopy categories by
\index{Homotopy categories!$\HH^{\A^{1}}_{*}(B)$} 
\index{Homotopy categories!$\HH^{\A^{1}}_{\tau,*}(B)$} 
\begin{equation}\label{eq:notation:homcatinftycatpointed}
\begin{array}{llll} 
\Spc_*(B) &=\mathrm{Ho}\cPre(B), & \Shvbf_{\tau,*}(B)&=\mathrm{Ho}\cPre_\tau(B), \\
\HH^{\A^{1}}_{*}(B)&=\mathrm{Ho}\cPre_{\A^{1}}(B),&   \HH^{\A^{1}}_{\tau,*}(B)&=\mathrm{Ho}\cPre_{\A^1,\tau}(B),
\end{array}
\end{equation}
and similarly in the framed setting (indicated by superscript $\fr$).
In contrast to the homotopy categories in the first row in \eqref{eq:hCatms}
the ones in \eqref{eq:notation:homcatinftycatpointed} are pointed. 
For any $\mathcal{S}\in \cPre(B)$, 
e.g., 
$\mathcal{S}=(\Gm,{1})$ or $\mathcal{S}=(\PP^{1},\infty)$,
we have the diagram
\[
\begin{tikzcd}
&\catSpt_{\mathcal{S}}(B)\arrow[d, shift left = 0.5ex, "L_{\mathcal{S}}"]&\\
\cPre(B)\arrow[r, shift left = 0.5ex, "\Sigma^\infty_{\mathcal{S}}"]\arrow[ru] & 
\PSpt_\mathcal{S}(B) 
\arrow[l, shift left = 0.5ex, "\Omega^\infty_{\mathcal{S}}"]
\arrow[u, shift left = 0.5ex, "{R}_{\mathcal{S}}"]\arrow[r] & \cPre(B)[\mathcal{S}^{-1}].
\end{tikzcd}
\]
Here \index{$\infty$-categories!$\PSpt(B)$} $\PSpt_\mathcal{S}(B)$ and \index{$\infty$-categories!$\catSpt(B)$} $\catSpt_\mathcal{S}(B)$ are the $\infty$-categories of 
$\mathcal S$-prespectra and $\mathcal S$-spectra in $\cPre(B)$, 
see \Cref{section:StablelocalizationDefinitionsNotations}.
The adjoint functors ${L}_{\mathcal{S}}$ and ${R}_{\mathcal{S}}$ are constructed in 
Section \ref{notat:PSptSptLSRS}.
For the composite functors we use the shorthand notation 
\[
\begin{array}{lcl}
\OmegaSigma_\mathcal{S} &=& \Omega^\infty_{\mathcal{S}}\Sigma^\infty_{\mathcal{S}} \colon \cPre(B)\to \cPre(B),\\
\Lrep_{\mathcal{S}} &=& R_{\mathcal{S}}L_{\mathcal{S}} \colon \PSpt_{\mathcal{S}}(B)\to \PSpt_{\mathcal{S}}(B). 
\end{array}
\]
The $\infty$-categories \cite[§2.2]{RobalnoncomKtheorybridge} of $S^1$-spectra, $(s,t)$-bispectra, $S^1\wedge \Gm$-spectra, 
and $\PP^{1}$-spectra are denoted%
\index{$\infty$-categories!$\cSpts(B)$}\index{$\infty$-categories!$\cSptst(B)$}
\index{$\infty$-categories!$\cSptswt(B)$}\index{$\infty$-categories!$\cSptP(B)$}
\[
\cSpts(B),\, \cSptst(B),\, \cSptswt(B),\, \cSptP(B).
\]
Moreover,
we write \index{$\infty$-categories!$\cSptst_{\A^1,\tau}(B)$} $\cSptst_{\A^1,\tau}(B)$ for $\A^1$-invariant $\tau$-local objects and $\SH_{\A^1,\tau}(B)$ 
for the stable motivic homotopy category of $B$ \cite[Definition 2.38]{RobalnoncomKtheorybridge} which is 
equivalent to the homotopy category
$\SH^{s,t}_{\A^1,\tau}(B):=\mathrm{Ho}\catSpt^{s,t}_{\A^1,\tau}(B)$.
Similar notation will be used for the suspension coordinates $S^1$, $S^1\wedge \Gm$ and $\PP^{1}$.
For example, 
there is an equivalence $\SH^{\A^{1}}_{s,\tau}(\mathcal S)\simeq\SH^{s}_{\A^1,\tau}(B)$, 
see \eqref{eq:hCatms}.

\printindex

\printbibliography
\end{document}